\documentclass[11pt]{article}
\usepackage{mymacros}

\usepackage{tikz} 
\usepackage{tikz-cd}
\usepackage{caption}
\usepackage{subcaption}

\numberwithin{equation}{section}

\makeatletter
\newcommand{\customlabel}[2]{%
   \protected@write \@auxout {}{\string \newlabel {#1}{{#2}{\thepage}{#2}{#1}{}} }%
   \hypertarget{#1}{#2}
}
\makeatother

\usepackage[titles]{tocloft}
\newcommand{\betaeff}{\beta_{\operatorname{eff}}}
\newcommand{\eff}{{\operatorname{eff}}}

\newcommand{\dagobs}{{\dagger}}
\newcommand{\dg}{{\operatorname{DG}}}
\newcommand{\Vil}{\operatorname{Vil}}

\newcommand{\str}{{\operatorname{str}}}

\newcommand{\dual}{{\divideontimes}}
\newcommand{\per}{\operatorname{per}}
\newcommand{\one}{\textnormal{\textbf{1}}}
\newcommand{\ry}{\operatorname{y}}
\newcommand{\rD}{\operatorname{D}}
\newcommand{\kn}{\mathfrak{n}}
\newcommand{\ke}{\mathfrak{e}}

\newcommand{\kM}{\mathfrak{M}}
\newcommand{\kV}{\mathfrak{V}}
\newcommand{\bulk}{{\o}}
\newcommand{\obs}{{\sigma}}

\newcommand{\ccE}{\mathcal{E}}
\newcommand{\frho}{\rho}

\newcommand{\bbe}{{\textnormal{\textbf{e}}}}
\newcommand{\bff}{{\textnormal{\textbf{f}}}}

\newcommand{\bi}{{\textnormal{\textbf{i}}}}
\newcommand{\bt}{{\textnormal{\textbf{t}}}}
\newcommand{\bw}{{\textnormal{\textbf{w}}}}
\newcommand{\bx}{{\textnormal{\textbf{x}}}}
\newcommand{\by}{{\textnormal{\textbf{y}}}}
\newcommand{\bz}{{\textnormal{\textbf{z}}}}
\newcommand{\bxy}{{\textnormal{\textbf{xy}}}}

\newcommand{\bA}{{\textnormal{\textbf{A}}}}

\newcommand{\bC}{{\textnormal{\textbf{C}}}}
\newcommand{\bD}{{\textnormal{\textbf{D}}}}
\newcommand{\bE}{{\textnormal{\textbf{E}}}}
\newcommand{\bI}{{\textnormal{\textbf{I}}}}
\newcommand{\bL}{{\textnormal{\textbf{L}}}}

\newcommand{\bSeg}{{\bf\Sigma}}
\newcommand{\bT}{{\textnormal{\textbf{T}}}}
\newcommand{\bP}{{\textnormal{\textbf{P}}}}
\newcommand{\pt}{{\rm pt}}

\newcommand{\lseg}{{\mapsto}}
\renewcommand{\supp}{{\operatorname{supp}}}

\newcommand{\Loc}{{\operatorname{Loc}}}

\newcommand{\Tay}{\operatorname{Tay}}
\newcommand{\PM}{\mathcal{K}}
\newcommand{\LPM}{\mathcal{L K}}
\newcommand{\LK}{\operatorname{LK}}

\newcommand{\RPM}{\mathcal{R K}}
\newcommand{\RK}{\operatorname{R K}}
\newcommand{\Rae}{J}

\newcommand{\etabulk}{{\eta^\bulk}}

\newcommand{\Seg}{\Sigma}

\newcommand{\hash}{{\#}}

\newcommand{\Adag}{{A/4}}

\newcommand{\Conn}{\operatorname{Conn}} 
\newcommand{\Comp}{\operatorname{Comp}} 

\usepackage{scalerel} 

\usetikzlibrary{shapes.geometric, calc}
\usetikzlibrary{decorations.markings, arrows.meta, calc}

\title{Two dimensional Villain model in the BKT phase}

\author{
Jiwoon Park\footnote{jp711@cantab.ac.uk} 
\;and
Wei Wu\footnote{wei.wu@nyu.edu}
}

\date{\vspace*{-2em}} 

\usepackage[titles]{tocloft}
\newcommand{\cu}{\square}
\renewcommand{\tilde}[1]{\widetilde{#1}}
\renewcommand{\hat}[1]{\widehat{#1}}

\begin{document}
\iftrue

\maketitle

\abstract{
We establish sharp two-point function asymptotics for the two dimensional Villain model at sufficiently low temperatures, confirming a central prediction in the Berezinskii-Kosterlitz-Thouless theory. 
We characterize the power law exponent through the renormalisation group flow of a marginal line defect observable in the dual high-temperature Discrete Gaussian model.
}

\tableofcontents

\section{Introduction and Main Results}

In this paper, we study the Villain model \cite{villain:jpa-00208289}, a classical statistical mechanics model with two-component
spins, in two dimensions. For $M \in \N$, we consider the finite box $\cu_M := \left\{ - M , \ldots, M \right\}^2 \subset \Z^2$. We denote by $\partial\cu_M$ its boundary, 
and by $E(\cu_M)$ its edge set.
The Villain model at inverse temperature $\beta \in (0 , \infty)$ on the box $\cu_M$ with free boundary condition is given by the following Gibbs measure on the set of functions $\Omega_{M} := \left\{ \theta : \cu_M \to [0 , 2\pi) \right\} \simeq [0 , 2\pi)^{\cu_M}$
\begin{equation*}
    \P^{\Vil}_{\beta, \square_M} ( d\theta ) 
	:= \frac{1}{Z^{\Vil}_{\beta, \square_M}} \prod_{\{x,y\} \in E(\square_M)} \sum_{m \in 2\pi\Z} e^{-\frac{\beta}{2} (\theta_x - \theta_y - m)^2} \prod_{x\in \square_M} d\theta_x, 
\end{equation*}
where $Z^{\Vil}_{\beta, \square_M}$ is a normalization constant. The expectation with respect to the Gibbs measure $\P^{\Vil}_{\beta, \square_M}$ is denoted by $\E^{\Vil}_{\beta, \square_M}$, and the two point correlation function of this model is given by $\E_{\beta, \square_M}^{\Vil} [ \exp( i (\theta_x - \theta_y)  ) ]$. The Villain model is a fundamental example of the \textit{Abelian spin models}, and is believed to share the same large-scale properties of other models in this class, such as  the classical XY model. 

It is known that, as a consequence of the Ginibre inequality \cite{MR0269252} (see also \cite{BFL1977, MMSP1978} ), as $M \to \infty$, there exists a unique infinite volume translation invariant Gibbs measure on $\Z^2$.  We denote by $\P_{\beta, \Z^2}^{\Vil}$ the corresponding infinite-volume Gibbs measure and by $\E_{\beta, \Z^2}^{\Vil}$ its expectation.

The two-dimensional Villain model has attracted significant attention as it undergoes a 
special kind of phase transition. It is proved by Mermin and Wagner that in $d=2$ there is no spontaneous magnetization at any positive temperature \cite{PhysRevLett.17.1133}. In their seminal work, Kosterlitz and Thouless \cite{0022-3719-6-7-010} (see also Berezinskii \cite{MR0314399}) proposed a mechanism for the phase transition of the XY and Villain models based on the vortex-antivortex pairs.  The (non-rigorous) renormalization group analysis predicts a critical temperature separating a low temperature phase with power law decay of correlations from a high temperature phase with exponential decay. The Berezinskii-Kosterlitz-Thouless (BKT) transition has since become a fundamental paradigm of the two dimensional critical phenomena.  Its relation to Coulomb gases and roughening transitions of height models has been developed through duality and renormalisation group arguments \cite{CW,JKKN}. Related (non-rigorous) field-theoretic descriptions explain the BKT transition in the six-vertex model \cite{dN} and provide predictions for the critical behavior of loop $O(2)$ models \cite{Nie}.

Based on a non-rigorous renormalization group (RG) argument,  Kosterlitz and Thouless \cite{0022-3719-6-7-010, kos} introduced the famous Kosterlitz-Thouless RG flow, which leads to the following \textit{conjecture} of the two point function at low temperature 
\begin{equation}\label{e.KT}
    \E^{XY/\Vil}_{\beta, \Z^2} \big[ \exp ( i (\theta_\bx - \theta_\by)) \big] \sim \norm{\bx-\by}^{-\frac{1}{2\pi \beta_{\operatorname{eff}}  }+o(1) }.
\end{equation}
 Here, $\beta_{\operatorname{eff}}$ is the \textit{effective temperature}, which differs from $\beta$ and depends on the coarse-graining of the vortex-antivortex interactions in the \textit{dipole phase}. Moreover, the limiting two point function is conjectured to be \textit{scaling and rotation invariant}.

 Rigorous polynomial upper bounds on the correlation function were obtained by McBryan and Spencer \cite{MR0441179}, who proved 
 \begin{equation}
    \E^{XY/\Vil}_{\beta, \Z^2} \big[ \exp ( i (\theta_\bx - \theta_\by)) \big] \le C \norm{\bx-\by}^{-\frac{1}{2\pi \beta}+o(1)}.
\end{equation}
An important  step towards the mathematical understanding of the BKT transition was made by Fr\"ohlich and Spencer \cite{MR634447}. 
Combining the the spin-wave and spin-vortex decomposition,  with a multiscale analysis argument, they established a lower bound on the two-point functions at sufficiently low temperatures:
\begin{equation}
    \E^{XY/\Vil}_{\beta, \Z^2} \big[ \exp ( i (\theta_\bx - \theta_\by)) \big] \ge C \norm{\bx-\by}^{-\frac{1}{2\pi c(\beta) }+o(1)},
\end{equation}
where $c(\beta)\neq \beta$ and satisfies $c(\beta)= \beta(1+o(1))$ as $\beta\to \infty$. Together with the exponential decay at high temperatures, this rigourously established that the XY and Villain models undergo a BKT transition in two dimensions. 
The spin-vortex decomposition picture was also developed in the work of Garban and Sep{\'u}lveda \cite{garban2023quantitative} to obtain an improved upper bound on the two point functions of 2D Villain model. 
 
 Conjecture \eqref{e.KT} (also known as the spin wave conjecture) remained open for the XY and Villain models (see discussions in the related works below).  In this paper we resolve the conjecture for two dimensional Villain models, and derive precise asymptotics for its correlation functions, stated below. 

\begin{theorem} \label{thm:main}
For any $\kn \in \Z \backslash \{0\}$, there exists $\beta_0 (\kn) >0$ such that, whenever $\beta \ge \beta_0 (\kn)$,
\begin{align}
    \E^{\Vil}_{\beta, \Z^2} \big[ \exp ( i \kn(\theta_\bx - \theta_\by)) \big] = C \norm{\bx-\by}^{-\frac{\kn^2}{2\pi \beta_{\operatorname{eff}} } } \big(1 + \psi (\bx-\by) \big)
\end{align}
for some constants $C \equiv C (\beta, \kn) >0$, $\beta_{\operatorname{eff}} \equiv \beta_{\operatorname{eff}} (\kn) = \beta + O_{\kn} (e^{-c\beta})$ 
and some function $\psi$ such that  $\psi (\bx) = O(1/\log \norm{\bx})$ as $\norm{\bx} \rightarrow \infty$.
\end{theorem}

In particular, the asymptotic correlation function is scaling and rotation invariant. A detailed statement is given in Corollary~\ref{cor:Vil2ptfnc} of Section~\ref{sec:partIIOutline}, with the proof given in the subsequent sections.

While polishing the final version of this paper, we became aware of the recent posting by OpenAI \cite[Theorem~1.2 and Corollary~5.9]{BKTHeightSpin2026}. For the Villain model at sufficiently low temperature, they prove that in a box of side length $n$, the spin correlation between two points at distance of order $n$ decays as a power of $n$. Since this estimate concerns pairs of points with macroscopic separation, Theorem~\ref{thm:main} is \emph{not} implied by their result. The same preprint also contains analogous results for the XY model at sufficiently low temperature.
The core of our argument is a new renormalisation group analysis of the marginal line observable (see Section~\ref{sec:obsRGmap}), which is not presented in \cite{BKTHeightSpin2026}.

Also, we expect that the methods developed in this paper can be extended to multi-point correlation functions and the full scaling limit of the spin field. Specifically, for the two dimensional Villain model at low temperatures and any fixed collection of charges $\kn_j\in\Z\setminus\{0\}$ satisfying the neutrality condition $\sum_j\kn_j=0$, we conjecture that
$$
 \E^{\Vil}_{\beta, \Z^2} \big[ \prod_j\exp ( i \kn_j(\theta_{\bx_j})) \big]  
 \sim 
 C(\beta,\kn) \prod_{i<j} \norm{\bx_i - \bx_j}^{ -\frac{\kn_i \kn_j}{2\pi \beta_{\operatorname{eff}}}},
$$
where $C(\beta,\kn)$ is independent of the insertion points. 
We plan to address these questions in future work.

\subsection{Related works}

The two dimensional XY and Villain models have attracted renewed interest in recent years. In particular, a new proof of the BKT transition was given in \cite{van2023elementary}, established a polynomial lower bound of the correlation of 2D XY model at and below the critical temperature; see also \cite{aizenman2021depinning} for a related argument. Both proofs exploit the spin-height duality and a result \cite{lammersdeloc} for the delocalisation of the dual height functions. 

We outline the duality between the two dimensional Villain model and the Discrete Gaussian model in Section~\ref{sec:funcObs}. This duality can be explored further to characterize the critical point of the XY model in terms of dual height functions \cite{lammers2023bijecting}.  More results for the dual height functions were obtained in the work of \cite{MR2251117, lammers2024delocalisation, arXiv.2211.14365, vDEL}.

Rigorous renormalisation group (RG) analysis for the discrete Gaussian model has been developed by Bauerschmidt, Park and Rodriguez \cite{dgauss1,dgauss2}. They proved that the Discrete Gaussian model at sufficiently high temperature has its scaling limit as a multiple of the Gaussian free field. By duality, this regime corresponds to the low temperature Villain model considered here. The argument has been carried out further to obtain asymptotics of the multipoint functions of the Discrete Gaussian model \cite{dgauss3}. We will explain how our results and methods relate to and differ from those of these works.

In dimensions $d \ge 3$, the XY and the Villain model exhibit long range order at low temperature \cite{MR0421531}. There, the truncated spin correlation is expected to display massless free field behavior, decaying as $|x|^{2-d}$, as predicted by the Goldstone Theorem. Precise asymptotics for the truncated two-point function were established by Dario and Wu in \cite{dario2024massless}, and further structural properties of the massless phase were obtained in \cite{dario2025parallel}, confirming the conjectures in \cite{BFLLS}. Both works rely on the Helffer--Sjöstrand representation, PDE regularity theory, and quantitative stochastic homogenization \cite{AKM}.

In the context of 2D Coulomb gas with small activity, Conjecture \eqref{e.KT} has been proved using RG analysis in the work of Falco \cite{MR2917175, 1311.2237} (see also \cite{MR1777310} for the rigorous free energy RG flow away from the BKT transition point). Both works were developed from the general RG approach of Brydges and Yau \cite{MR1048698} (see also \cite{MR2523458}). The RG analysis in these works requires a sufficiently small activity and therefore do not apply directly  to the study of the XY and Villain models, which corresponds to a neutral Coulomb gas with activity of order 1. The canonical ensemble of the two-dimensional Coulomb gas below the Kosterlitz-Thouless transition point was recently studied rigorously in \cite{boursier2025multipole}. In a related direction, Bauerschmidt and Bodineau \cite{MR4303014} used the Polchinski flow to establish logarithmic Sobolev inequalities and spectral gaps, for the massive continuum sine-Gordon model for $\beta$ up to $6\pi$.

For interacting two-dimensional lattice models, the rigorous determination of critical exponents beyond exactly solvable cases remains notoriously difficult. Important examples include constructive fermionic RG analyses of weakly non-integrable Ising, Ashkin–Teller,  and interacting dimer models \cite{MR1797142, giuliani2020non, giuliani2005anomalous}, and rigorous relations between anomalous correlation exponents \cite{MR3637384}. On the bosonic side, recent progress also include the GFF scaling limit of the six-vertex model \cite{duminil2026gaussian}. The present work shares the similar spirit by identifying a continuously varying family of spin-correlation exponents in the scaling limits.

\subsection{Overview of the method}

This paper is organized into three components. After reviewing the spin-height duality in Section~\ref{sec:funcObs}, Part I (Section~\ref{sec:RGcoordSpace}--\ref{sec:fiarotp}) develops and analyses the renormalization group (RG) map for the Discrete Gaussian (DG) model. In Part II (Section~\ref{sec:partIIOutline}--\ref{sec:sqcupPaths}), we use geometric arguments to extract precise correlation function asymptotics from the RG estimates.
We explain each component in detail below.

In Section~\ref{sec:funcObs},
we begin with the duality relation between the Discrete Gaussian model and the Villain model in two-dimensions. This duality goes back to the work of \cite{MR634447,JKKN}, and we give the details in Section~\ref{sec:dualityVilDG}, formulate using discrete differential forms. Discrete differential forms will also be useful in Part II, where  the polynomial decay of the correlation function is encoded in the geometry of the dual lattice.  In the rest of Section~\ref{sec:funcObs}, we prepare for the RG analysis of the DG model. This involves some standard constructions from \cite{dgauss1,dgauss2,dgauss3}. 

In Part I, we construct and analyse the RG on the DG model in the presence of a line defect along a one-dimensional path connecting $\bx$ and $\by$. This line defect arises from the duality transformation of the two-point correlation function of the Villain model.
Under successive coarse-graining operations, the associated observable generates an \emph{enhanced line observable} $f_j$ and a \emph{local line energy} $\ke_j (\Lambda)$. This yields a multi-scale expansion of the tension of the line defect,
\begin{align}
    \E^{\Vil}_{\beta',\Lambda_N} [e^{i\kn (\theta_{\bx} - \theta_{\by})}] \sim \exp\Big(\sum_{j=0}^{N-1} \Big( \frac{1}{2} ( d^* f_j , \Gamma_{j+1}  d^* f_j ) + \ke_j (\Lambda) \Big) \Big)  \quad \text{as} \;\; N\rightarrow \infty ,
    \label{eq:VilRGconclRes}
\end{align}
see \eqref{eq:VilRGconcl} for the precise statement and Section~\ref{sec:funcObs} for the definition of $d^*$ and $\Gamma_{j+1}$.

However, since the line observable is \emph{marginal} under coarse-graining operations and does not contract, its analysis requires an extension of the bulk RG framework developed in \cite{dgauss1}.
The bulk RG map and RG flow of the background (bulk) free energy of \cite{dgauss1} are recalled in Section~\ref{sec:RGcoordSpace}--\ref{sec:bulkRGflow}. The central idea of the line observable RG is to apply the Cameron--Martin change of variables at each scale of the finite-range decomposition. This extracts a quadratic energy increment and produces a scale-dependent field shift determined by the coarse-grained line defect. 
This argument is described in Section~\ref{sec:domObsRG}--\ref{sec:fiarotp}.

Part II establishes the power law decay exponent of the correlation function of the Villain model from \eqref{eq:VilRGconclRes}.
The crucial observation is to divide the quadratic energy of the enhanced line observable into a spin wave/Coulomb energy on the dual lattice and a local energy on the original lattice. The interaction of spin vortices grows logarithmically in the distance between the two endpoints of the line observable, $\bx$ and $\by$.
Its analysis involves a discrete differential calculus of 2-forms on the dual lattice, and the logarithmic energy ultimately give rise to the polynomial decay of the correlation function. The analysis of the spin wave energy is presented in Section~\ref{sec:enhLineObs}.

The local line energy grows proportional to the length of the line observable. It dominates the spin wave interaction, but the linear growth is expected to cancel with the exponential term in the duality transformation (see \eqref{eq:order-to-disorder_correlator}). Formulation of this cancellation first appears in Section~\ref{sec:contractible}, where we prove that the normalized free energy vanishes for every contractible loop.  In Section~\ref{sec:straightlineenergy} we prove that for the straight line observable, the local line energy, after cancellation of its leading order, has a constant order contribution with a quantification of error. This is extended to general rectangular line observable in \ref{sec:rectangularpath}.
Section~\ref{sec:straighlinesegment} proves that the magnitude of the spin wave energy is independent of the path configuration. When $\bx$ and $\by$ are on the same coordinate axis, Theorem \ref{thm:main} can be concluded after Section \ref{sec:straighlinesegment}; for general $\bx$ and $\by$, the proof of Theorem \ref{thm:main} is given at the end of the paper. 

Part II can be read independently from Part I, once \eqref{eq:VilRGconclRes} is established.

\subsection{Relation to the RG for the Discrete Gaussian model}

The analysis of this paper builds upon the (bulk) renormalisation group framework developed by Bauerschmidt, Park and Rodriguez in \cite{dgauss1,dgauss2}, which establishes the macroscopic GFF scaling limit for the Discrete Gaussian model, and the analysis of local observables in \cite{dgauss3}. The present paper addresses a different problem: the Villain spin correlations, after the duality transform, is given by an exponential line observables in the Discrete Gaussian model:
\begin{equation}\label{e.path}
    e^{-\frac{2\pi^2 \kn^2}{\beta} |\bP| } \E^{\dg, 0}_{\beta, \square_M} \big[ \exp \big(- 2\pi \beta^{-1} \kn (\sigma, d \one_\bP )  \big) \big],
\end{equation}
where $\bP$ is a path connecting $\bx$ and $\by$. The line observable has a support that grows with $|\bx-\by|$, and therefore its asymptotic is not a consequence of the GFF scaling limit. 

An essential difficulty is that the line observable remains sensitive to local interactions of the field. Indeed, simply replacing the Discrete Gaussian expectation in \eqref{e.path} by the GFF expectation with an effective stiffness produces a term that grows exponentially with the path length, rather than the expected power law decay. Therefore a separate RG analysis for the line defect is needed, both to identify its effective coupling and to track the local line energy along the path. This analysis is further complicated by the marginality of the observable, which does not contract under coarse-graining.

Sections~\ref{sec:RGcoordSpace}--\ref{sec:bulkRGflow} review the bulk RG map and the bulk RG flow from \cite{dgauss1}, with a more general formulation of the progression map in Section~\ref{sec:PM} to accommodate the line observables. This bulk RG flow serves as a reference trajectory for the analysis for the rest of the paper. 

The new observable RG analysis is developed in Sections~\ref{sec:domObsRG}--\ref{sec:fiarotp}. The marginally linear term is incorporated into an enhanced line observable along the RG flow. Applying the Cameron--Martin theorem scale by scale then converts its increment into Gaussian shifts and quadratic energy contributions. The observable RG flow tracks both the enhanced line observable and the local line energy.

The analysis of the observable RG flow adapts several tools from the earlier analyses: the flow dependent regulators and norms in Section~\ref{sec:flowdepNorm} build on those of \cite{dgauss2}, and the reblocking procedure in Section~\ref{sec:reblPert} is inspired by the same work. The additional task here is to make these estimates compatible with the geometry of the defect and with the dependence of the Gaussian shifts along the RG flow.

Part II derives precise correlation asymptotics based on the output of the RG flow \eqref{eq:VilRGconclRes}. This part is deterministic and is independent of the RG analysis. The split of the quadratic energy of the enhanced line observable into a Coulomb part on the dual lattice, and a local energy on the original lattice is original, and is independently of the earlier RG works.

\subsection{Basic Notations}
\label{sec:basicNots}

Given a finite set $X$ and a covariance matrix $C : X \times X \rightarrow \R$, we denote $\E_C$ for the Gaussian expectation with covariance $C$. We usually use $\zeta \sim \cN (0,C)$ for the Gaussian random field and denote $\E^{\zeta}$ when performing a conditional expectation just in the variable $\zeta$.
When $F : \R^X \rightarrow \R$ and $\varphi \in \R^X$, we denote $\theta F(\varphi) = F(\varphi + \zeta)$ when the covariance is clear from the context.

We denote $f \lesssim g$ or $f = O(g)$ if there exists a constant $C$ such that $f \le C g$. If $C$ depends on the parameter $L$, then we make it explicit by denoting $f \lesssim_L g$ or $f = O_L (g)$.

In \eqref{sec:covDecomp}, we will introduce a collection of covariances $(\Gamma_j)_{j\ge 1}$. We will also often abbreviate
\begin{align}
    \E_{j} = \E_{\Gamma_j} , \qquad \E_{\le j} = \E_{\Gamma_1 + \cdots +\Gamma_j}
\end{align}

\subsubsection{Block structure on $\Lambda$}
\label{sec:blockLambda}

The basic unit of our RG analysis is the functions on blocks and sets formed from blocks. They are defined as follows. If $\Lambda$ is either $\Z^2$ or $\Lambda_N$ and $j\in \N$, let $\cB_j$ be the partition of $\Lambda$ by $L^j \Z^2$-translations of $[0,L^{j}-1]^2 \subset \Lambda$. They form the $j$-blocks. A $j$-polymer is an arbitrary union of $j$-blocks, and the set of $j$-polymer is denoted $\cP_j$. For any $X \in \cP_j$, we let $|X|_j = L^{-2j} |X|$.
If $Z \subset \Lambda$, we use notation $\cB_j (Z)$ and $\cP_j (Z)$ for $j$-blocks and $j$-polymers, respectively, that are subsets of $Z$.

A set $Z$ is connected if there is a path $(t_1, \cdots, t_n)$ between two arbitrary points of $Z$ that satisfies $\dist_\infty (t_i ,t_{i+1}) =1$. 
We denote $Z \sim Z'$ if two sets are connected. The set of connected $j$-polymers (excluding empty sets) is $\Conn_j$, and the set of connected components of $X \in \cP_j$ is $\Comp (X)$. When $Y \in \Conn_j$ and $|Y|_{j} \le 4$, then $Y$ is called \emph{small}, and the set of small $j$-polymers is denoted $\cS_j$.
For any $Z \in \cP_j$, we let $Z^* = \cup_{Y \in \cS_j}^{Y\cap Z \neq \emptyset} Y$. If $X \in \cP_j$, $\bar{X}$ is the smallest $j+1$-polymer containing $X$.

\subsubsection{Choice of parameters}
\label{sec:choicePar}

We introduce the following parameters for the RG analysis. 
Parameters $c_h$, $h$, $\kappa_0$, $c_2$ are chosen the same as in \cite{dgauss1}.
\begin{itemize}
\item Throughout the paper we denote $\beta'$ the temperature of the Villain model in Theorem~\ref{thm:main}, and $\beta = 4\pi^2 /\beta'$. $\beta$ is the temperature of the Discrete Gaussian model dual to the Villain model. 
We denote $\alpha = \sqrt{\beta}$, and $h = c_h r \alpha$ for a sufficiently large constant $c_h$, independent of any other parameters, and some $r>0$ to be chosen sufficiently small. The role of $r>0$ can also be observed in Lemma~\ref{lemma:DVbnd}. $h$ will be the radius of convergence of the Taylor expansion of effective potentials.
\item $\kappa_0 = c_\kappa / \log L$ for a sufficiently small constant $c_{\kappa}$ and a sufficiently large parameter $L$. It is the exponent of the large field regulator, see \eqref{eq:GjXphiDefi}.
It is also used to define the strong large field regulator, as in \eqref{eq:GstrXvphiDefi}, but with $c_w$, an extra small coefficient.
$c_{w}$ is also used in \eqref{eq:Gjdaggerdefi}, for the observable large field regulator.

In \eqref{eq:GjXphiDefi}, there is an additional constant $c_2$ for the boundary terms, which is a small constant.

\item $A$ is a large $L$-dependent constant. 

\item We take $\etabulk = 0.9999$ and $\eta= 0.999$.
\end{itemize}

These parameters are fixed after Assumption~\ref{asmp:critical}, so we will not often make these choices explicit from Section~\ref{sec:domObsRG}.

Index $\bulk$ is used to label objects related to the bulk RG, while $\sigma$, $\dagger$ and $\ddagger$ are used to label objects related to the observable RG.

\section{Functional integral form of the spin observables}
\label{sec:funcObs}

\subsection{Models}

We consider spin models on either $\square_M = [-\frac{M}{2}, \frac{M-1}{2}]^2 \cap \Z^2$ (with free boundary) or $\Lambda_N = [-\frac{L^N }{2} , \frac{L^N - 1}{2}] \cap \Z^2$ (with periodic boundary) or their infinite volume limits on $\Z^2$. 
$M,N,L \in \Z_{\ge 2}$ are integers, and the indicated boundaries determine the graph structure on $\square_M$ and $\Lambda_N$.

\subsubsection{Discrete differential forms}
\label{sec:discDiffForm}

Let $\square$ be either $\square_M$ or $\Lambda_N$, or more generally a graph that can locally be embedded in $\Z^2$. 
Then the vertex set, edge set and the face can be defined and denote $V(\square),  E (\square)$ and $F (\square)$, respectively.
The oriented version of the vertex set is $\vec{V}(\square) = V(\square) \cup V^{-1} (\square)$,  where $V^{-1} (\square)$ is the copy of $V(\square)$, whose elements are denoted $x^{-1}$ for each $x\in V(\square)$.
The corresponding objects on the edge set and the face set are defined and denoted $\vec{E} (\square)$ and $\vec{F} (\square)$, respectively. 
They are defined so that they can be represented as subsets of $V^{-1} (\square) \times V (\square)$ and $E^{-1} (\square) \times E^{-1} (\square) \times E(\square) \times E (\square)$, respectively.
For example,  if $e \in \vec{E} (\square)$ is represented by $\{x^{-1},y\}$, the starting and the ending points of $e$, then $e^{-1}$ just means $\{ x ,  y^{-1} \}$, 
and if $f\in \vec{F} (\square)$ is represented by $\{e_1^{-1},e_2^{-1}, e_3, e_4\} \in F(\square)$, then $f^{-1} = \{e_1,e_2, e^{-1}_3, e^{-1}_4\}$.
The set of $\F$-valued functions (for either $\F=\Z$ or $\R$) on $\vec{V}(\square)$ such that $g (x^{-1}) = -g(x)$ is denoted $\square^{(0)} (\F)$. 
Analogous functions on $\vec{E} (\square)$ and $\vec{F} (\square)$ are denoted $\square^{(1)} (\F)$ and $\square^{(2)} (\F)$, respectively. 
The exterior derivatives ($d$) and the exterior coderivatives ($d^*$) are functions
$d : \square^{(i)} (\F) \rightarrow \square^{(i+1)} (\F)$ ($i= 0,1$) and
$d^* : \square^{(i)} (\F) \rightarrow \square^{(i-1)} (\F)$ ($i=1,2$)
defined by
\begin{align}
	d g_1 (e) = \sum_{x \in \vec{V} (\square) : x \in e } g_1(x)
	, \quad 
	d g_2 (f) = \sum_{e \in \vec{E} (\square) : e \in f} g_2 (e),  \\
	d^* g_3 (e) = \sum_{f \in \vec{F} (\square) : f \ni e } g_3 (f)
	, \quad
	d^* g_4 (v) = \sum_{e \in \vec{E} (\square) : e \ni v} g_4 (e)
\end{align}
for $g_1 \in \square^{(0)} (\F)$, $g_2, g_4 \in \square^{(1)} (\F)$ and $g_3 \in \square^{(2)} (\F)$.
For $X \subset \vec{V} (\square)$, $\one_X \in \square^{(0)} (\F)$ denotes the function $\one_X (x)=1$ if $x \in X$ and $\one_X (x) = 0$ if $x, x^{-1} \not\in X$. Similar notations are used for $X \subset \vec{E} (\square)$ or $\vec{F} (\square)$.

Denote by $\square^{\dual}$ the dual graph of $\square$, namely the graph that can locally be embedded in $\Z^2$ and has $F(\square)$ and $V(\square)$ as its vertex set and face set, respectively. 
Moreover, if $\square$ is $\Z^2$ or has periodic graph structure, 
$V(\square)$ and $F(\square)$ have the same cardinality, and $E(\square)$ and $E (\square^\dual )$ have the same cardinality via the map
\begin{align}
    \dual : E(\square) \rightarrow E (\square^\dual ) , \quad  e \mapsto e^\dual = \frac{\pi}{2} \text{-counterclockwise rotation of } e .
\end{align}
Also, for a function $f \in \square^{(1)} (\F)$, we can find a dual function $f^\dual \in ( \square^\dual )^{(1)} (\F)$ by
\begin{align}
    f^{\dual} (e^\dual) = f (e ) , \qquad e \in E(\square) .
\end{align}

When we talk about the Dirichlet boundary condition of a spin model defined on $\square_M^{\dual}$, it is useful to define an extended system $\square_M^+$
defined by $V(\square_{M}^+) = V(\square_{M})$ but $E(\square_M^+) = \{e \in E(\square_{M+2}) : e \cap \square_M \neq \emptyset \}$ and $F(\square_M^+) = \{f \in F(\square_{M+2}) : f \cap \square_M \neq \emptyset \}$.

\subsubsection{The Villain XY model}

On $\square_M$,  two vertices $x, y$ are neighbors if $\{ x^{-1} , y\} \in \vec{E} (\square_M)$ and we denote $x \sim y$.
The Villain XY model (or Villain model) with free boundary condition and inverse temperature $\beta$ is given by
\begin{align}
	\P^{\Vil}_{\beta, \square_M} ( d\theta ) 
	&= \frac{1}{Z^{\Vil}_{\beta, \square_M}} \prod_{\{x,y\} \in E(\square_M)} \sum_{m \in 2\pi\Z} e^{-\frac{\beta}{2} (\theta_x - \theta_y - m)^2} \prod_{x\in \square_M} d\theta_x, 
	\\
	Z^{\Vil}_{\beta, \square_M} 
	&= \int_{[0, 2\pi)^{\square_M}}  \prod_{\{x,y\} \in E(\square_M)} \sum_{m \in 2\pi\Z} e^{-\frac{\beta}{2} (\theta_x - \theta_y - m)^2} \prod_{x\in \square_M} d\theta_x
\end{align}
on $\theta \in \Omega^{\Vil}_{\square_M} = \{ \theta_x \in [0, 2\pi) : x \in \square_M  \}$, where $d\theta_x$ is the Lebesgue measure on $[0,2\pi)$. 
The function $(\theta_x, \theta_y) \mapsto \sum_{m\in \Z} e^{-\frac{\beta}{2} (\theta_x - \theta_y - m)^2}$ has a natural continuous periodic extension to $\R \times \R$, so $\theta$ can be considered as the angle of $\mathbb{S}^1$-valued spin configuration on $\square_M$.
Here, $\mathbb{S}^1$ is the unit circle in $\R^2$.
A natural, and also extensively studied observable of this model is the infinite volume two-point correlation function given by
\begin{align}
	\E_{\beta, \Z^2}^{\Vil} \big[ \exp( i (\theta_x - \theta_y)  )\big]
	:= \lim_{M\rightarrow \infty} \E_{\beta, \square_M}^{\Vil} \big[ \exp( i (\theta_x - \theta_y)  ) \big]
	.
	\label{eq:Villain_model_infinite_volume_limit}
\end{align}
In fact, $\E_{\beta, \Z^2}^{\Vil} $ is a translation invariant Gibbs measure on $\Z^2$, 
whose existence and uniqueness are proved by \cite{BFL1977, MMSP1978} (the original proofs are given for the two-dimensional XY model, but they rely on the Ginibre inequality \cite{MR0269252} and can be extended to the Villain model using the fact that the Villain model can be represented as a limit of XY model on the cable graph of $\Z^2$).

\subsubsection{The Discrete Gaussian model}

The Discrete Gaussian (DG) model is defined on $(\square^+_M)^{\dual}$ (the dual of $\square_M^+$) with 0-Dirichlet boundary condition and on $\Lambda_N$ with periodic boundary condition. 
On $( \square^+_M )^{\dual}$, the state space is $\Omega^{\dg}_{\square_M} = \{ \sigma \in 2\pi (\square_M^+ )^{(2)} ( \Z) : \sigma \equiv 0 \text{ on } F(\square^+_M) \backslash F(\square_M)  \}$ and the DG model with temperature $\beta$ and 0 boundary condition is given by
\begin{align}
	\P^{\dg, 0}_{\beta, \square_M} (\sigma) = \frac{1}{Z^{\dg, 0}_{\beta}} e^{-\frac{1}{2\beta} |d^* \sigma|^2}, 
	\qquad 
	Z^{\dg, 0}_{\beta, \square_M} = \sum_{\sigma \in \Omega^{\dg}_{\square_M} } e^{-\frac{1}{2\beta} |d^* \sigma|^2}
	.
	\label{eq:DG_with_Dirichlet_boundary}
\end{align}
Notice that $\beta$ is the temperature of the model, not the inverse temperature.  
On $\Lambda_N$, we endow a periodic structure and fix a special point $0\in \Lambda_N$.
Then the DG model with temperature $\beta$ and periodic boundary condition is given by the probability measure
\begin{align}
	\P^{\dg, \per}_{\beta, \Lambda_N} (\sigma) = \frac{1}{Z^{\dg, \per}_{\beta, \Lambda_N}} e^{-\frac{1}{2\beta} |d^* \sigma|^2}, 
	\qquad
	Z^{\dg, \per}_{\beta, \Lambda_N} = \sum_{\sigma \in \Omega^{\dg}_{\Lambda_N}} e^{-\frac{1}{2\beta} |d^* \sigma|^2 }
\end{align}
on $\sigma \in \Omega^{\dg}_{\Lambda_N} = \{ \sigma \in (2\pi\Z)^{\Lambda_N} : \sigma_0 = 0  \}$.
We note that 
$|d^* \sigma|^2 = \sum_{\{x,y\}: x\sim y} |\sigma_x - \sigma_y|^2$, so the notation is consistent with \cite{dgauss1}.
From these expressions, 
the DG model on $\Lambda_N$ can be formulated without the language of discrete differential forms. We therefore use this simplified definition once the duality between the DG and Villain models has been established.

For sufficiently high $\beta$, the DG model does not have a natural infinite volume Gibbs measure, 
but has a slope-zero (translation-)ergodic infinite volume \emph{gradient} Gibbs measure by \cite{MR2251117}.
Denote this measure by $\E^{\dg}_{\beta, \Z^2}$---it is a measure that resides on the $\sigma$-algebra generated by the nearest neighbor gradient of functions.
This measure can be obtained by taking an infinite volume limit of either the 0 boundary or the periodic DG model. 
In particular, if $f : \Z^2 \rightarrow \R$ is a function with compact support and $\iota^{\square}_M : F(\square_M) \rightarrow \Z^2$ and $\iota^{\Lambda}_N : \Lambda_N \rightarrow \Z^2$ are isometric embeddings,
then
\begin{align}
	\E_{\beta, \Z^2}^{\dg}  \big[  e^{(f,\sigma)} \big]
	=  \lim_{M\rightarrow \infty} \E_{\beta, \square_M}^{\dg,0}  \big[ e^{ ( f \circ \iota^{\square}_{M}   , \sigma) } \big]
	=  \lim_{N\rightarrow \infty} \E_{\beta, \Lambda_N}^{\dg,\per}  \big[ e^{ ( f \circ \iota^{\Lambda}_{N}   , \sigma) } \big]
	.
	\label{eq:infinite_volume_DG}
\end{align}
$\E_{\beta, \Z^2}^{\dg}$ is invariant under Euclidean symmetries (translations, reflections, $\pi/2$ rotations), so this result does not depend on the choice of the embedding.
We will often use $f$, $f \circ \iota^{\square}_{M}$ and $f \circ \iota^{\Lambda}_{N}$ indistinctively, whenever their meanings are clear from the context.

\subsection{Duality between the Villain and the DG models}
\label{sec:dualityVilDG}

The correlation function of the Villain model can be mapped to a moment-generating function of the DG model via a duality transformation \cite{MR634447}. 
We say that $\bP$ is a path connecting two points $y_1, y_2 \in \square_M$ in $\square_M$ if $\bP = \{ \{a_0^{-1}, a_1 \},  \{a_1^{-1}, a_2 \} \cdots, \{a_{n-1}^{-1}, a_n \}  \} \subset \vec{E} (\square_M)$ for $a_0 = y_1$, $a_n = y_2$ and some $a_1, \cdots, a_{n-1} \in \square_M$. 
(We do not allow repeated edges for notational clarity, although it is not a technical problem.)

\begin{proposition} \label{prop:VilotoDG}
Let $\kn \in \Z$,  $y_1, y_2\in \square_M$ and let $\bP$ be any path connecting $y_1$ and $y_2$. 
If $\theta \sim \P^{\Vil}_{\beta', \square_M}$ and $\sigma \sim \P^{\dg}_{\beta, \square_M}$ for $\beta = 4\pi^2 \beta'$, then
\begin{align}
	\E^{\Vil}_{\beta', \square_M} \big[ e^{i \kn (\theta_{y_2} - \theta_{y_1}) } \big]  
	= 
	e^{-\frac{2\pi^2 \kn^2}{\beta} |\bP| } \E^{\dg, 0}_{\beta, \square_M} \big[ \exp \big(- 2\pi \beta^{-1} \kn (\sigma, d \one_\bP )  \big) \big]
    .
	\label{eq:order-to-disorder_correlator}
\end{align}
\end{proposition}
\begin{proof}
Denote $g = \delta_{y_2} - \delta_{y_1} \in \square_M^{(0)} (\Z)$ so that $\theta_{y_2} - \theta_{y_1} = (\theta,g)$.
We want to study
\begin{align}
	\E^{\Vil}_{\beta' , \square_M} \big[ e^{i \kn (\theta, g)} \big] 
	= \frac{1}{Z^{\Vil}_{\beta' , \square_M} } \int_{[0, 2\pi)^{\square_M}} e^{i \kn (\theta, g)} 
\prod_{x\in \square_M} d\theta_x  \prod_{e = \{x,y \} \in E(\square_M)} \sum_{m_e \in 2\pi \Z} e^{-\frac{\beta' }{2} (\theta_x - \theta_y - m_e )^2} .
\end{align}
By the Poisson summation formula, 
\begin{align}
	\E^{\Vil}_{\beta' , \square_M} \big[ e^{i \kn (\theta, g)} \big] 
	& \propto \int_{[0, 2\pi )^{\square_M}} \prod_{x \in \square_M} d \theta_x  \sum_{w \in \square_M^{(1)} (\Z)} e^{-\frac{1}{2 \beta'} |w|^2 } e^{i \kn (\theta, g)  - i ( w ,  d \theta)} 
	\\
	& = \sum_{w \in 2\pi \square_M^{(1)} (\Z) } e^{-\frac{1}{2 \beta} |w|^2 }  \prod_{x \in \square_M} \int_{[0, 2\pi )} d \theta_x   e^{i \theta_x  (\kn g - d^* w / 2\pi )(x) } .
\end{align}
Since
$\int_{[0, 2\pi)} d\theta \,  e^{i k \theta } = 2\pi \delta_{0,k}$
for $k\in \Z$, we obtain
\begin{align}
\begin{split}
	\E^{\Vil}_{\beta', \square_M} \big[ e^{i \kn (\theta, g)} \big]  
        \propto \sum_{w \in 2\pi \square_M^{(1)} (\Z) }  1_{d^* w = 2\pi \kn g} e^{-\frac{1}{2\beta} |w|^2} .
\end{split}
\end{align}
 Since we endowed $\square_M$ with a free boundary condition, 
$\{ w \in 2\pi \square_M^{(1)} (\Z) : d^* w = 0 \} = \{ d^* \sigma : \sigma \in \Omega^{\dg}_{\square_M} \}$ (recall the paragraph before \eqref{eq:DG_with_Dirichlet_boundary} for $\Omega^{\dg}_{\square_M}$).
Also,  we have $d^* \one_{\bP} = g$, so any $w$ satisfying $d^* w = 2\pi \kn g$ can be written as $w = 2\pi \kn \one_{\bP} + d^* \sigma$ for some $\sigma \in \Omega^{\dg}_{\square_M}$.
Hence we obtain
\begin{align}
	 \E^{\Vil}_{\beta', \square_M} \big[ e^{i \kn (\theta, g)} \big] 
	 \propto \sum_{\sigma \in \Omega^{\dg}_{\square_M}}  e^{-\frac{1}{2\beta} |d^* \sigma + 2\pi \kn \one_{\bP}  |^2}
\end{align}
that becomes the right-hand side of \eqref{eq:order-to-disorder_correlator} after expanding the square in the exponent.
\end{proof}

The duality relation is an instance of the generalised Kramers--Wannier duality, see \cite{dubedat2011topics}.
The path observable on the right-hand side of \eqref{eq:order-to-disorder_correlator} is also called the disorder correlator of the model. 
As an immediate consequence \eqref{eq:infinite_volume_DG} and \eqref{eq:order-to-disorder_correlator}, we get
\begin{align}
	\E^{\Vil}_{\beta', \Z^2} \big[ e^{ i \kn (\theta_{y_2} - \theta_{y_1})  } \big]  
        &= e^{- \frac{ 2 \pi^2 \kn^2 }{\beta} |\bP | } \lim_{N\rightarrow \infty}  \E^{\dg, \per}_{\beta, \Lambda_N} \big[ e^{2\pi \beta^{-1} \kn (\sigma, d^* (\one_{\bP} )^\dual )  } \big]
    \label{eq:villain_two_point_via_dg_on_torus}
\end{align}
where on the right-hand side, $\bP$ should be interpreted as a path on $\vec{E} (\Lambda_N^{\dual})$ and $d \one_{\bP(y_1,y_2)} \in ( \Lambda_N^{\dual} )^{(2)} (\Z) \simeq  \Z^{\Lambda_N}$.
Thus, the infinite volume correlation function of the Villain model is stated in terms of the DG model with periodic boundary conditions.
Most of this paper will focus on the right-hand side of \eqref{eq:villain_two_point_via_dg_on_torus}. 

\begin{figure}[h]
\begin{centering}
\begin{tikzpicture}[scale=1, x=1cm, y=1cm, >=stealth]
\tikzset{
    middle arrow/.style={
        postaction={decorate, decoration={
            markings,
            mark=at position 0.6 with {\arrow[line width=1.5pt, black]{>}}
        }}
    },
    crossing arrow/.style={
        blue!60!cyan, thin, -{Latex[scale=0.8]}
    }
}

    \coordinate (X) at (0,0);
    \coordinate (Corner1) at (5,0);
    \coordinate (Corner2) at (5,3);
    \coordinate (Y) at (8,3);

    \draw[ultra thick, middle arrow] (X) -- (Corner1);
    
    \draw[ultra thick, middle arrow] (Corner1) -- (Corner2);
    
    \draw[ultra thick, middle arrow] (Corner2) -- (Y);

    \node[left=8pt] at (X) {\Large \bx};
    \node[right=8pt] at (Y) {\Large \by};

    \begin{scope}[crossing arrow]
        \foreach \x in {0.5, 1.5, 2.5, 3.5, 4.5} {
            \draw (\x, -0.5) -- (\x, 0.5);
        }
        
        \foreach \y in {0.5, 1.5, 2.5} {
            \draw (5.5, \y) -- (4.5, \y);
        }
        
        \foreach \x in {5.5, 6.5, 7.5} {
            \draw (\x, 2.5) -- (\x, 3.5);
        }
    \end{scope}    
\end{tikzpicture}
\caption{An example of $\bP$ connecting $\bx$ and $\by$ (black path) and its dual $\bP^\dual$ (blue arrows)}
\label{fig:PandPdual}
\end{centering}
\end{figure}

\begin{remark} \label{remark:wrongEquality}
In \cite{dgauss2}, it was proved that the scaling limit of the DG model on $\Z^2$ with temperature $\beta$ is a Gaussian free field with effective temperature $\beta_{\rm eff} >0$.
One may speculate that this result also gives a good approximation of the moment generating function of \eqref{eq:villain_two_point_via_dg_on_torus}, i.e.,
\begin{align}
    -\frac{2\pi^2 \kn^2}{\beta} |\bP| + \log \E_{\beta,\Z^2}^{\dg} \big[ e^{2\pi \beta^{-1} \kn (\sigma, d^* (\one_{\bP} )^\dual )  } \big]
    \text{``}\simeq\text{''}
    -\frac{B}{\beta} |\bP| + \frac{B \beta_{\eff}}{\beta^2}  (d^* (\one_\bP)^\dual , (-\Delta)^{-1} d^* (\one_\bP)^\dual )
    \label{eq:wrongEquality}
\end{align}
where $B = 2\pi^2 \kn^2$.
To simplify the quadratic form, we move to the dual lattice, and recall that the Laplace-Beltrami operator is given by $-\Delta = dd^* + d^* d$. Then
\begin{align}
    (d^* (\one_\bP)^\dual , (-\Delta)^{-1} d^* (\one_\bP)^\dual )
    &= (d \one_{\bP} , (dd^* + d^* d)^{-1} d \one_{\bP} ) \nnb
    &= (\one_{\bP},\one_{\bP}) - (d^* \one_{\bP} , (-\Delta)^{-1} d^* \one_{\bP} ) .
    \label{eq:CoulombFundamental}
\end{align}
But since $d^* \one_{\bP} = \delta_{\by} - \delta_{\bx}$, we actually have
\begin{align}
    - \frac{B}{\beta} |\bP| + \frac{B \betaeff}{\beta^2} (d^* (\one_\bP)^\dual , (-\Delta)^{-1} d^* (\one_\bP)^\dual )
    = \frac{B|\bP|}{\beta} \Big( -1 + \frac{\betaeff}{\beta} \Big) - \frac{2B \betaeff}{\beta^2} G (\bx,\by)
\end{align}
where $G(\bx,\by)$ is the Green's function of $-\Delta$ on $\Z^2$.
The term $e^{-\frac{2B\betaeff}{\beta^2} G (\bx,\by)}$ gives a polynomial decay.
However, \cite{garban2023quantitative} shows $\beta \neq \beta_{\rm eff}$, so there is a term that is growing proportional to $|\bP|$, so the correlation function of the Villain model is exponentially growing or decaying! This contradicts the theorem of Fr\"ohlich-Spencer \cite{MR634447} on the polynomial lower bound of the correlation function and that the Villain model takes value on $\mathbb{S}^1$. 

What is going wrong? The approximate \eqref{eq:wrongEquality} can not be true, because the line observable is actually marginal in the RG sense and evolves by a non-trivial amount under coarse-graining operations.
Control of the marginality under the RG flow will be explained in Section~\ref{sec:obsRGcoords}--\ref{sec:obsRGmap}, while the exact mechanism for the cancellation is provided by Section~\ref{sec:contractible}. 

The (wrong) computation discussed in this remark provides useful intuition for the analysis in Part II, cf. Remark~\ref{remark:multiscaleSpinwave}.
\end{remark}

\subsection{Reformulation as a sine-Gordon field theory}
\label{sec:reformSG}

We start from the right-hand side of \eqref{eq:villain_two_point_via_dg_on_torus}.
To apply RG on the DG model, we first need to restate the problem in terms of a sine-Gordon field theory. We recall some of the main arguments of \cite{dgauss1}.
We first add a mass $m^2 >0$ to the DG model to define
\begin{align}
    \langle F(\sigma) \rangle_{m^2,N} := \frac{\sum_{\sigma : \Lambda_N \rightarrow 2\pi\Z/\alpha} F(\sigma) e^{- \frac{1}{2} (\sigma , (-\Delta + m^2) \sigma)} }{\sum_{\sigma : \Lambda_N \rightarrow 2\pi\Z/\alpha} e^{- \frac{1}{2} (\sigma , (-\Delta + m^2) \sigma)} } ,
    \qquad
    \alpha = \sqrt{\beta} .
\end{align}
The massive approximation simplifies the analysis, by removing the pinning $\sigma_0 = 0$, both the configuration space of the integer-valued field $\sigma$ and its measure are  translation invariant.

\begin{lemma} \cite[Lemma~2.1]{dgauss1}
\label{lemma:DGm2approx}
If $F : (2\pi\Z/\alpha)^{\Lambda_N} \rightarrow \R$ is a function such that $F(\sigma + t \one) = F(\sigma)$ for any $t \in \R$ and satisfies an appropriate integrability condition, then
\begin{align}
    \E^{\dg}_{\beta,\Lambda_N} [F(\sigma / \alpha)] = \lim_{m^2 \rightarrow 0} \langle F(\sigma) \rangle_{m^2, N} .
\end{align}
\end{lemma}

For some parameters $s \in \R$, $m^2 > 0$ and $\gamma >0$, we define covariances
\begin{align}\label{eq:Csm}
    C (s,m^2) &= \big( ((-\Delta + m^2)^{-1} - \gamma\mathbb{I} )^{-1} + s(-\Delta) \big)^{-1} .
\end{align}
That $C(s,m^2)$ is positive semi-definite is part of Lemma~\ref{lemma:Cptv} and Proposition~\ref{prop:FRD}.
$\gamma$ is chosen only in Proposition~\ref{prop:FRD}.

\begin{lemma} \label{lemma:Cptv}
There exists $\gamma >0$ such that, whenever $s \in \R$ and $m^2 >0$ are sufficiently small, $C(s,m^2)$ is a covariance matrix on $\Lambda_N$.
\end{lemma}
\begin{proof}
This is explained in \cite[Section 2.3]{dgauss1}.
\end{proof}

In this setting, we restate the massive DG model as a lattice scalar field theory with a smooth potential. Recall Section~\ref{sec:basicNots} for the Gaussian expectation $\E_C$.

\begin{lemma} \cite[Lemma~2.1]{dgauss3}
\label{lemma:sinGordapprox}
Let $f \in \Lambda_N^{(1)} (\R)$ and $\tilde{f} = (1 + s\gamma \Delta) f$. Then
\begin{align}
    \langle e^{(d^*f,\sigma)} \rangle_{m^2, N}
        = e^{\frac{\gamma}{2} (d^* f, d^* \tilde{f})} \frac{\E_{C(s,m^2)} [ e^{(d^*f,\varphi)} Z_0^\obs (\varphi) ]}{\E_{C(s,m^2)} [Z_0^\bulk (\varphi)] } 
        \label{eq:lemmasinGordapprox}
\end{align}
where
\begin{align}
    & Z_0^\obs (\varphi) = e^{s\gamma(\Delta d^* f, \varphi)} Z_0^\bulk (\varphi + \gamma d^* f) , \\
    & Z_0^{\bulk} (\varphi)
        = \exp\big( \frac{1}{2} s (\nabla \varphi, \nabla \varphi) \big) \prod_{x \in \Lambda_N} \exp \big( V_0^Q (\{x\}, \varphi) \big) , \\
    & V_0^Q (\{x\}, \varphi) 
        = \sum_{q=1}^{\infty} \tilde{z}^{(q)} \cos (\alpha \varphi_x)
\end{align}
for some $(\tilde{z}^{(q)})_{q\ge 1} \subset \R$ satisfying $|\tilde{z}^{(q)}| \le 16 e^{-\frac{1}{4} \gamma\beta (1+q)}$.
\end{lemma}

For a later use, the coefficients of $V_0^Q$ are denoted
\begin{align}
    \tilde{\underline{z}} = (\tilde{z}^{(q)})_{q\ge 1} . \label{eq:tildeunderlinez}
\end{align}
With our application to \eqref{eq:villain_two_point_via_dg_on_torus} in mind, we will usually consider 
$f = 2\pi \kn (\one_{\bP} )^\dual / \alpha$.

\subsection{Covariance Decomposition}
\label{sec:covDecomp}

The other crucial input to the RG analysis is the finite-range decomposition of the covariance matrix. We only state the result here.

\begin{definition}
We say that a pair $(t_N , (\Gamma_j)_{j=1}^N)$ is a \emph{finite-range decomposition} of a covariance matrix $C : \Lambda_N \times \Lambda_N \rightarrow \R$ if $t_N >0$ and $(\Gamma_j)_{j=1}^N$ are covariance matrices on $\Lambda_N$ satisfying
\begin{align}
    C = \sum_{j=1}^N \Gamma_j + t_N Q_N
\end{align}
and $\Gamma_j (x,y) = 0$ whenever $|x-y|_{\infty} \ge L^j /4$, where $Q_N (x,y) = L^{-2N}$ is a constant-valued matrix.
\end{definition}

We compile all the required properties in the next statement.

\begin{proposition} \label{prop:FRD}
Let $d =2$ and $s \in \R$ and $m^2 > 0$ be sufficiently small.
There exists $\gamma >0$ such that $C(s,m^2)$ admits a finite-range decomposition $t_N >0$ and $(\Gamma_j)_{j=1}^N$, satisfying the following.
\begin{enumerate}
\item Each $(\Gamma_j)_j$ is invariant under lattice symmetries. Thus, we also simply denote $\Gamma_j (x,y) \equiv \Gamma_j (y-x)$.
\item For each $n\in \N$, there exists $C_n > 0$ such that
\begin{align}
    \norm{\nabla^n \Gamma_j}_{\ell^\infty} \le C_n \times \begin{cases}
            L^{-n (j-1)} & (n \neq 0) \\
            \log (L) & (n = 0) .
        \end{cases}
\end{align}
\item As $j\rightarrow \infty$, $\Gamma_j (0) \sim \frac{1}{2\pi(1+ s)} \log L$. Also, if $L$ is sufficiently large, then $\Gamma_j (0) \ge C \log L$ for all $j\ge 1$ and some $C>0$.
\item There exists $C >0$ such that $t_N \in [ (m^{-2} - C L^{2N} ) \vee 0 , m^{-2} ]$.
\item If $j < N$, then $\Gamma_j$ is independent of $N$. That is, if $j< N < N'$ and $i_{N',N} : \Lambda_{N'} \rightarrow \Lambda_N$ is a natural covering projection, then $\Gamma_j ( i_{N',N} (x) - i_{N',N} (y))$ on $\Lambda_N$ is identical to $\Gamma_j (x-y)$ on $\Lambda_{N'}$.
\item $\Gamma_j$ and $t_N$ are continuous functions of $s$ and $m^2$. Moreover, $\Gamma_j$ can be continuously extended to $m^2 = 0$.
\end{enumerate}
\end{proposition}
\begin{proof}
Consider covariance matrices $(D^{\Z^2}_t (s,m^2))_{t \ge 0}$, $(D^{\Lambda_N}_t (s,m^2))_{t \ge 0}$ and $t_N >0$ given by \cite[Proposition~3.4]{dgauss1} and let
\begin{align}
    \Gamma_{j+1} = \begin{cases} 
        \int_{1 \vee (L^j / 4 )}^{L^{j+1}/4} D_t^{\Z^2} (s,m^2) d t & (j+1 < N) \\
        \int_{L^{N}/4}^{\infty} \tilde{D}_t^{\Lambda_N} (s,m^2) dt & (j+1 = N) .
    \end{cases}
\end{align}
The properties follow from \cite[Proposition~3.2.--3.4.]{dgauss1}, except for (iii).
The first part of (iii) follows after integrating the estimate of \cite[Proposition~3.3.(iii)]{dgauss1}, while the second part can also be obtained after taking $L$ sufficiently large (so that the large-$L$ effect dominates the small-scale part).
\end{proof}

Since the insertion of the mass regularization $m^2 >0$ was arbitrary, we need to take $m^2 = 0$ at the end. Fortunately, we can define an $m^2$-independent RG flow using point (vi). To be specific, we define new covariances
\begin{align}
    \bar{C} (s,m^2) = \sum_{j=1}^N \Gamma_j (m^2 = 0) + t_N (m^2) Q_N
\end{align}
(making the $m^2$-dependence explicit) and $\Gamma_{\le k} = \sum_{j\le k} \Gamma_{j}$.
Then
\begin{align}
    Z_N^{\obs} (\varphi) = \E_{\Gamma_{\le N} (m^2=0)} [ e^{(d^* f , \zeta)} Z_0^\obs (\varphi + \zeta) ] 
    \label{eq:ZNfDefi}
\end{align}
is independent of $m^2 $. The mass dependence is contained in the final integration against the constant mode covariance $t_N(m^2)Q_N$.
As will be shown in the next lemma, the moment generating function can be reduced to $Z_N$, and it will be the subject of the RG analysis.

\begin{lemma} \label{lemma:DGmgfasZNint}
For $f \in \Lambda_N^{(1)} (\R)$ and $(s,\gamma)$ as in Proposition~\ref{prop:FRD},
\begin{align}
    \E^{\dg}_{\beta,\Lambda_N} [e^{(d^*f,\sigma) / \alpha}] 
        &= 
        e^{\frac{\gamma}{2} (d^* f, d^* \tilde{f})} \lim_{m^2\rightarrow 0} \frac{\E_{\bar{C}(s,m^2)} [ e^{(d^* f,\varphi)} Z_0^\obs (\varphi) ] }{\E_{\bar{C} (s, m^2)} [Z_0^\bulk (\varphi)] } \nnb  
        &= e^{\frac{\gamma}{2} (d^* f, d^* \tilde{f})} \lim_{m^2\rightarrow 0} \frac{\E_{t_N Q_N} [ Z_N^\obs (\varphi' ) ] }{\E_{t_N Q_N} [Z_N^\bulk (\varphi')] }
        .
    \label{eq:DGmgfasZNint}
\end{align}
\end{lemma}
\begin{proof}
We start from Lemma~\ref{lemma:DGm2approx} and \ref{lemma:sinGordapprox}.
We can decompose $C(s,m^2) = \Gamma_{\le N} (m^2) + t_N Q_N$. If $\zeta \sim \cN (0,\Gamma_{\le N} (m^2))$ and $\varphi' \sim \cN (0, t_N Q_N)$ are independent Gaussian r.v.'s and $\varphi = \varphi' + \zeta$, then $\varphi \sim \cN (0, C(s,m^2))$, so 
\begin{align}
    \E_{C(s,m^2)} [e^{(d^* f, \varphi)} Z_0^\obs (\varphi)]
    &= \E [e^{(d^* \tilde f, \varphi)} Z_0^\bulk (\varphi + \gamma d^* f)] \nnb
    & = \E [ e^{(d^* \tilde f , \varphi' + \zeta)} Z_0^\bulk (\varphi + \gamma d^* f) ] . 
\end{align}
But since $\varphi'$ is a.s. constant-valued and $(d^* f, \one)=(d^* \tilde{f}, \one)=0$, we also have $(d^* \tilde f , \varphi') = 0$ a.s. Thus,
\begin{align}
    = \E [ e^{(d^* \tilde f , \zeta)} Z_0^\bulk (\zeta + \varphi' + \gamma d^* f) ] = \E^{\varphi'} \E^{\zeta} [ e^{(d^* f, \zeta)} Z_0^\obs (\varphi' + \zeta) ] .
\end{align}
The inner expectation converges to $Z_N^\obs (\zeta)$ as $m^2 \rightarrow 0$.
\end{proof}

\section*{\large\textnormal{\MakeUppercase{Part I. Renormalisation Group Analysis}}}
\addcontentsline{toc}{section}{Part I. Renormalisation Group Analysis}
\renewcommand{\thesection}{I}
\setcounter{subsection}{0}
\setcounter{subsubsection}{0}

\renewcommand{\thesection}{\arabic{section}}
\section{RG coordinate spaces}
\label{sec:RGcoordSpace}

We renormalise the functions that are defined on polymers---recall the notions from Section~\ref{sec:blockLambda} related to polymers.

A \emph{polymer activity} (at scale $j$) is a collection of smooth functions $K = (K (X,\varphi) : X \in \Conn_j)$ such that $K(X,\varphi)$ only depends on $\varphi|_{X^*}$. We extend $K$ to $\cP_j$ using the \emph{polymer power}
\begin{align}
    K (X) = K^{[X]} := \prod_{Z \in \Comp (X)} K (Z) .
    \label{eq:polypower}
\end{align}
We denote by $\cN_j$ the set of all scale-$j$ polymer activities defined above, extended multiplicatively to arbitrary $j$-polymers by \eqref{eq:polypower}.
In this section, we will introduce additional structures to define new spaces of polymer activities.

\subsection{Function spaces and norms}

When $\{\bbe_1, \bbe_2 \}$ is the standard basis of $\Z^2$, 
let $\hat{e} = \{ \pm \bbe_1, \pm \bbe_2 \}$. On each scale $j\ge 0$ and $f : \{ x : \dist_1 (x,X) \le n \} \rightarrow \C$ for $X\in \cP_j$, we take
\begin{align}
\begin{cases}
    & \norm{\nabla^n_j f}_{L^\infty (X)} = \max_{x\in X} \sum_{(\mu) \in \hat{e}^{n}} |\nabla^{(\mu)}_j f(x)| \\
    & \norm{\nabla^n_j f}^p_{L^p_j (X)} = L^{-2j} \sum_{x \in X} \sum_{(\mu) \in \hat{e}^n} 2^{-n} |\nabla^{(\mu)}_j f(x) |^p \\
    & \norm{\nabla^n_j f}^p_{L^p_j (\partial X)} = L^{-j} \sum_{x \in \partial X} \sum_{(\mu) \in \hat{e}^n} 2^{-n} |\nabla^{(\mu)}_j f(x)|^p \\
    & \norm{f}_{C_j^2 (X)} = \max_{n=0,1,2} \norm{\nabla^n_j f}_{L^\infty (X)} 
\end{cases}    
\end{align}
where $\nabla_j = L^j \nabla$ and $\partial X$ is the inner boundary of $X$.

For $F\in \cN_j$, $n\in \N$ and $h\ge 0$, we define Taylor semi-norms
\begin{align}
    & \norm{D^n F (X,\varphi)}_{n, X, T_j (\varphi)}
    = \sup\big\{ |D^n F (X,\varphi ; f_1, \cdots, f_n) | : \norm{f_k}_{C_j^2 (X^*)} \le 1 , \, k=1, \cdots, n \big\} 
    \label{eq:DnFseminormDefi}
    \\
    & \norm{F (X,\varphi)}_{h, X, T_j (\varphi)} = \sum_{n=0}^\infty \frac{h^n}{n!} \norm{D^n F (X,\varphi)}_{n, X, T_j (\varphi)} ,
\end{align}
where $D^n F (X,\varphi ; f_1, \cdots, f_n)$ should be interpreted as an $n$-form acting on $(f_1, cdots, f_n)$.
With the help of regulators, these can be used to define actual norms on polymer activities.

\subsection{Regulators} \label{sec:regulators}

We use a \emph{large-field regulator} that controls the fluctuation of $\nabla \varphi$ and $\nabla^2 \varphi$: for parameters $\kappa_0$ and $c_2$ as specified in Section~\ref{sec:choicePar}, let
\begin{align}
    & G_j (X,\varphi) = \exp \Big[ \kappa_0 \Big( \norm{\nabla_j \varphi}_{L^2_j (X)}^2 + c_2  \norm{\nabla_j \varphi}_{L^2_j (\partial X)}^2 +  W_j (X, \nabla_j^2 \varphi) \Big) \Big]
    \label{eq:GjXphiDefi}
    \\
    & W_j (X, \nabla_j^2 \varphi) = \sum_{B \in \cB_j (X)} \norm{\nabla_j^2 \varphi}_{L^\infty (B^*)}^2 .
\end{align}
They are the same as those of \cite[Definition~5.3]{dgauss1}, but we will need to extend them for the observable RG coordinates---see Section~\ref{sec:flowdepNorm}.

The \emph{large-set regulator} is a (possibly $L$-dependent) parameter $A > 1$. We take it sufficiently large in Assumption~\ref{asmp:RGasmpBulk}.

\subsection{Normed spaces}
\label{sec:NrmdSpcs}

Suppose $\cG : \cP_j \times \R^{\Lambda} \rightarrow \R$ is a polymer-factorising function, i.e., $\cG (X\cup Y,\varphi) = \cG (X,\varphi) \cG(Y,\varphi)$ whenever $X\not\sim Y$, and $A' > 1$. ($G_j$ is one example.)
For $F\in \cN_j$, we define a norm via
\begin{align}
    & \norm{F(X)}_{h, X, T_j (\cG)} = \sup_{\varphi \in \R^{X^*}} \norm{F(X,\varphi)}_{h, X, T_j (\varphi)} / \cG (X, \varphi) \\
    & \norm{F}_{h, A' , T_j (\cG) } = \sup_{X \in \Conn_j} (A')^{|X|_j} \norm{F (X)}_{h, X, T_j (\cG)} .
\end{align}
We use $A' = A$ and $\cG = G_j$ for the bulk RG flow and denote
\begin{align}
    \norm{F}_{j,\bulk} = \norm{F}_{h, A, T_j (G_j)} .
\end{align}
We say that $F$ is \emph{invariant under lattice symmetries} if $F(X,\varphi) = F( \iota (X) , \varphi \circ \iota^{-1} )$ for any isometry $\iota : \Lambda \rightarrow \Lambda$,
$F$ is \emph{even} if $F(X,\varphi) = F(X,-\varphi)$
and $F$ is \emph{$\alpha$-periodic} if $F(X,\varphi)=F(\varphi + 2\pi \one / \alpha)$.
The sets of \emph{bulk polymer activities} are given by
\begin{align}
    \cN_j^\bulk &= \{ F \in \cN_j : \norm{F}_{j,\bulk} < \infty , \; F \text{ is $\alpha$-periodic, even and invariant under lattice symmetries} \} .
    \label{eq:cKjbulkDefi}
\end{align}

\subsection{Effective potential}
\label{sec:effpot}

At scale $j$, effective potential is parametrised by $s_j \in \R$ and $\underline{z}_j = (z_j^{(q)})_{q=1}^\infty \in \R^{\Z_+}$. Given these coordinates, we denote
\begin{align}
    V_j (X, \varphi) 
        &= V_j^G (X,\varphi) + 
    V^Q_j (X,\varphi) , \qquad
    X \subset \Lambda, \; \varphi \in \R^{\Lambda_N}
    \label{eq:Vbulkjdefi} \\
    V_j^G (X,\varphi) &= \sum_{x\in X} \frac{1}{2} s_j |\nabla \varphi_x|^2 , \qquad
    V^Q_j (X,\varphi)
        = L^{-2j} \sum_{x\in X} \sum_{q=1}^{\infty} z_j^{(q)} \cos ( q\alpha \varphi_x) , 
        \label{eq:VQjdefi}
\end{align}
The space of functions \eqref{eq:Vbulkjdefi} is denoted $\cV$. 
For $V_j \in \cV$, we consider the norm
\begin{align}
    \norm{V_j}_{\cV_j} = A \max \Big\{ |s_j| , \; \sup_{q\ge 1} e^{\frac{1}{4} \gamma \beta q} |z_j^{(q)}| \Big\} .
    \label{eq:cVjNormDefi}
\end{align}
Note that $V_j$ is not a polymer activity, but $e^{V_j}$ is. We denote this
\begin{align}
   I_j^X (\varphi) = I_j (X,\varphi) = e^{V_j (X,\varphi)} = (e^{V_j (\varphi)} )^X ,
\end{align}
where for a general polymer function $F (X)$, we use
\begin{align}
    F^X = \prod_{b\in \cB_j (X)} F(b) 
    \label{eq:blockPower}
\end{align}
to denote the \emph{block power} of $F$.
We say that $F$ is block-multiplicative if $F^X = F(X)$.

\section{Progression map}

We define the progression map in this section.
We first briefly explain the inductive structure of the RG analysis to motivate the progression map.

We defined $Z_0^\bulk$ in Lemma~\ref{lemma:sinGordapprox} and $Z_N^\bulk$ in \eqref{eq:ZNfDefi}.
To control $Z_N^\bulk$ inductively, we consider, for $j = 1, \cdots, N$, let
\begin{align}
    Z_j^\bulk (\varphi) = \E_{\le j} \theta [Z_0^\bulk (\varphi)] 
    \qquad \Leftrightarrow \qquad    
    Z_{j}^\bulk (\varphi) = \E_{j} \theta Z_{j-1}^\bulk (\varphi ) .
    \label{eq:Zjbulkdefi}
\end{align}
For some polymer functions $I_j$ and $K_j$, we will express $Z_j^\bulk$ in terms of a polymer expansion 
\begin{align}
    Z_j^\bulk = e^{u_j^\bulk (\Lambda)} ( I_j^\bulk \circ_j K_j^\bulk ) (\Lambda)
\end{align}
that we define as the following: for a block-multiplicative $I$ (i.e., $I^X = I(X)$) and a polymer activity $K$, we define the polymer expansion 
\begin{align}
    (I \circ_j K) (X) = \sum_{Y \in \cP_j (X)} I^{X \backslash Y} K^{[Y]} ,\qquad X \in \cP_j . 
\end{align}
(Recall notations \eqref{eq:polypower} and \eqref{eq:blockPower}.)
The \emph{progression map} can be used to expand $Z_{j+1}^\bulk$ in terms of $I_j^\bulk$ and $K_j^\bulk$.

\subsection{Definition of Progression map}
\label{sec:PM}

The progression map has three segments: (1) fluctuation integral $\E_{j+1}$, (2) reblocking in the larger scale $j+1$ and (3) exchange of terms between $V_j$ and $K_j$, also called the reapportioning. The latter is encoded inside the \emph{reapportioning element} $\Rae$.
Suppose $( \Rae (B,X,\varphi) : B\in \cB_{j+1} (X) , X \in \cS_{j+1} )$ is a collection such that
\begin{align}
    \sum_{X \in \cS_{j+1}}^{X \supset B} \Rae (B,X) = 0 . \label{eq:keassump}
\end{align}
Specific choices of $J$ are given in Definition~\ref{defi:Vplusdefi} and \ref{defi:obsVplusdefi}.
Then for $X \in \cP_{j+1}$, let
\begin{align}
    \cE K (X,\varphi) = \one_{X \in \cS_{j+1}} \sum_{B \in \cB_{j+1} (X)} \Rae (B, X, \varphi) .
\end{align}

\begin{definition} \label{defi:PM}
Let $\Rae$ satisfy \eqref{eq:keassump} and $\delta u_+ (B) \in \R$ for each $B \in \cB_{j+1}$.
Suppose $I^Y = \prod_{b\in \cB_j (Y)}I^b$ and $I_+^X = \prod_{B \in \cB_{j+1} (X)} I_+^B$ are set-multiplicative polymer activities. 
The \emph{progression map} on a polymer activity $K$ with reapportioning element $\Rae$ is given by, for $X \in \cP_{j+1}$,
\begin{align}
    \label{eq:PMdefi}
    & \PM_{j+1}^{J} (X) \equiv \PM_{j+1}^{J} [I, I_+, \delta u_+, K] (X) \\
    & = \sum_{X_0, X_1, Z, (B_{Z''})}^{\rightarrow} e^{-\delta u_{+} (T)} I_{+}^{X\backslash T} \, \E_{j+1} \Big[ (\theta I - e^{\delta u_+} I_+ )^{X_0} (\theta \bar{K} - \cE K )^{[X_1]} \Big] \prod_{Z'' \in \Comp_{j+1} (Z)} \Rae (B_{Z''} , Z'') 
    \nonumber
\end{align}
where $\sum^{\rightarrow}$ sums over disjoint $j+1$-polymers $X_0, X_1, Z \in \cP_{j+1}$ such that $Z \not\sim X_1$, $B_{Z''} \in \cB_{j+1} (Z'')$ for each $Z'' \in \Comp_{j+1} (Z)$, $X = \cup_{Z''} B_{Z''}^* \cup X_0 \cup X_1$, $T = X_0 \cup X_1 \cup Z$ and
\begin{align}
    \bar{K} (X) = \sum_{Y\in \cP_j}^{\bar{Y} = X} I^{X \backslash Y} K (Y) .
\end{align}
\end{definition}

Due to the following algebraic property, the progression map can be used to give a polymer expansion of $Z_{j+1}^\bulk$.

\begin{proposition} \label{prop:PMproperty}
Consider $\Lambda = \Lambda_N$ for $N < \infty$.
Let $K_{+}^{\Rae} = \cK_{j+1}^{\Rae} [I, I_+ , \delta u_+ , K]$ for any choice of $I$, $I_+$, $\delta u_+$, $K$ and ${\Rae}$. Given the convergence of each integral, we have
\begin{align}
    \E_{j+1} \theta [I \circ_j K] = e^{\delta u_+ (\Lambda)} (I_+ \circ_{j+1} K^{\Rae}_{+} ) .
\end{align}
Moreover, $\PM_{j+1}^{\Rae}$ only acts locally in the following senses:
\begin{enumerate}
\item $K^{\Rae}_+ (X)$ only depends on $I^Y$, $I_+^Y$ and $K (Y)$ and $\Rae (B,Y)$ for $B\subset Y \subset X$ and
\item if $\Rae (B, Z, \varphi)$ only depends on $\varphi |_{Z^*}$ for any $B \subset Z \in \cS_{j+1}$, then $K_{+}^{\Rae} (Z,\varphi)$ also only depends on $\varphi |_{Z^*}$.
\end{enumerate}
\end{proposition}
\begin{proof}
The identity is essentially a restatement of \cite[Theorem~7.5]{dgauss1}, just with $(e^{U_j}, e^{U_{j+1}}, -E_{j+1})$ in the reference replaced by $(I, I_+, \delta u_{+})$.
Both of the locality properties (i) and (ii) are direct from the definition and the finite-range property of the covariance $\Gamma_{j+1}$. 
\end{proof}

\begin{remark} \label{remark:infvolPM}
The progression map can be defined directly on an infinite lattice $\Lambda = \Z^2$, although in this case, we do not have Proposition~\ref{prop:PMproperty}.

Due to the locality of $\PM_{j+1}^{\Rae}$,
this infinite volume map can be considered as a natural limit of the finite-volume maps in the following sense: if $\iota_N : \Z^2 \rightarrow \Lambda_N$ is a local isometry and $j < N-1$, then if $K_j^{\Lambda_N}$ and $K_j^{\Z^2}$ are polymer activities on $\Lambda_N$ and $Z^2$, respectively, such that
\begin{align}
    K_j^{\Lambda_N} ( \iota_N (X) , \varphi) = K_j^{\Z^2} (X, \varphi \circ \iota_N ) , \qquad X \in \cP_j (\Z^2), \;\; \varphi \in \R^{\Lambda_N} ,
\end{align}
then the same holds for $\PM_{j+1}^{\Rae}$.
We refer to \cite[Section~8.2]{dgauss1} for a more precise formulation.
\end{remark}

\subsection{Localisation}
\label{sec:localisation}

The localisation operator extracts the neutral low-order part of $K$ and is used to define the reapportioning element. These terms will be explained throughout this section.

\subsubsection{Charge decomposition}
\label{sec:chargedecomp}

Recall that $F$ is an $\alpha^{-1}$-periodic polymer activity if $F(X,\varphi) = F(X,\varphi + 2\pi\alpha^{-1} \one)$. Then we can consider the charge decomposition
\begin{align} \label{eq:chargedecomp}
    F (\varphi + t \one) = \sum_{q \in\Z} e^{i \alpha q t} \hat{F}_q (\varphi), \qquad
    \hat{F}_q (\varphi) = \frac{\alpha}{2\pi} \int_0^{2\pi/\alpha} e^{-i\alpha qs} F(X,\varphi + s\one) d s
\end{align}
for any $t \in \R$.
The \emph{neutral} component of $F$ is $\hat{F}_0$ and the \emph{charge-$q$} component of $F$ is $\hat{F}_q$.

For the charged terms, we refer to a contraction bound in \cite{dgauss1}.
Recall that $\bar{X}$ was defined in Section~\ref{sec:blockLambda} and $c_h$ is chosen in Section~\ref{sec:choicePar}.

\begin{proposition} \cite[Proposition~6.13]{dgauss1}
\label{prop:qchargeconstraction}
Let $L$ be sufficiently large, $r\in (0,1)$, $h\ge c_h r \alpha$ and $X \in \cS_j$. Then for any $q$-charged polymer activity $F \in \cN_j$ such that $\norm{F(X)}_{h, X, T_j (G_j)} < \infty$,
\begin{align}
    \norm{\E_+ \theta F (X,\varphi) }_{h, X,T_{j+1} (\varphi)}
        \lesssim e^{\alpha |q| h} e^{- r(|q|-1/2)\alpha^2 \Gamma_{j+1}(0)} \norm{F(X)}_{h,X,T_j(G_j)} G_{j+1} (\bar{X},\varphi) .
\end{align}
\end{proposition}

\subsubsection{Polynomial expansion} \label{sec:LocExpansion}

The localisation $\Loc$ is an approximation of the Taylor expansion of order 2 using local polynomials. For an $\alpha^{-1}$-periodic $F \in \cN_j$, $X \in \cS_j$, let
\begin{align}
    \Loc_{X,y}^\bulk F = \one_{y \in X} \big( \Loc_{X,y}^{(0)} + \Loc_{X,y}^{(2)} \big) F
\end{align}
where 
\begin{align}
    & \Loc_{X,y}^{(0)} F = \frac{1}{|X|} \hat{F}_0 (X,0) \label{eq:LocBulkdefiZero} \\
    & \Loc_{X,y}^{(2)} F = \frac{1}{2 |X| |B|} \sum_{x_0 \in B} \sum_{x_1, x_2 \in X^*} \frac{\partial^2 \hat{F}_0 (X,0)}{\partial \varphi_{x_1} \partial \varphi_{x_2}} \langle \nabla \varphi_y , \delta x_1 , \nabla \varphi_y , \delta x_2 \rangle
    \label{eq:LocBulkdefi}
\end{align}
when $\delta x = x - x_0$, $B \in \cB$ is the unique box containing $y$ and
\begin{align}
    \langle \nabla \varphi_y , \delta x_1 , \nabla \varphi_y , \delta x_2 \rangle 
    = \frac{1}{4} \sum_{\mu,\nu \in \hat{e}} (1+ \delta_{\mu, \nu} - \delta_{\mu, -\nu}) \nabla^{\mu} \varphi_y (\delta x_1)^{\mu} \nabla^\nu \varphi_y (\delta x_2)^{\nu} .
\end{align}
($\delta_{\mu,\nu}$ is the Kronecker delta.)
Then we consider
\begin{align}
    \Loc_{X,Y}^\bulk = \sum_{y \in Y} \Loc_{X,y}^\bulk , \qquad \Loc_X = \Loc_{X,X}^\bulk  .
\end{align}

\subsection{$I_+$ and reapportioning element}
\label{sec:reappelemnt}

We now specify the choices of $\Rae$ and $I_+$ in Definition~\ref{defi:PM} using localisation.

\begin{definition} \label{defi:Vplusdefi}
Given $(V_j, K_j^\bulk) \in \cV \times \cN_j^\bulk$, define $\delta u_{j+1}^\bulk$ and $V_{j+1}$ via a relation
\begin{align}
    \delta u_{j+1}^\bulk (b) + (V_{j+1} - \E_+ \theta V_j ) (b) = \sum_{Y \in \cS_j}^{Y \supset b} \Loc_{Y,b}^\bulk \E_{j+1} \theta K^\bulk_j 
    \label{eq:Vplusdefi}
\end{align}
for each $b \in \cB_j$. For $(B,X) \in \cB_{j+1} \times \cP_{j+1}$, take
\begin{align}
    \Rae^\bulk (B, X)
    = 1_{B \in \cB_{j+1} (X)} 1_{X \in \cS_j} \sum_{b \in \cB_j (B)} \sum_{Y \in \cS_j}^{b \in \cS_j} \Loc_{b, Y}^\bulk \E_{j+1} \theta K_j^\bulk ( 1_{\bar{Y}=X} - 1_{B = X} ) .
    \label{eq:kEdefi}
\end{align}
With these definitions and $I_{j'} = \exp(V_{j'})$ (for both $j' \in \{j,j+1\}$), we take
\begin{align}
    K^{\bulk}_{j+1} = \PM^{\Rae^\bulk}_{j+1} 
    [I_j ,  I_{j+1} , \delta u_{j+1}^\bulk , K_j^\bulk] .
\end{align}
\end{definition}

The next lemma verifies the validity of the above choices.

\begin{lemma} \label{lemma:Vplusexplicit}
Suppose $(V_j, K_j) \in \cV \times \cN_j^\bulk$.
\begin{enumerate}
\item $\Rae^\bulk$ given by \eqref{eq:kEdefi} satisfies the key assumption \eqref{eq:keassump} for the reapportioning element.
\item There exists a unique pair $\delta u_{j+1}^\bulk \in \R^{\cB_j}$ and $V_{j+1} \in \cV$ such that \eqref{eq:Vplusdefi} is valid. If $V_{j+1}$ has coordinates $s_{j+1}$ and $\underline{z}_{j+1}$, then they satisfy
\begin{align}
    & | \delta u_{j+1}^\bulk (b) | \lesssim A^{-1} \norm{K_j^\bulk}_{h, A,T_j(G_j)} , \\ 
    & |s_{j+1} - s_j| \lesssim A^{-1} h^{-2} \norm{K_j^\bulk}_{h, A,T_j(G_j)} \\
    & z_{j+1}^{(q)} = L^2 e^{-\frac{1}{2}q^2 \alpha^2 \Gamma_{j+1} (0)} z_j^{(q)} .
\end{align}
\end{enumerate}
\end{lemma}
\begin{proof}
(i) follows from an explicit computation.
(ii) is \cite[Proposition~6.5, Theorem~7.6]{dgauss1}: it implies the unique existence of $\bar{E}$ and $\bar{s}$ satisfying
\begin{align}
    \sum_{Y\in \cS_j}^{Y\supset b} \sum_{y\in b} \big( \Loc^{(0)}_{Y,y} + \Loc^{(2)}_{Y,y} \big) \E\theta K^\bulk_j (Y,\varphi) = \bar{E} |b| + \frac{1}{2} \bar{s} |\nabla \varphi|_b^2 .
\end{align}
with $|\bar{E}| \lesssim A^{-1} \norm{K_j}_{h,A,T_j(G_j)}$ and $|\bar{s}| \lesssim A^{-1} h^{-2} \norm{K_j}_{h,A,T_j(G_j)}$. 
Then, we simply set $\delta u_{j+1}^\bulk (b) = \bar{E} |b| $ and $s_{j+1}-s_j= \bar{s}$.
\end{proof}

\subsection{Linearisation of the progression map}

We briefly discuss the approximate behaviour of the progression map via its linearisation. This estimate will not be used to control the bulk RG map, but it clarifies the mechanism of the stable manifold construction, Proposition~\ref{prop:critpt}.
Also, Lemma~\ref{lemma:LPMbnd} and the linearisation process described here will be used as inputs for estimating the observable RG map later.

As in \cite[Section~7.4]{dgauss1}, the linearisation is obtained from Definition~\ref{defi:PM} by only considering the terms in \eqref{eq:PMdefi} with
\begin{align}
    |X_0|_{j+1} + |\Comp_{j+1} (X_1)| + |\Comp_{j+1} (Z)| \le 1 ,
\end{align}
replacing exponentials by $1$ outside the expectation and replacing exponentials by their first order Taylor expansion inside the expectation, when $I = \exp(V)$ and $I_+ = \exp(V_+)$.
This yields 
\begin{align}
    & \sum_{\bar{Y}=X} 1_{Y \in \Conn_j} \E_+ \theta K (Y,\varphi) - \sum_{B\in\cB_{j+1}(X)} \Rae (B,X) \\
    &\;\;
    + \sum_{b \in \cB}^{\bar{b}=X} \E_{j+1} \theta V (b, \varphi) - \delta u_+ (b) - V_+ (b, \varphi)  + \sum_{B \in \cB_{j+1}}^{B^* = X} \sum_{Z \in \cS_{j+1}}^{Z \supset B} \Rae (B,Z)
    \nonumber
\end{align}
(The expression looks different from \cite[(7.62)]{dgauss1} because $J$ is left unexpanded above.)
The summand of $\sum_{B \in \cB_{j+1}}^{B^* = X}$ vanishes due to \eqref{eq:keassump}, and we are left with the following.

\begin{definition} \label{defi:LPM}
When $I =\exp(V)$ and $I_+ = \exp(V_+)$, the \emph{linearised progression map} is defined by
\begin{align}
    \LPM_{j+1}^{\Rae} [I, I_{+} , \delta u_{+} , K] (X)
    &= \sum_{\bar{Y}=X} 1_{Y \in \Conn_j} \E_{j+1} \theta K (Y) - \sum_{B\in\cB_{j+1} (X)} \Rae (B,X) \nnb
    &\quad\;
    + \sum_{b \in \cB_j}^{\bar{b}=X} \E_+ \theta V (b) - \delta u_+ (b) - V_+ (b, \varphi) .
    \label{eq:LPMdefi}
\end{align}
Under the setting of Definition~\ref{defi:Vplusdefi}, we take
\begin{align}
    \LK^{\bulk}_{j+1} = \LPM^{\Rae^{\bulk}}_{j+1}  [I_j, I_{j+1}, \delta u_{j+1}^\bulk , K^\bulk_j ] .
\end{align}
\end{definition}

For the bulk RG map, a contraction mechanism is established in \cite{dgauss1}, under the the following set of assumptions.

\begin{assumption} \label{asmp:RGasmpBulk}
For sufficiently large constants $c_h,\alpha_0,L_0 \ge 1$, an $L$-dependent parameter $A_0 (L) >0$ and sufficiently small $r >0$, let $\alpha \ge \alpha_0$ (so $\beta \ge \alpha_0^2$), $L\ge L_0$, $A\ge A_0 (L)$ and $h = c_h r \alpha$. 
\end{assumption}

If we take $\etabulk = 0.9999$, by the assumptions along with Proposition~\ref{prop:FRD}(iii),
\begin{align}
    L^{-(2+\etabulk)} > (L / \sqrt{\log L})^{-3}  + \sum_{q\geq 1} e^{\alpha q h} e^{- r (q-1/2) \alpha^2 \Gamma_{j+1} (0)}  .
    \label{eq:etabulkProp}
\end{align}
This gives an important contraction inequality.

\begin{lemma} \label{lemma:LPMbnd}
Assume Assumption~\ref{asmp:RGasmpBulk},
$K_j^\bulk \in \cN_j^\bulk$, $X \in \cP_{j+1}^c$
and let $\LK_{j+1}^\bulk$ be given by Definition~\ref{defi:LPM}. Then
\begin{align}
    \norm{\LK_+^{\bulk} (X)}_{h, X, T_{j+1} (G_{j+1})}
    \lesssim L^{-\etabulk} A^{-|X|_{j+1}} \norm{K_j^\bulk}_{h, A, T_j(G_j)} .
\end{align}
\end{lemma}
\begin{proof}
By \cite[Theorem~7.7(i)]{dgauss1},
we have exactly the same bound only with $L^{-\etabulk}$ replaced by
\begin{align}
    L^2 \alpha_{\rm Loc} := C L^2 \Big[ (L / \sqrt{\log L})^{-3} + \sum_{q\geq 1} e^{\alpha q h} e^{-r (q-1/2) \alpha^2 \Gamma_{j+1} (0)} \Big]
\end{align}
for some constant $C>0$. But by \eqref{eq:etabulkProp}, this is bounded by $O( L^{-\etabulk} )$. 
\end{proof}

\section{Bulk RG flow} \label{sec:bulkRGflow}

The bulk RG flow considers 
$
    Z_{j}^{\bulk} (\varphi) := \E_{\Gamma_{\le j}} [ Z_0 (\varphi ) ]
$.
We seek for an inductive definition of the RG coordinates $(u_j^\bulk, V_j, K_j^\bulk)_{j\ge 0}$ satisfying the polymer expansion
\begin{align}
    Z_j^{\bulk} = e^{u_j^\bulk (\Lambda)} (e^{V_j} \circ K_j^{\bulk}) (\Lambda) .
    \label{eq:ZjbulkPlyExp}
\end{align}
The collection of RG coordinates forms the bulk RG flow. The stability of the flow is guaranteed for a specific choice of initial conditions. 

\subsection{Bulk RG map}

The bulk RG map is the map
\begin{align}
    \Phi_{j+1}^\bulk :
    (V_j, K^\bulk_j) \mapsto (\delta u^\bulk_{j+1} , V_{j+1} , K^\bulk_{j+1})
\end{align}
where the choices $\delta u_{j+1}^\bulk$, $V_{j+1}$, $\Rae^{\bulk}$ and $K^\bulk_{j+1}$ are given by Definition~\ref{defi:Vplusdefi}.
By Proposition~\ref{prop:PMproperty}, they satisfy
\begin{align}
    \E_{j+1} \theta Z_j^{\bulk} = e^{(u_j^\bulk + \delta u_{j+1}^\bulk) (\Lambda)} (e^{V_{j+1}} \circ K_{j+1}^\bulk)
    \label{eq:EthetaZjbulk}
\end{align}
when $N < \infty$ and we may take
\begin{align}
    u_{j+1}^\bulk = u_j^\bulk + \delta u_{j+1}^\bulk .
\end{align}
Thus, the bulk RG map provides an inductive definition of the RG coordinates that expand $Z_j^\bulk$.

A flow generated by $(\Phi^\bulk_{j+1})_{j\ge 0}$ is called an RG flow.
Note that, although the polymer expansion \eqref{eq:ZjbulkPlyExp} is convergent only when $N < \infty$, the RG flow itself can also be defined directly on $\Lambda = \Z^2$.

\begin{definition}
Let $N \in \N \cup \{\infty\}$. On $\Lambda_N$, a \emph{bulk RG flow} of length $k \le N$ is a sequence $(V_j, K^\bulk_j)_{j=0}^k \in \prod_{j=0}^k \cV_j \times \cN_j^\bulk$ satisfying
\begin{align}
    (\delta u^\bulk_{j+1},V_{j+1},K^\bulk_{j+1}) = \Phi_{j+1}^\bulk (V_j, K^\bulk_j) , \qquad j\le k-1 
\end{align}
for some $\delta u^\bulk_{j+1} \in \R^{\cB_{j+1}}$.
We allow $k = N = \infty$, in which case $(V_j, K^\bulk_j)_{j=0}^\infty$ is called an infinite bulk RG flow.
\end{definition}

The existence of the RG flow is not always guaranteed because we may have $K_{j+1}^{\bulk} \not\in \cN_{j+1}^{\bulk}$---in fact, it may even happen that the progression map Definition~\ref{defi:PM} is not convergent. This is discussed further in Section~\ref{sec:cripticalPoint}.

\subsection{Critical point}
\label{sec:cripticalPoint}

An infinite bulk RG flow can be constructed subject to a carefully chosen initial condition using a theory of dynamical systems. Together with an application of the Intermediate Value Theorem, this implies the existence of a critical point of the DG model. We do not show this argument that is already presented in \cite{dgauss1, dgauss3}. 
In the statement below, recall that $\underline{\tilde{z}}$ was defined by \eqref{eq:tildeunderlinez}.

\begin{proposition} \label{prop:critpt}
Let $N \in \N \cup \{\infty\}$ and assume Assumption~\ref{asmp:RGasmpBulk}. Consider the RG map on $\Lambda_N$.
Then there exists $s_0^c (\beta) = O(e^{-\frac{1}{4} \gamma \beta})$ such that the bulk RG flow of length $N$ exists with initial conditions $(s_0, \underline{z}_0) = (s_0^c (\beta) , \underline{\tilde{z}})$ and $K^\bulk_0 (X) = 1_{X = \emptyset}$
and the parameter $s = s_0^c (\beta)$ (implicit inside the definition of the covariance $\bar{C}(s,m^2)$, thus inside $\Gamma_j$'s).
Moreover, the RG coordinates satisfy
\begin{align}
    \max\big\{ \norm{V_j}_{\cV_j} , \;\; \norm{K^\bulk_j}_{j, \bulk}, \;\;
    |\delta u_{j}^\bulk (b)|
    \big\}
    \lesssim_L e^{-\frac{1}{4} \gamma \alpha^2 } L^{-\eta j}
    \label{eq:VKjbulkdecay}
\end{align}
for $\eta = 0.999 \in (0.99, \etabulk)$ and any $b \in \cB_j$. 
\end{proposition}

\begin{proof}
When $N < \infty$, the statement (except the bound on $\delta u_j^\bulk$) is \cite[Proposition~8.1]{dgauss1}, 
and we can use Remark~\ref{remark:infvolPM} for the statement on the infinite volume ($N=\infty$) once we obtain the finite-volume results.
The bound on $\delta u_j^\bulk$ follows from Lemma~\ref{lemma:Vplusexplicit} and the bound on $K_j^\bulk$.
\end{proof}

We take the initial condition of Proposition~\ref{prop:critpt} as the input of all the following analyses.

\begin{assumption} \label{asmp:critical} 
In addition to Assumption~\ref{asmp:RGasmpBulk}, we take $s_0 = s = s_0^c (\beta)$ as given by Proposition~\ref{prop:critpt}.
With initial condition $(s_0, z_0)= (s_0^c, \tilde{\underline{z}})$ and $K_0^\bulk (X) = \one_{X = \emptyset}$, the infinite bulk RG flow $(V_j, K_j^\bulk)_{j\ge 0}$ and $(u_{j}^\bulk)_{j\ge 1}$ get fixed.
\end{assumption}

Assumption~\ref{asmp:RGasmpBulk} fixed $\beta_0$, $L_0$ and $A_0$, but we may need to choose them larger in the RG analysis of the observable. We do not always make this explicit.
We will assume Assumption~\ref{asmp:critical} in the rest of the paper, often not mentioned explicitly.

\section{Domain of observable RG}
\label{sec:domObsRG}

In this section, we set up the spaces of the observable RG coordinates.
The bulk RG flow controlled the denominator of \eqref{eq:lemmasinGordapprox}.
For the numerator, we consider
\begin{align}
    Z_j^\obs (\varphi) := \E_{\le j} [e^{(d^* f, \zeta)} Z_0^\obs ( \varphi + \zeta )] ,
    \label{eq:Zjobsdefi}
\end{align}
cf. \eqref{eq:ZNfDefi} and \eqref{eq:Zjbulkdefi}. The observable RG map describes the evolution of RG coordinates in the expansion
\begin{align}
    Z_j^\obs (\varphi) = e^{u_{j}^\obs (\Lambda) + D_j (\Lambda, \nabla\varphi) } (I_j \circ_j K_j^{\obs} ) (\varphi + g_{j}) 
    \label{eq:Zjobsexpansion} .
\end{align}
where $u_j^\obs (\Lambda) \in \R$, $I_j = \exp (V_j)$ and $D_j$ is a linear function of $\nabla \varphi$ defined in Section~\ref{sec:obsRGcoords}.
The function $g_{j} \in \R^{\Lambda}$ is defined in Section~\ref{sec:flowdepNorm}.
We will also use $\Omega_j (d\varphi)$ symmetrising $D_j$, defined in Section~\ref{sec:obsRGcoords} so that $\Omega_j (d\varphi) = D_j (\Lambda, \nabla\varphi)$. This gives
\begin{align}
    Z_j^\obs (\varphi) = e^{u_{j}^\obs (\Lambda) + \Omega_j (d\varphi) } (I_j \circ_j K_j^{\obs} ) (\varphi + g_j) . 
    \label{eq:Zjobsinter1}
\end{align}
Then in Section~\ref{sec:reblPert}, we explain the reblocking operation that removes $g_j$.
In the rest of the section, we describe how to construct the reapportioning element for the observable RG map.

\subsection{Assumption on the observable}

The nature of the observable determines the space of relevant RG coordinates.
We will only need to consider observables $f \in \Lambda^{(1)} (\Z)$ arising from paths on the dual lattice as described in Section~\ref{sec:funcObs}, with the additional assumption on the upper bound on the box dimension. 

\begin{figure}[h]
\begin{centering}
\begin{tikzpicture}[scale=0.6, x=1cm, y=1cm, >=stealth]
\tikzset{
        mid arrow/.style={
            postaction={decorate, decoration={
                markings,
                mark=at position 0.5 with {\arrow[scale=1.5, thick]{>}}
            }}
        },
        mid arrow reverse/.style={
            postaction={decorate, decoration={
                markings,
                mark=at position 0.5 with {\arrow[scale=1.5, thick]{<}}
            }}
        }
    }

    \draw[ultra thick, mid arrow] (0.5, 0.5) -- (0.5, 3.5);

    \begin{scope}[shift={(2.5, 4.5)}] 
        \draw[ultra thick, mid arrow reverse] (0,2) -- (2,2); 
        \draw[ultra thick, mid arrow reverse] (2,2) -- (2,0); 
        \draw[ultra thick, mid arrow reverse] (2,0) -- (0,0);         
        \draw[ultra thick, mid arrow reverse] (0,0) -- (0,2);         
    \end{scope}

    \begin{scope}[shift={(6, 0.5)}] 
        \draw[ultra thick, mid arrow] (0, 0) -- (0, 3);    
        \draw[ultra thick, mid arrow] (0, 3) -- (2.5, 3);  
        \draw[ultra thick, mid arrow] (2.5, 3) -- (2.5, 5.5); 
        \draw[ultra thick, mid arrow] (2.5, 5.5) -- (5.5, 5.5);    
        \draw[ultra thick, mid arrow] (5.5, 5.5) -- (5.5, 0);      
        \draw[ultra thick, mid arrow] (5.5, 0) -- (0, 0);          
    \end{scope}
\end{tikzpicture}
\caption{An example of $\bP$ satisfying Assumption~\ref{asmp:1Dobs}}
\label{fig:ExampleP}
\end{centering}
\end{figure}
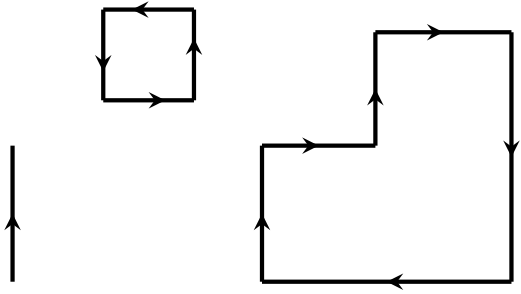

\begin{assumption} \label{asmp:1Dobs}
Let $\bP$ be a path on $\vec{E} (\Lambda^\dual)$ that is a finite collection of self-avoiding walks or self-avoiding loops that do not intersect each other. Assume that 
\begin{align}
    \big| \big( [x, x+ \ell ] \times [y , y+ \ell] \big) \cap \bP \big| \le 100 \ell
\end{align}
for any $x,y \in \Lambda^\dual$ and $\ell > 0$, where $|\cdot |$ is the length of the path considered as a subset of $\R^2$, thus in particular, the box dimension has an upper bound of 1. 
(The constant $100$ is arbitrary.)
Also, there exists $j_{\bP} \ge 0$ such that, for $j \le j_{\bP}$ and $b \in \cB_j$
\begin{align}
    b \cap \bP \neq \emptyset 
    \qquad \text{implies} \qquad
    |b^* \cap \bP | \ge L^{j} /100 ,
    \label{eq:bMinIntersect}
\end{align}
while $|\bP| \le 100 L^{j_{\bP}+1}$. 

Then we take
\begin{align}
    f = 2 \pi \alpha^{-1} \kn (\one_{\bP})^\dual
    \label{eq:fassump}
\end{align}
for $\kn \in \Z$ and $(\one_{\bP})^\dual \in \Lambda^{(1)} (\R)$ is the characteristic function of the collection of edges orthogonal to those of $\bP$, obtained by rotating each edge by angle $\pi/2$ in the counter-clockwise direction.
\end{assumption}

\subsection{Effective observable Hamiltonian}
\label{sec:obsRGcoords}

For the observable flow, we lose the symmetry under $\varphi \mapsto -\varphi$, and we need an extended effective Hamiltonian reflecting this.
This requires
\begin{align}
    & D_j (b, \nabla\varphi) = \sum_{x \in b,\, \mu\in\hat{e}} d_j^{\mu} (x) \nabla^{\mu} \varphi (x) ,
    \quad 
    \supp (d_j^{\mu} ) \subset \supp (d^* f) 
    \label{eq:Djform}
    \\
    & \tilde{D}_j (b, \nabla \varphi) = \sum_{x \in b,\, \mu\in\hat{e}} \tilde{d}_j^{\mu} (x) \nabla^{\mu} \varphi (x) , 
    \quad 
    \supp (\tilde{d}_j^\mu) \subset \cup \{ b' \in \cB_{j-1} : (b')^{***} \cap \supp (d^* f) \neq \emptyset  \}
    \label{eq:DjformTilde}
\end{align}
for $b\in \cB_j$. (The triple small set neighbourhood $(b')^{***}$ is taken at scale $j-1$, not $j$.) 
Their norms are defined as follows. 

\begin{definition}
We denote $\cD$ for the set of $D_j$ of the form \eqref{eq:Djform}
and $\tilde{\cD}_j$ for the set of $\tilde{D}_j$ of form \eqref{eq:DjformTilde}. 
For $D_j \in \cD$ and $\tilde{D}_j \in \tilde{\cD}_j$, let
\begin{align}
    & \norm{D_j}_{\cD_j} = A \max_{x \in \Lambda, \; \mu \in \hat{e}} |d_j^{\mu} (x)| \\
    & \norm{\tilde{D}_j}_{\tilde{\cD}_j} = A L^{j-1} \max_{x \in \Lambda, \; \mu \in \hat{e}} |\tilde{d}_{j}^\mu (x)| .
\end{align}
\end{definition}
(The factor $L^{j-1}$ in the second line is due to the fact that $\tilde{D}_j$ is supported on a band on width $O(L^{j-1})$.)

\begin{lemma} \label{lemma:VjDjbnd}
For any $b \in \cB_j$ and $(V_j, D_j) \in \cV \times \cD$,
\begin{align}
    \norm{V_j (b,\varphi)}_{h,b,T_j (\varphi)} &\le A^{-1} \norm{V_j}_{\cV_j} (h + \norm{\nabla_j \varphi}_{L^{\infty} (b^*)} )^2  , \\
    \norm{D_j (b, \nabla\varphi)}_{h,b,T_j(\varphi)} &\le A^{-1} L^{(j_{\bP}-j) \wedge 0} \norm{D_j}_{\cD_j} (h + \norm{\nabla_j \varphi}_{L^{\infty} (b)} ) , \\
    \norm{\tilde{D}_j (b, \nabla\varphi)}_{h,b,T_j(\varphi)} &\le A^{-1} \norm{\tilde{D}_j}_{\tilde\cD_j} (h + \norm{\nabla_j \varphi}_{L^{\infty} (b)} ) .
\end{align}
\end{lemma}
\begin{proof}
We consider the polynomial part and the periodic part of $V_j$ separately.
The bound on the periodic part $\sum_{q=1}^{\infty} z_j^{(q)} \cos ( q\alpha \varphi_x)$ (inside \eqref{eq:Vbulkjdefi}) is given in \cite[Lemma~7.4]{dgauss1}.
For the bounds on the polynomial part, we only need the simple bound
\begin{align}
    \norm{|\nabla \varphi_z|^n}_{h,b,T_j (\varphi)} \lesssim L^{-nj} \big( h + |\nabla_j \varphi_z| \big)^2 
    \label{eq:lemmaVjDjbnd2}
\end{align}
for both $n \in \{1,2\}$, which follows from the definition of the norm. The $n=1$ case also implies the bounds on $D_j$ and $\tilde{D}_j$.
\end{proof}

Since $D_j$ is a linear form on $\nabla \varphi$, we can also view it as an inner product of $d \varphi$ and a 1-form. This can be made more explicit after symmetrisation.
Our RG formulation is adapted to the use of $D_j$, but the calculus of Part II is adapted to the use of $\Omega_j$.

\begin{definition} \label{def:omega}
Given $D_j  \in \cD_j$, we define an 1-form $\omega_j \in \Lambda^{(1)} (\R)$ and a dual of 1-forms $\Omega_j (d\varphi)$ via
\begin{align}
    \omega_j ( x, x+\mu ) = d_j^\mu (x) - d_j^{-\mu} (x+\mu)  , 
    \qquad \Omega_j (d \varphi) = (\omega, d \varphi) = \sum_{e \in E(\Lambda)} \omega_j (e) (d\varphi) (e) .
\end{align}
Define
\begin{align}
    \supp_+ (f) := \{ e = \{x,y\} \in E(\Lambda) : e\cap e' \neq \emptyset \text{ for some } e' \in \supp (f) \} ,
    \label{eq:suppPlusDefi}
\end{align}
so that $\supp (\omega_j) \subset \supp_+ (f)$.
\end{definition}

We will make an initial choice $D_0 = 0$, so  $\Omega_0 = 0$ and $\omega_0 = 0$.

\subsection{Flow-dependent norm} \label{sec:flowdepNorm}

Before we define the space of observable remainder coordinates $K^\obs$, we need to define a norm that controls the fluctuation due to the line observable. The fluctuation is given by
\begin{align}
    g_j = \begin{cases} 
        \gamma d^* f  & (j =0) \\
        \Gamma_j d^* (f + \omega_{j-1}) & (j \ge 1) 
    \end{cases}
    \label{eq:gjDefi}
\end{align}
---we are defining the RG coordinates inductively, and we can assume that $\omega_{j-1}$ is known.
We introduce regulators and norms incorporating these fluctuations. 
In particular, the norm is dependent on the RG coordinate $\omega_{j-1}$.

\begin{definition} \label{defi:Gjobsdefi}
For $j\ge 0$, let $\kappa_0$ be as in Section~\ref{sec:regulators} and
\begin{align}
    k_j = 2 - 2^{-j}, \qquad \kappa_j = k_j \kappa_0   .
    \label{eq:kappajdefi} 
\end{align}
Given $c_w > 0$ as in Section~\ref{sec:choicePar},
\begin{align}
    G_j^\dagobs (X,\varphi) = \Big[ e^{\sum_{b \in \cB_j (X)} \sup \big\{ \norm{\nabla_j (\varphi + t_b g_j)}^2_{L^2_j (b)} + c_2 \norm{\nabla_j (\varphi + t_b g_j)}_{L_j^2 (b \cap \partial X)} \, : \, t_b \in [0,1] \big\}} \nnb
    \times e^{ W_j^\dagobs (X, \nabla_j^2 \varphi)} e^{2 c_w  \sum_{b \in \cB_j (X)} \norm{g_j}^2_{C_j^2 (b^*)}}  \Big]^{\kappa_j}
    \label{eq:Gjdaggerdefi}
\end{align}
for $X \in \cP_j$ and $\varphi \in \R^{\Lambda}$, where
\begin{align}
    W_j^\dagobs (X,\varphi) 
    = \sum_{b \in \cB_j (X)} \sup_{t'_b \in [0,1]} W_j (b, \nabla_j^2 (\varphi + t'_b g_j)) .
\end{align}
(Recall that $W_j$ was defined in Section~\ref{sec:regulators}.)

For $F \in \cN_j$, we define norms 
\begin{align}
    \norm{F}_{j,\obs} 
        &:= \norm{F}_{h, A/2, T_j (G_j^{k_j}) } 
        := \sup_{X \in \Conn_j} \sup_{\varphi\in \R^{\Lambda}} (A/2)^{|X|_j} \frac{\norm{F(X,\varphi)}_{h, X, T_j (\varphi)}}{G_j^{k_j} (X,\varphi)} ,  \label{eq:normjsigma}
        \\
    \norm{F}_{j,\dagger} 
        &:= \norm{F}_{h, A/4, T_j (G_j^{\dagger}) } 
        := \sup_{X \in \Conn_j} \sup_{\varphi\in \R^{\Lambda}} (A/4)^{|X|_j} \frac{\norm{F(X,\varphi)}_{h, X, T_j (\varphi)}}{G_j^{\dagger} (X,\varphi)} , 
        \\
    \norm{F}_{j,\ddagger} 
        &:= \norm{F}_{h, A, T_j (G_j^{k_{j-1}}) } 
        := \sup_{X \in \Conn_j} \sup_{\varphi\in \R^{\Lambda}} (A)^{|X|_j} \frac{\norm{F(X,\varphi)}_{h, X, T_j (\varphi)}}{G_j^{k_{j-1}} (X,\varphi)} ,    
\end{align}
cf. Section~\ref{sec:NrmdSpcs}.
\end{definition}

Aside from the inclusion of $g_j$, the observable norm introduces two features that differ from the bulk one. The first is the evolving exponent $\kappa_j$. This is required because of repeated removal of linear shifts along the RG. This effect appears in a definition \eqref{eq:Kobsplusdefi} and its bound is treated in the proof of Lemma~\ref{lemma:Kobsplusestimate}.
The second is the changing large set regulator $A/2$ and $A/4$. They are used to cancel the growth introduced by the reblocking operation of Section~\ref{sec:reblPert}. Still, they do not propagate along the scales because the large-set regulator always improves due to Proposition~\ref{prop:largesetcont}.

We can define related spaces, with an additional compatibility condition.

\begin{definition} \label{defi:cNobsdefi}
We let 
\begin{align}
    \cN^\obs_j &= \{ F \in \cN_j : \text{$F$ is $\alpha$-periodic and } \norm{F}_{j, \obs } < \infty \} , 
    \\  
    \cN^\dagobs_j &= \{F \in \cN_j : \text{$F$ is $\alpha$-periodic and } \norm{F}_{j, \dagger} < \infty \} ,
    \\
    \cN^\ddagger_j &= \{F \in \cN_j : \text{$F$ is $\alpha$-periodic and } \norm{F}_{j, \ddagger} < \infty \}
\end{align} 
We say that $F_1 , F_2 \in \cN_j$ are \emph{compatible} if $F_1 (X) = F_2 (X)$ whenever $X \in \cP_j$ satisfies $X^{**} \cap \supp (d^*f ) = \emptyset$. 
Then we define
\begin{align}
    \underline{\cN}_j^\obs 
        &= \{ (F^\bulk, F^\obs ) \in \cN_j^\bulk \times \cN_j^\obs  : \text{$F^\bulk$ and $F^\obs$ are compatible} \} ,
        \\
    \underline{\cN}_j^\dagobs 
        &= \{ (F^\bulk, F^\dagobs) \in \cN_j^\bulk \times \cN_j^{\dagobs} : \text{$F^\bulk$ and $F^\dagobs$ are compatible} \} ,
        \\
    \underline{\cN}_j^\ddagger 
        &= \{ (F^\bulk, F^\ddagger) \in \cN_j^\bulk \times \cN_j^{\ddagger} : \text{$F^\bulk$ and $F^\ddagger$ are compatible} \}    
\end{align}
\end{definition}

\subsubsection{Strong regulator}

In order to control $\exp( V (\varphi) )$, we need a \emph{strong regulator} defined by
\begin{align}
    G^{\str}_{j} (X,\varphi) = e^{c_w \kappa_0 w_j (X,\varphi)} , \qquad
    w_j (X,\varphi) = \sum_{b \in \cB_j (X)} \max_{n=1,2} \norm{\nabla^n_j \varphi}^2_{L^\infty (b^*)} .
    \label{eq:GstrXvphiDefi}
\end{align}
The related observable version is
\begin{align}
    G^{\str,\dagobs}_{j} (X,\varphi) &= e^{c_w \kappa_0 w_j (X,\varphi)} , \qquad
    w_j^\dagobs (X,\varphi) = \sum_{b \in \cB_j (X)} \sup_{t_b \in [0,1]}\max_{n=1,2} \norm{\nabla^n_j ( \varphi + t_b g_j ) }^2_{L^\infty (b^*)} .
\end{align}
Observable regulators are related by a domination inequality, which is actually a corollary of lattice functional inequalities of \cite[Appendix~A]{dgauss1}---see item (iii) of the next lemma, for example.

\begin{lemma} \label{lemma:Gjobsproperties}
Suppose $\norm{D_{j-1}}_{\cD_{j-1}} \lesssim A \alpha^{-1}$ and $\alpha$ is sufficiently large.
If $c_2 , c_w >0$ are sufficiently small, the following hold for any disjoint $X,Z \in \cP_j$.
\begin{enumerate}
\item $G_j^\dagger \ge G_j^{k_j} \ge G_j^{k_{j-1}} \ge G_j$.
\item $G_j^\dagger (X,\varphi) = \prod_{Y \in \Comp_j (X)} G_j^\dagger (Y,\varphi)$.

\item $G_j^{\str, \dagger} (X,\varphi) G_j^{\dagger} (Z,\varphi) \le  G_j^{\dagger} (X \cup Z , \varphi)$ 
and $G_j^{\str}  (X,\varphi) G_j^{k'} (Z,\varphi) \le  G_j^{k'} (X \cup Z , \varphi)$ for either $k' \in \{k_j, k_{j-1} \}$.

\item $\E_{j+1} \theta G_j^{\dagger} (X,\varphi) \le 3^{|X|_j} G_{j+1}^{k_{j}} (X,\varphi)$.
\end{enumerate}
\end{lemma}
\begin{proof}
First, we note that, Lemma~\ref{lemma:gjbound} implies $\alpha \norm{g_j}_{C_j^2} \lesssim L$. By this bound, (i)--(iv) are basically \cite[Proposition~3.4.]{dgauss2}, just differing in constants. (The $\Psi$-regulator \cite[Definition~3.2]{dgauss2} has exactly the same form as $G_j^{\dagger}$.) We only explain the first part of item (iii) for clarification.

To control $w_j^{\str, \dagger}$, we consider arbitrary collections $(t_b)_{b \in \cB_j (X)}$ and $(t'_b)_{b \in \cB_j (X)}$ taking values in $[0,1]$. Definition of $w_j$ gives
\begin{align}
    \sum_{b \in \cB_j (X)} w_j (b, \varphi + t_b g_j)
    &\le 2 \sum_{b \in \cB_j (X)} \max_{n=1,2} \big( \norm{\nabla_j^n (\varphi + t'_b g_j )  }^2_{L^\infty (b^*)} + (t_b - t'_b)^2 \norm{\nabla_j^n g_j}^2_{L^\infty (b^*)} \big) \nnb
    &\le 2\sum_{b \in \cB_j (X)} \max_{n=1,2} \big( \norm{\nabla_j^n (\varphi + t'_b g_j )  }^2_{L^\infty (b^*)} + 2 \sum_{b \in \cB_j (X)} \norm{g_j}_{C_j^2 (b^*)}^2 .
    \label{eq:lemmaGjobsproperties2}
\end{align}
We use functional inequalities \cite[(3.19)]{dgauss2}:
\begin{align}
    \norm{\nabla_j \varphi}^2_{L^\infty (b^*)} \le 2 \norm{\nabla_j \varphi}^2_{L_j^2 (b)} + C \norm{\nabla_j^2 \varphi}^2_{L^\infty (b^*)}
\end{align}
for some constant $C>0$. Plugging this into \eqref{eq:lemmaGjobsproperties2} and taking supremum over $(t_b)_b$,
\begin{align}
    w_j^\dagobs (X,\varphi) \le \sum_{b \in \cB_j (X)} \Big( 2\norm{g_j}_{C_j^2 (b^*)}^2 + C' \norm{\nabla_j (\varphi + t'_b g_j)}^2_{L_j^2 (b)} \Big) + C' W_j^\dagobs (X, \nabla_j^2 \varphi)
\end{align}
for some constant $C'>0$. 
Together with Lemma~\ref{lemma:traceExp}, we obtain
\begin{align}
    & c_w w_j^\dagger (X,\varphi) + 
    \sum_{b \in \cB_j (Y)} \Big( \norm{\nabla_j (\varphi + t_b g_j)}_{L_j^2 (b \cap Y)}^2 + c_2 \norm{\nabla_j (\varphi+ t_b g_j)}^2_{L^2_j (b\cap \partial Y)} + W_j^\dagger (Y,\nabla_j^2 \varphi) \Big)
    \nnb
    & \le 2 c_w \sum_{b \in \cB_j (X)} \norm{g_j}^2_{C_j^2 (b^*)}
    + c_w C' \sum_{b \in \cB_j (X)}  \norm{\nabla_j (\varphi + t'_b g_j)}^2_{L_j^2 (b)}
    \nnb
    &\quad + \sum_{b \in \cB_j (Y)} \norm{\nabla_j (\varphi + t_b g_j)}^2_{L^2_j (b\cap Y)} \nnb
    &\quad + c_2 \sum_{b \in \cB_j (X\cup Y)} \norm{\nabla_j (\varphi + t_b g_j)}^2_{L_j^2 (b\cap \partial (X\cup Y))} + c_2 C \sum_{b \in \cB_j (X)} \norm{\nabla_j (\varphi + t_b g_j )}^2_{L^2_j (b)} \nnb
    &\quad + (c_w C' + c_2 C) W_j^\dagger (X,\nabla_j^2 \varphi) + W_j^\dagger (Y,\nabla_j^2 \varphi) .
\end{align}
for some $(t_b)_{b\in \cB_j (X)} \in [0,1]^{\cB_j (X)}$. Since our choice of $t'_b$ was arbitrary, we can take $t'_b = t_b$. Then for sufficiently small $c_w$ and $c_2$ such that $c_w C' + c_2 C \le 1$, we get
\begin{align}
    &\le \sum_{b \in \cB_j (X\cup Y)}\Big( \norm{\nabla_j (\varphi + t_b g_j)}^2_{L^2_j (b)} + c_2 \norm{\nabla_j (\varphi + t_b g_j)}^2_{L^2_j (b \cap \partial (X\cup Y))} \Big) \nnb 
    &\qquad + W_j^\dagger (X\cup Y, \nabla_j^2 \varphi)
    + 2 c_w \sum_{b \in \cB_j (X)} \norm{g_j}^2_{C_j^2 (b^*)} ,
\end{align}
as desired.
\end{proof}

The following lemma is used to expand a boundary estimate on a polymer to a larger polymer.

\begin{lemma} \label{lemma:traceExp}
Let $X,Y \in \cP_j$ be disjoint polymers, and $(t_b)_{b \in \cB_j (Y)}$ be any collection of numbers in $[0,1]$ and $g\in \R^\Lambda$. Then there exists a collection $(t_{b'})_{b' \in \cB_j (X)} \subset [0,1]^{\cB_j (X)}$ such that, for any $\varphi \in \R^{\Lambda}$, 
\begin{align}
    \sum_{b\in \cB_j (Y)} \norm{\nabla_j (\varphi + t_b g)}_{L^2_j (b \cap \partial Y)}^2
    \le \sum_{b\in \cB_j (X\cup Y)} \norm{\nabla_j (\varphi + t_b g)}_{L^2_j (b \cap \partial (X \cup Y))}^2 \nnb
    + O(1) \sum_{b \in \cB_j (X)} \Big( \norm{\nabla_j (\varphi + t_b g)}_{L^2_j (b)}^2 + \norm{\nabla_j^2 (\varphi + t_b g)}_{L^\infty (b)}^2 \Big) .
    \label{eq:traceExp}
\end{align}
\end{lemma}
\begin{proof}
We use the discrete trace theorem \cite[Corollary~A.2]{dgauss1}: for any $\psi \in \R^{\Lambda}$, $b \in \cB_j$ and its outer boundary $\partial^+ b$,
\begin{align}
    L^{-j} \sum_{x \in \partial^+ b} |\nabla \psi (x)|^2 \le C \Big( L^{-2j} \sum_{x \in b} |\nabla \psi (x)|^2 + \norm{\nabla^2 \psi}^2_{L^\infty (b)} \Big)
    \label{eq:discTrace}
\end{align}
for some $C>0$. We apply this inequality inductively. To be specific, we choose any sequence $Z_0, \cdots, Z_n \subset \cP_j (X\cup Y)$ such that (i) $Z_0 = Y$, $Z_n = X\cup Y$, (ii) $Z_i \subset Z_{i+1}$ and (iii) $b_{i+1} := Z_{i+1} \backslash Z_i \in \cB_j$.

If $b_{i+1}$ shares an edge with $b' \in \cB_j (Z_i)$, then we take $t_{b_{i+1}} = t_{b'}$ and if $b_{i+1}$ does not share an edge with $Z_i$, then take $t_{b_{i+1}} = 0$. Then application of \eqref{eq:discTrace} with $b = b_{i+1}$ and $\psi = \varphi + t_b g$ gives
\begin{align}
    & \sum_{b' \in \cB_j (Z_j)} \norm{\nabla_j (\varphi + t_{b'} g)}^2_{L^2_j (b' \cap \partial Z_j)} 
    \le
    \sum_{b' \in \cB_j (Z_{j+1})} \norm{\nabla_j (\varphi + t_{b'} g)}^2_{L^2_j (b' \cap \partial Z_j)} \nnb
    & \qquad\qquad\qquad 
    + C \Big( \norm{\nabla_j (\varphi + t_{b_{i+1}} g)}_{L^2_j (b_{i+1})}^2 + \norm{\nabla_j^2 (\varphi + t_{b_{i+1}} g)}_{L^\infty (b_{i+1})}^2 \Big) .
\end{align}
Thus we see that \eqref{eq:traceExp} is verified inductively when $X$ is replaced by $Z_i \backslash Y$.
\end{proof}

Finally, the bound on $g_j$ is verified.

\begin{lemma} \label{lemma:gjbound}
Suppose $\norm{D_{j-1}}_{\cD_{j-1}} \lesssim A L^{(j-j_{\bP}) \vee 0} \alpha^{-1}$. Then $\alpha \norm{g_j}_{C_j^2} \lesssim L$. In particular, $h \norm{g_j}_{C_j^2} \lesssim L$. 
\end{lemma}
\begin{proof}
Since $j=0$ case is trivial, we only have to take care of $j\ge 1$.
Since $\omega_j$ is a symmetrisation of $D_j$, the bound on $D_{j-1}$ also implies $\norm{\omega_{j-1}}_{L^\infty} \lesssim \alpha^{-1}$, and $\omega_{j-1}$ is supported on $\supp_+ (f)$. Thus, it is actually more convenient to prove a slightly generalised statement, that 
$g' = \Gamma_j d^* f'$ satisfies $\norm{g'}_{C_j^2} \lesssim O_L (1)$
whenever $f' \in \Lambda^{(1)} (\R)$ is such that $d^* f'$ is supported on $\supp_+ (f)$ and $\norm{f'}_{L^\infty} \lesssim L^{(j-j_{\bP}) \vee 0}$.

To prove the generalised statement, observe that
\begin{align}
    \Gamma_j d^* f' (x) = \sum_{e = (y,z) \in \supp_+ (f)} \big( \Gamma_j (x-z) - \Gamma_j (x-y) \big) f' (e) .
\end{align}
But since $\big| \Gamma_j (x-z) - \Gamma_j (x-y) \big| \lesssim \norm{\nabla \Gamma_j}_{L^\infty} \lesssim L^{1-j}$ and there are only $O(L^j)$ number of edges $e \in \supp_+ (f)$ such that $\Gamma_j (x-z)$ is non-vanishing, we see that
\begin{align}
    |\Gamma_j d^* f' (x)| \lesssim L \norm{f'}_{L^\infty} \lesssim L^{1+ (j-j_{\bP}) \vee 0} ,
\end{align}
which shows the statement for $j \le j_{\bP}$. When $j > j_{\bP}$, we see that there are only $O(L^{j_{\bP}})$ number of edges $e \in \supp_+ (f)$, so 
\begin{align}
    |\Gamma_j d^* f' (x) | \lesssim L^{1-j} L^{j_{\bP}} \norm{f'}_{L^\infty} \lesssim L .
\end{align}
\end{proof}

\subsection{Reblocking the perturbation}
\label{sec:reblPert}

Before applying a renormalisation group step, we can transfer the perturbation $g$ from the effective Hamiltonian $V+D$ to $K^{\obs}$ by applying a reblocking operation. This results in $K^{\dagger}$.
We first give a bound on $DV$ and define $K^{\dagger}$ in Lemma~\ref{lemma:reblockObsg}.

\begin{lemma} \label{lemma:DVbnd}
Suppose $h = c_h r \alpha$ for a sufficiently small $r >0$. Then for $\norm{V_j}_{\cV_j} < \infty$, $g\in \R^{\Lambda}$ and $b \in \cB_j$,
\begin{align}
    \norm{D V_j (b,\varphi;g)}_{h, b, T_j (\varphi)}
    \lesssim
    A^{-1} h^{-1} \norm{V_j}_{\cV_j} \big( h + \norm{\nabla_j \varphi}_{\ell^\infty (b)} \big)^2
    \norm{g}_{C_j^2}
\end{align}
\end{lemma}
\begin{proof}
We consider the decomposition $V = V_j^G + V_j^Q$ of \eqref{eq:Vbulkjdefi} and bound them separately.

Since $D V_j^G$ is a degree 1 polynomial of $\varphi$, we have
\begin{align}
    \norm{D V_j^G (b,\varphi ;g)}_{h,b,T_j (\varphi)}
    &= \sum_{n=0,1} \frac{h^n}{n!} \norm{D^{n+1} V_j (b,\varphi ;g)}_{n, b, T_j (\varphi)} \nnb
    &\lesssim h^{-1} \sum_{n=1,2} \frac{h^n}{n !} \norm{D^{n} V^G_j (b,\varphi ;g)}_{n, b,T_j (\varphi)} \nnb
    &\le h^{-1} \norm{g}_{C_j^2} \norm{V^G_j}_{h, b,T_j (\varphi)} \nnb
    &\lesssim A^{-1} h^{-1} \norm{g_j}_{C_j^2} \norm{V}_{\cV_j} \big( h + \norm{\nabla_j \varphi}_{\ell^\infty (b)} \big)^2 ,
\end{align}
where we apply Lemma \ref{lemma:VjDjbnd} to obtain the last inequality. 

For $V_j^Q$, we can make an explicit expansion to obtain
\begin{align} 
    \norm{D V_j^Q (b,\varphi ;g)}_{h,b,T_j (\varphi)}
    &\le \sum_{q\in \Z} \sum_{n=0}^{\infty} \frac{(\alpha |q| h)^n}{n!}  |q| \alpha |z_j^{(q)}| \norm{g}_{\ell^\infty} \nnb
    &\le \sum_{q \in \Z} e^{2\alpha |q| h} |z_j^{(q)}|
\end{align}
and by the definition of the norm \eqref{eq:cVjNormDefi},
\begin{align}
    \le A^{-1} \sum_{q \in \Z} e^{2\alpha |q| h - \frac{1}{4} \gamma |q| \alpha^2} \norm{V_j}_{\cV_j} .
\end{align}
Since $2\alpha |q| h = 2c_h r \alpha^2$, the sum converges for $r < \gamma / 8c_h$.
\end{proof}

\begin{lemma} \label{lemma:reblockObsg}
Given $K_j^{\obs} \in \cN_j^\obs$ and $V_j \in \cV$, let $I_j = \exp(V_j)$ and
\begin{align}
    K_j^{\dagobs} (Y , \varphi) &= \sum_{Z \in \cP_j (Y)} \big( I_j (\varphi + g_j) - I_j (\varphi) \big)^Z K_j^\obs (Y \backslash Z, \varphi + g_j) , \qquad 
    (Y,\varphi) \in \cP_j \times \R^\Lambda.
\end{align}
Then
\begin{align}
    (I_j \circ K_j^\obs) (\varphi + g_j) = \big( I_j \circ K_j^\dagobs \big) (\varphi)
\end{align}
Whenever $\norm{V_j}_{\cV_j} \vee \norm{K_j^\obs}_{j,\obs}$ is sufficiently small and 
$\norm{D_{j-1}}_{\cD_{j-1}} \lesssim A \alpha^{-1}$,
\begin{align}
    \norm{K_j^\dagobs}_{j,\dagger} \lesssim O_L \big(\norm{V_j}_{\cV_j} \big) \vee \norm{K_j^\obs}_{j,\obs} .
    \label{eq:reblockingbnd}
\end{align}
Moreover, if $K_j^\bulk \in \cN^\bulk_j$ and $K_j^\obs$ are compatible (recall Definition~\ref{defi:cNobsdefi}), then
\begin{align}
    (K_j^\bulk, K_j^\dagobs) \in \underline{\cN}_j^\dagger . \label{eq:KbulkobsdagobsCompatible}
\end{align}
\end{lemma}
\begin{proof}
We have
\begin{align}
    (I_j \circ K_j^\obs) (\varphi + g) = \sum_{X \in \cP_j} ( I_j (\varphi +g_j ) )^{X} K_j^\obs (\Lambda \backslash X, \varphi + g_j)
\end{align}
Since $( I_j (\varphi +g_j) )^{X} = \sum_{Y \subset X} ( I_j (\varphi) )^{Y} \big( I_j (\varphi + g_j) - I_j (\varphi) \big)^{X \backslash Y}$,
after reparametrizing $Y' = \Lambda \backslash Y$ and $Z=X\setminus Y$,  
\begin{align}
    (I_j \circ K_j^\obs) (\varphi + g)
    & = 
    \sum_{Y' \in \cP_j} \big( I (\varphi) \big)^{\Lambda \backslash Y'} 
    \sum_{Z \in \cP_j (Y')} \big( I_j (\varphi + g_j ) - I_j (\varphi) \big)^Z K_j^\obs (Y' \backslash Z, \varphi + g_j ) . \nonumber
\end{align}

To prove \eqref{eq:reblockingbnd}, we first bound the difference in $I_j$. By the Fundamental Theorem of Calculus,
\begin{align}
    \norm{I_j (b,\varphi + g_j ) - I_j (b,\varphi)}_{h,b,T_j (\varphi)}
        \le \sup_{t \in [0,1]} \norm{D V_j (b,\varphi + t g_j ; g_j)}_{h,b,T_j (\varphi)} \norm{I (b,\varphi+tg_j)}_{h,b,T_j (\varphi)} .
\end{align}
By Lemma~\ref{lemma:VjDjbnd}, we have $\norm{V_j (b,\varphi)}_{h,b,T_j (\varphi)} \lesssim (h + \norm{\nabla_j \varphi}_{\ell^\infty (b^*)})^2 \norm{V_j}_{\cV_j}$ and
\begin{align}
    \norm{I_j (b,\varphi+tg) }_{h,b,T_j (\varphi)} 
    \lesssim ( G_j^\str (b,\varphi+tg) )^{1/2}
    \le ( G_j^{\str, \dagger} (b,\varphi) )^{1/2}
\end{align}
for sufficiently small $\norm{V_j}_{\cV_j}$.
Also, by Lemma~\ref{lemma:DVbnd} and Lemma~\ref{lemma:gjbound} (on the bound on $g_j$)
\begin{align}
    \norm{D V_j (b,\varphi ;g)}_{h,b,T_j (\varphi)}
    \lesssim L A^{-1} \norm{V}_{\cV_j} \big( G_j^{\str} (b,\varphi) \big)^{1/2}
\end{align}
and we have 
\begin{align}
    \norm{I_j (b,\varphi+ g) - I_j (\varphi)}_{h, b,T_j (\varphi)}
    \le L A^{-1} G_j^{\str,\dagger} (b,\varphi) \norm{V_j}_{\cV_j} .
    \label{eq:IobsplusgminusIobs}
\end{align}
On the other hand, for $K_j^\obs$, we use the definition of the observable norm \eqref{eq:normjsigma} to obtain
\begin{align}
    \norm{K_j^\obs (Y \backslash Z , \varphi + g)}_{h, X,T_j (\varphi)}
    &\le (A/2)^{-|Y \backslash Z|_j} G^{k}_j (Y\backslash Z , \varphi + g)  \norm{K_j^\obs}_{h, T_j}^{|\Comp (Y\backslash Z)|}  \nnb
    &\le (A/2)^{-|Y \backslash Z|_j} G_j^\dagger (Y\backslash Z , \varphi) \norm{K_j^\obs}_{h, T_j}^{|\Comp (Y\backslash Z)|}  .
    \label{eq:KjobsBnd}
\end{align}
Putting together \eqref{eq:IobsplusgminusIobs} and \eqref{eq:KjobsBnd}, 
\begin{align}
    &\norm{K_j^\dagger (Y,\varphi)}_{h, X,T_j (\varphi)} \nnb
    &\quad \lesssim \sum_{Z \in \cP_j (Y)} A^{-|Z|_j} (A/2)^{-|Y \backslash Z|_j} G_j^{\str,\dagger} (Z,\varphi) G_j^\dagger (Y\backslash Z , \varphi) \big( L \norm{V_j}_{\cV_j} \vee \norm{K_j^\obs}_{j,\obs} \big) \nnb
    &\quad \le G_j^{\dagger} (Y,\varphi) (A/2)^{-|Y|_j}  \big(  L \norm{V_j}_{\cV_j}  \vee \norm{K_j^\obs}_{j,\obs} \big) \sum_{Z \in \cP_j (Y)} 1 .
\end{align}
Since $\sum_{Z \in \cP_j (Y)} 1 = 2^{|Y|_j}$, we have \eqref{eq:reblockingbnd}.

The final result \eqref{eq:KbulkobsdagobsCompatible} is trivial due to \eqref{eq:reblockingbnd} and the definition of $K^\dagobs$.
\end{proof}

Together with \eqref{eq:Zjobsinter1}, we obtain an intermediate form
\begin{align}
    Z_j^\obs (\varphi) = e^{u_j^\obs (\Lambda) + \Omega_j (d\varphi)} (I_j \circ K_j^\dagger) (\varphi) .
    \label{eq:Zjobsinter2}
\end{align}
We will apply the progressive map operation (Definition~\ref{defi:PM}) on this expression, but we will first  determine the reapportioning element. This is discussed in the rest of the section.

\subsection{Observable localisation} \label{sec:obsLoc}

As for the bulk case, the reapportioning element is determined by a localisation of $K_j^\obs$.
For the observable, we localise only to first order in $\nabla\varphi$, yielding an external gradient field. 
If $z \in b \in \cB_j$, let
\begin{align}
    & \Loc_{X,z}^{(1)} F (\varphi) = \frac{1}{|b||X|} \sum_{x_0 \in b} 
    \sum_{x_1 \in X^*} \partial_{\varphi (x_1)} \hat{F}_0 (X,0) \langle \nabla \varphi_z , \delta x_1 \rangle   ,
    \label{eq:LocLinDefi}
\end{align}
where, as in Section~\ref{sec:LocExpansion}, $\delta x_1 = x_1 - x_0$ and $\langle \nabla \varphi_z , \delta x_1 \rangle = \frac{1}{2} \sum_{\mu \in \hat{e}} \nabla^\mu \varphi_z (\delta x_1)^{\mu}$.
Then we take
\begin{align} \label{e.locX}
    \Loc_{X,z}^\obs F = 
    \Loc_{X,z}^{(0)} F + \Loc_{X,z}^{(1)} F
\end{align}
where the zero-order term $\Loc_{X,z}^{(0)}$ was given by \eqref{eq:LocBulkdefiZero}. $\Loc^{(1)}$ contributes to the evolution of $D_j$ and $\Loc^{(0)}$ contributes to the evolution of $u_j^{\obs}$.
They define, for $Y\subset X$,
\begin{align}
    \Loc_{X,Y}^\obs = \sum_{z \in Y} \Loc_{X,z}^\obs , 
    \qquad 
    \Loc_{X}^\obs = \Loc_{X,X}^\obs .
\end{align}

An immediate property of $\Loc$ is the boundedness as a linear map. In Lemma~\ref{lemma:LocBndd}, 
the factor $\log L$ is growing in $L$, but it will be absorbed by a polynomial contraction rate in Section \ref{sec:obsRGmap}. 
($\log L$ originates from $\kappa_0^{-1}$ of \eqref{eq:GjXphiDefi} when we bound polynomials of $\nabla\varphi$ in terms of the regulator.)
Also, recall that $\Loc^{(2)}$ was defined in Section~\ref{sec:LocExpansion}.

\begin{lemma} \label{lemma:LocBndd}
For $X \in \cS_j$ and $b \in \cB_j (X)$, $F \in \cN_j$ be a neutral polymer activity such that $\norm{F (X)}_{h, X, T_j (G_j)} < \infty$ and each $i=0,1,2$,
\begin{align}
    \norm{\Loc_{X,b}^{(i)} \E_{j+1} \theta F (X,\varphi)}_{h,b,T_j (\varphi)}
    \lesssim \log L \norm{F(X)}_{h,X,T_j (G_j)} G^\str_j (b , \varphi)  .
    \label{eq:LocBndd}
\end{align}
Analogous bounds hold when $G_j$ is replaced by $G_j^\dagger$. 
\end{lemma}
\begin{proof}
When $i=0,2$, the bound is \cite[(6.16)]{dgauss1}---although the proof there assumes $\norm{F(X)}_{h,X,T_j(G_j)} < \infty$, it only uses
the properties listed in Lemma~\ref{lemma:Gjobsproperties} for large-field regulators. Since these properties also hold for $G_j^\dagger$, the same argument yields the corresponding bound with $G_j$ replaced by $G_j^\dagger$.

For $i=1$, let $F' (\varphi) = \E_{j+1} \theta F (X,\varphi)$. 
We fix $b\in \cB_j$, $x_0 \in b$, $z \in X$ and $\mu \in \hat{e}$ and consider
\begin{align}
    \Loc' F' (\varphi) = \sum_{x_1 \in X^*} \partial_{\varphi (x_1)} F' (0) (\delta x_1)^{\mu} \nabla^{\mu} \varphi_z .
\end{align}
Since $\Loc^{(1)}_{X,b}$ is the average of $\Loc'$ over $( x_0, z ) \in b \times X$, it will be sufficient to prove
\begin{align}
    \norm{\Loc' F' (\varphi)}_{h, b, T_j (\varphi)}
    \lesssim \log L \norm{F(X)}_{h, X,T_j (G_j)} G^\str_j (b, \varphi)  .
    \label{eq:LocBnddi1}
\end{align}
If we let $\psi : X^* \rightarrow \R$, $\psi (x) = (\delta x)^{\mu}$, then by the definition of the Taylor semi-norm,
\begin{align}
    \Big| \sum_{x_1 \in X^*} \partial_{\varphi (x_1)} F' (0) (\delta x_1)^{\mu} \Big| 
    = \big| D F' (0 ; \psi) \big|
    \le h^{-1} \norm{F' (0)}_{h,X,T_j(0)} \norm{\psi}_{C_j^2 (X^*)} . 
    \label{eq:LocBnddi2}
\end{align}
Since $\norm{\psi}_{C_j^2 (X^*)} \lesssim L^j$, by \eqref{eq:lemmaVjDjbnd2}, and recall that $\kappa_0 = \frac{c_\kappa }{\log L}$
\begin{align}
    \norm{\Loc' F' (\varphi)}_{h,b,T_j (G_j)} 
    &\lesssim h^{-1} \norm{F' (0)}_{h,X,T_j (0)} \norm{\nabla_j \varphi_z}_{h,X,T_j (\varphi)} \nnb
    &\lesssim \kappa_0^{-1} \norm{F' (0)}_{h,X,T_j(0)} G^\str_j (b, \varphi) \nnb
    &\lesssim \log L \norm{F' (0)}_{h,X,T_j (0)} G^\str_j (b, \varphi).
    \label{eq:LocBnddi3}
\end{align}
But since
\begin{align}
    \norm{F'(0)}_{h,X,T_j (0)} \le \norm{F (X)}_{h,X,T_j (G_j)} \E_{j+1} G_j (X,\zeta) \le 3^{|X|_j} \norm{F (X)}_{h,X,T_j (G_j)}
    \label{eq:FprimeBnd}
\end{align}
due to Lemma~\ref{lemma:Gjobsproperties}(iv), and $X \in \cS_j$ implies $|X|_j \le 4$, we obtain the desired bound.
\end{proof}

\subsubsection{Taylor expansion}
\label{sec:TayExp}

When the polymer activity is neutral, we can improve Lemma~\ref{lemma:LocBndd}, since $\Loc_{X,b}^{(i)}$ acts as a good approximation of the original function.
To see this, we define 
\begin{align}
    \Tay_1 F(X) = F(X,0) + \frac{1}{|b|} \sum_{x_0 \in b} \sum_{x_1 \in X^*} \frac{\partial F(X,\psi)}{\partial \psi_{x_1} } \Big|_{\psi=0} \delta \varphi_{x_1} 
    \label{eq:Tayndefi}
\end{align}
for $X \in \cP_j$, and we will be able to show that both $1-\Tay_1$ and $\Tay_1 - \Loc^\obs$ are sufficiently smaller than the original function.

\begin{lemma} \label{lemma:Taybound}
Let $X \in \cS_j$ and $F$ be a neutral polymer activity such that $\norm{F}_{j,\dagger} <\infty$. Then
\begin{align}
    \norm{(1-\Tay_1) \E_{j+1} \theta F(X) }_{h,\bar{X}, T_{j+1} (\varphi)}
        \lesssim \Big( \frac{L}{\log L} \Big)^{-2} A^{-|X|_j} \norm{F}_{j,\dagger} G^{k_j}_{j+1} (\bar{X}, \varphi) .
\end{align}
\end{lemma}
\begin{proof}
The proof is the same as
\cite[Lemma~6.17]{dgauss1}---the assumption is slightly different because we are assuming $\norm{F}_{j,\dagger} < \infty$. The Gaussian-integration and large-field estimates required in the argument hold with $G_j^\dagger$ in place of $G_j$, by Lemma~\ref{lemma:Gjobsproperties}.
\end{proof}

Next, we approximate $\Tay_1$ by $\Loc^\obs$.

\begin{lemma} \label{lemma:TayLocobsError}
Let $X \in \cS_j$ and $\norm{F}_{j,\dagger} < \infty$. Then
\begin{align}
    & \big\| \big(\Tay_1  - \Loc_{X}^\obs \big) \E_{j+1} \theta F (X) \big\|_{h, X, T_{j+1} (\varphi)} 
    \lesssim L^{-2} \norm{F(X)}_{h,X,T_j(G_j^\dagger)} G_{j+1}^{\str} (\bar{X},\varphi) .
\end{align}
\end{lemma}
\begin{proof}
Let $F' (\varphi)= \E_{j+1} \theta F (X,\varphi)$. We can reorder the sum so that
\begin{align}
    \Loc^{(1)}_X F' (\varphi)
    & = \frac{1}{|X|} \sum_{x_0 \in X} \frac{1}{|b (x_0)|} 
    \sum_{z \in b(x_0)}
    \sum_{x_1 \in X^*}  \partial_{\varphi (x_1)} F' (0) \langle \nabla \varphi_z , \delta x_1 \rangle .
\end{align}
$\Tay_0$ and $\Loc^{(0)}_X$ cancel exactly, and the remaining difference is
\begin{align}
    & \big( \Tay_1 - \Loc^{(0)}_X - \Loc^{(1)}_X  \big) F' (\varphi)
    = \frac{1}{|X|} \sum_{x_0 \in X} \frac{1}{|b(x_0)|} \sum_{z \in b(x_0)} \cE_{x,z}
    \label{eq:Tay1LocX0LocX1}
\end{align}
where
\begin{align}
    \cE_{x,z} = \sum_{x_1 \in X^*} \partial_{\varphi (x_1)} F' (0) \big( \varphi_{x_1} - \varphi_{x_0} -  \langle \nabla \varphi_z , x_1 - x_0 \rangle  \big) ,
    \label{eq:cExzDefi}
\end{align}
thus it is sufficient to bound each $\cE_{x,z}$.
Similarly as in \eqref{eq:FprimeBnd}, we get
\begin{align}
    \norm{F' (0)}_{h,X,T_j (0)} \le 3^{|X|_j} \norm{F(X)}_{h,X,T_j (G_j^{\dagger})} ,    
\end{align}
so for $X \in \cS_j$,
\begin{align}
    \norm{D F' (0)}_{n=1,X,T_j (0)} \lesssim  h^{-1} \norm{F(X)}_{h,X,T_j(G_j^\dagger)} .
\end{align}
The terms following $\partial_{\varphi (x_1)} F' (0)$ in \eqref{eq:cExzDefi} are bounded using Lemma~\ref{lemma:deltavarphiapprox}, so
\begin{align}
    &\norm{ (\Tay_1 - \Loc^\obs_X ) F(X,\varphi) }_{h,X,T_{j+1}(\varphi)} \nnb
    & \qquad \lesssim L^{-2} \norm{F(X)}_{h,X,T_j(G_j^\dagger)} h^{-1} \big(h+ \norm{\nabla_{j+1} \varphi}_{L^{\infty} (X^*)} + \norm{\nabla_{j+1}^2 \varphi}_{L^{\infty} (X^*)} \big) \nnb
    & \qquad \lesssim L^{-2} \norm{F(X)}_{h,X,T_j(G_j^\dagger)} G_{j+1}^{\str} (X,\varphi)
\end{align} 
\end{proof}

The next lemma is just a more sophisticated version of \eqref{eq:lemmaVjDjbnd2}.

\begin{lemma} \label{lemma:deltavarphiapprox}
If $X \in \cS_j$, $x_0,y \in X$ and $\delta \varphi_x = \varphi_x - \varphi_{x_0}$, we have
\begin{align}
    \norm{\delta \varphi - \langle \nabla \varphi_y , \delta x \rangle }_{C_{j+1}^2 (X^*)} & \le C L^{-2} (h + \norm{\nabla_{j+1}^2 \varphi}_{L^{\infty} (X^*)} ) .
\end{align}
\end{lemma}
\begin{proof}
For $\delta \varphi$, there is a path $(z_0 = x_0, z_1, \cdots, z_m =x)$ of length $m = O( L^j)$ connecting $x_0$ and $x$, such that $\dist_1 (x_i , x_{i+1}) =1$ for each $i$.
If we let $\mu_i = x_{i+1} - x_i$ for each $i$, then 
$\delta \varphi_x = \sum_{i=0}^{m-1} \nabla^{\mu} \varphi_{x_i}$, so
\begin{align}
    \delta \varphi_x - \langle \nabla \varphi_{y} , \delta x \rangle = \sum_{i=0}^{m-1} \big( \nabla^{\mu} \varphi_{x_i} - \nabla^{\mu} \varphi_{y} \big) .
\end{align}
Next, by consider a similar path connecting $y$ and $x_{i}$, of length $m' = O(L^j)$,
we get
\begin{align}
    \big| \nabla^\mu \varphi_{x_i} - \nabla^\mu \varphi_y \big|
    &\le m' \max_{z \in X^*, \; \mu, \nu\in \hat{e}} \big| \nabla^{(\mu,\nu)} \varphi_z \big| \nnb
    & \lesssim L^j \times L^{-2(j+1)} \norm{\nabla^2_{j+1} \varphi}_{L^\infty (X^*)}) .
\end{align}
Derivatives can be bounded similarly.
\end{proof}

\subsection{Contraction estimates on the observables}
\label{sec:constractionObs}

It was noted earlier that the contraction estimate Lemma~\ref{lemma:LPMbnd} plays a key role in the construction of a stable bulk RG flow, given by Proposition~\ref{prop:critpt}. In this section we obtain similar estimates for the observables and use them to prove the stability of the observable RG flow.

The contraction follows from three different mechanisms. The first is on the charged terms (Lemma~\ref{lemma:chargetermcontraction}), the second is on the neutral terms (Corollary~\ref{cor:obsneutralcont}) and the third is on the 'large' polymers (Proposition~\ref{prop:largesetcont}). We state them here.

\begin{lemma}  \label{lemma:chargetermcontraction}
Let $F$ be a $q$-charged polymer activity (i.e., $F = \hat{F}_q$) and $\norm{F}_{j, \dagger} < \infty$ and $X \in \cS_j$. Then
\begin{align}
    \norm{\E_{j+1} \theta F (X,\varphi)}_{h,X,T_{j+1}(\varphi)}
    \le \norm{F(X)}_{h,X,T(G_j^\dagger)} G_{j+1}^{k_j} (\bar{X},\varphi) 
    e^{\alpha |q| h} e^{- r \frac{2|q| -1}{2} \alpha^2 \Gamma_{j+1} (0)}
\end{align}
\end{lemma}
\begin{proof}
This is \cite[Lemma~6.13]{dgauss1}, just with the condition $\norm{F(X)}_{h,X,T_j(G_j)} < \infty$ replaced by $\norm{F(X)}_{h,X,T_j(G_j^\dagger)} < \infty$.
\end{proof}

The contraction on neutral polymer activities is already partially explained in Section~\ref{sec:TayExp}, so we only have to collect the results.

\begin{corollary} \label{cor:obsneutralcont}
Let $X \in \cS_j$ and $F$ be a neutral polymer activity such that $\norm{F(X)}_{j,\dagger} < \infty$. Then
\begin{align}
    \norm{(1-\Loc^\obs_X) \E_{j+1} \theta F (X)}_{h,\bar{X},T_{j+1}(\varphi)}
    \lesssim (L/\log L)^{-2} A^{-|X|_j} \norm{F}_{j,\dagger} G^{k_j}_{j+1} (\bar{X},\varphi) . 
\end{align}
\end{corollary}
\begin{proof}
After decomposing
\begin{align}
    (1-\Loc^\obs_X) \E_{+} \theta F (X) = (1-\Tay_1) \E_{+} \theta F (X) + (\Tay_1 -\Loc^\obs_X) \E_{+} \theta F (X) ,
\end{align}
and applying Lemma~\ref{lemma:Taybound} and \ref{lemma:TayLocobsError} to bound the first and the second terms, respectively, we get
\begin{align}
    &\norm{(1-\Loc^\obs_X) \E_{+} \theta F (X)}_{h, \bar{X}, T_{j+1} (\varphi)} \nnb
    &\qquad \lesssim 
    (L/\log L)^{-2} \norm{F (X)}_{h,X,T_j(G_j^\dagger)} \big( G^\str_{j+1} (\bar{X}, \varphi) + G^{k_j}_{j+1} (\bar{X},\varphi) \big) .
\end{align}
Using $\norm{F (X)}_{h,X,T(G_j^{\dagger})} \le (A/4)^{-|X|_j} \norm{F}_{j,\dagger}$ and Lemma~\ref{lemma:Gjobsproperties}(iii) to bound $G^\str_{j+1}$, we get the desired result. 
\end{proof}

The first and the second mechanism is applied in the following unified form, where we recall $\etabulk = 0.9999$.

\begin{corollary} \label{cor:obssmallsetcont}
Assume Assumption~\ref{asmp:RGasmpBulk}. If $F$ is a polymer activity such that $\norm{F}_{j,\dagger} < \infty$, then for $X \in \cS$,
\begin{align}
    \norm{(1-\Loc_X) \E_{j+1} \theta F (X,\varphi)}_{h,\bar{X},T_{j+1} ( \varphi)}
    \lesssim L^{-1-\etabulk} A^{-|X|_j} \norm{F}_{j,\dagger} G_{j+1}^{k_j} (\bar{X},\varphi) .
\end{align}
\end{corollary}
\begin{proof}
After decomposing $F = \sum_{q\in \Z} \hat{F}_q$, the bound on $\hat{F}_0$ follows from Corollary~\ref{cor:obsneutralcont}, and the bound on the rest follows from Lemma~\ref{lemma:chargetermcontraction} (notice that  $\operatorname{Loc}_X$ acts only on the neutral component, hence $\operatorname{Loc}_X\mathbb E\theta\widehat F_q=0$ for every $q\neq0$). This gives the desired bound, just with $L^{-1-\etabulk}$ replaced by $(L/\log L)^{-2} + \sum_{q\ge 1} e^{\alpha qh} e^{-(q-1/2) \alpha^2 \Gamma_{j+1} (0) }$. If $L$ and $\alpha$ are sufficiently large, this factor is bounded by $L^{-1-\etabulk}$ due to \eqref{eq:etabulkProp}.
\end{proof}

The third mechanism works on $X \not\in \cS$.

\begin{proposition} \label{prop:largesetcont}
For some $C>0$, any $L \ge 5$, $A \ge L^{C}$, $X \in \cP^c_{j+1}$ and $F \in \cN_j$ such that $\norm{F}_{j,\dagger} < \infty$,
\begin{align}
    \Big\| \sum_{Y \in \Conn_j \backslash \cS_j}^{\bar{Y}=X} \E\theta F(Y,\varphi) \Big\|_{h,X,T_{j+1} (\varphi)} \le (L\Adag)^{-|X|_{j+1}} \norm{F}_{j,\dagger} G_{j+1}^{k_j} (X,\varphi) .
\end{align}
\end{proposition}
\begin{proof}
This is \cite[Proposition~6.20]{dgauss1}, just with $\norm{F}_{j}$ replaced by $\norm{F}_{j,\dagger}$.
\end{proof}

\section{Observable RG map} \label{sec:obsRGmap}

Now we explain the observable RG map. 
The observable RG map is
\begin{align}
    \Phi_{j+1}^\obs : (D_{j-1} , D_j, V_j , K_j^\obs , K_j^\bulk) \mapsto (\delta u_{j+1}^\obs , D_{j+1} , V_{j+1}, K_{j+1}^\obs , K_{j+1}^\bulk )  
\end{align}
satisfying \eqref{eq:Zjobsexpansion} at scales $j$ and $j+1$.
Under Assumption~\ref{asmp:critical}, the bulk RG flow $(V_j, K_j^\bulk)_{j\ge 0}$ has been constructed and can be used as input data. 
We use the bulk RG flow to construct the observable RG map and the bounds on the bulk flow to establish stability of the observable flow. 

Suppose at scale-$j$,  $(V_j, K_j^\bulk, K_j^\obs) \in \cV \times \underline{\cN}_j^\obs$ satisfies the smallness condition of Lemma~\ref{lemma:reblockObsg}, then we have $(K_j^\bulk, K_j^\dagobs) \in \underline{\cN}_j^{\dagobs}$,
and we have the scale-$j$ representation 
\eqref{eq:Zjobsinter2}, which states
\begin{align}
    Z_j^\obs (\varphi) = e^{u_j^\obs (\Lambda) + \Omega_j (d\varphi)} ( e^{V_j} \circ_j K_j^\dagobs) (\varphi) .
    \label{eq:ZjobswithKjdag}
\end{align}
$D_{j-1}$ was involved in this process because $g_j$ is defined in terms of $\Omega_{j-1}$, see \eqref{eq:gjDefi}. 
The observable RG map constructs RG coordinates satisfying \eqref{eq:Zjobsexpansion} at scale $j+1$,
\begin{align}
    Z_{j+1}^\obs (\varphi) &= \E_{j+1} e^{(d^* f, \zeta)} Z_j^\obs (\varphi + \zeta) \nnb
        &= e^{(u_j^\obs + \delta u_{j+1}^\obs) (\Lambda) + \Omega_j (d\varphi)} ( e^{V_{j+1}} \circ_{j+1} K_{j+1}^{\obs} ) (\varphi + g_{j+1}) ,
    \label{eq:Zjp1obsexpansion}
\end{align}
see Lemma~\ref{lemma:obsRGmapWorks}.

In Section~\ref{sec:K0obsdefi}, we describe a reformulation of $Z_0^\obs$ in the form of \eqref{eq:Zjobsexpansion}.
In Section~\ref{sec:quadEnergy}, we perform a change of variable that explains the origin of the shift $g_{j+1}$ in \eqref{eq:Zjp1obsexpansion}.
In Section~\ref{sec:pertAnalysisObs}, $\delta u_{j+1}^\obs$, $D_{j+1}$ and $\Omega_{j+1}$ are defined,
and we also perform the fluctuation integral to obtain a polymer activity $K_{j+1}^\ddagger$ at scale $j+1$.
After projecting the observable localisation on the line observable in Section~\ref{sec:finalStep},
$K_{j+1}^\obs$ is defined in Definition~\ref{defi:Kobsplusdefi}.
The rest of the section is devoted to the analysis of $\Phi^\obs_{j+1}$ and the flow generated by it.

For brevity, we abbreviate $j+1$ by $+$, $j-1$ by $-$ and omit $j$ if there is no source of confusion.

\subsection{Initialisation}
\label{sec:K0obsdefi}

Recall Lemma~\ref{lemma:sinGordapprox}, 
\begin{align}
    Z_0^\obs (\varphi) = e^{s\gamma(\Delta d^* f, \varphi)} Z_0^\bulk (\varphi + g_0) .
\end{align}
We would like to restate this in form \eqref{eq:Zjobsexpansion}.
Under Assumption~\ref{asmp:critical} and initial choices $u_0^\obs (\Lambda) = - s\gamma (\Delta d^* f, g_0)$, $D_0 = 0$ and $V_0 (\{x\} , \varphi) = \frac{1}{2} s_0 |\nabla \varphi_x|^2 + W_0^Q (\{x\} ,\varphi)$, we have
\begin{align}
    Z_0^\obs (\varphi) = e^{u_0^\obs (\Lambda)} e^{E (\Lambda, \varphi + g_0) + V_0 (\Lambda, \varphi + g_0)} ,
\end{align}
where $E(\Lambda,\varphi) = \sum_{x \in \Lambda} E(\{x\}, \varphi) = s\gamma \sum_{x \in \Lambda} (d^* f)_x (\Delta \varphi_x)$.
We perform a reblocking operation to remove $E$. The next lemma also serves as a definition of $K_0^\obs$.

\begin{lemma} \label{lemma:K0obsdefi}
Suppose $\alpha$ and $(s A + h^2 \norm{V_0}_{\cV_0})^{-1}$ are sufficiently large. Then there exists $K_0^\obs \in \cN_0^\obs$ such that
\begin{align}
    Z_0^\obs (\varphi) = e^{u_0^\obs (\Lambda)} ( I_0 \circ K_0^\obs ) (\varphi + g_0) \label{eq:lemmaK0obsdefi1}
\end{align}
for $I_0 = e^{V_0}$. Moreover, $(K_0^\bulk , K_0^\obs) \in \underline{\cN}_0^{\obs}$ with
\begin{align}
    \norm{K_0^\obs}_{j,\obs} \lesssim sA .
    \label{eq:lemmaK0obsdefi2}
\end{align}
\end{lemma}
\begin{proof}
By direct expansion,
\begin{align}
    e^{E(\Lambda,\varphi)} e^{V_0 (\Lambda, \varphi)} 
    &= \prod_{x \in \Lambda} \big( e^{E(\{x\},\varphi)} - 1 +1 \big) e^{V_0 (\{x\},\varphi)} \nnb
    & = \sum_{X \subset \Lambda} e^{V_0 (X,\varphi) }(e^{E(\varphi)}-1)^{\Lambda \backslash X} =: 
    (e^{V_0} \circ K_0^\obs) (\varphi)
\end{align}
when $K_0^\obs (X,\varphi) = \prod_{x \in X} (e^{E(\{x\},\varphi)} -1 )$.
$K_0^\obs$ is obviously compatible with $K_0^\bulk (X) = \one_{X=\emptyset}$ because $K_0^\obs (X) = 0$ whenever $X\not\subset  \supp (d^*f)$.

For the estimate \eqref{eq:lemmaK0obsdefi2}, we abbreviate $\norm{\cdot} = \norm{\cdot}_{h,X,T_0 (\varphi)}$ for a moment.
By Assumption~\ref{asmp:critical}. Lemma~\ref{lemma:sinGordapprox} and \ref{lemma:VjDjbnd}, we see that 
\begin{align}
    \norm{e^{V_0 (\{x\},\varphi)}} 
    \le 
    e^{A^{-1} \norm{V_0}_{\cV_0} (h+ \norm{\nabla \varphi}_{L^{\infty} (\{x\}^*)})^2 } 
    \le 
    e^{h^2 A^{-1} \norm{V_0}_{\cV_0} (1+ \alpha^{-1} \norm{\nabla \varphi}_{L^{\infty} (\{x\}^*)})^2 }  .
\end{align}
Also, Taylor expansion of the exponential implies
\begin{align}
    \norm{e^{E(\{x\},\varphi)} -1} \le \sum_{n=0}^\infty \frac{\norm{E(\{x\},\varphi)}^{n+1}}{(n+1)!} \le \norm{E(\{x\},\varphi)} e^{\norm{E(\{x\},\varphi)}} ,
\end{align}
so it will be sufficient to bound $\norm{E(\{x\},\varphi)}$. Since $E$ is a linear function of $\Delta \varphi$ and $|E(\{x\},\varphi)| \lesssim |s| \alpha^{-1} \norm{\nabla^2 \varphi}_{L^\infty (\{x\}^*)}$ due to \eqref{eq:fassump}, we have
\begin{align}
    \norm{ E(\{x\},\varphi) } \lesssim |s| \alpha^{-1} \big( h + \norm{\nabla^2 \varphi}_{L^\infty (\{x\}^*)} \big) \lesssim |s| \big( 1 + \alpha^{-1} \norm{\nabla^2 \varphi}_{L^\infty (\{x\}^*)} \big) . 
\end{align}
By assumptions, these imply
\begin{align}
    \norm{ E(\{x\},\varphi) } e^{\norm{ E(\{x\},\varphi) }} \lesssim |s| G^{\str}_0 (\{x\},\varphi) .
\end{align}
In other words,
\begin{align}
    \norm{K_0^\obs (X,\varphi)} \lesssim s^{|X|} G_0^\str (X,\varphi) \lesssim s^{|X|} G_0^{k_0} (X,\varphi) ,    
\end{align}
where the final inequality follows from Lemma~\ref{lemma:Gjobsproperties}.
Thus $\norm{K_0^\obs}_{0,\obs} \lesssim s A$ by definition of the norm, whenever $s A$ is sufficiently smaller than $1$.
\end{proof}

Again, due to Lemma~\ref{lemma:reblockObsg}, \eqref{eq:lemmaK0obsdefi1} in turn can be restated in the form \eqref{eq:ZjobswithKjdag}, so we work only with $K^\dagger$ from now on.

\subsection{Quadratic energy} \label{sec:quadEnergy}

We explain a step that contributes to the Gaussian part of the free energy.

\begin{lemma} \label{lemma:ZplusIntrmd}[Cameron-Martin]
Assume \eqref{eq:ZjobswithKjdag}. Then
\begin{align}
    Z_+^\obs (\varphi) = e^{u^{(i)} (\Lambda) + \Omega(d\varphi ) } \E_+ \theta \big[ (I \circ K^\dagger ) (\varphi + g_+) \big] 
\end{align}
where $u^{(i)} (\Lambda) = \sum_{b\in \cB_j} u^{(i)} (b)$ with
\begin{align}
    u^{(i)} (b) &= u^\obs (b) + \frac{1}{2} \big( d^* (f+ \omega) , \Gamma_{+} d^* (f+ \omega) \big)_b .
    \label{eq:uinter}
\end{align}
\end{lemma}
\begin{proof}
By \eqref{eq:Zjobsdefi}, $Z_{+}^\obs (\varphi) = \E [ e^{(d^* f , \zeta)} Z^\obs ( \varphi + \zeta ) ]$ for $\zeta \sim \cN (0, \Gamma_{+})$.
Plugging in \eqref{eq:ZjobswithKjdag},
\begin{align}
    Z_{+}^\obs (\varphi) 
    &= e^{u^\obs (\Lambda)} \E \Big[ e^{(d^* f,\zeta)} e^{(\omega, d\zeta + d\varphi)} (I \circ K^\dagger) (\varphi + \zeta) \Big] \nnb
    &= e^{u^\obs (\Lambda) + ( \omega, d\varphi)} \E \Big[ e^{(d^* (f+\omega) ,\zeta)} (I \circ K^\dagger) (\varphi + \zeta) \Big] .
\end{align}
Using the change of variable $\zeta \mapsto \zeta + \Gamma_{+} (d^* (f+\omega))$,
\begin{align}
    = e^{u^\obs (\Lambda) + \frac{1}{2} (  d^* (f+ \omega) , \Gamma_{+} d^* (f + \omega) ) + (d^* \omega , \varphi) } 
    \E \Big[ \big(I \circ K^\dagger \big) \big(\varphi + \zeta + \Gamma_{+} d^* (f+ \omega) \big) \Big] .
\end{align}
By \eqref{eq:gjDefi}, $g_+ = \Gamma_+ d^* (f+\omega)$, and the desired conclusion is satisfied.
\end{proof}

The field $g_+$ can be problematic if we try to control it at scale $j$, since it is supported on a neighbourhood of $\supp (d^* f)$ at a distance $\asymp L^{j+1}$, 
In particular, it is not supported on a small-set neighborhood of $\supp (d^* f)$. To circumvent this problem, we simply keep $g_+$ as a shift of the field in the output of the scale-$j$ RG map, and pass it to the next scale. This leads to the presence of the perturbation $g$ inside \eqref{eq:Zjobsexpansion}.
In most of the subsequent RG analysis, we will simply consider the substitution $\varphi + g_+ \mapsto \varphi$ and invert the substitution only at the very end, so it will be sufficient to study
$\E_+ \theta [ (I \circ K^\dagger )(\varphi) ]$.

\subsection{Reblocking-reapportioning with the observables}
\label{sec:pertAnalysisObs}

We extend Definition~\ref{defi:Vplusdefi} to reflect the evolution of $D$.

\begin{definition} \label{defi:obsVplusdefi}
For $V \in \cV$ and $(K^\bulk , K^\obs) \in \underline{\cN}^\obs$, 
let $\delta u_{+}^\bulk$ be given by Definition~\ref{defi:Vplusdefi} and $K^\dagger$ be given by Lemma~\ref{lemma:reblockObsg}.
Then we define $\delta u_{+}^\obs$ and $\tilde{D}_+ = D + \delta \tilde{D}_+$ by
\begin{align}
    \delta u_+^{K} (b) &:= (\delta u_+^\dagger - \delta u_+^\bulk )(b) 
        := \sum_{Y \in \cS}^{Y \supset b} \Loc^{(0)}_{Y,b} \E_{j+1} \theta (K^\dagobs - K^\bulk) (Y,\varphi)
    \label{eq:deltauplusKDefi}
    \\
    \delta \tilde{D}_+ (b,\varphi) &:= \sum_{Y \in \cS}^{Y \supset b} \Loc^{(1)}_{Y,b} \E_{j+1} \theta (K^\dagobs - K^\bulk) (Y,\varphi) 
    \label{eq:deltaDplusDefi}
\end{align}
for each $b \in \cB$ and
\begin{align}
    \delta u_+^\obs (b) 
    &= \delta u_{+}^{\dagger} (b) + u^{(i)} (b) - u^\obs (b) + \delta \tilde{D}_+ (b, \nabla g_+) \nnb
    &= \delta u_{+}^{\dagger} (b) + \frac{1}{2} ( d^* (f+\omega) , \Gamma_+ d^* (f+\omega))_b + \delta \tilde{D}_+ (b, \nabla g_+) .
\end{align}
Also, for $B \in \cB_{+}$ and $X \in \cP_{+}$, let
\begin{align} 
    & \Rae^{\dagger} (B, X) 
    = 1_{B \subset X} 1_{X \in \cS_+} \sum_{b \in \cB (B)} \sum_{Y \in \cS}^{Y \supset b} \Loc_{Y,b} (\E_+ \theta K^\dagger , \E_+ \theta K^\bulk)  ( 1_{\bar{Y}=X} - 1_{B = X} ) 
    \label{eq:obskEdefi}
\end{align}
where
\begin{align}
    \Loc_{Y,b} (K_1 , K_2) = \Loc_{Y,b}^\bulk K_2 + \Loc_{Y,b}^\obs (K_1 - K_2) .     
\end{align}
\end{definition}

By \eqref{eq:deltauplusKDefi} and \eqref{eq:deltaDplusDefi}, it is straightforward to see that
\begin{align}
    & \big( \delta u_{+}^{\dagger} - \delta u_{+}^\bulk \big) (b) + \delta \tilde{D}_+ (b)  
    = \sum_{Y \in \cS_j}^{Y \supset b} \Loc_{Y,b}^\obs \E_{j+1} \theta \big( K^\dagger - K^\bulk \big)
    \label{eq:obsVplusdefi}
\end{align}
Bounds are provided in Lemma~\ref{lemma:ptobsRG}.

\begin{lemma} [Perturbative RG map] \label{lemma:ptobsRG}
$(\delta u_+^\obs, \delta \tilde{D}_+) \in \R^{\cB} \times \tilde{\cD}_+$ satisfies
\begin{align}
    & (\delta u_+^\obs - \delta u_+^\bulk ) (b) 
    = \frac{1}{2} (d^* (f+\omega), \Gamma_+ d^* (f+\omega))_b + \delta \tilde{D}_+ (b, \nabla g_+) + \delta u_+^{K} (b)
    \label{eq:uplusobsformula}
\end{align}
for some $\delta u_+^{K} (b) \in \R$ such that
\begin{align}
    \max\{ A |\delta u_+^{K} (b)|, \;\; \norm{\delta \tilde{D}_+}_{\tilde\cD_+} \} \lesssim \log L ( \norm{K^\dagger}_{j,\dagger} + \norm{K^\bulk}_{j,\bulk} )  .
\end{align}
\end{lemma}
\begin{proof}
The fact that $\delta \tilde{D}_+ \in \tilde{\cD}_+$ and
\eqref{eq:uplusobsformula} are straightforward from the definitions.
The bound on $\delta u_+^K$ follows from Lemma~\ref{lemma:LocBndd}, and the bound on $\delta \tilde{D}_+$ is in a similar spirit. 
To see this in detail, let $K' = \E_{j+1} \theta (K^\dagger - K^\bulk)$ so that
\begin{align}
    &\delta \tilde{D}_+ (b,\varphi) 
    = \sum_{Y \in \cS}^{Y\supset b} \Loc^{(1)}_{Y,b} K' (y,\varphi) \\
    &\qquad 
    = \sum_{Y \in \cS}^{Y \supset b} \sum_{x_0 \in b, \, x_1 \in Y^*} \frac{\partial_{\varphi (x_1)} \hat{K}'_0 (Y,0)}{|b| |Y|} \frac{1}{2} \sum_{\mu \in \hat{e}} (\delta x_1)^{\mu} \sum_{z \in b} \nabla^\mu \varphi_z \nonumber ,
\end{align}
so for each $\mu \in \hat{e}$ and $z \in b$,
\begin{align}
    & \tilde{d}_{j+1}^\mu (z) - d_{j}^\mu (z)
    = \sum_{Y \in \cS}^{Y \supset b}  \frac{1}{2 |b||Y|} 
    \sum_{x_0 \in b, \, x_1 \in Y^*} \partial_{\varphi (x_1)} \hat{K}'_{0} (Y,0) (\delta x_1)^\mu . \nonumber
\end{align}
Moreover, \eqref{eq:LocBnddi1}--\eqref{eq:FprimeBnd} imply
\begin{align}
    \Big| \sum_{x_1 \in Y^*} \partial_{\varphi (x_1)} \hat{K}'_{0} (Y,0) (\delta x_1)^\mu \Big| 
    \lesssim A^{-1} \log L  \big( \norm{K^\dagger}_{j,\dagger} + \norm{K^\bulk}_{j,\bulk} \big) L^j .
\end{align}
Since $|Y| \ge L^{2j}$, this gives
\begin{align}
    \norm{\delta\tilde{D}_+}_{\tilde{\cD}_+}
    =
    A L^{j} \sup_{z,\mu} |\tilde{d}_{j+1}^\mu (z) - d_{j}^\mu (z)|
    \lesssim 
    \log L  ( \norm{K^\dagger}_{j,\dagger} + \norm{K^\bulk}_{j,\bulk} ) . 
\end{align}
\end{proof}

We are now ready to define the observable RG map. 
Recall that $\PM_+^J$ was defined in Section~\ref{sec:PM},
$(\delta u_+^\bulk , V_+)$ was defined in Definition~\ref{defi:Vplusdefi}, $K^\dagger$ was defined in Lemma~\ref{lemma:reblockObsg} 
and $(\delta u_+^\obs, \delta \tilde{D}_+)$ and $\Rae^{\dagger}$ were defined in Definition~\ref{defi:obsVplusdefi}.

\begin{definition} \label{defi:obsRG}
Let $V \in \cV$ and $(K^\bulk, K^\obs) \in \underline{\cN}^\obs$. 
We take $V_+^\dagger = V_+ + \delta \tilde{D}_+$, $I_+^\dagger = \exp (V_+^\dagger)$ and
\begin{align}
    & K^\ddagger_+ = \PM^{ \Rae^{\dagger}}_+ [I, I^\dagger_+, \delta u_+^{\dagger} , K^\dagger]  .
    \label{eq:Kobsplusdefi}
\end{align}
\end{definition}

The next lemma gives an intermediate summary.

\begin{lemma} \label{lemma:obsRGmapWorks}
Let $Z^\obs (\varphi) = e^{u^\obs (\Lambda) + D (\Lambda, \nabla\varphi) } (I \circ K^{\obs} ) (\varphi + g)$.
Assuming convergence of the integral, 
\begin{align}
    Z_{+}^\obs (\varphi) := 
    \E_+ \big[ e^{(d^* f, \zeta)} \theta Z^\obs (\varphi) \big]
    = e^{u^\obs_+ (\Lambda) + D (\Lambda, \nabla\varphi)} (e^{V_j + \delta \tilde{D}_j} \circ K_+^{\ddagger} ) (\Lambda, \varphi + g_+) .
\end{align}
\end{lemma}
\begin{proof}
By Lemma~\ref{lemma:reblockObsg} and the definition of $D (\Lambda, \nabla\varphi )$, we have
\begin{align}
 \E_+ \big[ e^{(d^* f, \zeta)} \theta Z^\obs (\varphi) \big]
 &= 
   e^{u^\obs (\Lambda)} \E_+ \big[ e^{(d^* f, \zeta)  + D (\Lambda, \nabla\varphi + \nabla \zeta) } (I \circ K^{\obs} ) (\varphi +\zeta+  g) \big] \\
   &=  e^{u^\obs (\Lambda)+ D (\Lambda, \nabla\varphi )} \E_+ \big[ e^{(d^*( f+\omega), \zeta)   } (I \circ K^{\dagger} ) (\varphi +  \zeta) \big]
\end{align}
We then apply Lemma \ref{lemma:ZplusIntrmd} to see 
\begin{align}
\label{e.CMZ+}
     Z_{+}^\obs (\varphi) =
     \E_+ \big[ e^{(d^* f, \zeta) } \theta Z^\obs (\varphi) \big] 
    = e^{u^{(i)} (\Lambda) + D (\Lambda, \nabla \varphi)} \E \theta \big( I \circ K^\dagger \big) (\Lambda, \varphi + g_+) .
\end{align}
Now, by the choice of $K_+^\ddagger$ in Definition \ref{defi:obsRG} and Proposition~\ref{prop:PMproperty}
applied with $(J^\dagger, I,I_+^\dagger,\delta u_+^{\dagger},K^{\dagger})$ in place of $(J, I, I_+,\delta u_+,K)$,
\begin{align}
    \E_+ \theta (I \circ K^\dagger) 
        &= e^{\delta u^{\dagger}_+ (\Lambda)} \sum_{X \in \cP_+} e^{ (V_+ + \delta \tilde{D}_+ ) (\Lambda \backslash X) } K_+^\ddagger (X)
\end{align}
By Definition~\ref{defi:obsVplusdefi}, 
$\delta u_+^\obs = \delta u_{+}^{\dagger}  + u^{(i)}  - u^\obs  + \delta \tilde{D}_+ (\cdot, \nabla g_+)$, 
thus in the display above, $ \delta u_+^\dagger + u^{(i)}  + \delta \tilde{D}_+ (\cdot, \nabla g_+) = u^\obs_+$. 
\end{proof}

\subsection{Final step: projection on line observable}
\label{sec:finalStep}

The RG step is almost complete, except that $\delta\tilde{D}_+$ is not yet supported on $\supp (d^* f)$.
We consider a projection of $\delta\tilde{D}_+$ on $\supp (d^* f)$ defined by
\begin{align}
    & w(z,z') = \one_{z \in \supp (d^* f)} \sum_{b' \in \cB ( b(z)^{****} )}  \one_{z' \in b'} \frac{1}{|(b')^{****} \cap \supp (d^* f)|} , \\
    & \delta\hat{D}_+ (b,\varphi) = \sum_{\mu \in \hat{e}} \sum_{z \in \Lambda} \sum_{z' \in b} w(z,z') \delta \tilde{d}_+^\mu (z') \nabla^\mu \varphi (z) 
\end{align}
for $\mu \in \hat{e}$ and $b \in \cB_j$,
where $b(z)$ is the unique $j$-block containing $z$ and the quartic small-set neighbourhood $b^{****}$ is taken at scale $j$.
Since $\sum_{z \in \Lambda} w(z,z') = 1$, we can restate 
\begin{align}
    \delta \tilde{D}_+ (b,\varphi) = \sum_{\mu\in \hat{e} , \, z\in\Lambda, \, z' \in b} w(z,z') \delta \tilde{d}_+^{\mu} (z') \nabla^\mu \varphi (z') .
    \label{eq:tildeDAlt}
\end{align}
It is in this sense that $\delta \hat{D}_+$ is a projection of $\delta \tilde{D}_+$ on $\supp (d^* f)$.

Next, define
\begin{align}
    & \delta D_+ (Z,\varphi) 
    = \sum_{z \in Z , \, \mu \in \hat{e} } \delta d_+^\mu (z) \nabla^\mu \varphi (z) , \\
    & \delta d_+^\mu (z) = \sum_{b \in \cB} \one_{z \in \supp (d^* f) \cap b^{****}} \sum_{z' \in b} w(z,z') \delta \tilde{d}_+^\mu (z') .
\end{align}
By construction, we have
\begin{align}
    \delta D_+ (\Lambda, \varphi) = \delta \hat{D}_+ (\Lambda, \varphi)
    = \sum_{\mu\in\hat{e}} \sum_{b \in \cB} \sum_{z\in \supp (d^* f) \cap b^{****}} \sum_{z' \in b} w(z,z') \delta \tilde{d}_+^{\mu} (z') \nabla^\mu \varphi (z) ,
    \label{lemma:deltaDplusDhatEqual}
\end{align}
so we will be able to replace $\delta \hat{D}_+$ by $\delta D_+$ at the very last step.
But for substitution of $\delta \tilde{D}_+$, the role of $\delta \hat{D}_+$ is crucial, due to the next lemma.

\begin{lemma} \label{lemma:deltatildeDMinuHatD}
Suppose $\delta \tilde{D}_+ \in \tilde{\cD}_+$. Then for $B \in \cB_+$,
\begin{align}
    \norm{\delta \tilde{D}_+ (B,\varphi) - \delta \hat{D}_+ (B,\varphi)}_{h,B,T_+ (\varphi)}
    \lesssim A^{-1} L^{-1} \norm{\delta \tilde{D}_+}_{\tilde{\cD}_+} \norm{\nabla_{j+1}^2 \varphi}_{L^\infty (B^*)} .
\end{align}
\end{lemma}
\begin{proof}
By \eqref{eq:tildeDAlt} and the definition of $\delta \hat{D}_+$,
\begin{align}
    \big( \delta \tilde{D}_+ - \delta \hat{D}_+ \big) (B,\varphi)
    = 
    \sum_{\mu\in \hat{e} , \, z\in\Lambda, \, z' \in B} w(z,z') \delta \tilde{d}_+^{\mu} (z') \big( \nabla^\mu \varphi (z') - \nabla^\mu \varphi (z) \big) .
\end{align}
Also, since $\dist (z,z') \le O(L^j)$ whenever $w(z,z') \neq 0$, $\nabla^\mu \varphi (z') - \nabla^\mu \varphi (z)$ can be expressed as a sum of $O(L^j)$ number of $\nabla^2 \varphi$ terms inside $B^*$, so
\begin{align}
    \norm{\delta \tilde{D}_+ (B,\varphi) - \delta \hat{D}_+ (B,\varphi)}_{h,B,T_+ (\varphi)}
    & \lesssim L^{j-1} \norm{\nabla_{j+1}^2 \varphi}_{L^\infty (B^*)} \sup_{z', \mu} |\delta \tilde{d}_+^{\mu}| 
    \nnb
    &= A^{-1} L^{-1} \norm{\nabla_{j+1}^2 \varphi}_{L^\infty (B^*)} \norm{\delta \tilde{D}_+}_{\tilde{\cD}_+}  
    .
\end{align}
\end{proof}

The next bound is used to trace the evolution of $D$.

\begin{lemma} \label{lemma:deltaDplusBnd}
Suppose $\delta \tilde{D}_+ \in \tilde{\cD}_+$. Then
\begin{align}
    \norm{\delta D_+}_{\cD_+} \lesssim
    L^{ (j-j_{\bP}) \vee 0} \norm{\delta \tilde{D}_+}_{\tilde{\cD}_+}
\end{align}
\end{lemma}
\begin{proof}
If $\delta \tilde{d}^\mu_+ (z') \neq 0$ and $z' \in b$, 
then we are already assuming $\supp (d^* f) \cap b^{***} \neq \emptyset$.
Now, by taking another neighbourhood $b^{****}$, 
we would have $|\supp (d^* f) \cap b^{****} | \ge L^{j \wedge j_{\bP}} /100$ (see Assumption~\ref{asmp:1Dobs}),
so we would have
\begin{align}
    w(z,z') \lesssim L^{-(j \wedge j_{\bP})} 
\end{align}
Thus,
\begin{align}
    A \sup_{z,\mu} |\delta d_+^\mu (z)| \lesssim A L^{2j} L^{-(j \wedge j_{\bP})} \sup_{z,\mu} |\delta \tilde{d}_+^{\mu} (z)| 
    = L^{ (j-j_{\bP}) \vee 0} \norm{\delta \tilde{D}_+}_{\tilde{\cD}_+} .
\end{align}
(The factor $L^{2j}$ originates from summing over $z' \in b$.)
\end{proof}

With these at hand, we can complete the definition of the observable RG map.

\begin{definition} \label{defi:Kobsplusdefi}
Let $V \in \cV$ and $(K^\bulk, K^\obs) \in \underline{\cN}^\obs$. 
\begin{align}
    K_+^{\obs} (X) = e^{-\delta \hat{D}_+ (X)} \sum_{Y \in \cP_+ (X)} e^{V_+ (X \backslash Y)} \big( e^{\delta \tilde{D}_+} - e^{\delta \hat{D}_+} \big)^{X \backslash Y} K_+^\ddagger (Y) 
    .
\end{align}
This defines the observable RG map
\begin{align}
    \Phi^\obs_+ : (D_- , D , V , K^\bulk, K^\obs) \mapsto ( \delta u^\obs_+ , D_+ , V_+ , K^\obs_+) .
\end{align}
\end{definition}

The full observable RG map can be summarised diagrammatically as:
\begin{equation}
\label{eq:observableRGstages}
K_j^\sigma
\xrightarrow{\text{absorb }g_j}
K_j^\dagger
\xrightarrow{\text{RG step}}
K_{j+1}^\ddagger
\xrightarrow{\text{projection and} \times e^{-\delta \hat{D}_{j+1}}\;}
K_{j+1}^\sigma .
\end{equation}

\begin{remark} \label{remark:localityRG}
The observable RG map is a local operation, but in a slightly weaker sense compared to Proposition~\ref{prop:PMproperty}. To be precise, given $X \in \cP_+$, $(\delta u_+^\obs (X) , \delta D_+ (X) , V_+ (X) , K_+^\obs (X))$ depend only on $D_{-} (Y)$, $D(Y)$, $V(Y)$, $K^\bulk (Y)$ and $K^\obs (Y)$ for either $Y \in \cP(X)$ or $Y \in \cS$ such that $Y\cap X \neq \emptyset$.
This follows from the locality of the progression map (Proposition~\ref{prop:PMproperty}) and the definition of $\Rae^{\dagger}$: since $\Rae^{\dagger} (B,B)$ depends on any $Y \in \cS$ such that $Y \cap B \neq \emptyset$, the RG map on $X$ also depends on small polymers at scale $j$ smearing out from $X$.
\end{remark}

\begin{proposition} \label{prop:KplusobsBnd}
With the definition of $K_+^{\obs}$ as in Definition~\ref{defi:Kobsplusdefi}, we have
\begin{align}
    (e^{V_+ + \delta\tilde{D}_+} \circ K_+^{\ddagger} ) (\Lambda, \varphi) 
    = e^{\delta D_+ (\Lambda, \varphi)} (I_+ \circ K_+^{\obs}) (\Lambda, \varphi) .
    \label{eq:KplusobsId}
\end{align}
If $\norm{V_+}_{\cV_+}$, $\norm{\delta \tilde{D}_+}_{\tilde{\cD}_+}$ and $\norm{K_+^\ddagger}_{j+1, \ddagger}$ are sufficiently small, then
\begin{align}
    \norm{K_+^\obs}_{j+1, \obs} \lesssim \norm{K_+^\ddagger}_{j+1, \ddagger} + \norm{V_+}_{\cV_+} + \norm{\delta \tilde{D}_+}_{\tilde{\cD}_+} .
\end{align}
\end{proposition}
\begin{proof}
By a proof that is essentially the same as that of Lemma~\ref{lemma:reblockObsg}, we get
\begin{align}
    (e^{V_+ + \delta\tilde{D}_+} \circ K_+^{\ddagger} )
    = e^{V_+ + \delta \hat{D}_+} \circ \big( ( e^{V_+ + \delta \tilde{D}_+} - e^{V_+ + \delta \hat{D}_+} ) \circ K_+^\ddagger \big) ,
\end{align}
and using Lemma~\ref{lemma:VjDjbnd} and \ref{lemma:deltatildeDMinuHatD} to bound $e^{V_+ + \delta \tilde{D}_+} - e^{V_+ + \delta \hat{D}_+}$, we get
\begin{align}
    & \Big\| \big( ( e^{V_+ + \delta \tilde{D}_+} - e^{V_+ + \delta \hat{D}_+} ) \circ K_+^\ddagger \big) (X,\varphi) \Big\|_{h,X,T_+ (\varphi)}
    \nnb
    &\qquad \lesssim (A/3)^{-|X|_{j+1}} \big( \norm{K_+^\ddagger}_{j+1, \ddagger} + \norm{V_+}_{\cV_+} + \norm{\delta \tilde{D}_+}_{\tilde{\cD}_+} \big) G_+^{(k + k_{+})/2} (X,\varphi)
    \label{prop:KplusobsBndInt}
\end{align}
The final step amounts to factoring out $\delta \hat{D}_+$ from the circle product, which gives
\begin{align}
    &(e^{V_+ + \delta\tilde{D}_+} \circ K_+^{\ddagger} ) (\Lambda)
        = e^{\delta \hat{D}_+ (\Lambda)} (e^{V_+} \circ K_+^\obs) (\Lambda)  \\
    &K_+^\obs (X,\varphi) 
        = e^{-\delta \hat{D}_+ (X,\varphi)} \big( ( e^{V_+ + \delta \tilde{D}_+} - e^{V_+ + \delta \hat{D}_+} ) \circ K_+^\ddagger \big) (X,\varphi) .
    \label{eq:KplusobsDefi2}
\end{align}
But $\delta \hat{D}_+ (\Lambda) = \delta D_+ (\Lambda)$ due to \eqref{lemma:deltaDplusDhatEqual}, so the proof of \eqref{eq:KplusobsId} is complete.

For the bound, if $\norm{\tilde{D}_+}_{\tilde{D}}$ is sufficiently small, we can use Lemma~\ref{lemma:VjDjbnd} and \ref{lemma:deltatildeDMinuHatD} to bound
\begin{align}
   \norm{e^{-\delta \hat{D}_+ (X)}}_{h,X,T_+(\varphi)} \le
   \prod_{B \in B_+(X)} (4/3) \big( G_+^\str (B,\varphi) \big)^{2^{-j-2}} 
   \leq
   (4/3)^{|X|_{+1}} \big( G_+^\str (X,\varphi) \big)^{2^{-j-1}} .
\end{align}
Putting this together with \eqref{prop:KplusobsBndInt} and \eqref{eq:KplusobsDefi2},
\begin{align}
    & \norm{K_+^\obs (X,\varphi)}_{h,X,T_+ (\varphi)} \nnb
    & \lesssim (A/4)^{-|X|_{j+1}} \big( \norm{K_+^\ddagger}_{j+1, \ddagger} + \norm{V_+}_{\cV_+} + \norm{\delta \tilde{D}_+}_{\tilde{\cD}_+} \big) G_+^{(k+k_+)/2} (X,\varphi)\big( G_+^\str (X,\varphi) \big)^{2^{-j-2}} \nnb
    &\le (A/4)^{-|X|_{j+1}} \big( \norm{K_+^\ddagger}_{j+1, \ddagger} + \norm{V_+}_{\cV_+} + \norm{\delta \tilde{D}_+}_{\tilde{\cD}_+} \big) G_+^{k_{j+1}}
\end{align}
where we used Lemma~\ref{lemma:Gjobsproperties}(iii) for the final bound.
\end{proof}

\subsection{Estimate on the full observable RG map}

The key estimate on $K_+^\obs$ is summarised in Lemma~\ref{lemma:Kobsplusestimate}. We will soon see in Section~\ref{sec:obsRGflow} that these estimates are enough for proving the stability of the observable RG flow.

\begin{lemma} \label{lemma:Kobsplusestimate}
Assume $\norm{D_-}_{\cD_-} \lesssim A L^{(j-j_{\bP}) \vee 0} /\alpha$.
There exist $\epsilon_r (\alpha) > 0$ (that is only polynomially small in $\alpha$) such that, whenever $(K^\bulk, K^\dagger) \in \underline{\cN}^\dagger$ and $V \in \cV$ satisfy $\norm{K^\bulk}_{\bulk}, \norm{K^\dagger}_{\dagger} < 2^{-j} \epsilon_{r}$ and $\norm{V}_{\cV} < \epsilon_{r}$, then
\begin{align}
    \norm{K_+^\obs}_{+,\obs} &\lesssim L^{-\etabulk}  \big(\norm{K^\bulk}_{\bulk} + \norm{K^\dagger}_{ \dagger} + \norm{V}_{\cV} \big)  .
    \label{eq:lemmaKobsplusestimate1}
\end{align}
Also, $(\delta u_+^\obs - \delta u_+^\bulk ) (B) = \frac{1}{2} (d^* (f+ \omega), \Gamma_{+} d^* (f+ \omega))_B + \delta u_+^r (B)$ for each $B \in \cB_+$ and satisfies
\begin{align}
    |\delta u_+^r (B)| \le O_L \big( e^{-\gamma \alpha^2 /4} L^{-\eta j} + \norm{K^\dagger}_{\dagger} \big) .
    \label{eq:deltaruplusB}
\end{align}
\end{lemma}

The proof is divided into three steps. We first discuss the easy part, the estimate on $\delta u_+^K$. To estimate $K_+^\obs$, the linearisation $\LPM_+^\obs$ is bounded in Section~\ref{sec:obsLPMbnd} and the bound on the high-order terms is given in Section~\ref{sec:obsRPMbnd}, whose proof is deferred to Appendix~\ref{app:remainderBnd}. These two bounds can be put together to give the bound on $K_+^\obs$. 

\begin{proof}[Proof of \eqref{eq:deltaruplusB}]
We use \eqref{eq:uplusobsformula} and bound each of its components.
By Lemma~\ref{lemma:ptobsRG} and Proposition~\ref{prop:critpt}, we have 
\begin{align}
    |\delta u_{+}^{K} (B)| \le O_L (e^{-\gamma\alpha^2/4} L^{-\eta j} + \norm{K^\dagger}_{j,\dagger} ) .
\end{align}
We use Lemma~\ref{lemma:ptobsRG} for $\delta \tilde{D}_{+}$ and Lemma~\ref{lemma:gjbound} (using the assumption $\norm{D_-}_{\cD_-} \lesssim A L^{(j-j_{\bP}) \vee 0}/\alpha$) for $g_+$ to see 
\begin{align}
    |\delta \tilde{D}_{+} (B, \nabla g_+)| \le O_L (\alpha^{-1}) \big( \norm{K^\bulk}_{j, \bulk} + \norm{K^\dagger}_{j, \dagger} \big) .
\end{align}
We get the desired estimate on $\delta u_+^r$ since $\delta u_+^r (B) = \delta \tilde{D}_+ (B,\nabla g_+) + \delta u_+^K$.
\end{proof}

\subsubsection{Linearisation}
\label{sec:obsLPMbnd}

The dominant contribution to $K_+^\obs$ originates from the linearised progression map. Recalling Definition~\ref{defi:LPM}, let
$\LK_+^\ddagger = \LPM_+^{\Rae^{\dagger}} [I,I_+^\dagger,\delta u_+^\dagger,K^\dagger]$.

\begin{lemma} \label{lemma:obsLPMbnd}
Under the assumptions of Lemma~\ref{lemma:Kobsplusestimate}, if $X \in \Conn_+$,
\begin{align}
    &\norm{ \LK_+^\ddagger (X,\varphi)}_{h,X,T_{+} (\varphi)}
    \lesssim L^{-\etabulk} \big( \norm{K^\dagger}_{j,\dagger} + \norm{K^\bulk}_{j,\bulk} \big) A^{-|X|_{j+1}} G_+^{k} (X,\varphi) .
    \label{eq:lemmaobsLPMbnd}
\end{align}
\end{lemma}
\begin{proof}
We plug in \eqref{eq:obskEdefi} into the definition \eqref{eq:LPMdefi}, notice that $V_+^\dagger = V_+ + \delta \tilde{D}_+$ to obtain
\begin{align}
    & \LK_+^\ddagger (X)
    = \sum_{\bar{Y}=X} 1_{Y \in \Conn} \E_+ \theta K^\dagger (Y) - 1_{Y \in \cS} \Loc_{Y} ( \E_+ \theta K^\dagger , \E_+ \theta K^\bulk) \\
    &
    \quad +\sum_{b \in \cB}^{\bar{b}=X} \Big[ \E_+ \theta V (b, \varphi) - \delta u_{+}^{\dagger} (b) - \delta \tilde{D}_+ (b) - V_+ (b) + \sum_{Y \in \cS}^{Y \supset b} \Loc_{Y,b} (\E_+ \theta K^\dagger , \E_+ \theta K^\bulk) \Big] .
    \nonumber
\end{align}
The second line vanishes by adding together \eqref{eq:Vplusdefi}, \eqref{eq:deltauplusKDefi}, and \eqref{eq:deltaDplusDefi}. 
Thus we can expand $\Loc$ using $K^\dagger = K^\bulk+ (K^\dagger - K^\bulk)$ to see that
\begin{align} \label{eq:LKddagger}
    \LK_+^\ddagger (X) 
    &= \sum_{\bar{Y} = X} 1_{Y\in \Conn} \E_+ \theta K^\bulk (Y) - 1_{Y\in \cS} \Loc_{Y}^{\bulk} \E_+ \theta K^\bulk \nnb
    &\quad + \sum_{\bar{Y} = X} 1_{Y\in \Conn} \E_+ \theta (K^\dagger - K^\bulk) (Y) - 1_{Y\in \cS} \Loc_{Y}^{\obs} \E_+ \theta (K^\dagger - K^\bulk) .
\end{align}
Recall the definition of $\LK_+^\bulk$ in \eqref{defi:LPM}, and the definition of  $\Rae^\bulk (B, X) $ in $\eqref{eq:kEdefi}$, we get
\begin{align}
    \LK_+^\bulk(X)
    =
    \sum_{\bar Y=X}
    \left[
    1_{Y\in\Conn}\E_+\theta K^\bulk(Y)
    -
    1_{Y\in S}\Loc^\bulk_Y \E_+\theta K^\bulk(Y)
    \right]
    \nnb
    +
    \sum_{\bar b=X}
    \left[
    \E_+\theta V(b,\phi)
    -
    \delta u^\bulk_+(b)
    -
    V_+(b,\phi)
    +
    \sum_{\substack{Y\in S}}^{Y\supset b}
    \Loc^\bulk_{Y,b}\E_+\theta K^\bulk(Y)
    \right].
\end{align}
(Note that, for the bulk case, $\Loc_{Y,b}^\bulk$ vanishes unless $Y \supset b$.)
By Definition~\ref{defi:Vplusdefi}, the second line above vanishes,  thus the first line of \eqref{eq:LKddagger} equals $\LK_+^\bulk(X)$,
whose bound is already given in Lemma~\ref{lemma:LPMbnd}, which is bounded by the right-hand side of \eqref{eq:lemmaobsLPMbnd}.

For the second line of \eqref{eq:LKddagger}, we split cases $Y \in \cS$ and $Y \not\in \cS$. For $Y \in \cS$. Corollary~\ref{cor:obssmallsetcont} implies 
\begin{align}
    & \Big\| \sum_{\bar{Y}=X}^{Y \in \cS} (1-\Loc_Y^\obs) \E_+ \theta (K^\dagger - K^\bulk) (Y) \Big\|_{h,X,T_{+}(\varphi)} \nnb
    & \qquad
    \lesssim L^{-1-\etabulk} \norm{K^\dagger - K^\bulk}_{\dagger} G_+^{k} (X,\varphi) \sum_{\bar{Y}=X}^{Y\in \cS} A^{-|Y|_j} \one_{K^\dagger (Y) \neq K^\bulk (Y)} . 
\end{align}
Due to the compatibility condition (Definition~\ref{defi:cNobsdefi}), $(K^\bulk,K^\dagger) \in \underline{\cN}^\dagger$ implies $K^\dagger (Y) \neq K^\bulk (Y)$ only for $O(L)$ number of $Y \in \cS$. Also, $\bar{Y}=X$ implies $|Y|_j \ge |X|_{j+1}$, so
\begin{align}
    \lesssim L^{-\eta^\bulk} A^{-|X|_{j+1}} G_+^{k} (X,\varphi) \norm{K^\dagger -K^\bulk}_{\dagger} .
\end{align}
When $Y \in \Conn \backslash \cS$, then Proposition~\ref{prop:largesetcont} implies
\begin{align}
    & \Big\| \sum_{\bar{Y}=X}^{Y\in \Conn \backslash \cS} \E_+ \theta (K^\dagger - K^\bulk) (Y) \Big\|_{h,X,T_+ (\varphi)}  
        \lesssim (LA/4)^{-|X|_{j+1}} G_+^{k} (X,\varphi) \norm{K^\dagger - K^\bulk}_{\dagger}
\end{align}
and $(L/4)^{-1} \le L^{-\eta^\bulk}$ for suffciently large $L$. Putting these together, we get the desired bound.
\end{proof}

\subsubsection{Bound on the remainder}
\label{sec:obsRPMbnd}

The remainder is given by
\begin{align}
    \RK_+^\ddagger := K_+^{\ddagger} - \LK_+^\ddagger .
\end{align}
The remainder consists of terms of order $\ge 2$ in $V$ and $K$. Lemma~\ref{lemma:RPMboundObs} gives a precise meaning to this in terms a norm with large-field regulator $G_+^k$. It is proved in Appendix~\ref{app:remainderBnd}.

\begin{lemma} \label{lemma:RPMboundObs}
With the assumptions of Lemma~\ref{lemma:Kobsplusestimate},
\begin{align}
    \big\| \RK_+^\ddagger \big\|_{h,A,T_{+}(G_+^{k}) } \le O_L (1)  \big( \norm{K^\bulk}_{\bulk} + \norm{K^\dagger}_{\dagger} + \norm{V}_{\cV} \big)^2 .
\end{align}
\end{lemma}

Bounds on $\LK_+^\ddagger$ and $\RK_+^\ddagger$ almost imply the bound on $K_+^\obs$. But due to the prefactor $e^{-\delta D_+}$ present inside \eqref{eq:Kobsplusdefi}, the large field regulator $G_+^k$ needs to be replaced by $G_+^{k_+}$. This is the motivation of having \eqref{eq:kappajdefi}. 

\begin{proof}[Proof of Lemma~\ref{lemma:Kobsplusestimate}]
Let $\epsilon_r >0$ be sufficiently small, $\norm{K^{\bulk}}_{\bulk}$, $\norm{K^\dagger}_{\dagger} < 2^{-j} \epsilon_r$ and $\norm{V}_{\cV} < \epsilon_r$.
Since $K_+^\ddagger = \LK_+^\ddagger + \RK_+^\ddagger$, Lemma~\ref{lemma:obsLPMbnd} and \ref{lemma:RPMboundObs} imply
\begin{align}
    \norm{K_+^\ddagger (X)}_{h,X,T_+(\varphi)} & \lesssim 
    A^{-|X|_{j+1}} G_+^{k} (X,\varphi) \Big( L^{-\etabulk} \big(\norm{K^\bulk}_{\bulk} + \norm{K^\dagger}_{\dagger} \big) \qquad \nnb
    & \qquad\qquad\qquad + O_L \big(\norm{K^\bulk}_{\bulk} + \norm{K^\dagger}_{\dagger} + \norm{V}_{\cV} \big)^2  \Big) \nnb
    &\lesssim 
    A^{-|X|_{j+1}} G_+^{k} (X,\varphi) L^{-\etabulk} \big(\norm{K^\bulk}_{\bulk} + \norm{K^\dagger}_{\dagger} + \norm{V}_{\cV} \big)
    \label{eq:lemmaKobsplusestimate3}
\end{align}
for $X\in \Conn_+$.
Since we also have a bound on $\delta \tilde{D}_+$ from Lemma~\ref{lemma:ptobsRG},
the bound on $K_+^\obs$ now follows from Lemma~\ref{lemma:deltatildeDMinuHatD} and Proposition~\ref{prop:KplusobsBnd}.

The bound on $\delta u^r_+$ was already proved right below the statement of the lemma.
\end{proof}

\subsection{Observable RG flow}
\label{sec:obsRGflow}

We now prove that the dynamical system generated by the observable RG map is stable under Assumption~\ref{asmp:critical}.

\begin{proposition} \label{prop:obsRGflowExists}
Let $N\in \N \cup \{\infty\}$.
Consider the bulk RG maps on $\Lambda_N$ and assume Assumption~\ref{asmp:critical} holds.
There exists an observable RG flow $(D_j, V_j , K_j^\bulk , K^\obs_j)_{j=0}^N$ and free energy $(\delta u^\obs_j )_{j=0}^N$ with initial condition $\Omega_0 = D_0=0$ and $K_0^\obs$ constructed in Lemma~\ref{lemma:K0obsdefi}.
Namely, there exists
\begin{align}
    (\delta u^\obs_j , D_j, K_j^\obs )_{j\ge 0} \in 
    \prod_{j=0}^\infty \R^{\cB_j} \times \cD \times K_j
\end{align}
such that $(\delta u_{j+1}^\obs , D_{j+1} , V_{j+1} , K_{j+1}^\obs ) = \Phi_{j+1}^\obs (D_{j-1}, D_j, V_{j}, K_j^\bulk, K_j^\obs)$ for all $j\ge 0$ (with the convention $D_{-1} =0$) and $(K_j^\bulk, K_j^\obs) \in \underline{\cN}_j^\obs$.

Moreover, the flow satisfies, for each $b \in \cB_j$ and every $\eta \in (0, \eta^\bulk)$,
\begin{align}
\begin{cases}
    \max \big\{ \norm{K^\obs_j}_{j, \obs} , \;\; |\delta u_j^r (b)| \big\}
    \le O_L (e^{-\gamma\alpha^2/4} L^{-\eta j}) \\
    \norm{\delta D_j}_{\cD_j}
    \le O_L (e^{-\gamma\alpha^2/4} L^{-\eta j}) L^{(j-j_{\bP}) \vee 0} .
\end{cases}    
\label{eq:propobsRGflowExists}
\end{align}
\end{proposition}
\begin{proof} 
As an induction hypothesis, assume that for all $j' \le j$,  $(K_{j'}^\bulk, K_{j'}^\obs) \in \underline{\cN}_{j'}^\obs$, and 
\begin{align}
\begin{cases}
    \norm{V_{j'}}_{\cV_{j'}} \le M L^{-\eta {j'}} , \\
    \max\big\{ 
    \norm{K_{j'}^\bulk}_{j' ,\bulk }, \; \norm{K_{j'}^\obs}_{j',\obs} \big\} \le M' L^{-\eta {j'}} , \\
    \norm{\delta \tilde{D}_{j'}}_{\tilde\cD_{j'}} \le M'' L^{-\eta {j'}}.
\end{cases}
\end{align}
(The bound on $\delta \tilde{D}$ implies the bound on $\delta D$ due to Lemma~\ref{lemma:deltaDplusBnd}.)
Here, the constants are chosen as $M = M_1 e^{-\frac{1}{4}\gamma \alpha^2}$, $M' = M'_1 e^{-\frac{1}{4}\gamma \alpha^2}$ and $M'' = M''_1 e^{-\frac{1}{4}\gamma \alpha^2}$ for some $M_1, M'_1, M''_1<\infty$. Notice that the bulk flow is already controlled by Proposition~\ref{prop:critpt}, which implies 
$V_{j'}$ and $K^\bulk_{j'}$ already satisfy the desired bounds.
To  verify the base case $j=0$ for the other observables, notice that since $K_0^\obs$ is constructed in Lemma~\ref{lemma:K0obsdefi}, it follows from that lemma $\norm{K_{0}^\obs}_{0,\obs} \le A|s|$,  by the choice $s= s_0^c(\beta)$ in Assumption~\ref{asmp:critical} (see also Proposition~\ref{prop:critpt} for the bound $s_0^c (\beta) = O(e^{-\gamma \alpha^2 /4})$). The initial bound for $\delta \tilde{D}_0$ is trivial.

We now verify the induction hypothesis for $j+1$. By Lemma~\ref{lemma:reblockObsg}, there exists an $L$-independent constant $C_1 >0$ and and $L$-dependent constant $C_L >0$ such that
\begin{align}  
    \norm{K_j^\dagobs}_{j,\dagger} \lesssim  O_L \big(\norm{V_j}_{\cV_j} \big)  \vee \norm{K_j^\obs}_{j,\obs}  \le C_1 (M' + C_L M) L^{-\eta j}  .
\end{align}
By the induction hypothesis and that $D_j = \sum_{k\le j} \delta D_k$,  we see that the conditions of Lemma~\ref{lemma:Kobsplusestimate} are satisfied for sufficiently large $L$ and $\beta = \alpha^2$, so we also have
\begin{align}
    \norm{K^\obs_{j+1}}_{j+1, \obs} &\le C_2 L^{-\etabulk} (M' + C_L M) L^{-\eta j} \le M' L^{-\eta (j+1)} ,
\end{align}
for $M' \gg C_L M$. Here $\eta$ is chosen such that $\eta<\eta^\bulk$. 
The bound on $\delta D_{j+1}$ follows due to Lemma~\ref{lemma:ptobsRG}, when $M'' \gg M' \log L$. 
This completes the induction.
The bound on $\delta u_{j}^r$ follows from \eqref{eq:deltaruplusB}.
\end{proof}

\section{Final integral and reduction of the problem}
\label{sec:fiarotp}

Existence of an infinite observable RG flow also allows polymer expansions of $(Z_j^\obs)_{j=0}^N$ on a finite volume $\Lambda_N$.
In words, Proposition~\ref{prop:obsRGflowExists} gives an observable RG flow of arbitrary length under a suitable initial condition. Thus, we can apply Lemma~\ref{lemma:obsRGmapWorks} and Lemma~\ref{lemma:reblockObsg} repeatedly up to scale $N$ to get
\begin{align}
    Z_N^\obs (\varphi) = e^{u_N^\obs (\Lambda) + \Omega_N (d\varphi)} \big( e^{V_N} \circ K_N^\obs \big) (\Lambda, \varphi + g_N) = e^{u_N^\obs (\Lambda) + \Omega_N (d\varphi)} \big( e^{V_N} \circ K_N^\dagger \big) (\Lambda, \varphi) ,
\end{align}
Thus, the application of Lemma~\ref{lemma:DGmgfasZNint} is now upto a final fluctuation integral in $\varphi_{\hat{N}} \sim \cN (0, t_N Q_N)$. 

\begin{lemma} \label{lemma:EZNobsEZNbulkRatio}
Let $\E_{\hat{N}}$ denote the expectation in $\varphi_{\hat{N}} \sim \cN (0, t_N Q_N)$.
Then under the setting of Proposition~\ref{prop:obsRGflowExists},
\begin{align}
    \frac{\E_{\hat{N}} [Z_N^\obs (\varphi_{\hat{N}} )] }{\E_{\hat{N}} [Z_N^\bulk (\varphi_{\hat{N}})]}
        = e^{(u_N^\obs - u_N^\bulk)(\Lambda)} \big( 1+ O_L (L^{-\eta N} ) \big) .
\end{align}
\end{lemma}
\begin{proof}
Recall that $Q_N$ is a constant-valued matrix, so any random variable with covariance $t_N Q_N$ is almost surely constant-valued. Thus, $\Omega_N (d \varphi_{\hat{N}})=0$ and $|\nabla \varphi_{\hat{N}}|^2 = 0$ almost surely. 
We are only left with 
\begin{align}
    \E_{\hat{N}} Z_N^\obs (\varphi_{\hat{N}}) = e^{u_N^\obs (\Lambda)} \E_{\hat{N}} [ e^{V_N^Q (\Lambda,\varphi_{\hat{N}})} \circ K_N^\dagger (\Lambda, \varphi_{\hat{N}})  ] 
    =
    e^{u_N^\obs (\Lambda)} \E_{\hat{N}} [ e^{V_N^Q (\Lambda,\varphi_{\hat{N}})} +K_N^\dagger (\Lambda, \varphi_{\hat{N}})  ] 
\end{align}

Since $\norm{V_N}_{\cV_N} \le O_L (L^{-\eta N})$, we get
\begin{align}
    |V_N^Q (\Lambda,\varphi_{\hat{N}})| \le \sum_{q \ge 0} |z_N^{(q)}| \le O_L ( L^{-\eta N}) .
\end{align}
For $K_N^\dagger$, the definition of the $\dagger$-norm, Proposition~\ref{prop:obsRGflowExists} and Lemma~\ref{lemma:reblockObsg} together imply
\begin{align}
    |K_N^\dagger (\Lambda, \varphi_{\hat{N}})| \le (A/4)^{-1} \norm{K_N^\dagger}_{N,\dagger} G_N^\dagger (\Lambda, \varphi_{\hat{N}})
    \le O_L( L^{-\eta N} ) G_N^\dagger (\Lambda, \varphi_{\hat{N}}) .
\end{align}
Again since $\nabla \varphi_{\hat{N}} =0$, \eqref{eq:Gjdaggerdefi} and Lemma~\ref{lemma:gjbound} imply $G_j^\dagger (\Lambda, \varphi_{\hat{N}}) \lesssim O_L (1)$, so we actually have
\begin{align}
    |K_N^\dagger (\Lambda, \varphi_{\hat{N}})| \le O_L (L^{-\eta N}) ,
\end{align}
and thus
\begin{align}
    \E_{\hat{N}} Z_N^\obs (\varphi_{\hat{N}}) = e^{u_N^\obs (\Lambda)} (1 + O_L (L^{-\eta N}) ). 
\end{align}
The same happens for $\E_{\hat{N}} Z_N^\bulk (\varphi_{\hat{N}})$ just with $u_N^\obs$ replaced by $u_N^\bulk$, so we get the desired result.
\end{proof}

Therefore, we obtain a cornerstone representation of the moment generating function, written in terms of the \emph{local line energy}
\begin{align}
    \ke_j (b) = \begin{cases}
        u_0^\sigma (b) - \frac{2\pi^2\kn^2}{\beta} (\one_\Upsilon, \one_\Upsilon)_b + \frac{\gamma}{2} (d^* f, d^* \tilde{f})_b & (j = 0) \\
        \delta u_j^r (b) & (1\le j \le N-1) .
    \end{cases}
    \label{eq:kejbDefi}
\end{align}
for $f = \frac{2\pi\kn}{\alpha} \one_\Upsilon$, where $\delta u_j^r$ satisfies the bound \eqref{eq:deltaruplusB}.
Also, $u_0^{\sigma} (b) = -s\gamma (\Delta d^* f , g_0)_b$.
$1_{\Upsilon}$ is an edge function, so the inner product requires a subtle modification, see \eqref{eq:innProdOnEdges}.
The representation is stated in the next result.

\begin{corollary} \label{cor:finalintegralconcl}
Let $\Upsilon$ satisfy Assumption~\ref{asmp:1Dobs}. If $\alpha, L$ are sufficiently large, then
\begin{align}
    & -\frac{2\pi^2\kn^2}{\beta} |\Upsilon| + \log \E^{\dg}_{\beta, \Lambda} [e^{(d^*f, \varphi)/\alpha}] 
    = \sum_{j=0}^{N-1} \Big[ \frac{1}{2} ( d^* f_j , \Gamma_{j+1}  d^* f_j ) +  \ke_j (\Lambda) \Big] + O_L (L^{-\eta N})  
    \label{eq:finalintegralconcl}
\end{align}
where $f_j = f + \omega_j \in \Lambda^{(1)} (\R)$ and $\ke_j (\Lambda) = \sum_{b \in \cB_j} \ke_j (b)$.
\end{corollary}
\begin{proof}
By Lemma~\ref{lemma:DGmgfasZNint} and \ref{lemma:EZNobsEZNbulkRatio},
\begin{align}
    \E^{\dg}_{\beta,\Lambda_N} [e^{(d^*f,\sigma)/\alpha}] = e^{\frac{\gamma}{2} (d^* f , d^* \tilde{f})} e^{(u_N^{\obs}-u_N^\bulk)(\Lambda)} (1 + O_L (e^{-\eta N}) ) .
\end{align}
The conclusion follows once we recall the formula on $\delta u_j^\obs - \delta u_j^\bulk$ from Lemma~\ref{lemma:Kobsplusestimate}.
\end{proof}

By Proposition~\ref{prop:VilotoDG} and Lemma~\ref{lemma:DGmgfasZNint}, this gives a representation for the correlation function of the Villain model: when $\Upsilon$ is a path connecting $\bx$ and $\by$ and satisfies Assumption~\ref{asmp:1Dobs}, 
\begin{align}
    \E^{\Vil} [e^{i\kn (\theta_{\bx} - \theta_{\by})}] = \exp\Big(\sum_{j=0}^{N-1} \Big( \frac{1}{2} ( d^* f_j , \Gamma_{j+1}  d^* f_j ) + \ke_j (\Lambda) \Big) \Big) (1 + O_L (L^{-\eta N}) ) .
    \label{eq:VilRGconcl}
\end{align}
We will show how to control this exponent in the second half of the paper.
To gain a quantitative control over the summand,
we crucially rely on the bounds \eqref{eq:propobsRGflowExists} along with the locality of $\omega_j$ and $\ke_j$, stated in Lemma~\ref{lemma:omegakeLocality}. Recall that $\omega_j$ is defined in Definition \ref{def:omega}. 

\begin{lemma} \label{lemma:omegakeLocality}
$\omega_j$ and $\ke_j$ have small-set local dependence on $d^* f$ in the following sense.
If $N\le M \in \N \cup \{\infty\}$ and $\iota : \Lambda_M \rightarrow \Lambda_N$ is a local isometry such that $\iota (\cB_j (\Lambda_M)) = \cB_j (\Lambda_N)$ and $f^{(M)} \in \Lambda_M^{(1)} (\R)$ and $f^{(N)} \in \Lambda_N^{(1)} (\R)$ satisfy
$d^* f^{(M)} |_{b^*} = (d^* f^{(N)} \circ \iota )|_{b^*}$ for some $b \in \cB_{j} (\Lambda_M)$, then $\ke_j (b)$ on $\Lambda_M$ is the same as $\ke_j (\iota(b))$ on $\Lambda_N$.
The same holds for $\omega_j$.
\end{lemma}
\begin{proof}
Clearly, $\ke_0 (b)$ and $\omega_0 (b) = 0$ are local by their definitions.
Due to Remark~\ref{remark:localityRG}, using induction, it also follows that the whole flow $(\delta u_j^\obs - \delta u_j^\bulk, D_j, K_j^\obs)_{j\ge 0}$ has the small-set local dependence. 
$\omega_j$ is just a symmetrization of $D_j$, so the same holds. It follows from the finite-range property of $\Gamma_{j+1}$ and \eqref{eq:gjDefi} the small-set locality of $g_{j+1}$. Also, by \eqref{eq:deltauplusKDefi}, the same holds for $\delta u_+^K$.
Then for $\ke_j$, the small-set locality follows from \eqref{eq:kejbDefi}, since $\delta u^r_{j+1} = \delta D_{j+1} (B, \nabla g_{j+1}) + \delta u^K_{j+1}$.
\end{proof}

\section*{\large\textnormal{\MakeUppercase{Part II. Geometric Analysis of Line Observables}}}
\addcontentsline{toc}{section}{Part II. Geometric Analysis of Line Observables}
\renewcommand{\thesection}{II}
\setcounter{subsection}{0}
\setcounter{subsubsection}{0}

\renewcommand{\thesection}{\arabic{section}}

\section{Outline of Part II}
\label{sec:partIIOutline}

In the second half of the paper (Section~\ref{sec:partIIOutline}--\ref{sec:sqcupPaths}), we derive precise correlation asymptotics of the Villain model through a deterministic analysis of the multiscale representation of the free energy obtained in Corollary~\ref{cor:finalintegralconcl}. Proposition \ref{post:standard} collects this representation and the properties of the local line energy and enhanced line observable required below. Theorem \ref{thm:standard} establishes the desired asymptotics from these properties.
Overview of the proof is given in Section~\ref{sec:overviewPartIIProof}.

In Part II, we make the dependence of constants on $L$ implicit, because they do not play any role, e.g., we just write $A \le O(B)$ and $A \lesssim B$ whenever $A \le O_L (B)$. 
In Part II, we focus on a class of observables in the DG model to those relevant to the correlation function of the Villain model. Besides paths connecting two points, this class also includes pairs of straight lines winding around the torus and rectangular loops. They are required to control the local line energies. The class of such observables is specified in the following definition.

\begin{definition} \label{def:standardPath}
A path $\bP$ on $E(\Lambda^\dual)$ has a \emph{standard shape} if it has one of the following forms: (1) a pair of  straight lines winding around the torus (see Section~\ref{sec:straightlineenergy}), or (2) a rectangle (see Section~\ref{sec:rectangularpath}, 
or (3) a straight line segment (see Section~\ref{sec:straighlinesegment}),
or (4) a $\sqcup$-shaped path connecting $\bx$ and $\by$ (see Section~\ref{sec:sqcupPaths}). 

We define the \emph{coalescence scale} as
\begin{align}
    j_\bxy = \begin{cases} 
        \min\{ j \ge 0 : \text{there exists $b\in \cB_j$ such that } \bx, \by\in b^* \} & (\text{case (3) and (4)}) \\
        N & (\text{case (1) and (2)}) .
        \end{cases}
\end{align}
We assume that the line segments constituting $\bP$ have lengths $\in [ \frac{1}{100} L^{j_\bxy} , 100 L^{j_\bxy} ]$, and in case (1), the straight lines are at least distance $\frac{1}{100} L^{j_\bxy}$ apart.

Given a path $\bP$, let $\bT$ be the set of its terminal points and let $\bI$ be the corners of $\bP$ excluding $\bT$. Together, we call $\bC = \bT \cup \bI$ the characteristic points of $\bP$.
To be specific, when $\bP$ has type (1), we take $\bT = \bI = \emptyset$.
When $\bP$ has type (2), its corner points are denoted $\bI = \{\bi_1,\bi_2,\bi_3,\bi_4 \}$ and $\bT = \emptyset$.
When $\bP$ has type (3), $\bT = \{t_1, t_2\}$ and $\bI = \emptyset$.
When $\bP$ has type (4), its corner points are denoted $\bI = \{ \bi_1, \bi_2 \}$ and the terminal points are denoted $\bT =\{ \bt_1, \bt_2 \}$.
\end{definition}

\begin{remark} \label{remark:whyUshaped}
Since the error estimates improve over scale in the RG flow, it is always better to work in scales as large as possible. This is the reason why we include the assumption that the length of each line segment constituting $\bP$ is comparable to $\norm{\bx-\by}$.

However, if we use an L-shaped path to connect $\bx$ and $\by$, it is not always possible to make the length of each line segment comparable to $\norm{\bx-\by}$.
We can avoid this problem if we use a $\sqcup$-shaped path as given in Figure~\ref{fig:UpsSqcup}.
This is why we derive the two-point correlation function using the $\sqcup$-shaped path connecting $\bx$ and $\by$
in Section~\ref{sec:sqcupPaths}, instead of an $L$-shaped one.
\end{remark}

As before, we consider the observable $f = 2\pi\alpha^{-1} \kn \one_{\Upsilon}$ when $\Upsilon =\bP^\dual$ is the dual of $\bP$.
For the estimates we see below, we fixed $\eta = 0.999$.

\begin{definition} \label{defi:locallineE}
A collection of block functions $(\ke_j : \cB_j \rightarrow \R )_{j=0}^{N-1}$ is called a local line energy if it satisfies the following.
\begin{enumerate}
\item $\ke_j (b) = 0$ if $b \in \cB_j$ is such that $b^* \cap \supp (d^* f ) = \emptyset$.

\item $\ke_j$ only \emph{depends locally} on $f$ in the following sense. If $N\le M$ and $\iota : \Lambda_M \rightarrow \Lambda_N$ is a local isometry such that $\iota (\cB_j (\Lambda_M)) = \cB_j (\Lambda_N)$ and $f^{(M)} \in \Lambda_M^{(1)} (\R)$ and $f^{(N)} \in \Lambda_N^{(1)} (\R)$ satisfy
$d^* f^{(M)} |_{b^*} = (d^* f^{(N)} \circ \iota )|_{b^*}$ for some $b \in \cB_{j} (\Lambda_M)$, then $\ke_j (b)$ on $\Lambda_N$ is the same as $\ke_j (\iota(b))$ on $\Lambda_M$.
In particular,
\begin{itemize}
    \item if $\iota' : \Lambda \rightarrow \Lambda$ is an isometry such that $\iota' (\cB_j)=\cB_j$ and $d^* f|_{b^*} = (d^* f \circ \iota' )|_{b^*}$, then $\ke_j (b) = \ke_j (\iota(b))$ and
    \item $\ke_j$ is independent of $N$, the size of the discrete torus, as long as $j < N$.
\end{itemize}

\item $\ke_j$ is an even function of $f$, i.e., $\ke_j$ is invariant under inverting the direction of $\bP$.

\item $|\ke_j (b)| \le O( L^{- (2\eta - 1) j} )$, uniformly in $j$ and $b \in \cB_j$.
\end{enumerate}
\end{definition}

The actual decay rate we had for $\ke_j$ was $L^{-\eta j}$, but we will need a slightly slower decay rate due to Lemma~\ref{lemma:flddLinObs}.
Next, we recall $\supp_+ (f)$ from \eqref{eq:suppPlusDefi}.

\begin{definition} \label{defi:enhcLineObs}
Suppose $\bP$ has a standard shape and $\Upsilon = \bP^\dual$. Let $f = 2\pi\alpha^{-1}\kn \one_{\Upsilon}$ for some $\kn \in \Z_{\neq 0}$.
We call that the collection $(f_j)_{j\ge 0} \subset \Lambda^{(1)} (\R)$ is an enhanced line observable if it satisfies the following.
\begin{enumerate}
\item $f_j = f+ \omega_j$ for $\omega_j \in \Lambda^{(1)} (\R)$ such that $\supp (\omega_j) \subset \supp_+ (f)$.
\item $\omega_0=0$, and $\sup_{j < j_{\bxy}} L^{\eta j} \norm{\omega_{j+1}-\omega_j}_{\ell^\infty} \le O(e^{-\gamma \alpha^2/4})$, 
while $\sup_{j\ge j_{\bxy}} L^{-(1-\eta) j + j_{\bxy}} \norm{\omega_{j+1} - \omega_j}_{\ell^\infty} \le O(e^{-\gamma \alpha^2 / 4})$.
\item $\omega_j$ has local dependence in the sense of Definition~\ref{defi:locallineE}~(ii).
\item $\omega_j$ is an odd function of $f$, i.e., $\omega_j$ changes sign under inverting the direction of $\bP$.
\item $\omega_{j+1}-\omega_{j}$ is a symmetrisation (in the sense of Definition~\ref{def:omega}) of a function on directed edges that is constant on each $j$-block, i.e., there exists $(a^\mu (x) :x \in \supp (d^* f), \mu \in \hat{e})$ such that,
\begin{align}
    \big( \omega_{j+1} - \omega_j \big) (x, x+\mu) = a^{\mu} (x) - a^{-\mu} (x+\mu) , \qquad x \in \Lambda, \mu \in \hat{e}
\end{align}
and whenever $x,x' \in b\in \cB_{j}$, we have $a^\mu (x)=a^{\mu}(x')=-a^{-\mu}(x)$.
\end{enumerate}
\end{definition}

As we will explain in Section~\ref{sec:foldExtArms}, we will mostly refer to (v') rather than (v).

\begin{proposition}[RG input for Part II]
\label{post:standard}
We assume $\alpha$ is sufficiently large and $\bP$ has a standard shape. For $\Upsilon = \bP^{\dual}$ and $f = 2\pi\alpha^{-1}\kn \one_{\Upsilon}$ for some $\kn \in \Z_{\neq 0}$, we let
\begin{align}
    \bL [\bP] = - \frac{2\pi^2 \kn^2}{\beta} |\Upsilon| + \log \E^{\dg}_{\beta,\Lambda} [ e^{(d^* f, \varphi)/\alpha} ] .
\end{align}
There there exists a collection of local line energy $(\ke_j)_j$ (Definition~\ref{defi:locallineE})
and an enhanced line observable $(f_j)_j$ (Definition~\ref{defi:enhcLineObs}) such that
\begin{align}
    \bL [\bP] = \frac{1}{2} \sum_{j=0}^{N-1} (d^* f_j, \Gamma_{j+1} d^* f_j ) + \sum_{j=0}^{N-1} \ke_j (\Lambda) + O(L^{-\eta N}) ,
    \label{eq:localineenergyForm1}
\end{align}
where we denoted $\ke_j (X) = \sum_{b \in \cB_j (X)} \ke_j (b)$.
\end{proposition} 

The following theorem shows that the RG input collected in Proposition \ref{post:standard} suffice to determine the precise power law asymptotics of the Villain two point function. Its proof is given in Section~\ref{sec:sqcupPaths}.

\begin{theorem} \label{thm:standard}
Suppose $\bP$ is a path with endpoints $\bx,\by$ and that $\bL [\bP]$ satisfies the representation and properties stated in Proposition \ref{post:standard}. Then there exist $C>0$ and $\bff =\bff (\kn,\alpha) > 0$ such that
\begin{align}
    \lim_{N\rightarrow \infty} e^{\bL[\bP]} = C \norm{\bx - \by}^{-\frac{1}{2\pi (1+s)} \bff^2} \Big( 1 + \psi(\bx-\by) \Big)
\end{align}
for some function $\psi (\bz) = O(1/\log \norm{z})$ and $\bff$  satisfying $\bff = 2\pi \kn \alpha^{-1} + O(e^{-\gamma\alpha^2/4})$. 
\end{theorem}

The main theorem of this paper, Theorem~\ref{thm:main} follows from the above statement.

\begin{corollary} \label{cor:Vil2ptfnc}
Suppose $\beta$ is sufficiently large and $\beta' = 4\pi^2/\beta$. Then there exist $C>0$ and $\beta'_\eff (\kn) = \beta' + O(e^{-\pi^2\gamma\beta'})$ such that
\begin{align}
    \E^{\Vil}_{\beta', \Z^2} \big[ e^{i \kn (\theta_{\bx} - \theta_{\by}) } \big]  
    = C\norm{\bx-\by}^{-\frac{\kn^2}{2\pi\beta'_\eff}}
\bigl(1+\psi(\bx-\by)\bigr)
\end{align}
for all distinct $\bx,\by\in\Z^2$, where
$$
\psi(\bz)=O_{\beta',\kn}\bigl((\log\norm{\bz})^{-1}\bigr)
\qquad\text{as }\norm{\bz}\to\infty.
$$
\end{corollary}
\begin{proof}
By \eqref{eq:villain_two_point_via_dg_on_torus},
\begin{align}
	\E^{\Vil}_{\beta', \Z^2} \big[ e^{i \kn (\theta_{\bx} - \theta_{\by}) } \big]  
	= \lim_{N\rightarrow\infty}
	e^{-\frac{2\pi^2 \kn^2}{\beta} |\bP| } \E^{\dg}_{\beta, \Lambda_N} \big[ \exp \big(- 2\pi \beta^{-1} \kn (\sigma, d \one_\bP )  \big) \big]
\end{align}
for any path $\bP$ connecting $\bx$ and $\by$,
and Proposition~\ref{post:standard} follows from Corollary~\ref{cor:finalintegralconcl} and Lemma~\ref{lemma:omegakeLocality}, together with the symmetry properties of $\ke_j$ and $\omega_j$.
Thus, the conclusion follows from Theorem~\ref{thm:standard}, with
\begin{align}
    \frac{\kn^2}{2\pi \beta'_\eff (\kn)} = \frac{\bff^2}{2\pi(1+s)} .
\end{align}
The desired estimate on $\beta'_{\eff}$ follows from the estimate on $\bff$ from Theorem~\ref{thm:standard} and the estimate on $s$ from Proposition~\ref{prop:critpt} (where an estimate on $s_0^c (\beta)$ is given and Assumption~\ref{asmp:RGasmpBulk} takes $s=s_0^c(\beta)$).
\end{proof}

\subsection{Block structure on plaquettes and edges}

The exterior derivative $df_j$ is a function on $\Lambda^\dual$, so we need to consider a block structure on $\Lambda^\dual$, with the following convention---let $\bD:\Lambda\rightarrow\Lambda^\dual$ be the bijection given by
\begin{align}
\begin{split}
    \bD(x)&=\text{the plaquette with edges } (x,x+\bbe_1),(x+\bbe_1,x+\bbe_1+\bbe_2), \\
    &\qquad (x+\bbe_1+\bbe_2, x+\bbe_2),(x+\bbe_2,x) ,
\end{split}
\end{align}
and let $\cB_j (\Lambda^\dual) := \{ \bD (b) : b\in \cB_j \}$.
Similarly, we consider injections $\bE_1, \, \bE_2 : \Lambda \rightarrow E(\Lambda)$ given by 
\begin{align}
    \bE_1 (x) = \{ x, x+\bbe_1 \}, \qquad \bE_2 (x) = \{ x, x+\bbe_2 \}
\end{align}
and we let $\cB_j (E(\Lambda)) := \{ \bE_1 (b) \cup \bE_2 (b) : b\in \cB_j \}$.
Whenever there is no source of confusion, we do not distinguish $b \in \cB_j (\Lambda)$, $\bE_1 (b) \cup \bE_2 (b) \in \cB_j (E(\Lambda))$ and $\bD (b) \in \cB_j (\Lambda^\dual)$ for notational simplicity.
We use similar notations for other block-induced objects, but we modify the inner product to preserve symmetries. For $b \in \cB_j (\Lambda)$, we define the following.
\begin{itemize}
\item For $f,g \in \Lambda^{(0)}(\R)$, we take $(f,g)_b = \sum_{x \in b} f(x) g(x)$ as usual.
\item For $f,g \in \Lambda^{(1)} (\R)$, let $E_{\operatorname{int}} (b) =\{ (x,y) \in E : x,y \in b \}$ and $E_{\partial} (b) = \{ (x,y) \in E : x \in b , \; y\not\in b\}$. Then we take
\begin{align}
    (f,g)_b = \sum_{e \in E_{\operatorname{int}} (b)} f(e) g(e) + \frac{1}{2} \sum_{e \in E_{\partial} (b)} f(e) g(e) .
    \label{eq:innProdOnEdges}
\end{align}
\item For $f,g \in \Lambda^{(2)} (\R)$, consider $x = \{e_1,e_2, e_3, e_4\} \in \Lambda^*$. We say $x \in F_{\operatorname{int}} (b)$ if $x \subset E_{\operatorname{int}} (b)$, say $x \in F_{\partial} (b)$ if $|x \cap E_{\operatorname{int}} (b)|=1$ and say $x \in F_{\partial\partial} (b)$ if $|x \cap E_{\partial} (b)|=2$ and $x\cap E_{\operatorname{int}} (b) = \emptyset$.
Then we take
\begin{align}
    (f,g)_b = \sum_{x \in F_{\operatorname{int}} (b)} f(x) g(x) + \frac{1}{2} \sum_{x \in F_{\partial} (b)} f(x) g(x) + \frac{1}{4} \sum_{x \in F_{\partial\partial} (b)} f(x) g(x) .
\end{align}
\end{itemize}
For general $X \in \cP_j$, we let $(f,g)_X = \sum_{b \in \cB_j (X)} (f,g)_b$. 

\subsection{Folding external arms}
\label{sec:foldExtArms}

Due to our construction of the RG on the line observable, the support of $\omega_j$ also contained edges that do not belong to $\supp (f)$.
We can think of a folding operation that removes (most of) these edges.
We start by considering a horizontal line segment
\begin{align}
    \bP = \big\{ \big( (x,y) ,(x+1,y) \big) \in \vec{E} (\Lambda^*) : x\in [n,m] \big\} .
\end{align}
If $f = 2\pi\alpha^{-1} ( \one_{\bP} )^\dual$, we define the \emph{internal edges} of $\supp_+ (f) \subset E(\Lambda)$ as
\begin{align}
    \supp_I (f) &= \big\{ \big( (x,y), (x,y+1) \big) : x\in [n+1,m+1] \big\} \\
    &\qquad \cup \big\{ \big( (x,y+\sigma), (x+1,y + \sigma) \big) : x\in [n,m+1] , \; \sigma \in \{ 0,1 \} \big\} , \nonumber
\end{align}
see Figure~\ref{fig:suppPf}.
For a more general standard shaped path $\bP$, we can find line segments $\bP_1, \cdots, \bP_k$ such that $\bP = \cup_{i=1}^k \bP_i$. Then we take
\begin{align}
    \supp_I (f) = \bigcup_{i=1}^k \supp_I (f_i) , \qquad f_i  =  2\pi\alpha^{-1} ( \one_{\bP_i} )^\dual
\end{align}
for $f = 2\pi\alpha^{-1}(\one_\bP)^\dual$. The \emph{external edges} are given by
\begin{align}
    \supp_E (f) = \supp_+ (f) \backslash \supp_I (f) .
\end{align}
We cannot expect the folding operation to preserve Definition~\ref{defi:enhcLineObs}, but it still satisfies an alternative condition.

\begin{definition} \label{defi:foldenhcLineObs}
Suppose $\bP$ has a standard shape and $\Upsilon = \bP^\dual$. Let $f = 2\pi\alpha^{-1}\kn \one_{\Upsilon}$ for some $\kn \in \Z_{\neq 0}$.
We call that the collection $(f_j)_{j\ge 0} \subset \Lambda^{(1)} (\R)$ is a \underline{folded} enhanced line observable if it satisfies (iii),(iv) of Definition~\ref{defi:enhcLineObs}, and also:
\begin{itemize}
\item[(i')]  $\hat{f}_j = f+ \hat\omega_j$ for $\hat\omega_j \in \Lambda^{(1)} (\R)$ such that $\supp (\hat\omega_j) \subset \supp_I (f)$.
\item[(ii')] $\hat\omega_0=0$, and $\sup_{j \ge 0} L^{\eta j} \norm{\hat\omega_{j+1}-\hat\omega_j}_{\ell^\infty} \le O(e^{-\gamma \alpha^2/4})$.
\item[(v')]  Let $b \in \cB_{j}$ and $\mu , \nu \in \hat{e}$ be such that $\mu \perp \nu$.
Suppose $x,x+\nu\in b \cap \supp (d^* f)$ are such that $\dist_{\infty} (x,\bC) > 1$ and, $\dist_{\infty} (x+\nu, \bC) > 1$. 
If we let $e_{x,\mu} = (x,x+\mu)$ and $e_{x+\nu,\mu} = (x+\nu,x+\nu+\mu)$, then $(\omega_{j+1}-\omega_j) (e_{x,\mu}) = (\omega_{j+1}-\omega_j) (e_{x+\nu,\mu})$, 
i.e., the increment of $\omega$ in $\mu$-direction does not change in a direction perpendicular to $\mu$ on each $j+1$-block, away from $\bC$.
\end{itemize}
\end{definition}

\begin{lemma} \label{lemma:deltaHatomegaDefi}
There exists $\omega'_j$ that is supported on $\supp_I (f)$ and satisfies Definition~\ref{defi:foldenhcLineObs}. 
Moreover, $\omega'_j - \omega_j = d \sigma_j$ for some $\sigma_j \in \Lambda^{(0)} (\R)$ satisfying 
$\norm{\sigma_j}_{\ell^\infty} \lesssim \norm{\omega_j}_{\ell^\infty}$.
\end{lemma}
\begin{proof}
Observe that, for each $e\in \supp_E (f)$, there exists a unique $x_e \in \supp (d^*f)$ such that $x_e \in e$. Moreover, $e \mapsto x_e$ is injective. We consider the function $\sigma_j \in \Lambda^{(0)} (\R)$ given by
\begin{align}
    \sigma_j (x) = \begin{cases}
        \omega_j (e) & (x=x_e, \; e= (x,y) \in \supp_E (f) ) \\
        0 & (\text{otherwise}) , 
    \end{cases}
\end{align}
so $\omega_j (e) + d \sigma_j (e) = 0$ for each $e \in \supp_E (f)$. 
By definition, $\sigma_j$ satisfies the same $\ell^\infty$ as that of $\omega_j$.
Thus, $\omega'_j := \omega_j + d\sigma_j$ satisfies the desired properties.
\end{proof}

\begin{figure}
\begin{centering}
\begin{tikzpicture}[scale=0.6, x=1cm, y=1cm, >=stealth]
    \draw[->, >=latex, thick, black!100] (-0.5, 0) -- (9.5, 0);
    \node[below] at (9.2, 0) {$\bP$};

    \draw[line width=4pt, opacity=0.5, gray] (-1.5, 1) -- (9.5, 1);
    \draw[line width=4pt, opacity=0.5, gray] (-1.5, -1) -- (9.5, -1);
    
    \foreach \x in {0.3, 2.1, 4, 5.8, 7.6} {
        \draw[line width=4pt, opacity=0.5, gray] (\x, -1) -- (\x, 1);
    }

    \foreach \x in {0.3, 2.1, 4, 5.8, 7.6} {
        \draw[dashed, thick] (\x, -3) -- (\x, 3);
    }

    \draw[dashed, thick] (-1.5, 1) -- (9.5, 1);
    \draw[dashed, thick] (-1.5, -1) -- (9.5, -1);

    \fill (-0.6, 0) circle (2.5pt);
    \node[below] at (-0.6, -0.1) {\Large $\bx$};
\end{tikzpicture}
\caption{$\supp_+ (f)$ (dotted lines) and $\supp_I (f)$ (grey shadowed lines)}
\label{fig:suppPf}
\end{centering}
\end{figure}
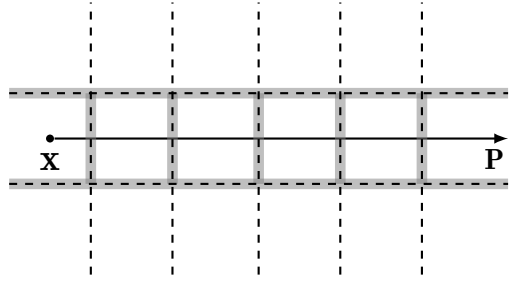

We see that Proposition~\ref{post:standard} holds similarly with $f_j$ replaced by $\hat{f}_j = f + \hat{\omega}_j$, but $\hat{\omega}_j$ satisfies a slightly different definition.

\begin{lemma} \label{lemma:flddLinObs}
There exist folded enhanced line observable $(\hat{f}_j)_j$ supported on $\supp_I (f)$ and local line energy $(\hat{\ke}_j)_j$ such that
\begin{align}
    \bL [\bP] = \frac{1}{2} \sum_{j=0}^{N-1} (d^* \hat{f}_j, \Gamma_{j+1} d^* \hat{f}_j ) + \sum_{j=0}^{N-1} \hat\ke_j (\Lambda) + O(L^{-\eta N}) ,
    \label{eq:localineenergyForm2}
\end{align}
\end{lemma}

\begin{proof}
Let $\sigma_j$ be as in Lemma~\ref{lemma:deltaHatomegaDefi} $f'_j = f_j + d \sigma_j$. We consider
\begin{align}
    \hat{f}_j = f'_{j \wedge j_{\bxy}} .
\end{align} 
When we replace $(d^* f_j, \Gamma_{j+1} d^* f_j)$ by $(d^* \hat{f}_j, \Gamma_{j+1} d^* \hat{f}_j)$, 
we get an error
\begin{align}
    \frac{1}{2} \sum_{j=0}^{N-1}  \big( (d^* f_j, \Gamma_{j+1} d^* f_j) - (d^* \hat{f}_j, \Gamma_{j+1} d^* \hat{f}_j) \big) 
    =
    \frac{1}{2} \sum_{j=0}^{N-1}  \big( f_j + \hat{f}_j , d\Gamma_{j+1} d^* ( f_j - \hat{f}_j )  \big)
\end{align}
Thus, \eqref{eq:localineenergyForm2} holds with
\begin{align}
    \hat{\ke}_j (b) := \ke_j (b) - \frac{1}{2} \times \begin{cases}
    0 & (j =0) \\
    \big( f_{j-1} + \hat{f}_{j-1} , d\Gamma_{j} d^* d \sigma_{j-1} \big)_b & (1 \le j \le j_\bxy) \\
    \big( f_{j-1} + \hat{f}_{j_{\bxy}} , d\Gamma_{j} d^* (f_{j-1} - \hat{f}_{j_{\bxy}}) \big)_b & (j > j_{\bxy}) .
    \end{cases}
\end{align}

To see that $(\hat{\ke}_j)_j$ is a local line energy, we just need to verify Definition~\ref{defi:locallineE}(iv').
When $j \le j_{\bxy}$, this holds because
\begin{align}
    \big| \big( f_j + \hat{f}_j , d\Gamma_{j+1} d^* d \sigma_j \big)_b \big| 
    &\lesssim 
    (L^{j})^2 \times L^{-3j} \times 1 \big)^2
    \le L^{-(2\eta -1) j}
\end{align}
where each factor of the right-hand side bounds each component of the left; 
$L^{j}$ counts $|\supp (f_j) \cap b|$ and $|\supp (\sigma_j) \cap b|$; 
$L^{-3j}$ bounds $\norm{d \Gamma_{j+1} d^* d}_{\ell^\infty}$; 
$L^{((1-\eta)j - j_\bxy )\vee 0}$ is the bound on $\norm{f_j+\hat{f}_j}_{\ell^\infty}$ and $\norm{\sigma_j}_{\ell^\infty}$ due to Lemma~\ref{lemma:deltaHatomegaDefi}.

When $j > j_{\bxy}$, we decompose again to obtain
\begin{align}
    &\big( f_{j-1} + \hat{f}_{j_{\bxy}} , d\Gamma_{j} d^* (f_{j-1} - \hat{f}_{j_{\bxy}}) \big)_b
    \nnb
    &\qquad = \big( f_{j-1} + f_{j_{\bxy}} , d\Gamma_{j} d^* (f_{j-1} - f_{j_{\bxy}}) \big)_b
    +
    \big( \hat{f}_{j_{\bxy}} + f_{j_{\bxy}} , d\Gamma_{j} d^* d \sigma_{j_{\bxy}} \big)_b 
    .
\end{align}
The second term is bounded as above. 
For the first term, we have $\norm{f_{j-1} + f_{j_{\bxy}}}_{\ell^\infty} \lesssim L^{((1-\eta)j - j_\bxy ) \vee 0}$ and $\norm{f_j - f_{j_\bxy}}_{\ell^\infty} \lesssim L^{(1-\eta)j - j_{\bxy}}$, 
and both are supported on a set of size $O(L^{j_{\bxy}})$. Thus,
\begin{align}
    \big| \big( f_{j-1} + f_{j_{\bxy}} , d\Gamma_{j} d^* (f_{j-1} - f_{j_{\bxy}}) \big)_b \big|
    \lesssim 
    (L^{j_{\bxy}})^2 \times L^{-2j} \times L^{(1-\eta)j - j_{\bxy}} \times L^{((1-\eta)j - j_\bxy ) \vee 0} .
\end{align}
In any case, this is bounded by $O(L^{-\eta j})$.
\end{proof}

We will always make reference to \eqref{eq:localineenergyForm2} rather than \eqref{eq:localineenergyForm1}, while we will tend to refer to Definition~\ref{defi:enhcLineObs} without (v) rather than \ref{defi:foldenhcLineObs}.
In fact, we can derive an even better alternative expression, as we describe.

Recall that the Laplace-Beltrami operator $\Delta$ is defined as, on $\Lambda^{(0)}$, $-\Delta = d^* d$, on $\Lambda^{(2)}$, $-\Delta = dd^*$, and on $\Lambda^{(1)}$, $-\Delta = d d^* + d^* d$.
Thus, we always have $d\Delta=\Delta d$ and $d^*\Delta=\Delta d^*$. (Be aware that $\Delta$ on the left-hand side and $\Delta$ on the right-hand side act on different spaces.)
Since $\Gamma_j$ is defined as a polynomial of $-\Delta$ (see \cite[Section~3]{dgauss1}), we see that the commutativity is shared by $\Gamma_j$, i.e., $d\Gamma_j= \Gamma_j d$ and $d^* \Gamma_j =\Gamma_j d^*$. Also,
\begin{align}
    (d^* \hat{f}_j , \Gamma_{j+1} d^* \hat{f}_j) = (\hat{f}_j , d d^* \Gamma_{j+1} \hat{f}_j ) = (\hat{f}_j, -\Delta \Gamma_{j+1} \hat{f}_j) - (d \hat{f}_j , \Gamma_{j+1} d\hat{f}_j) ,
    \label{eq:CoulombMultiscale}
\end{align}
and \eqref{eq:localineenergyForm2} can be equivalently rewritten as
\begin{align}
    \bL = \frac{1}{2} \sum_{j=0}^{N-1} \Big[ (\hat{f}_j, -\Delta \Gamma_{j+1} \hat{f}_j ) - (d \hat{f}_j, \Gamma_{j+1} d \hat{f}_j ) \Big] + \sum_{j=0}^{N-1} \hat\ke_j (\Lambda) + O(L^{-\eta N}) .
    \label{eq:localineenergy}
\end{align}
with $\bP$ made implicit.

\begin{remark} \label{remark:multiscaleSpinwave}
The decomposition \eqref{eq:CoulombMultiscale} is motivated by \eqref{eq:CoulombFundamental}. It isolates the spin-wave contribution of the enhanced line observables. Unlike the formal calculation in Remark~\ref{remark:wrongEquality}, the relation
\eqref{eq:localineenergy} retains the scale dependence of the spin-wave amplitudes, rather than replacing them by a single fixed amplitude. Controlling this dependence is essential to avoid the exponential term in Remark~\ref{remark:wrongEquality}, and constitutes a central difficulty in deriving precise correlation asymptotics in the BKT phase. A related picture was also employed in \cite{MR634447}; the present work develops the finer quantitative estimates needed to obtain sharp asymptotics.
\end{remark}

To control \eqref{eq:localineenergy}, we will need all 0-forms, 1-forms and 2-forms, so we consider all faces and vertices that can be obtained by performing derivatives and coderivatives on $\hat{f}_j$'s:
\begin{align}
\begin{cases}
    \supp^+ (d f) = \{ x \in \Lambda^\dual : x \cap \supp_I (f) \neq \emptyset \} , \\
    \overline{\supp} (d^*f) = \{ x \in \Lambda : x\in e\in \supp_I (f) \} .
\end{cases}   
\label{eq:suppfextensions}
\end{align}

\subsection{Overview of the proof of Theorem~\ref{thm:standard}}
\label{sec:overviewPartIIProof}

To study the contribution of $(d\hat{f}_j, \Gamma_{j+1} d\hat{f}_j)$,
one could take $\hat{f}_j \approx f$ and $\sum_{j} \Gamma_j = (-\Delta )^{-1}$ (see Section~\ref{sec:covDecomp})
as a first approximation.
Then $df = 2\pi\kn\alpha^{-1}(\delta_{\by} - \delta_{\bx})$, so
\begin{align}
    \frac{1}{2} \sum_{j=0}^{N-1} (d\hat{f}_j, \Gamma_{j+1} d \hat{f}_{j}) \approx  \frac{2\pi^2\kn^2}{\beta}(\delta_{\by} - \delta_{\bx}, (-\Delta)^{-1} (\delta_{\by} - \delta_{\bx}))  ,
    \label{eq:coulombGoal}
\end{align}
which is the Coulomb interaction between two particles with opposite charges. Section~\ref{sec:enhLineObs} verifies the validity of this approximation.
As a correction to this approximation, one should put effective charges at the endpoints of the path, with charges determined by the corrections $(\hat \omega_j)_{j\ge0}$ to the line observable.
See Proposition~\ref{prop:enhancedCorr}for the precise statement.

The rest of Part II focuses on the quadratic term $(\hat{f}_j, -\Delta \Gamma_{j+1} \hat{f}_j)$ and the local line energy $\hat\ke_j$.
In Section~\ref{sec:sqcupPaths}, we eventually use a $\sqcup$-shaped path $\bP$ connecting $\bx$ and $\by$ to compute the correlation function. To control the energy of such paths, we divide the path into elements of three categories: straight line segments, corners and tips. These three elements contribute differently, and they are studied in Section~\ref{sec:straightlineenergy}, \ref{sec:rectangularpath} and \ref{sec:straighlinesegment}, respectively. As a result, we obtain
\begin{align}
     \lim_{N\to\infty} \sum_{j=0}^{N-1} \Big[ \frac{1}{2}  (\hat{f}_j, -\Delta \Gamma_{j+1} \hat{f}_j ) +  \hat\ke_j (\Lambda) \Big] = C + O(1/\log \norm{\bx-\by})
     \label{eq:lineEnergyCancel}
\end{align}
for some constant $C$. Together with \eqref{eq:coulombGoal}, this proves Theorem~\ref{thm:standard}.

There are two main difficulties in the proof of \eqref{eq:lineEnergyCancel}. The first is to capture the effect of the flowing enhanced line observables within $(\hat{f}_j, -\Delta \Gamma_{j+1} \hat{f}_j)$. To overcome this difficulty, we use Definition~\ref{defi:enhcLineObs}(ii) to deduce the existence of a limit $\hat{f}_\infty := \lim_{j\rightarrow\infty} \hat{f}_j$ and approximate
\begin{align}
    \sum_{j\ge 0} (\hat{f}_j, -\Delta \Gamma_{j+1} \hat{f}_j ) \approx \sum_{j\ge 0} (\hat{f}_\infty, -\Delta \Gamma_{j+1} \hat{f}_\infty) = (\hat{f}_\infty, - \Delta  \Gamma_{>0} \hat{f}_\infty) ,
    \label{eq:hatfjhatfinfApprox}
\end{align}
where $\Gamma_{>0} = \sum_j \Gamma_j$. Since $(\Gamma_j)_j$ was a decomposition of a deformed lattice Green's functions (see Section~\ref{sec:covDecomp}), we may treat $-\Delta\Gamma_{>0}$ as a diagonal-concentrated operator. Indeed, this is what is observed in Lemma~\ref{lemma:inbtwparlines}.
However, the approximation \eqref{eq:hatfjhatfinfApprox} may not be treated within our framework because $\hat{f}_\infty$ does not depend locally on $\bP$---without locality, it becomes extremely difficult to estimate the energies exactly, especially for paths with complicated geometry. As a detour, we consider an infinitely extended observable $\bar{f}$ that replaces $\hat{f}_\infty$ in Section~\ref{sec:digInfExtObs}, with each edge weighted by the limiting amplitude of the corresponding infinite straight line observable. Its coefficients therefore depend on the local geometry and orientation of the path.

The second difficulty stems from the fact that a rough estimate on the left-hand side of \eqref{eq:hatfjhatfinfApprox} indicates that it may be proportional to the length of $\bP$. This should not happen because the energy should be path-independent. This can be resolved by extracting the energy contribution from the contractible loops. Indeed, we see from Section~\ref{sec:contractible} that contractible loops contribute only trivially, so we apply this result to paths with the shape of straight lines and rectangles to see that the straight line segments and corners attain exact cancellations. These are formulated in Lemma~\ref{lemma:straightlinecancellation} and \ref{lemma:asquarevanishes}.

\textbf{Notation.} We hereby summarize the following three variations of the line observable and its local energy used in Part II of the paper:

\begin{itemize}
    \item \(f_j\) and \(e_j\) denote the enhanced line observable and local line energy supplied by the RG analysis.
    \item \(\hat f_j\) denotes the folded observable constructed in Section \ref{sec:foldExtArms}, supported on the internal edges. The corresponding energy \(\hat \ke_j\) incorporates the change in quadratic energy caused by folding.
    \item  In Section \ref{sec:digInfExtObs} we define \(\bar f_j\),  denotes the  observable obtained by infinitely extended straight lines, and \(\bar f=\lim_{j\to\infty}\bar f_j\). The modified energy \(\bar e_j\), defined in Section \ref{sec:digInfExtObs}, incorporates the change in quadratic energy when \(\hat f_j\) is replaced by \(\bar f\).
\end{itemize}

In particular, \(\bar f\) is constructed from infinitely extended straight lines and should not be confused with the limit \(\hat f_\infty\) associated with the original path.

\renewcommand{\thesection}{\arabic{section}}

\section{Coulomb energy}

\label{sec:enhLineObs}
We start with the term $\sum_{j=0}^{N-1} (d\hat{f}_j, \Gamma_{j+1} d \hat{f}_{j})$ in \eqref{eq:localineenergy}.
In Proposition~\ref{prop:enhancedCorr}, we will see that it can be effectively approximated by a Coulomb interaction between particles with net charge 0, placed at the characteristic points $\bC$.
This proposition will be proved at the end of the section.

\begin{proposition} \label{prop:enhancedCorr}
There exist $(\bA (t), \bw(t))_{t \in \bC} \subset \R$ such that $\sum_{t \in \bC} \bw(t) = 0$ and
\begin{align}
    \sum_{j = 0}^{N-1} (d\hat{f}_{j} , \Gamma_{j+1} d\hat{f}_{j} ) 
    &= \sum_{t\in \bC} \bA (t) + \Big( \sum_{t \in \bC} \bw (t) \delta_t , \Gamma_{\le N} \sum_{t\in \bC} \bw (t) \delta_t \Big) + O(1/j_\bxy) .
\end{align}
\end{proposition}
The weights $(\bA (t), \bw(t))_{t \in \bC}$ are defined in \eqref{eq.Aw}. 
We will show in Section~\ref{sec:chargesoastndpath} its leading term is given by the effective charges take the values $\pm\bff$ at the endpoints, where $\bff$  is determined by the thickness of the enhance line observable and is independent of the path configuration.

\subsection{Relevant scales}

We first show that the spin wave energy decays on scales beyond $j_\bxy$.

\begin{lemma} \label{lemma:coulombenergyvanishlargej}
We have
\begin{align}
    \big| (d\hat{f}_j, \Gamma_{j+1} d\hat{f}_j) \big| 
    \lesssim L^{-2(j-j_\bxy)} 
\end{align}
\end{lemma}
\begin{proof}
Since $(d\hat{f}_j, \Gamma_{j+1} d\hat{f}_j) = (\hat{f}_j, d^* d \Gamma_{j+1} \hat{f}_j)$ and $\norm{d^* d \Gamma_{j+1}}_{\ell^\infty} \lesssim L^{-2j}$,
\begin{align}
    \big| (d\hat{f}_j, \Gamma_{j+1} d\hat{f}_j) \big| \lesssim L^{-2j} \times (L^{j_\bxy})^2 ,
\end{align}
where $O (L^{j_\bxy})$ counts the number of points in $\supp (\hat{f}_j)$ and $O(1)$ bounds $\norm{\hat{f}_j}_{\ell^\infty}$.
\end{proof}

When $j\ge 1.1 j_\bxy$, we will use this lemma to bound the energy of vortex interaction.
When $j < 1.1 j_\bxy$, we divide the bounds into those on the straight line component and the characteristic component. 
In the definition, $b^{**} = (b^*)^*$ is the small-set neighbourhood applied on $b$ twice.

\begin{definition} 
Let $\cB_j [f] \subset \cB_j (\Lambda^\dual)$ be the collection of $b \in \cB_j$ such that $b \cap \supp^+ (d f) \neq \emptyset$.
Let $\cB_j^s [f] \subset \cB_j [f]$ be the set of $b \in \cB_j [f]$ such that $(\Upsilon |_{b^{**}} )^{\dual}$ has the shape of a straight line segment. Let
\begin{align}
    \Seg_j = \bigcup_{b \in \cB_j^s [f]} b , \qquad \cB_j^c [f] = \cB_j [f] \backslash \cB_j^s [f] .
    \label{eq:SegjDefi}
\end{align}
\end{definition}

\begin{figure}
\begin{centering}
\begin{tikzpicture}[x=1cm, y=1cm, >=stealth]
    \coordinate (X) at (0,4.25);
    \coordinate (Y) at (5,5.15);
    \coordinate (I1) at (0,1);
    \coordinate (I2) at (5,1);

    \draw[ultra thick] (X) -- (I1) -- (I2) -- (Y);
    
    \draw[line width=4pt, opacity=0.4, gray, line cap=rect] (0,3.55) -- (0,4.25);
    \draw[line width=4pt, opacity=0.4, gray, line cap=rect] (0,1.14) -- (0,1.4);
    \draw[line width=4pt, opacity=0.4, gray, line cap=rect] (5,1.14) -- (5,1.4); 
    \draw[line width=4pt, opacity=0.4, gray, line cap=rect] (5,4.55) -- (5,5.15);
    \draw[line width=4pt, opacity=0.4, gray, line cap=rect] (0,1) -- (0.7,1);
    \draw[line width=4pt, opacity=0.4, gray, line cap=rect] (4.3,1) -- (5,1);    

    \draw[dashed, gray, ultra thick] (-0.3, 1.5) rectangle (0.3, 3.5);
    \draw[dashed, gray, ultra thick] (4.7, 1.5) rectangle (5.3, 4.5);
    \draw[dashed, gray, ultra thick] (0.8, 0.6) rectangle (4.2, 1.4);

    \fill (X) circle (1.5pt) node[above=5pt] {$\bx=t_1$};
    \fill (Y) circle (1.5pt) node[above=5pt] {$\by=t_2$};
    
    \node[below left] at (X) {$\bC_j(t_1)$};
    \node[below right] at (Y) {$\bC_j(t_2)$};
    \node[left] at (I1) {$\bC_j(i_1)$};
    \node[right] at (I2) {$\bC_j(i_2)$};
    
    \node at (2.5, 1.9) {$\Seg_j$};
\end{tikzpicture}
\caption{$\Seg_j$ (dotted boxes, \eqref{eq:SegjDefi}) and $\bC_j$ (thick grey lines, \eqref{eq:bCjtDefi}) for a $\sqcup$-shaped path}
\label{fig:SjCjt}
\end{centering}
\end{figure}

For a standard path $\bP$, we have
\begin{align}
    b \in \cB_j^s [f] \quad &\Rightarrow \quad \dist_\infty (b, \bC) \ge 2^d L^j , \\
    b \in \cB_j^c [f] \quad &\Rightarrow \quad \dist_\infty (b, \bC) \le 2^{d+1} L^j  ,
\end{align}
(the coefficients of $L^j$ are rough) where we recall that $\bC$ is the set of characteristic points of $f$.

Later in Lemma~\ref{lemma:dfinftyawayfromC}, it will be seen that we only have to consider the effective charge at each characteristic point accumulated below a relevant scale $j_{\Seg}$. This also gives the irrelevant straight line component $\bSeg$ defined as
\begin{align}
    j_{\Seg} &= \min \{ j \ge 0 : L^{0.4j} \ge j_\bxy \}, \qquad 
    \bSeg = \Seg_{j_{\Seg}} .
    \label{eq:jSegDefi}
\end{align}

\subsection{Straight line component}

Ideally, when $\hat{\omega}_j = 0$, $df$ can be expressed in terms of delta functions at the endpoints of $\bP$. We would like to use this property to reduce the computations for $d\hat{f}_j$.

\begin{lemma} \label{lemma:dfvanishesoutT}
If $x \in \Lambda^\dual \backslash \bT$, then $df (x) =0$.
\end{lemma}
\begin{proof}
By the choice $f = 2\pi\alpha^{-1} \kn \one_{\Upsilon}$, $df (x)$ is non-vanishing only when $x$ is the tip of $\bP = \Upsilon^\dual$. These tips are precisely $\bT$.
\end{proof}

Even when $\hat{\omega}_j \neq 0$, we can use symmetries to isolate the part of the enhanced line observable that does not contribute.
On each connected component of $\Seg_j$, the path $\bP$ looks like a line segment, so we can let $\hat\omega^\parallel_k$ be $\hat\omega_k$ restricted to the edges in $\Seg_j$  parallel to $\bP$ and $\hat\omega^\perp_k$ be $\hat\omega_k$ restricted to the edges in $\Seg_j$ perpendicular to $\bP$, so that $\hat\omega_k |_{\Seg_j} = \hat\omega_k^{\parallel} |_{\Seg_j} + \hat\omega^\perp_k |_{\Seg_j}$. 

\begin{lemma} \label{lemma:domegazeroStraightline}
On $\Seg_j$, let $\hat\omega^\parallel_k$ and $\hat\omega^\perp_k$ be as above.
\begin{enumerate}
\item If $k < j$, then $(d \hat\omega_k^{\perp}) |_{\Seg_j} = 0$ and $\hat\omega_k^{\parallel} |_{\Seg_j} = 0$.
\item If $k \ge j$,
$d(\hat\omega_k^{\perp} |_{\Sigma_j})$ and $(d\hat\omega_k^{\perp} ) |_{\Sigma_j}$ are supported on $O(1)$ number of points, $(d \hat\omega_k^\parallel) |_{\Seg_j} = d(\hat\omega_k^\parallel |_{\Seg_j})$ and 
\begin{align}
    \norm{\hat\omega_k^\parallel}_{\ell^\infty (\Seg_j)}\vee \norm{d\hat\omega_k^{\perp}}_{\ell^\infty (\Seg_j)} 
    \lesssim L^{-\eta j} 
    \label{eq:domegazeroStraightlinebnd1}
\end{align}
\end{enumerate}
\end{lemma}
\begin{proof} 
For (i), consider $k < j$ and $b \in \cB_j^s [f]$.
Since $\hat\omega_0 =0$, we will prove that $\hat\omega_k^\parallel=0$ and $(d\hat\omega_{k}^\perp - d\hat\omega_{k-1}^\perp) |_b = 0$ for each $k < j$.
Since the choices of $b$ and $f$ are arbitrary, we take $b = b_0$ to be the unique $k$-block containing 0, 
and edge $(y\bbe_2, (y+1)\bbe_2 ) \in \supp(f) \cap b_0$ for some $y \in \Z$.
In particular, $\bP$ is parallel to the $x$-axis. (Not the $y$-axis!)
By local dependence of $\omega$, we can assume $\Upsilon$ is given by
\begin{align}
    \one_{\Upsilon} ( (\ell \bbe_1 + y \bbe_2 , \ell\bbe_1 + (y+1)\bbe_2 ) ) = 1  \qquad (\ell\in \Z) .
\end{align} 
Now, consider the reflection $r: \Lambda \rightarrow \Lambda$ given by
\begin{align}
    r( x \bbe_1 + y\bbe_2) = (L^{k} - 1 - x)\bbe_1 + y\bbe_2 , \qquad x,y \in \Z 
\end{align}
so that $r(b_0) = b_0$ and $f\circ r |_{b_0}  = f |_{b_0}$.
By the locality of the increment (Definition~\ref{defi:enhcLineObs}(iii)), 
$r$ preserves $\hat\omega_{k}$, but at the same time, it also reverses the sign of $\hat\omega_{k} (e)$ when $e$ is parallel to the $x$-axis, so we have
\begin{align}
    \hat\omega^{\parallel}_{k} \big( x\bbe_1 + y\bbe_2 , (x+1)\bbe_1 + y\bbe_2 \big) = 0 .
\end{align}
For the edges parallel to the $y$-axis, we use Definition~\ref{defi:foldenhcLineObs}(ii'),(v') and invariance under the translation map $t: \Lambda \rightarrow \Lambda$, 
$t (x\bbe_1 +y\bbe_2) = (x + L^{k}) \bbe_1 + y \bbe_2$ to see that
\begin{align}
    (\hat\omega^\perp_{k} - \hat\omega^\perp_{k-1} ) \big(  (x + 1) \bbe_1 + y\bbe_2 ,  (x + 1) \bbe_1 + (y + 1)\bbe_2 \big) \qquad \nnb
    = (\hat\omega^\perp_{k} - \hat\omega^\perp_{k-1} ) \big( x\bbe_1 + y\bbe_2 , x \bbe_1 + (y+ 1)\bbe_2 \big) 
    \label{eq:domegazeroStraightlineverticalInv}
\end{align}
whenever $x \bbe_1 + y\bbe_2 \in b_0$.
Thus, $(d \hat\omega^\perp_{k} - \hat\omega^\perp_{k-1}) |_b =0$. This completes (i).

For (ii), we first verify the bounds. Definition~\ref{defi:foldenhcLineObs}(ii') says $\norm{\hat\omega_{k'+1} - \hat\omega_{k'}}_{\ell^\infty} \lesssim L^{- \eta k'}$ for any $k' \ge j$, and point (i) of this lemma says $d\hat\omega_{j}^\perp |_{\Seg_j} \equiv 0$ and $\hat\omega_{j}^\parallel |_{\Seg_j} \equiv 0$ so they imply \eqref{eq:domegazeroStraightlinebnd1}.

For the other property of $\hat\omega_k^{\perp}$, take $b \in \cB_k$ and observe that \eqref{eq:domegazeroStraightlineverticalInv} still holds whenever $x\bbe_1 + y\bbe_2 \in b$ \emph{and} $(x+1)\bbe_1 + y\bbe_2 \in b$ due to Definition~\ref{defi:enhcLineObs}(v') (translation invariance cannot be applied in this case), so $d (\hat\omega_k^\perp |_{\Seg_j})$ does not vanish only on the boundary of $b$. But there are at most $O(1)$ number of intersections between $\Seg_j$, $\supp^+ (df)$ and the union of all boundary edges of $k$-blocks for $k\ge j$.
Also, $d (\hat\omega_k^\perp |_{\Seg_j})$ and $(d\hat\omega_k^\perp) |_{\Seg_j}$ differ on only $O(1)$ number of points, so the same is true for $(d\hat\omega_k^\perp) |_{\Seg_j}$.
For $\hat\omega_k^\parallel$, we see that $(d \hat\omega_k^\parallel )|_{\Seg_j}$ depends only on $\hat\omega_k$ on $\Seg_j$, so we can say $(d \hat\omega_k^\parallel )|_{\Seg_j} = d ( \hat\omega_k^\parallel |_{\Seg_j} )$. 
\end{proof}

Next, we show that the contribution of the vortices on $\Sigma_j \backslash \Sigma_{j+1}$ have interaction energy that decays in $j$.

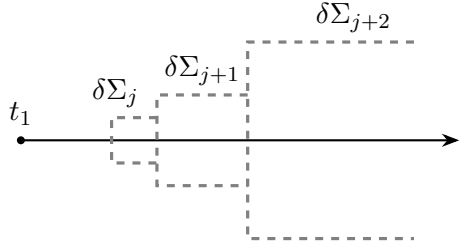
\begin{figure}
\centering
\begin{tikzpicture}[scale=1, x=1cm, y=1cm, >=stealth]
    \draw[thick, -Stealth] (0, 0) -- (5.8, 0);

    \fill[black] (0, 0) circle (1.5pt) node[above=2pt] {$t_1$};

    \draw[very thick, gray, dashed] (1.8, 0.3) -- (1.2, 0.3) -- (1.2, -0.3) -- (1.8, -0.3);
    \node[above] at (1.25, 0.3) {$\delta \Sigma_j$};

    \draw[very thick, gray, dashed] (3.0, 0.6) -- (1.8, 0.6) -- (1.8, -0.6) -- (3.0, -0.6);
    \node[above] at (2.4, 0.6) {$\delta \Sigma_{j+1}$};

    \draw[very thick, gray, dashed] (5.2, 1.3) -- (3.0, 1.3) -- (3.0, -1.3) -- (5.2, -1.3);
    \node[above] at (4.4, 1.3) {$\delta \Sigma_{j+2}$};
\end{tikzpicture}
\caption{An example of $\delta \Sigma_j = \Sigma_j \backslash \Sigma_{j+1}$ for a straight line segment path.
\label{fig:deltaSigma}}
\end{figure}

\begin{lemma} \label{lemma:dfinftyawayfromC}
Let $j\ge 0$ and $X\subset \supp^+ (df)$ be any set that can be expressed as a union of sets of the form $\Seg_{j'} \backslash \Seg_{j'+1}$ and points of $\bC$. Then for any $k\ge 0$,
\begin{align}
    \Big| \big( (d\hat{f}_k) |_{\Seg_{j} \backslash \Seg_{j+1}} , \Gamma_{k+1} (d \hat{f}_k) |_X \big) \Big|
    \lesssim L^{-\eta j}  .
    \label{eq:lemmadfinftyawayfromC}
\end{align}
\end{lemma}
\begin{proof}
For brevity, denote $\delta\Seg_j = \Seg_j \backslash \Seg_{j+1}$. 
If $x \in \delta\Seg_j$, then we have $d(f+ \hat\omega_{k} ) (x) = d \hat\omega_k (x)$ by Lemma~\ref{lemma:dfvanishesoutT}, and by Lemma~\ref{lemma:domegazeroStraightline}, for any $k\ge 0$,
\begin{align}
   d\hat{f}_k (x) = d (f + \hat\omega_k) (x) = d \hat\omega_k (x) = \begin{cases}
        0 & (k < j) \\
        d \hat\omega_k^\perp (x) + d\hat\omega_k^\parallel (x) & (k \ge j) ,
   \end{cases}
\end{align}
We have a similar decomposition of $d\hat{f}_k (y)$ if $y \in \delta\Seg_{j'}$ for some $j'$,
so for $k \ge j \vee j'$,
\begin{align}
    & \big( (d\hat{f}_k) |_{\delta\Seg_j} , \Gamma_{k+1} (d\hat{f}_k) |_{\delta\Seg_{j'} } \big)
    = \big(  ( d \hat\omega_k^\perp + d\hat\omega_k^\parallel )|_{\delta\Seg_j} , \Gamma_{k+1} ( d \hat\omega_k^\perp + d\hat\omega_k^\parallel )|_{\delta\Seg_{j'} }  \big) \nnb
    &\qquad\qquad = \big(  ( d \hat\omega_k^\perp )|_{\delta\Seg_j} , \Gamma_{k+1} (d\hat\omega_k^\perp)|_{\delta\Seg_{j'}} \big) 
    + \big(  ( d \hat\omega_k^\perp )|_{\delta\Seg_j} , \Gamma_{k+1} d( \hat\omega_k^\parallel |_{\delta\Seg_{j'}}  ) \big) \nnb
    &\qquad\qquad\qquad + \big(  d( \hat\omega_k^\parallel |_{\delta\Seg_j}  ) , \Gamma_{k+1} ( d \hat\omega_k^\perp )|_{\delta\Seg_{j'}} \big)
    + \big(  d( \hat\omega_k^\parallel |_{\delta\Seg_j}  ) , \Gamma_{k+1} d( \hat\omega_k^\parallel |_{\delta\Seg_{j'}}  ) \big) ,
\end{align}
where the second equality used the property of $\hat\omega^\parallel$ that $d$ and $|_{\delta\Seg_j}$ commute.
We bound $d\hat\omega_k^\perp$ using that it is supported on $O(1$) points and bound $d\hat\omega_k^\parallel$ using the derivative. To be specific, we have
\begin{align}
    \big|  d \hat\omega_k^\perp (x) \Gamma_{k+1} (x,y) d\hat\omega_k^\perp (y) \big| \lesssim  L^{-\eta (j+j')}  
\end{align}
by \eqref{eq:domegazeroStraightlinebnd1}, but since $d\hat\omega_k^\perp$ is supported on $O(1)$ number of points in $\delta \Sigma_j \cup \delta \Sigma_{j'}$, we have
\begin{align}
    \big| \big(  ( d \hat\omega_k^\perp )|_{\delta\Seg_j} , \Gamma_{k+1} (d\hat\omega_k^\perp)|_{\delta\Seg_{j'}} \big) \big|
        \lesssim L^{-\eta(j+j')} .
\end{align}
On the other hand, 
\begin{align}
    & \big| \big(  d( \hat\omega_k^\parallel |_{\delta\Seg_j}  ) , \Gamma_{k+1} d( \hat\omega_k^\parallel |_{\delta\Seg_{j'}}  ) \big) \big| 
        = 
        \big| \big(  \hat\omega_k^\parallel |_{\delta\Seg_j} ,  d^* d \Gamma_{k+1} \hat\omega_k^\parallel |_{\delta\Seg_{j'}}   \big) \big| \nnb
    &\qquad\qquad\qquad\qquad
    \lesssim L^{-2k} L^{-\eta(j+j')} \times L^{j} \times L^{j'} 
        \lesssim L^{-\eta(j+j')} 
\end{align}
where $L^{-2k}$ is due to the upper bound on $d^* d \Gamma_{k+1}$, $L^{j}$ is due to the number of points in $\delta\Seg_j$, and $L^{j'}$ is due to the number of points in $\delta\Seg_{j'}$. 
The other two terms can be bounded using the same strategies, so
\begin{align}
    \big| (d\hat{f}_k |_{\delta\Seg_j} , \Gamma_{k+1} d\hat{f}_k |_{\delta\Seg_{j'} }) \big| \lesssim  L^{-\eta(j+j')} .
\end{align}
Summing over $j' \ge 0$  gives the desired bound.

If $t \in \bC$, then we can apply similar arguments to see that
\begin{align}
    & \big| \big( (d \hat\omega^\perp_{k} ) |_{\delta\Seg_j}  , \,\Gamma_{k+1} (\cdot -t) \big) d\hat{f}_k (t) \big| \lesssim L^{-\eta j} \\
    & \big| \big( d (\hat\omega^\parallel_{k}  |_{\delta\Seg_j} ) , \,\Gamma_{k+1} (\cdot -t) \big) d\hat{f}_k (t) \big| \lesssim L^{-\eta j} .
\end{align}
Summing over $j' \ge 0$ and $t\in \bC$ also implies the desired bound.
\end{proof}

Unfortunately, \eqref{eq:lemmadfinftyawayfromC} has no decay in $k$, However, when it is summed over $k < 1.1j_\bxy$, the sum $O(j_\bxy) L^{-\eta j}$ grows polynomially in $j_\bxy$, 
so the energy contribution is negligible as long as $j\ge j_{\Seg}$, since $L^{\eta j_\Seg} \gg (j_\bxy)^2$, see \eqref{eq:jSegDefi}.

\begin{corollary} \label{eq:dfiftyS}
For any $X \subset \supp^+ (df)$ that is an arbitrary union of sets of form $\Seg_{j'} \backslash \Seg_{j'+1}$ and points of $\bC$, 
\begin{align}
    \sum_{k<1.1j_\bxy} \Big|  ( d \hat{f}_{k} |_{\bSeg} , \Gamma_{k+1} (d \hat{f}_{k} |_X) ) \Big|
    \lesssim j_\bxy L^{-\eta j_{\Seg}} \le (j_\bxy)^{-1} .
\end{align}
\end{corollary}
\begin{proof}
By Lemma~\ref{lemma:dfinftyawayfromC},
\begin{align}
    \sum_{k<1.1 j_\bxy} \Big|  ( d \hat{f}_{k} |_{\bSeg} , \Gamma_{k+1} (d \hat{f}_{k} |_X) ) \Big|
    \lesssim j_\bxy \sum_{j\ge j_{\Seg}} L^{-\eta j} \lesssim j_\bxy L^{-\eta j_{\Seg}}  ,
\end{align}
but $\eta = 0.999$, so $L^{-\eta j_\Sigma} \le j_\bxy^{-2}$.
\end{proof}

\subsection{Characteristic points}

So far, we have seen that $d\hat{f}_j (x)$ has negligible contribution whenever the distance of $x$ from $\bC$ is larger than $L^{j_{\Seg}}$.
Thus, we only need to focus on 
\begin{align}
    \supp^+ (df) \backslash \bSeg \subset \{ x\in \supp^+ (df) : \dist_\infty (x, \bC) \le 2^{d+1} L^{j_{\Seg}} \} . 
\end{align}
For each $t \in \bC$, we take
\begin{align}
    \bC_{j} (t) &= \text{(Connected component of $\supp^+ (df) \backslash \Seg_j$ containing $t$)} \nnb
    & \subset \{ x \in \supp^+ (df) : \dist_{\infty} (x, t) \le 2^{d+1} L^{j} \} .
    \label{eq:bCjtDefi}
\end{align}
(See Figure~\ref{fig:SjCjt}.)
Since $\supp^+ (df) \backslash \bSeg = \sqcup_{t\in \bC} \bC_{j_{\Seg}} (t)$ ($\sqcup$ is a disjoint union), Corollary~\ref{eq:dfiftyS} indicates that it will be enough to bound
\begin{align}
    \sum_{k < 1.1j_\bxy} \big( (d \hat{f}_k) |_{\bC_{j_{\Seg}} (t) } , \Gamma_{k+1} (d \hat{f}_k ) |_{\bC_{j_{\Seg}} (t')} \big) , 
    \qquad t,t' \in \bC . 
\end{align}
We first show that each $d \hat{f}_{k} |_{\bC_j (t)}$ can be approximated by a point charge located at $t$:
\begin{align}
    w_{k,j} (t) \delta_t = \sum_{x \in \bC_j (t)} d \hat{f}_{k} (x) \delta_t 
    \quad\; \text{where} \quad\;
    w_{k,j} (t) := \sum_{x \in \bC_j (t)} d \hat{f}_{k} (x) .
    \label{eq:wkjDefi}
\end{align}
For each pair $t, t' \in \bC$, let the error of approximation be
\begin{align}
    A_{j',j} (t, t') := \sum_{k<j'} \Big[  \big( (d \hat{f}_{k} ) |_{\bC_j (t)} , \Gamma_{k+1} (d \hat{f}_{k} )|_{\bC_j (t')} \big) - \big( w_{k,j} (t) \delta_t , \Gamma_{k+1} w_{k,j} (t') \delta_{t'} ) \Big] .    
\end{align}

\begin{lemma} \label{lemma:Ajttapprox}
There exists $A_{\infty,\infty} = \lim_{j\to\infty}\lim_{j'\to\infty} A_{j',j}$ such that, for $j' \ge j$,
\begin{align}
    |A_{\infty,\infty} (t, t') - A_{j',j} (t, t')| &\lesssim  
    \begin{cases} 
        L^{-\eta j} & (t=t') \\
        L^{-\eta (j \vee j_\bxy) } & (t \neq t') .
    \end{cases}
\end{align}
\end{lemma}
\begin{proof}
We decompose $A_{j',j} (t,t') = \sum_{k<j'} a_{k,j} (t,t')$ where
\begin{align}
    a_{k,j} (t,t') 
    = \big( d \hat{f}_{k} |_{\bC_j (t)} - w_{k, j}  \delta_t , \Gamma_{k+1}  d \hat{f}_{k} |_{\bC_j (t')} \big) + \big( w_{k,j}  \delta_t , \Gamma_{k+1} ( d \hat{f}_{k} |_{\bC_j (t')} - w_{k,j} \delta_{t'} ) \big) . 
\end{align}
We only consider
\begin{align}
    a'_{k,j} = \big( d \hat{f}_{k} |_{\bC_j (t)} - w_{k,j} \delta_t , \Gamma_{k+1} d \hat{f}_{k} |_{\bC_j (t')} \big)
\end{align}
and the convergence of the rest can be proved similarly. Also, let
\begin{align}
    \Delta a'_{k,j} := a'_{k,j+1} - a'_{k,j} 
    &= \big( d \hat{f}_{k} |_{\bC_{j+1} (t) \backslash \bC_j (t)} - (w_{k,j+1} - w_{k,j}) \delta_t , \, \Gamma_{k+1} d \hat{f}_k |_{\bC_{j+1} (t')} \big) \nnb 
    &\qquad\qquad
    + \Big( d \hat{f}_{k} |_{\bC_j (t)} - w_{k,j} \delta_t , \, \Gamma_{k+1}  d \hat{f}_k |_{\bC_{j+1} (t') \backslash \bC_j (t')}   \Big) .
\end{align}
Again, we only consider
\begin{align}
    \tilde{a}_{k,j} 
    = \big( d \hat{f}_{k} |_{\bC_{j+1} (t) \backslash \bC_j (t)} - (w_{k,j+1} - w_{k,j}) \delta_t , \, \Gamma_{k+1} d \hat{f}_k |_{\bC_{j+1} (t')} \big) 
\end{align}
and the other term can be bounded similarly.
We let $\delta \bC_{j+1} (t) = \bC_{j+1} (t) \backslash \bC_j (t)$ for the sake of brevity.

To bound $\tilde{a}_{k,j}$, since $\delta \bC_{j+1} (t) \subset \Seg_j \backslash \Seg_{j+1}$, we can use Lemma~\ref{lemma:dfvanishesoutT} and \ref{lemma:domegazeroStraightline} to decompose 
\begin{align}
    (d\hat{f}_k ) |_{\delta\bC_{j+1} (t)} = (d\hat\omega_k) |_{\delta\bC_{j+1} (t)} = (d\hat\omega_k^\perp) |_{\delta\bC_{j+1} (t)} + d ( \hat\omega_k^\parallel |_{\delta\bC_{j+1} (t)} ) .
\end{align} 
For the parallel part,
\begin{align}
    \sum_{x \in \delta\bC_{j+1} (t)} d ( \hat\omega_k^\parallel |_{\delta\bC_{j+1} (t)} ) (x) = (d ( \hat\omega_k^\parallel |_{\delta\bC_{j+1} (t)} ) , \one) 
    =
    ( \hat\omega_k^\parallel |_{\delta\bC_{j+1} (t)} , d^* \one) = 0 ,
    \label{eq:smxdCdhwp}
\end{align}
so in fact,
\begin{align}
    w_{k,j+1} - w_{k,j} 
        = \sum_{x\in \delta\bC_{j+1} (t)} d\hat{f}_{k} (x) 
        = \sum_{x\in \delta\bC_{j+1} (t)} d\hat\omega_k^\perp (x)  .
\end{align}
Since $(d\hat\omega_k^\perp) |_{\delta\bC_{j+1} (t)}$ is supported on $O(1)$ number of points by Lemma~\ref{lemma:domegazeroStraightline},
\begin{align}
    |w_{k,j+1} - w_{k,j}| \le
    \norm{(d\hat\omega_k^\perp) |_{\delta\bC_{j+1} (t)}}_{\ell^1} 
        \lesssim \norm{(d\hat\omega_k^\perp) |_{\delta\bC_{j+1} (t)}}_{\ell^\infty}  
        \lesssim \one_{k\ge j} L^{-\eta j} 
    \label{eq:dfinftySjbnd1}
\end{align}
Thus, inside
\begin{align}
    ( d\hat\omega_k^\perp ) |_{\delta\bC_{j+1} (t)} - (w_{k,j+1} - w_{k,j}) \delta_t
    = \sum_{x \in \delta\bC_{j+1} (t)} d\hat\omega_k^\perp (x) (\delta_x - \delta_t ) ,
\end{align}
we can write the summand of the right-hand side as a sum of gradients of size $O(L^{-\eta j})$ along a line of length $O(L^j)$.
This implies, whenever $k \ge j \ge l$, 
\begin{align}
\big( \delta_x - \delta_t , \Gamma_{k+1} d \hat{\omega}^{\perp}_k |_{\delta\bC_{\ell} (t')} \big)
\le
|x-t| \| \nabla \Gamma_{k+1} \|_{\ell^\infty} \norm{(d\hat\omega_k^\perp) |_{\delta\bC_{\ell} (t')}}_{\ell^\infty}  
\lesssim
L^{-k+j}  L^{-\eta \ell}.
\end{align}
Putting together, again by Lemma~\ref{lemma:domegazeroStraightline},
\begin{align}
    & \big| \big( d \hat\omega_{k}^\perp  |_{\delta\bC_{j+1} (t)} - (w_{k,j+1} - w_{k,j}) \delta_t , \, \Gamma_{k+1} d \hat{\omega}^{\perp}_k |_{\delta\bC_{\ell} (t')} \big) \big| \nnb
    &\quad \lesssim 
    \max_{x\in \bC^{\delta}_{j+1} (t)} |d\hat\omega_k^\perp (x)| \big( \delta_x - \delta_t , \Gamma_{k+1} d \hat{\omega}^{\perp}_k |_{\delta\bC_{\ell} (t')} \big)
    \lesssim
    \one_{k>j \vee \ell}  L^{-\eta j} L^{-k+j} L^{-\eta \ell}
    \label{eq:dopwwdGdwp}
\end{align}
where $L^{-\eta j}$ is due to \eqref{eq:dfinftySjbnd1}.
Similarly, but now using the property of $\hat{\omega}_k^\parallel$ instead,
\begin{align}
    & \big| \big( d \hat\omega_{k}^\perp  |_{\delta\bC_{j+1} (t)} - (w_{k,j+1} - w_{k,j}) \delta_t , \, \Gamma_{k+1} d \hat{\omega}^{\parallel}_k |_{\delta\bC_{\ell} (t')} \big) \big| \nnb
    & = \big| \big( d \hat\omega_{k}^\perp  |_{\delta\bC_{j+1} (t)} - (w_{k,j+1} - w_{k,j}) \delta_t , \, d \Gamma_{k+1}  \hat{\omega}^{\parallel}_k |_{\delta\bC_{\ell} (t')} \big) \big|
    \lesssim \one_{k \ge j\vee \ell} L^{-\eta j} L^{-2k+j} L^{(1-\eta)\ell}
    \label{eq:dopwwdGdwpl}
\end{align}
where $L^{-\eta j}$ is due to \eqref{eq:dfinftySjbnd1}, $L^{-2k+j}$ is from $\max_{x,t}|d\Gamma_{k+1} (x) - d\Gamma_{k+1} (t)|$ and the final $L^{\ell}$ is due to the $\ell^1$ norm on $d \hat{\omega}_K^{\parallel}|_{\delta \bC_{\ell} (t')}$.
Summing \eqref{eq:dopwwdGdwp} and \eqref{eq:dopwwdGdwpl} over $\ell \le j$, we get
\begin{align}
    \big| \big( d \hat\omega_{k}^\perp  |_{\delta\bC_{j+1} (t)} - (w_{k,j+1} - w_{k,j}) \delta_t , \, \Gamma_{k+1} d \hat{f}_k |_{\bC_{j+1} (t')} \big) \big| 
    \lesssim \one_{k\ge j} L^{-\eta j} L^{-k+j} .
\end{align}
The same principle applies to
\begin{align}
    & \big| \big( d (\hat\omega_{k}^\parallel |_{\delta\bC_{j+1} (t)} )  , \, \Gamma_{k+1} d \hat{f}_k |_{\bC_{j+1} (t')} \big) \big|  
    = 
    \big| \big( \hat\omega_{k}^\parallel |_{\delta\bC_{j+1} (t)}   , \, d^* \Gamma_{k+1} d \hat{f}_k |_{\bC_{j+1} (t')} \big) \big| \nnb &\qquad \lesssim
    \sum_{\ell=0}^{j} \one_{k \ge j \vee \ell} L^{(1-\eta) j} L^{-k} L^{-\eta \ell} (1 + L^{-k+\ell}) 
    \lesssim \one_{k\ge j} L^{(1-\eta) j} L^{-k}
\end{align}
where we used $\norm{\hat\omega_{k}^\parallel}_{\ell^\infty (\Seg_j)} \lesssim L^{-\eta j}$ and $\norm{ (d^* )^n \Gamma_{k+1}}_{\ell^\infty} \lesssim L^{-n k}$. These bounds prove $|\tilde{a}_{k,j}| \lesssim \one_{k\ge j} L^{-k + (1-\eta) j}$. 

The other terms are bounded identically, so we obtain
\begin{align}
    |a_{k,j}| \lesssim \sum_{j''<j} \one_{k\ge j''} L^{-k + (1-\eta)j''}
    \lesssim
    \one_{k\ge j} L^{- k + (1-\eta) j} .     
\end{align}
Thus $(A_{j',j})_{j'}$ form a Cauchy sequence with $|A_{j'+1,j} -A_{j',j}| \lesssim L^{-j'+(1-\eta)j}$, and there exists $A_{\infty,j}$ such that $|A_{\infty,j} - A_{j',j}| \lesssim L^{-j'+(1-\eta)j}$. Similarly, $(A_{\infty,j})_j$ is a Cauchy sequence with
\begin{align}
    |A_{\infty,j} - A_{\infty,j+1}| \le \sum_{k=0}^\infty |a_{k,j}-a_{k,j+1}| \lesssim \sum_{k=0}^\infty \one_{k\ge j} L^{-k+(1-\eta)j} \lesssim L^{-\eta j}  .
\end{align}
so there exists $A_{\infty,\infty}$ with $|A_{\infty,\infty}- A_{\infty,j}| \lesssim L^{-\eta j}$.
This gives the desired result when $t= t'$. 

When $t \neq t'$, the bound can be improved. Since 
\begin{align}
    \big( \delta_t , \Gamma_{k+1} \delta_{t'} ) =
    \big( d \hat{f}_{k} |_{\bC_j (t)}  , \Gamma_{k+1}  d \hat{f}_{k} |_{\bC_j (t')} \big) =  
    \big( \delta_t , \Gamma_{k+1}  d \hat{f}_{k} |_{\bC_j (t')}  \big) =0 
\end{align}
for $k \le j_\bxy -1$, so we get a better bound $|\tilde{a}_{k,j}| \lesssim \one_{k > j_{\bxy}} L^{- k + (1- \eta ) j}$. We can repeat the same process to obtain the desired bound.
\end{proof}

By definition of $j_{\Seg}$, we have
\begin{align}
    |A_{1.1j_\bxy,j_{\Seg}} (t, t) - A_{\infty,\infty} (t,t)|
    \lesssim
    \begin{cases}
        L^{-\eta j_{\Seg}} \le (j_\bxy)^{-1} & (t=t') \\
        L^{- \eta j_\bxy} & (t\neq t') ,
    \end{cases}
\end{align}
and this approximation only gives a negligible error.
Thus, the computation of the quadratic forms reduces to the computation of $w_{k,j} (t)$.

\begin{lemma} \label{lemma:winftywj}
For each $t \in \bC$ and $j\ge 0$, there exist $w_{k,\infty} (t)$ and $w_{\infty,\infty} (t)$ such that
\begin{align}
    | w_{\infty,\infty} (t) - w_{k,\infty} (t) | 
        &\lesssim L^{-0.99 k} \label{eq:lemmawinftywj1} \\
    | w_{k,\infty} (t) - w_{k,j} (t) | 
        &\lesssim L^{-\eta j} \label{eq:lemmawinftywj2}
\end{align}
and for any $j \in \N \cup \{\infty\}$,
\begin{align}
    \Big| \sum_{t \in \bC} w_{k,j} (t) \Big| \lesssim L^{-\eta j} .
    \label{eq:lemmawinftywj3}
\end{align}
\end{lemma}
\begin{proof}
By \eqref{eq:dfinftySjbnd1},
\begin{align}
    | w_{k,j+1} - w_{k,j} | \lesssim \one_{k\ge j} L^{-\eta j} ,
\end{align}
so \eqref{eq:lemmawinftywj2} follows.
To bound $w_{k+1,j} - w_{k,j}$, we recall \eqref{eq:smxdCdhwp} so only $\hat{\omega}^\perp$ contributes on each $\delta\bC_l (t)$.
Also, by Lemma~\ref{lemma:domegazeroStraightline}, $d \hat\omega^\perp_{k+1} |_{\Seg_l} = 0$ for $l > k$, so 
\begin{align}
    |w_{k+1,j} - w_{k,j} | 
    &\le \Big| \sum_{x\in \bC_0 (t)} d (\hat\omega_{k+1} - \hat\omega_k) (x) \Big| + \sum_{l=1}^{j\wedge k} \sum_{x \in \delta\bC_l (t)} \Big| d(\hat\omega^\perp_{k+1}-\hat\omega^\perp_k) (x) \Big| \nnb
    &\lesssim L^{-\eta k} +  \sum_{l=1}^{j\wedge k} \sum_{x \in \delta\bC_l (t)} \Big| d(\hat\omega^\perp_{k+1}-\hat\omega^\perp_k) (x) \Big| ,
\end{align}
where we used that there are $O(1)$ points in $\bC_0 (t)$.
Also,
$d(\hat\omega_{k+1}^\perp -\hat\omega_k^\perp) |_{\bC^{\delta}_l (t)}$ is supported on $O(1)$ number of points, so
\begin{align}
    |w_{k+1,j} - w_{k,j} | \lesssim k L^{-\eta k} \lesssim L^{-0.99 k} .
\end{align}
This bound is uniform in $j$, so
\eqref{eq:lemmawinftywj1} follows.

For the final inequality, observe that $\sum_{x \in \Lambda} d (f+\hat\omega_j ) (x) = 0$ for each $j$, so by Lemma~\ref{lemma:domegazeroStraightline},
\begin{align}
    \Big| \sum_{t \in \bC} w_{k,j} (t) \Big| &= \Big|\sum_{x \in \Lambda} d (f+ \hat\omega_k) (x) - \sum_{x \in \Seg_j} d(f+ \hat\omega_k) (x) \Big| \nnb
    &= 
    \Big|\sum_{x \in \Seg_j} d \hat\omega_k (x) \Big|
    = 
     \Big|\sum_{x \in \Seg_j} d \hat\omega_k^\perp (x) \Big|
    \lesssim L^{-\eta j}
\end{align}
where the final inequality follows from Lemma~\ref{lemma:domegazeroStraightline}. This bound is uniform in $k$, so we also have the same bound for $\sum_{t \in \bC} w_{\infty,j} (t)$.
\end{proof}

As can be seen in Lemma~\ref{lemma:Ajttapprox}, $A_{\infty, \infty} (t,t')$ is negligible when $t \neq t'$ (note that $A_{0,0}(t,t')=0$), so we only need to keep track of
\begin{align} \label{eq.Aw}
    & \bA (t) := A_{\infty,\infty} (t,t) + \sum_{k < \infty} \sum_{t \in \bC} \big( w_{k,\infty} (t)^2 - w_{\infty,\infty} (t)^2  \big) \Gamma_{k+1} (0) \\
    & \bw (t) := w_{\infty,\infty} (t) = \lim_{k\rightarrow \infty} \lim_{j\rightarrow \infty} w_{k,j} (t) = \lim_{k\rightarrow \infty} w_{k,\infty} (t) .
\end{align}
As an immediate consequence of \eqref{eq:lemmawinftywj3}, 
\begin{align}
    \sum_{t\in \bC} w_{k,\infty} (t) = 0 \qquad \Rightarrow \qquad \sum_{t \in \bC} \bw (t) = 0 ,
    \label{eq:sumwinfty}
\end{align}
so the quadratic form $(\sum_t \bw(t)\delta_t , \sum_{j} \Gamma_j \sum_{t} \bw(t) \delta_t)$ also makes sense on $\Z^2$.
We are now ready to prove the main result of this section.

\begin{proof}[Proof of Proposition~\ref{prop:enhancedCorr}]
By Lemma~\ref{lemma:coulombenergyvanishlargej}, 
we have
\begin{align}
    \sum_{j\ge 0} (d\hat{f}_{j} , \Gamma_{j+1} d\hat{f}_j ) 
    = \sum_{j < 1.1 j_\bxy} (d\hat{f}_j , \Gamma_{j+1} d\hat{f}_j ) + O(L^{-0.2 j_\bxy}) .
\end{align}
By Corollary~\ref{eq:dfiftyS},
\begin{align}
    \sum_{j< 1.1j_\bxy} (d\hat{f}_{j} , \Gamma_{j+1} d\hat{f}_j ) 
    = 
   \sum_{t,t'\in \bC} \sum_{k < 1.1j_\bxy} \big( (d \hat{f}_k) |_{\bC_{j_{\Seg}} (t) } , \Gamma_{k+1} (d \hat{f}_k ) |_{\bC_{j_{\Seg}} (t')} \big) + O(\frac{1}{j_\bxy} ).
  \end{align} 
By the definition of $A_{j',j}$, 
  \begin{align}
\sum_{j < 1.1j_\bxy} \big( (d \hat{f}_j) |_{\bC_{j_{\Seg}} (t) } , \Gamma_{j+1} (d \hat{f}_j ) |_{\bC_{j_{\Seg}} (t')} \big)
= 
\sum_{j < 1.1j_\bxy}\Big( w_{j, j_{\Seg}} (t) \delta_t , \Gamma_{j+1} w_{j,j_{\Seg}} (t') \delta_{t'} \Big) + A_{1.1j_\bxy, j_{\Seg}}(t,t') .
  \end{align} 
By Lemma~\ref{lemma:Ajttapprox}, the off-diagonal terms $A_{1.1j_\bxy, j_{\Seg}}(t,t')$ (for $t \neq t'$) are small, and the diagonal terms $A_{1.1j_\bxy, j_{\Seg}}(t,t)$ can be replaced by $A_{\infty, \infty} (t,t)$ with a small error. Thus,
\begin{align}
    & \Big| \sum_{j < 1.1 j_\bxy} \Big[ (d\hat{f}_j , \Gamma_{j+1} d\hat{f}_j )
    - \sum_{t,t' \in \bC} \Big( w_{j, j_{\Seg}} (t) \delta_t , \Gamma_{j+1} w_{j,j_{\Seg}} (t') \delta_{t'} \Big) \Big]
    - \sum_{t\in \bC} A_{\infty,\infty} (t,t)  \Big| \nnb
    & \lesssim 
    \frac{1}{j_\bxy} + L^{-\eta j_{\Seg}} + L^{-\eta (j_\bxy \vee j_{\Seg})} 
    \lesssim \frac{1}{j_\bxy} .
\end{align}
Now, by Lemma~\ref{lemma:winftywj}, we can replace $w_{j,j_{\Seg}}$ by $w_{j,\infty}$ with error $\lesssim (\log L) j_\bxy L^{-\eta j_{\Seg}} \lesssim 1/j_\bxy$, so
\begin{align}
    \sum_{j\ge 0} (d\hat{f}_j , \Gamma_{j+1} d\hat{f}_j ) 
    &= \sum_{t\in \bC} A_{\infty,\infty} (t,t) + \sum_{j < 1.1 j_\bxy} \Big( \sum_{t \in \bC} w_{j,\infty} (t) \delta_t , \Gamma_{j+1} \sum_{t\in \bC} w_{j,\infty} (t) \delta_t \Big) + O(\frac{1}{j_\bxy}) .
\end{align}
But since $\sum_{t\in \bC} w_{j,\infty} (t) = 0$ due to \eqref{eq:sumwinfty}, we can write $\sum_{t\in \bC} w_{j,\infty} (t) \delta_t$ as a sum of $O(L^{j_{\bxy}})$ number of gradients, so
\begin{align}
     \one_{j\ge 1.1j_\bxy} \Big| \Big( \sum_{t \in \bC} w_{j,\infty} (t) \delta_t , \Gamma_{j+1} \sum_{t\in \bC} w_{j,\infty} (t) \delta_t \Big) \Big|
     \lesssim
     L^{-2j} L^{2j_\bxy} \lesssim L^{-0.2j_\bxy} ,
\end{align}
and we can replace $\sum_{j<1.1j_\bxy}$ by $\sum_{j\ge 0}$ at a negligible cost.

Next, we approximate $w_{j,\infty}$ by $\bw$. When $j \ge j_{\Seg}$, since
$|w_{j,\infty} (t) - \bw (t) |\le L^{-0.99 j}$ by Lemma~\ref{lemma:winftywj}, we get
\begin{align}
    &\Big| \sum_{j\ge j_{\Seg}} \sum_{t,t' \in \bC} \Big[  \Big( w_{j,\infty} (t) \delta_t , \Gamma_{j+1}   w_{j,\infty} (t') \delta_{t'} \Big)
        - \Big( \bw (t) \delta_t , \Gamma_{j+1}  \bw (t') \delta_{t'} \Big) \Big] \Big|
    \nnb
    &\qquad \le
    \max_{j\ge 0, t \in \bC} |w_{j,\infty} (t) |^2 \sum_{j\ge j_{\Seg}} L^{-0.99 j} \norm{\Gamma_{j+1}}_{\ell^\infty} \lesssim L^{-0.99 j_{\Seg}} \lesssim O(1/ j_\bxy ) .
\end{align}
When $j < j_{\Seg}$, then $\Gamma_{j+1} (t,t') =0$ whenever $t \neq t' \in \bC$, so
\begin{align}
    \sum_{j< j_{\Seg}} \sum_{t,t' \in \bC} \Big( w_{j,\infty} (t) \delta_t , \Gamma_{j+1} w_{j,\infty} (t') \delta_{t'} \Big)
    & = \sum_{j < j_{\Seg}} \sum_{t \in \bC}  w_{j,\infty} (t)^2 \Gamma_{j+1} (0) ,
\end{align}
and upon approximating $w_{j,\infty}$ by $\bw$, we get an error
\begin{align}
    &\sum_{j < j_{\Seg}} \sum_{t \in \bC} \big( w_{j,\infty} (t)^2 - \bw (t)^2  \big)\Gamma_{j+1} (0) + O(1/j_\bxy) \nnb
    &= \sum_{j < \infty} \sum_{t \in \bC} \big( w_{j,\infty} (t)^2 - \bw (t)^2  \big)\Gamma_{j+1} (0) + O(1/j_\bxy)
    = \sum_{t\in \bC} \bA (t) - A_{\infty,\infty} (t,t) + O(1/j_\bxy)
\end{align}
where the first estimate holds because $w_{j,\infty} (t)- \bw (t) = O(L^{-0.99 j})$. This completes the proof.
\end{proof}

\section{contractible loops}
\label{sec:contractible}

Recall the expression \eqref{eq:localineenergy} for $\bL$. In Section~\ref{sec:enhLineObs} we derived the leading order term for the Coulomb energy part $(d\hat{f}_j , \Gamma_{j+1} d\hat{f}_j)$. We now focus on the other terms, $(\hat{f}_j, -\Delta\Gamma_{j+1}\hat{f}_j)$ and $\ke_j$.

Among standard-shaped paths (Definition~\ref{def:standardPath}), rectangular paths and straight lines play a crucial role, since we have an exact identity for contractible loops. After formulating the identity in Lemma~\ref{lemma:contractible}, we will see its applications in Section~\ref{sec:straightlineenergy} and \ref{sec:rectangularpath}, to analyze the contributions of $(\hat{f}_j, -\Delta\Gamma_{j+1}\hat{f}_j)$ and $\ke_j$ associated with straight line components and corner components of the path.

\begin{lemma} \label{lemma:contractible}
If $\bP$ is a contractible loop, then
\begin{align}
    \bL [\bP] = 0 .
\end{align}
\end{lemma}
\begin{proof}
By definition,
\begin{align}
    \E^{\dg}_{\beta,\Lambda} [ e^{2\pi\kn\beta^{-1} (d^* \one_\Upsilon , \varphi ) } ]
    &\propto \sum_{\varphi \in (2\pi\Z)^{\Lambda}: \varphi_0 = 0} e^{2\pi\kn\beta^{-1} (\one_\Upsilon , d\varphi)} e^{-\frac{1}{2\beta} |d \varphi|^2} \nnb
    &= e^{\frac{2\pi^2\kn^2}{\beta} |\one_{\Upsilon}|^2} \sum_{\varphi \in (2\pi\Z)^{\Lambda}: \varphi_0 = 0} e^{-\frac{1}{2\beta} |d \varphi - 2\pi \kn \one_{\Upsilon} |^2} .
    \label{eq:lemmacontractibleEq1}
\end{align}    
On the other hand, since $\bP$ is contractible, there exists $\psi \in \Lambda^{(2)} (\Z)$ such that $\one_\bP = d^* \psi$, thus
\begin{align}
    \one_\Upsilon = (\one_\bP)^\dual = (d^* \psi)^\dual = d (\psi)^\dual ,
\end{align}
so we can make change of variable $\varphi \mapsto \varphi+ 2\pi \kn \psi^\dual$ in \eqref{eq:lemmacontractibleEq1} to obtain
\begin{align}
    \E^{\dg}_{\beta,\Lambda} [ e^{2\pi\kn\beta^{-1} (d^* \one_\Upsilon , \varphi ) } ] 
    \propto e^{\frac{2\pi^2\kn^2}{\beta} |\one_{\Upsilon}|^2} \sum_{\varphi \in (2\pi\Z)^{\Lambda}: \varphi_0 = c} e^{-\frac{1}{2\beta} |d \varphi |^2}
\end{align}
where $c = (2\pi \kn \psi^\dual)(0) \in 2\pi \Z$. But the value of $c$ does not affect the sum, and
\begin{align}
    = e^{\frac{2\pi^2\kn^2}{\beta} |\one_{\Upsilon}|^2} Z^{\dg}_{\beta,\Lambda} .
\end{align}
Thus, $\bL [\bP] := - \frac{2\pi^2 \kn^2}{\beta} |\one_{\Upsilon}|^2 + \log \E^{\dg}_{\beta,\Lambda} [ e^{2\pi \kn \beta^{-1} (d^* \one_\Upsilon , \varphi ) } ]=0$.
\end{proof}

\section{Straight line energy}
\label{sec:straightlineenergy}

\begin{figure}
\begin{centering}
\begin{tikzpicture}[scale=0.6, x=1cm, y=1cm, >=stealth]
    \tikzset{
        middle arrow/.style={
            postaction={decorate, decoration={
                markings,
                mark=at position 0.55 with {\arrow[scale=1.5, thick]{>}}
            }}
        }    
    }
    \tikzset{
        middle arrow reversed/.style={
            postaction={decorate, decoration={
                markings,
                mark=at position 0.55 with {\arrow[scale=1.5, thick]{<}}
            }}
        }    
    }

    \draw[thick] (0,0) rectangle (6,6);

    \draw[thick, middle arrow] (0,4) -- (6,4);
    
    \draw[thick, middle arrow reversed] (0,2) -- (6,2);

    \node[left] at (0,4) {\Large $y_1$};
    \node[left] at (0,2) {\Large $y_2$};
\end{tikzpicture}
\caption{$\bP_{y_1,y_2}^{\rightleftharpoons}$. The upper line is $\bP^{\rightarrow}_{y_1}$ and the lower line is $\bP^{\leftarrow}_{y_2}$}
\label{fig:lines}
\end{centering}
\end{figure}

Given $y_1 \neq y_2 \in \Z / L^N \Z$ such that $y_2 = y_1 + m L^{N-1}$ for some $|m| > 2^d$, consider a pair of dual of paths
\begin{align}
    \begin{cases}
        \Upsilon^{\rightarrow}_{y_1} = \{ ( (x,y_1) , (x,y_1 + 1)) : x \in \Z / L^N \Z \} \\
        \Upsilon_{y_2}^{\leftarrow} = \{ ( (x,y_2+1) , (x,y_2) ) : x \in \Z / L^N \Z \}
    \end{cases}
\end{align}
and take
\begin{align}
    \Upsilon^{\rightleftharpoons}_{y_1,y_2} = \Upsilon_{y_1}^{\rightarrow} \cup \Upsilon_{y_2}^{\leftarrow} , \qquad
    \bP^{\rightleftharpoons}_{y_1,y_2} = (\Upsilon^{\rightleftharpoons}_{y_1,y_2})^{\dual} .
\end{align}
We denote $\hat{f}_j^{\rightarrow,y_1}$ be the $\hat{f}_j$ arising from $\Upsilon_{y_1}^{\rightarrow}$, 
i.e., set $f_0^{\rightarrow, y_1} = 2\pi\alpha^{-1}\kn\one_{\Upsilon^\rightarrow_{y_1}}$, and let $\hat{f}_j^{\rightarrow, y_1}$ be generated by the RG flow with line observable $f_0^{\rightarrow, y_1}$.
Other objects follow the corresponding definitions, e.g., $\hat{f}_j^{\leftarrow, y_2}$, $\hat\omega_j^{\rightarrow, y_1}$ and $\hat\ke_j^{\rightarrow, y_1}$.

We consider each term of \eqref{eq:localineenergy}.
Due to Definition~\ref{defi:locallineE}(iii), the local line energies are invariant under taking a sign change on $f$, so $\hat\ke_j^{\rightarrow,y_1} = \hat\ke_j^{\leftarrow, y_1}$ and independent of $N$ for $j \le N-1$.
Due to the locality and symmetry (Definition~\ref{defi:enhcLineObs}(iv)) of the enhanced line observable, we see that $\hat{f}_j^{\leftarrow, y_1} =-\hat{f}_j^{\rightarrow, y_1}$ and independent of $N$ whenever $j \le N-1$. Moreover, due to locality and Definition~\ref{defi:enhcLineObs}(v'), we may define
\begin{align}
    \frho_{j} (y_1) := \hat{f}_j^{\rightarrow, y_1} (e_{y_1}) = \hat{f}_j^{\rightarrow, y_1} ( (x,y_1) , (x,y_1 + 1)) , \qquad 
    x\in \Z / L^N \Z
    \label{eq:frhojDefi}
\end{align}
where $e_{y_1} \in E (\Lambda)$ is the edge $((0,y_1),(0, y_1+1))$.

\begin{corollary} \label{cor:kfinftyDefi}
There exists $\frho_{\infty} (y_1)$ such that 
$|\frho_{\infty} (y_1) - \frho_j (y_1)| \lesssim L^{-\eta j}$.
\end{corollary}
\begin{proof}
By locality,  $\frho_j (y_1)$ are consistently defined for all $j\ge0$. For $k>j$, Definition~\ref{defi:foldenhcLineObs}(ii') gives
\begin{align}
    |\frho_k (y_1)-\frho_j (y_1)|
    \le \sum_{\ell=j}^{k-1}
    |\frho_{\ell+1} (y_1)-\frho_\ell (y_1)|
    \lesssim L^{-\eta j}.
\end{align}
Thus $(\frho_j (y_1) )_{j\ge0}$ is a Cauchy sequence. Denoting its limit by $\frho_\infty (y_1)$ and letting $k\to\infty$ in the preceding estimate yields the claimed bound.
\end{proof}

\subsection{Digression: substitution by an infinitely extended observable}
\label{sec:digInfExtObs}

For a moment, we go back to the general setting, when $\bP$ has a standard shape and $\Upsilon = \bP^\dual$.
We consider a substitution of $\hat{f}_j$ by ``$\hat{f}_\infty^{\rightarrow,y_`}$'', whose meaning will become clear soon. For $e \in \Upsilon$, we can read off the `$y$-coordinate' of the edge as the following, which we call $\ry_e$.
\begin{itemize}
\item When $e = \big( (x,y), (x, y+1) \big)$, then $\ry_e = y$, and when $e = \big( (x,y+1), (x, y) \big)$, then $\ry_e = y + 1$.
\item When $e = \big( (x,y) ,(x+1, y) \big)$, then $\ry_e = x$ and when $e = \big( (x+1,y), (x,y) \big)$, then $\ry_e = x+1$.
\end{itemize}

Then we define $\bar{f} \in \Lambda^{(1)} (\R)$ as
\begin{align}
    \bar{f}_j (e) := \frho_{j} (\ry_e) \, \one_{\Upsilon} (e) ,
    \qquad
    \bar{f} (e) := \bar{f}_{\infty} (e) := \frho_{\infty} (\ry_e) \, \one_{\Upsilon} (e)  ,
    \label{eq:barfDefi}
\end{align}
i.e., it replaces $f$ by its limiting infinitely extended enhanced line observable. We use $\bar f$ rather than $f_\infty$, because the latter  (1) retain dependence on corners and tips of the path and (2) it loses locality in the limiting process. By contrast, $\bar f$ is constructed from infinite straight-line observables and depends only on the local orientation of the path.

We will estimate the error of replacing $\hat{f}_j$ by $\bar{f}$. Recall $\bC$ from Definition~\ref{def:standardPath} and define
\begin{align}
    \bC_+ = \{ e \in E(\Lambda) : e \in x , \; \text{some} \; x \in \bC \} . 
\end{align}

\begin{lemma} \label{lemma:fjbndRectangle}
If $e \in \supp_+ (f)$, then for $j\le N-1$,
\begin{align}
    \big| \hat{f}_j (e) - \bar{f} (e) \big| \lesssim \Big( \frac{1}{\dist_\infty (e,\bC_+)} \vee L^{-j} \Big)^{\eta} .
\end{align}
\end{lemma}
\begin{proof}
Since Corollary~\ref{cor:kfinftyDefi} implies $\norm{\bar{f}_j - \bar{f}}_{\ell^\infty} \lesssim L^{-\eta j}$, it will be sufficient to prove
\begin{align}
    \big| \hat{f}_j (e) - \bar{f}_{j} (e) \big| \le \Big( \frac{1}{\dist_\infty (e,\bC_+)} \Big)^{\eta} .
    \label{eq:lemmafjbndRectangle1}
\end{align}

First, suppose $\dist_\infty (e,\bC_+) \ge 10 L^j$, then by the locality of $\hat{f}_j$, we see that $\hat{f}_j (e) = \bar{f}_j (e)$, so there is nothing to prove.

Next, suppose $\dist_\infty (e,\bC_+) < 10 L^j$, and let $k = \max\{ k' \in \Z : \dist_\infty (e,\bC_+) \ge 10 L^{k'} \}$. 
Then by the same reasoning as above, we have $\hat{f}_k (e) = \bar{f}_k (e)$. 
Also, due to Definition~\ref{defi:foldenhcLineObs}(ii'), we have $|\hat{f}_{k} (e) - \hat{f}_{j} (e)| \lesssim L^{-\eta k}$ and 
$|\bar{f}_{k} (e) - \bar{f}_{j} (e)| \lesssim L^{-\eta k}$,
so in fact $|\hat{f}_j (e) - \bar{f}_j (e)| \lesssim L^{-\eta k}$. But by the definition of $k$, we have $L^{-k} \lesssim 1/ \dist_\infty (e,\bC_+)$, so we have \eqref{eq:lemmafjbndRectangle1}.
\end{proof}

\begin{lemma} \label{lemma:fjreplacementbyhatf}
For each $j \ge 0$ and $B \in \cB_{j+1}$, 
\begin{align}
    \big| (\hat{f}_j, -\Delta \Gamma_{j+1} \hat{f}_j)_B - (\bar{f}, -\Delta \Gamma_{j+1} \bar{f} )_B \big|
        \lesssim L^{-\eta j} .
\end{align}
\end{lemma}
\begin{proof}
If $B \cap \supp_+ (f) = \emptyset$ (recall \eqref{eq:suppfextensions}) then $\hat{f}_j |_{B} = \bar{f} |_B \equiv 0$, so we only need to consider $B$ that intersects $\supp_+ (f)$.

If $B^{**} \cap \bC_+ = \emptyset$, then $B$ does not feel the effect of the corners, so we have $\hat{f}_j = \bar{f}_j$.
In this case, 
\begin{align}
    (\bar{f}_j, -\Delta \Gamma_{j+1} \bar{f}_j )_B - (\bar{f}, -\Delta \Gamma_{j+1} \bar{f} )_B
    = \big( \bar{f}_j - \bar{f}, -\Delta \Gamma_{j+1} (\bar{f}_j + \bar{f}) \big)_B ,
\end{align}
while by Corollary~\ref{cor:kfinftyDefi}, $\norm{\bar{f}_j - \bar{f}}_{\ell^\infty} \lesssim L^{-\eta j}$ and $\norm{\bar{f}_j + \bar{f}}_{\ell^\infty} \lesssim 1$. Also, $\norm{\Delta \Gamma_{j+1}}_{\ell^\infty} \lesssim L^{-2j}$, so
\begin{align}
    \big| \big( \bar{f}_j - \bar{f}, -\Delta \Gamma_{j+1} (\bar{f}_j + \bar{f} )_B \big) \big| \lesssim L^{-\eta j} \times L^{-2j} \times (L^{j})^2 \lesssim L^{-\eta j}
\end{align}
where $(L^{j})^2$ is due to the fact that both $\bar{f}_j - \bar{f}$ and $\bar{f}_j + \bar{f}$ are supported on $\supp_+ (f)$, and $|\supp_+ (f) \cap B| \lesssim L^j$. 

If $B^{**} \cap \bC_+ \neq \emptyset$ and $B \cap \supp_+ (f) \neq \emptyset$, then we bound
\begin{align}
    \big| (\hat{f}_j, -\Delta \Gamma_{j+1} \hat{f}_j )_B - (\bar{f}, -\Delta \Gamma_{j+1} \bar{f} )_B \big|
    &= \big| \big(\hat{f}_j - \bar{f}, -\Delta \Gamma_{j+1} (\hat{f}_j + \bar{f} ) \big)_B \big| \nnb
    &\lesssim L^j \big\| -\Delta \Gamma_{j+1} (\hat{f}_j + \bar{f}) \big\|_{\ell^\infty} \, \norm{\hat{f}_j - \bar{f}}_{\ell^1 (B)} \nnb
    & \lesssim L^{-j} \sum_{k=1}^{10 L^{j}} \Big( \frac{1}{k} \vee L^{-j} \Big)^\eta 
    \lesssim L^{-\eta j} ,
\end{align}
where for the second inequality, we used $\norm{\Delta \Gamma_{j+1}}_{\ell^\infty} \lesssim L^{-2j}$ and Lemma~\ref{lemma:fjbndRectangle} to bound $\norm{\hat{f}_j - \bar{f}}_{\ell^1 (B)}$.
\end{proof}

We can use this result to restate Proposition~\ref{post:standard}. In what follows we define 
\begin{align}
    R := - \Delta \Gamma_{>0} := -\Delta \sum_{j=1}^{N} \Gamma_j , \qquad N \in \N \cup \{\infty \} .
\end{align}

\begin{proposition} \label{prop:standardModified}
There exists a collection $(\bar\ke_j : \cB_j \rightarrow \R )_{j=0}^{N-1}$ that satisfies Definition~\ref{defi:locallineE}---we call this modified line energy---such that
\begin{align}
\label{eq:propstandardModified}
\begin{split}
    \bL [\bP] 
    &= -\frac{1}{2} \sum_{t \in \bC} \bA(t) - \frac{1}{2} \Big( \sum_{t \in \bC} \bw (t) \delta_t, -  \Gamma_{>0} \sum_{t \in \bC} \bw (t) \delta_t \Big)
    \\
    &\qquad + \frac{1}{2} ( \bar{f} , R \bar{f} ) + \sum_{j\ge 0} \bar\ke_j (\Lambda) + O(1/j_\bxy) .
\end{split}    
\end{align}
\end{proposition}
\begin{proof}
We start from \eqref{eq:localineenergy}.
Due to Proposition~\ref{prop:enhancedCorr}, we can replace $\sum_j (d \hat{f}_j, \Gamma_{j+1} d \hat{f}_j)$ by the first line of the right-hand-side of \eqref{eq:propstandardModified}.
Due to Lemma~\ref{lemma:fjreplacementbyhatf}, we can replace $\sum_j (\hat{f}_j, -\Delta \Gamma_{j+1} \hat{f}_j)$ by $(\bar{f},-\Delta \Gamma_{>0} \bar{f})$, and we obtain the desired result once we let
\begin{align}
    \bar\ke_j (b) = 
        \begin{cases}
            \hat\ke_j (b) & (j=0) \\
            \hat\ke_j (b) + {\frac 12}(\hat{f}_{j-1}, -\Delta \Gamma_{j} \hat{f}_{j-1})_b - \frac 12 (\bar{f}, -\Delta \Gamma_j \bar{f} )_b & (j \ge 1) .
        \end{cases}
\end{align}
for $j\ge 1$ and $b \in \cB_j$. We can verify the conditions for $(\bar{\ke}_j)_j$ being a local line energy from the properties of the enhanced line observables $(\hat{f}_j)_j$. (Note that, $\bar{f}$ makes reference to to the infinite RG flow, but this does \emph{not} mean that it breaks the locality of $\bar{f}$. By definition, $\bar{f}$ only depends on its local environment.)
\end{proof}

To apply Proposition~\ref{prop:standardModified}, we need estimates for the interaction $(\bar{f}, R \bar{f})$.
In Lemma~\ref{lemma:inbtwparlines}, we show that $R = -\Delta \Gamma_{>0}$ is exponentially decaying, using the fact that $\Gamma_{>0}$ is a tilted Green's function of $-\Delta$.

\begin{lemma} \label{lemma:inbtwparlines}
If $|s|$ is sufficiently small, then $R$ is an exponentially decaying kernel, i.e., there exist $C,c>0$ such that
\begin{align}
    |R (x,y)| \le C e^{-c \dist_1 (x,y)} .
\end{align}
\end{lemma}
\begin{proof}
Recall the definition of $C(s,m^2)$ in \eqref{eq:Csm},  the decomposition $C(s,m^2) = \sum_{j=1}^{N} \Gamma_j (m^2) + t_N (m^2) Q_N$, and $\Gamma_{>0} = \lim_{m^2 \rightarrow 0} \sum_{j=1}^{N} \Gamma_j (m^2)$ from Section~\ref{sec:reformSG}. Since $\Delta Q_N = 0$,
\begin{align}
   -\Delta \Gamma_{>0} = \lim_{m^2\rightarrow 0} (-\Delta C(s,m^2) ) 
        &= \frac{-\Delta}{ ((-\Delta)^{-1} - \gamma \operatorname{id} )^{-1} + s (-\Delta)} \nnb
         &=  \frac{1}{(\operatorname{id}+\gamma \Delta)^{-1}+ s\operatorname{id}} \nnb
        &= \frac{1}{1+s} + \frac{ \gamma \Delta}{(1+s) (1+s-s\gamma (-\Delta))} .
\end{align}
To see the decay bound, using that $s$ is sufficiently small, we expand the second term as
\begin{align}
    R' 
    :=  \frac{\gamma \Delta}{(1+s) (1+s-s\gamma (-\Delta))} 
    = -\frac{\gamma}{(1+s)^2} \sum_{k=0}^{\infty} \gamma^k \Big( \frac{s}{1+s} \Big)^k (-\Delta)^{k+1} (x,y)  .
    \label{eq:EOperatorExpansion}
\end{align}
Since $-\Delta$ is nearest-neighbourhood, the summand does not vanish only for $k+1 \ge \dist_1 (x,y)$, so
\begin{align}
    |R' (x,y)| \lesssim \sum_{k+1 \ge \dist_1 (x,y)} (C|s|)^k \lesssim (C|s|)^{\dist_1 (x,y)-1} 
\end{align}
for some $C>0$, implying the exponential decay.
\end{proof}

\subsection{Interaction between parallel lines}

We return to the setting $\Upsilon = \Upsilon_{y_1, y_2}^{\rightleftharpoons}$.
The exponentially decaying bound of Lemma~\ref{lemma:inbtwparlines} implies that the self-interaction of $\Upsilon_{y_1,y_2}^{\rightleftharpoons}$ can effectively be decomposed into those of $\Upsilon^\rightarrow_{y_1}$ and $\Upsilon^\rightarrow_{y_2}$.

\begin{lemma} \label{lemma:DeltaGammaNoneexpntldecay}
Let $\dist_1 (y_1, y_2) > 2^d L^{N-1}$. Then for some $c>0$,
\begin{align}
    & ( \one_{\Upsilon^{\rightleftharpoons}_{y_1, y_2}}, R \one_{\Upsilon^{\rightleftharpoons}_{y_1, y_2}} )
    = (\one_{\Upsilon^\rightarrow_{y_1}}, R \one_{\Upsilon^\rightarrow_{y_1}}) 
    + ( \one_{\Upsilon^\rightarrow_{y_2}}, R \one_{\Upsilon^\rightarrow_{y_2}})
    + O(e^{-c \dist_1 (y_1, y_2)}) .
\end{align}
\end{lemma}
\begin{proof}
Since $\one_{\Upsilon^{\rightleftharpoons}_{y_1, y_2}} = \one_{\Upsilon^\rightarrow_{y_1}} + \one_{\Upsilon^\rightarrow_{y_2}}$,
\begin{align}
    & ( \one_{\Upsilon^{\rightleftharpoons}_{y_1, y_2}}, R \one_{\Upsilon^{\rightleftharpoons}_{y_1, y_2}} )
    = 
        ( \one_{\Upsilon^\rightarrow_{y_1}}, R  \one_{\Upsilon^\rightarrow_{y_1}}) 
        +
        (\one_{\Upsilon^\rightarrow_{y_2}}, R  \one_{\Upsilon^\rightarrow_{y_2}}) 
        +
        2 ( \one_{\Upsilon^\rightarrow_{y_1}}, R \one_{\Upsilon^\rightarrow_{y_2}}) ,
\end{align}
and by Lemma~\ref{lemma:inbtwparlines},
$( \one_{\Upsilon^\rightarrow_{y_1}}, R \one_{\Upsilon^\rightarrow_{y_2}})$ decays exponentially in $\dist_1 (y_1, y_2)$.
\end{proof}

For the self-interaction $(\one_{\Upsilon^\rightarrow_{y_1}} , R \one_{\Upsilon^\rightarrow_{y_1}})$, the first moment of $R$ will be enough.

\begin{lemma} \label{lemma:EOneDefi}
There exists $R^{(1)} \in \R$ such that
\begin{align}
    \sum_{k=1}^{L^N} R ( k \bbe_1 ) = R^{(1)} + O( e^{-c L^N})
\end{align}
for some $c>0$.
\end{lemma}
\begin{proof}
We can also consider the expansion \eqref{eq:EOperatorExpansion} on $\Z^2$, which defines an operator $R^{\Z^2}$. We define
\begin{align}
    R^{(1)} = \sum_{k \in \Z} R^{\Z^2} (k\bbe_1) . 
\end{align}
By the exponential decay of $R$ and $R^{\Z^2}$, there exists $c>0$ such that
\begin{align}
    \Big| \sum_{k \in \Z} R^{\Z^2} (k\bbe_1) - \sum_{k \in \Z/ L^N \Z} R (k\bbe_1) \Big| \lesssim e^{-cL^N} .
\end{align}
\end{proof}

\begin{corollary}  \label{cor:UpsilonDeltaCsUpsilon}
If $\Upsilon = \Upsilon^{\rightarrow}_{y_1}$, then for some $c>0$, 
\begin{align}
    (\one_{\Upsilon}, R \one_{\Upsilon}) =  R^{(1)} |\Upsilon| + O(e^{-cL^N}) .
\end{align}
\end{corollary}
\begin{proof}
By the choice of $\Upsilon$,
\begin{align}
    (\one_\Upsilon, R \one_\Upsilon) = L^N \sum_{k=1}^{L^N} R( k\bbe_1) = L^N  R^{(1)} + O(e^{-c L^N}) .
\end{align}
\end{proof}

\subsection{Cancellation identity for straight line energy}

We come back to the setting $\Upsilon = \Upsilon^{\rightleftharpoons}_{y_1}$. Then $\bC = \emptyset$ (the path $\bP = \Upsilon^\dual$ does not have any characteristic points), so Proposition~\ref{prop:standardModified} gives
\begin{align}
    \bL [\Upsilon] = \frac{1}{2} (\bar{f} , R \bar{f}) + \sum_{j=1}^{N-1} \bar{\ke}_j (\Lambda) + O(L^{-\eta N}) .
    \label{eq:prop:standardModifiedForStraightLine}
\end{align}
For the (modified) local line energy, let $b_j^{y_1}$ and $b_j^{y_2}$ be the unique $j$-boxes containing the points $y_1 \bbe_2$ and $y_2 \bbe_2$, respectively.
Also, if we let
\begin{align}
    & X_j^{y_1} = \cup \{ b = b_j^{y_1} + n L^{j} \bbe_2 : n\in \Z , \; b^* \cap \overline{\supp} (d^* f) \neq \emptyset \} \\
    & X_j^{y_2} = X_j^{y_1} + mL^{N-1}
\end{align}
(recall that $y_2 = y_1 + mL^{N-1}$). When we consider copies $X_j^{y_1}$ of translated horizontally, $L^{N-j}$ copies fit into $\Lambda_N$, so together with the symmetries of local line energy,
\begin{align}
    \sum_{j=0}^{N-1} \bar\ke_j (\Lambda) = \sum_{j=0}^{N-1} L^{N-j} \big( \bar\ke_j^{\rightarrow,y_1} (X_j^{y_1}) + \bar\ke_j^{\leftarrow,y_2} (X_j^{y_2}) \big)
    = 2 L^N \sum_{j=0}^{N-1} L^{-j} \bar\ke_j^{\rightarrow,y_1} (X_j^{y_1}) .
\end{align}
We will abbreviate
$\ccE_j^{\rightarrow} (y_1) \equiv \bar\ke_j^{\rightarrow,y_1} (X_j^{y_1})$ from here on, and also use
\begin{align}
    \ccE_{<j}^{\rightarrow} (y_1) = \sum_{k=0}^{j-1} L^{-k} \ccE_k^{\rightarrow} (y_1)  .
    \label{eq:kerarrowDefi}
\end{align}
They admit limits as $j\rightarrow \infty$.
In future uses in general, $\ke$ will be polymer functions and $\ccE$ will be just scalars.

\begin{lemma} \label{lemma:barkeinfty}
There exist $\ccE_{<\infty}^{\rightarrow} (y_1)$ such that 
$|\ccE_{<\infty}^{\rightarrow} (y_1) - \ccE_{<j}^{\rightarrow} (y_1) | \lesssim L^{-(1+\eta) j}$.
\end{lemma}
\begin{proof}
It is clear from the bound of Definition~\ref{defi:locallineE}(ii).
\end{proof}

We finally deduce the cancellation of the line energies from Lemma~\ref{lemma:contractible}. 

\begin{lemma} \label{lemma:straightlinecancellation}
We have
\begin{align}
    \frac{1}{2} R^{(1)} (\frho_\infty (y_1))^2 + \ccE_{<\infty}^{\rightarrow} (y_1) = 0 .
\end{align}
\end{lemma}
\begin{proof}
Let $\Upsilon = \Upsilon^{\rightleftharpoons}_{y_1}$. 
By Lemma~\ref{lemma:DeltaGammaNoneexpntldecay} and Corollary~\ref{cor:UpsilonDeltaCsUpsilon}, 
since $|\Upsilon_{y_1}^{\rightarrow}| = |\Upsilon_{y_2}^{\leftarrow}| = L^N$,
we get
\begin{align}
    (\bar{f}, R \bar{f}) = (\frho_\infty (y_1))^2 (\one_{\Upsilon}, R \one_{\Upsilon}) = 2L^N (\frho_\infty (y_1))^2  R^{(1)}  + O(L^{-\eta N}) ,
\end{align}
while by Lemma~\ref{lemma:contractible}, we get $\bL [\bP] =0$.
Thus by \eqref{eq:prop:standardModifiedForStraightLine}--\eqref{eq:kerarrowDefi} and Lemma~\ref{lemma:barkeinfty}, we get
\begin{align}
    \frac{1}{2} 2L^N  (\frho_\infty (y_1))^2  R^{(1)}  + 2 L^N \ccE_{<\infty}^{\rightarrow} (y_1) + O(L^{-\eta N}) = \bL[\bP] = 0 .
\end{align}
The conclusion follows after dividing both sides by $2L^N$ and taking $N\rightarrow \infty$.
\end{proof}

\section{Line segment}
\label{sec:straighlinesegment}

We consider a line segment observable.
Given $M>0$, $\upsilon = (\upsilon_x, \upsilon_y)$ and $\varpi= (\varpi_x, \varpi_y)$ with $\upsilon_y = \varpi_y$, let
\begin{align}
    \Upsilon^\lseg_{\upsilon, \varpi} = 
    \big\{ \big( 
    (x , \upsilon_y ) , (x , \upsilon_y +1 ) \big)  \, : \,
    x \in [\upsilon_x +1 , 100L^M - \varpi_x - 2] \big\} .
\end{align}
Then let $\bP^\lseg_{\upsilon,\varpi} = (\Upsilon^\lseg_{\upsilon, \varpi})^\dual$, and $L^M$ denotes the length scale of the line segment. 
The endpoints $t_1,t_2 \in \Lambda^\dual$ of $\bP^\lseg_{\upsilon,\varpi}$ are given by $t_1 = (\upsilon_x, \upsilon_y)$ and $t_2 = (100 L^M - \varpi_x - 2 , \upsilon_y)$.
(See Figure~\ref{fig:lineSeg}.)
For such a path, due to \eqref{eq:sumwinfty}, we have
\begin{align}
    \bw (t_1) + \bw (t_2) = 0 ,
    \label{eq:lsegwSumToZero}
\end{align}
so the Coulomb energy is 
(see Proposition~\ref{prop:standardModified})
\begin{align}
\label{eq:coulomb}
    \Big( \sum_{t \in \bC} \bw (t) \delta_t, \Gamma_{>0} \sum_{t \in \bC} \bw (t) \delta_t \Big) 
    =  \bw (t_1)^2 \Big(\delta_{t_1} - \delta_{t_2} , \Gamma_{>0} (\delta_{t_1} - \delta_{t_2} ) \Big) + O(1/M) .
\end{align}

\begin{figure}[h]
\begin{centering}
\qquad\qquad\qquad
\begin{tikzpicture}[scale=0.6, x=1cm, y=1cm, >=stealth]
    \tikzset{
        middle arrow/.style={
            postaction={decorate, decoration={
                markings,
                mark=at position 0.5 with {\arrow[scale=1.5, thick]{>}}
            }}
        }
    }

    \draw[->, thin] (-0.5, 0) -- (10, 0); 
    \draw[->, thin] (0, -0.5) -- (0, 7);  

    \coordinate (Origin) at (0,0);
    \coordinate (V) at (1.5, 3.5); 
    \coordinate (W) at (7, 3.5);   
    \coordinate (XEnd) at (9,0);   

    \draw[thick, middle arrow] (V) -- (W);

    \fill (V) circle (2.5pt);
    \fill (W) circle (2.5pt);

    \node[below=25pt, font=\Large] at (V) {$\upsilon$};
    \node[above=10pt, left=3pt, font=\Large] at (V) {$t_1$};    
    \node[below=25pt, font=\Large] at (W) {$\varpi$};
    \node[above=10pt, right=3pt, font=\Large] at (W) {$t_2$};

    \node[below=5pt, xshift=5pt] at (11.3, 6.5) {$x=100 L^M-1$};

    \draw[->, blue, thin] (Origin) -- (1.3, 3.1); 
    \draw[<-, blue, thin] (7.1, 3.1) -- (9, 0); 

    \draw[dashed] (9, 0) -- (9, 6.5);
\end{tikzpicture}
\caption{$\bP_{\upsilon,\varpi}^{\lseg}$ (bold arrowed line). Notice that $\varpi$ is not the coordinate of $t_2$, but its relative position with respect to $(100L^M -1 , 0)$}
\label{fig:lineSeg}
\end{centering}
\end{figure}
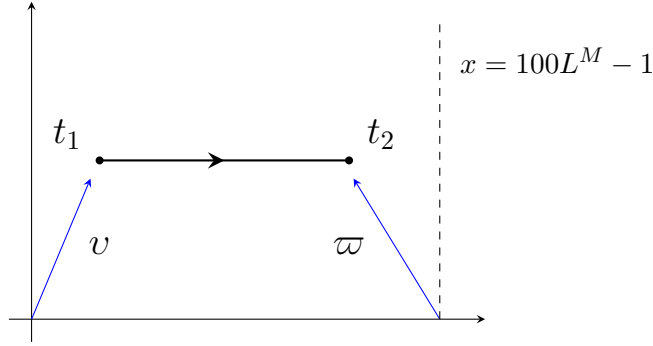

We now count the local line energy.  
If we define
\begin{align}
    \hash_{j} ( \upsilon_x ) = \min \{ n \in \Z : n \times L^j - 1 \ge \upsilon_x \} + 2^d ,
\end{align}
then we can find
\begin{align}
    & L^j \big( 100 L^{M-j} - \hash_{j} ( \upsilon_x ) - \hash_{j} ( \varpi_x ) \big)
\end{align}
points in the line segment that do not feel the effect of the tip at scale $j$.
To make this quantitative, let $(x_1, x_2, x_3, x_4) = (\upsilon, \upsilon +\bbe_1 , \upsilon + \bbe_2 , \upsilon+\bbe_1 +\bbe_2)$ and let 
\begin{align}
    \cB_j^{1} (\upsilon) 
    = \big\{ b \in \cB_j : b^* \cap \{x_1, x_2, x_3, x_4\} \neq \emptyset \big\} ,
    \label{eq:cBjlsegdefi} 
\end{align}
a collection of $j$-blocks  whose enlargements intersect the plaquette associated with the tip $t_1$. The corresponding collection $\cB_j^{2}(\varpi)$ is defined analogously for the tip $t_2$.
Then we can decompose the (modified) local line energy as
\begin{align}
    \bar\ke_{j} (\Lambda) = \big( 100 L^{M-j} - \hash_{j} (\upsilon_x) - \hash_j (\varpi_x) \big) \ccE_j^{\rightarrow} (\upsilon_y)
    + \sum_{b \in \cB_j^{1} (\upsilon)} \bar\ke_j (b) 
    + \sum_{b \in \cB_j^{2} (\varpi)} \bar\ke_j (b)
    .
    \label{eq:hatkelinesegmentdecomp}
\end{align}

\begin{lemma} \label{lemma:locallinenEnergylseg}
If we let
\begin{align}
    \label{eq:ketupsDefi}
    \ccE_{<N}^{t} (\upsilon) 
        &= \sum_{j<N} \Big[ - \hash_{j}( \upsilon_x ) \ccE_j^{\rightarrow} (\upsilon_y) + \sum_{b \in \cB_j^{1} (\upsilon)} \bar\ke_j (b) \Big]
\end{align} 
there exist $\ccE_{<\infty}^{t}(\upsilon)$ and $\ccE_{<\infty}^{t} (\varpi)$ such that for every $N>M$,
\begin{align}
    \sum_{j=0}^{N-1} \bar\ke_j (\Lambda) 
    = 100 L^{M}  \ccE_{<\infty}^{\rightarrow} (\upsilon_y) + \ccE_{<\infty}^{t} (\upsilon) + \ccE_{<\infty}^{t} (\varpi)  + O(L^{-\eta M}) .
\end{align}
and $| \ccE_{<\infty}^{t} (\upsilon) - \ccE_{<M}^{t} (\upsilon) | \lesssim L^{-\eta M}$.
\end{lemma}
\begin{proof}
It follows directly from \eqref{eq:hatkelinesegmentdecomp} that
\begin{align}
    \sum_{j<N} \bar\ke_j (\Lambda) 
        &= 100 L^{M} \sum_{j<N} L^{-j } \ccE_j^{\rightarrow}(\upsilon_y) +  \ccE_{<N}^{t}(\upsilon) + \ccE_{<N}^{t} (\varpi)  \nnb
        &= 100L^M \ccE_{<\infty}^{\rightarrow}(\upsilon_y) + \ccE_{<N}^{t} (\upsilon) + \ccE_{<N}^{t}(\varpi) + O(L^{-\eta M})
\end{align}
But since $|\ccE_{<j+1}^{t}(\upsilon) - \ccE_{<j}^{t}(\upsilon) | \lesssim L^{-\eta j}$ due to Definition~\ref{defi:locallineE}, we see that $\ccE_{<N}^{t} (\upsilon)$ also forms a Cauchy sequence with its increment bounded by $O(L^{-\eta j})$.
This shows the existence of the limits $\ccE_{<\infty}^{t} (\upsilon), \ccE_{<\infty}^{t}(\varpi)$ with the desired error bound.
\end{proof}

For the self-interaction of $\bar{f}$, recall from \eqref{eq:barfDefi}, that
\begin{align}
    \bar{f} = \frho_\infty (\upsilon_y) \one_{\Upsilon}  .
    \label{eq:VprimelsegdefiPre}
\end{align}

\begin{lemma} \label{eq:Vprimelsegdefi}
Let $\Upsilon = \Upsilon^\lseg_{\upsilon,\varpi}$. Then there exists $V^{\lseg} \in \R$ such that
\begin{align}
    (\one_\Upsilon , R \one_\Upsilon )
    = R^{(1)} |\Upsilon| + 4 V^{\lseg} + O(L^{-\eta M}),
\end{align}
where $R^{(1)}$ is defined in Lemma~\ref{lemma:EOneDefi} and $V^{\lseg}$ is independent of the choice of $\upsilon$ and $\varpi$.
\end{lemma}
\begin{proof}
We consider $R^{\Z^2}$ as in the proof of Lemma~\ref{lemma:EOneDefi}, and it will be sufficient to prove the statements with $R$ replaced by $R^{\Z^2}$ due to the exponentially decaying errors.
We define
\begin{align}
    V^{\lseg} 
    = - \frac{1}{2} \sum_{k' \le 0} \sum_{k\ge 1}^\infty R^{\Z^2} ( (k-k') \bbe_1 ) , 
\end{align}
and consider
\begin{align}
    R^{(1)} |\Upsilon|  - (\one_\Upsilon , R^{\Z^2} \one_\Upsilon) 
    &= \sum_{k=1}^{|\Upsilon|} \sum_{l \in \Z \backslash [1, |\Upsilon|] } R^{\Z^2} ( (k-l) \bbe_1 ) \nnb
    &= \sum_{k=1}^{|\Upsilon|} \Big[ \sum_{l \le 0} R^{\Z^2} ( (k-l) \bbe_1 ) 
    + \sum_{l > |\Upsilon|} R^{\Z^2} ( (k-l) \bbe_1 )
    \Big]  .
\end{align}
When $k \le |\Upsilon| /2$ and $l > |\Upsilon|$ or $|k| > |\Upsilon|/2$ and $l\le 0$, then $|k-l| \ge |\Upsilon|/2$, thus by the exponential decay of $R^{\Z^2}$ (Lemma~\ref{lemma:inbtwparlines})
\begin{align}
    \sum_{k=1}^{|\Upsilon| /2} \Big[ \sum_{l \le 0} R^{\Z^2} ( (k-l) \bbe_1 ) 
    + \sum_{l > |\Upsilon|} R^{\Z^2} ( (k-l) \bbe_1 ) \Big]
    & = \sum_{k\ge 1} \sum_{l \le 0} R^{\Z^2} ( (k-l) \bbe_1 ) + O(e^{-c |\Upsilon|}) \nnb
    & =: - 2 V^{\lseg} + O(e^{-c' |\Upsilon|}) .
\end{align}
We can perform a similar operation for $k > |\Upsilon| /2$, so
\begin{align}
    R^{(1)} |\Upsilon|  - (\one_\Upsilon , R^{\Z^2} \one_\Upsilon)  = - 4V^{\lseg}  + O(e^{-c' |\Upsilon| }) .
\end{align}
\end{proof}

We find that the energy contribution of a tip of the line segment can be summarised as
\begin{align}
    \label{eq:alsegDefi}
    a^\lseg (\upsilon) 
        &:= - \frac{1}{2} R^{(1)}(\frho_\infty (\upsilon_y))^2  (\upsilon_x +1) - \frac{1}{2} \bA(t_1) + \ccE^{t}_{<\infty} (\upsilon) + V^\lseg (\frho_\infty (\upsilon_y))^2 .
\end{align}

\begin{corollary} \label{cor:straightlinesegIntermediate}
When $\Upsilon = \Upsilon^{\lseg}_{\upsilon, \varpi}$ and $\bP = \Upsilon^\dual$, 
\begin{align}
    & \bL [\bP] 
    = a^\lseg (\upsilon) + a^\lseg (\varpi)
    -\frac{1}{2} \bw (t_1)^2 \Big( \delta_{t_1} - \delta_{t_2} , \Gamma_{>0} \big( \delta_{t_1} - \delta_{t_2} \big) \Big) + O(1/M) .
\end{align}
\end{corollary}
\begin{proof}
Putting \eqref{eq:coulomb}, \eqref{eq:VprimelsegdefiPre}, Lemma~\ref{lemma:locallinenEnergylseg} and \ref{eq:Vprimelsegdefi} with Proposition~\ref{prop:standardModified}, we get
\begin{align}
    \bL [\bP] 
    &= -\frac{1}{2} \sum_{i=1,2}
    \bA (t_i)
    + \ccE^{t}_{<\infty}(\upsilon) + \ccE^{t}_{<\infty} (\varpi)
    + 2 V^{\lseg} (\frho_\infty (\upsilon_y))^2
    + 100 L^M \ccE_{<\infty}^{\rightarrow} (\upsilon_y)
    + \frac{1}{2} R^{(1)} |\Upsilon| (\frho_\infty (\upsilon_y))^2 
    \nnb
    &\qquad -\frac{1}{2} \bw (t_1)^2 \Big( \delta_{t_1} - \delta_{t_2} , \Gamma_{>0} \big( \delta_{t_1} - \delta_{t_2} \big) \Big) + O(1/M)
    \label{eq:lsegIntest}
\end{align}
but since $|\Upsilon| = 100 L^M - \upsilon_x - \varpi_x - 2$,
\begin{align}
    &= a^\lseg (\upsilon) + a^\lseg (\varpi)
    -\frac{1}{2} \bw (t_1)^2 \Big( \delta_{t_1} - \delta_{t_2} , \Gamma_{>0} \big( \delta_{t_1} - \delta_{t_2} \big) \Big)  \nnb
    &\qquad + 100 L^M \Big[ \ccE_{<\infty}^{\rightarrow} (\upsilon_y)
    + \frac{1}{2} R^{(1)} (\frho_\infty (\upsilon_y))^2 
    \Big] .
\end{align}
Also, by Lemma~\ref{lemma:straightlinecancellation},
\begin{align}
    & \ccE_{<\infty}^{\rightarrow} (\upsilon_y)
    + \frac{1}{2} R^{(1)}(\frho_\infty (\upsilon_y))^2
    = O(L^{-(1+\eta) N}) ,
\end{align}
so we have the desired identity.
\end{proof}

\subsection{Invariance of $\bw (t)$}

 In this section we show that $\bw (t_1)$ is independent of $\upsilon$ (see Lemma~\ref{lemma:wt1isfinfty} and \ref{lemma:bffinftydefi}) .
We first show its independence of the horizontal coordinate $\upsilon_x$ by identifying it with the limiting amplitude of the enhanced line observable.

\begin{lemma} \label{lemma:wt1isfinfty}
If $\Upsilon = \Upsilon^\lseg_{\upsilon, \varpi}$, then $\bw (t_1) = \frho_\infty (\upsilon_y)$.
\end{lemma}
\begin{proof}
Recall $w_{k,j}$ from \eqref{eq:wkjDefi}:
$w_{k,j} (t_1) = \sum_{x \in \bC_j (t_1)} d\hat{f}_k (x)$.
Since $\Upsilon$ is a line segment, if we take $k < j \le j_\bxy - 1$, then there exists $n>0$ such that $\bC_j (t_1)$ can be expressed as
\begin{align}
    \bC_j (t_1) = \{ (\upsilon_x + m , \upsilon_y + \sigma) : m\in [0,n] , \; \sigma \in \{ \pm 1 , 0 \}  \} ,
\end{align}
i.e., it has the shape of a box.

\begin{figure}[h]
\begin{centering}
\begin{tikzpicture}[scale=0.9, x=1.5cm, y=1.2cm, >=stealth]
    \filldraw[fill=black!9, draw=none, rounded corners=12pt] (-0.5, -1.5) rectangle (5.5, 1.5);

    \draw[thick] (-0.5, -0.5) -- (7.2, -0.5);
    \draw[thick] (-0.5,  0.5) -- (7.2,  0.5);

    \foreach \x in {0.5, 1.5, 2.5, 3.5, 4.5, 5.5, 6.5} {
        \draw[thick] (\x, -1.5) -- (\x, 1.5);
    }

    \draw[dashed, gray, -stealth] (0, 0) -- (7.3, 0) node[right, black] {$\mathbf{P}$};

    \draw[line width=2.5pt, gray!80!black] (5.5, -0.5) -- (5.5, 0.5);

    \draw[dotted, -stealth] (5.6, 0.2) to[out=0, in=-90] (6.2, 1.6) node[above] {$\bar{e}_j$};

    \foreach \x in {0, 1, 2, 3, 4, 5, 6} {
        \foreach \y in {-1, 0, 1} {
            \ifnum\x=0
                \ifnum\y=0
                \else
                    \draw[fill=white, draw=black] (\x, \y) circle (2pt);
                \fi
            \else
                \draw[fill=white, draw=black] (\x, \y) circle (2pt);
            \fi
        }
    }

    \fill[black] (0, 0) circle (2.5pt);

    \node[left=2pt] at (-0.5, 0) {$t_1 = (v_x, v_y)$};

    \node at (2.5, 2.0) {$\mathbf{C}_j(t_1)$};
\end{tikzpicture}
\caption{An illustration of $\bC_j (t_1)$ (grey area). Since $\hat{f}_j$ is supported on $\supp_I (f)$, it vanishes on the upper, left and lower edges. Circles represent $\supp^+ (df)$.}
\label{fig:Cjt}
\end{centering}
\end{figure}

If we denote $\partial \bC_j (t_1)$ for its boundary, the Discrete Stokes' Theorem gives 
\begin{align}
    \sum_{x\in \bC_j (t_1) } d\hat{f}_k (x) = \sum_{e \in \partial \bC_j (t_1)} \hat{f}_k (e) .
\end{align}
However, $\partial \bC_j (t_1) \cap \supp_I (f)$ only consist of one right-most edge (see Figure~\ref{fig:Cjt} and \ref{fig:suppPf}), which we call $\bar{e}_j$, so that
\begin{align}
    w_{k,j} (t_1) = \sum_{e \in \partial \bC_j (t_1)} \hat{f}_k (e) = \hat{f}_k (\bar{e}_{j})
\end{align}
Since $\bar{e}_{j} \in \Seg_{j-1}$, we have $\hat{f}_{j-1} ( \bar{e}_{j} ) = \hat{f}^{\rightarrow , \upsilon_y}_{j-1} (\bar{e}_{j})$, so
\begin{align}
    | \hat{f}_k ( \bar{e}_{j} ) - \hat{f}_\infty^{\rightarrow, \upsilon_y} (\bar{e}_{j}) | 
    &\le
    | \hat{f}_k ( \bar{e}_{j} ) - \hat{f}_{j-1} ( \bar{e}_{j} ) | +
    | \hat{f}_{j-1}^{\rightarrow, \upsilon_y} (\bar{e}_{j}) - \hat{f}_\infty^{\rightarrow, \upsilon_y} (\bar{e_{j}}) | \nnb
    &\lesssim L^{-\eta (k\wedge (j-1))} + L^{-\eta (j-1)} \lesssim L^{-\eta k} ,
    \label{eq:wt1isfinftyEq1}
\end{align}
and since $\hat{f}_\infty^{\rightarrow, \upsilon_y} = \frho_\infty (\upsilon_y)$,
we see that $|w_{k,\infty} (t_1) -  \frho_\infty (\upsilon_y) | \lesssim L^{-\eta k}$. Thus
$\bw (t_1) = \lim_{k\rightarrow \infty} w_{k,\infty} = \frho_\infty (\upsilon_y)$.
\end{proof}

Before proceeding, we note a standard Green's function estimate.

\begin{lemma} \label{lemma:Csasymptotic}
There exist $C_1 \in \R$, such that
\begin{align}
    (\delta_{\bx} - \delta_{\by} , \Gamma_{>0} (\delta_{\bx} - \delta_{\by})) = \frac{1}{\pi (1+s)} \log |\bx - \by| + C_1 + O\Big( \frac{1}{\log |\bx -\by|} \Big) .
\end{align}
Here $s:= s_0^c(\beta)$ is defined in Assumption \ref{asmp:critical}.
\end{lemma}
\begin{proof}
By \cite[Lemma~3.3]{dgauss3}, 
\begin{align}
    (\delta_{\bx} - \delta_{\by} , \Gamma_{>0} (\delta_{\bx} - \delta_{\by})) = 
    \frac{1}{1+s} (\delta_{\bx} - \delta_{\by} , (-\Delta)^{-1} (\delta_{\bx} - \delta_{\by})) + C + O( \norm{\bx-\by}^{-2} ) , 
\end{align}
for some $C\in \R$, and the result follows from a standard asymptotic on the lattice Green's function $(-\Delta)^{-1}$.
\end{proof}

Next, we show that $\bw (t_1)$ is also independent of $\upsilon_y$.

\begin{lemma} \label{lemma:bffinftydefi}
$\bw(t_1)$ does not depend on $\upsilon_y$. In particular, there exists an $\upsilon$-independent constant $\bff$ such that $\bff = 2\pi\kn\alpha^{-1} + O(e^{-\gamma\alpha^2 /4})$ and
\begin{align}
    \bff = \frho_\infty (\upsilon_y) = \bw(t_1) .
\end{align}
\end{lemma}
\begin{proof}
Since we have already seen that $\bw(t_1) = \frho_\infty (\upsilon_y)$, we only need to show that this does not depend on $\upsilon_y$.
Thus, let $\upsilon \neq \upsilon'$ and $\varpi \neq \varpi'$ have $\upsilon_x = \upsilon'_x$, $\varpi_x=\varpi'_x$, $\upsilon_y = \varpi_y$ and $\upsilon'_y = \varpi'_y$.
By their definitions, and the translation invariance of the DG model, we have
\begin{align}
   \bL [\bP^\lseg_{\upsilon, \varpi}]
    = 
    \bL [\bP^\lseg_{\upsilon', \varpi'}] , 
\end{align}
applying Corollary~\ref{cor:straightlinesegIntermediate},
\begin{align}
    & \big( (\frho_\infty (\upsilon_y))^2 - (\frho_\infty (\upsilon'_y))^2 \big) \Big( \delta_{t_1} - \delta_{t_2}, \Gamma_{>0} (\delta_{t_1} - \delta_{t_2}) \Big) \nnb
    &\qquad \qquad = a^\lseg (\upsilon) + a^\lseg (\varpi) - ( a^\lseg (\upsilon') + a^\lseg (\varpi') )+ O(1/M)  .
\end{align}
In the limit $N > M \rightarrow \infty$, we also have $|t_2 - t_1| \rightarrow \infty$, so $( \delta_{t_1} - \delta_{t_2}, \Gamma_{>0} (\delta_{t_1} - \delta_{t_2})) \rightarrow \infty$ by Lemma~\ref{lemma:Csasymptotic},
while the right-hand side remains bounded.
This can only happen when 
$(\frho_\infty (\upsilon_y))^2 = (\frho_\infty (\upsilon'_y))^2$.
Moreover, Definition~\ref{defi:foldenhcLineObs}(ii') together with the intial condition $\frho_0 (\upsilon_y) = 2\pi\kn\alpha^{-1}$ imply $|\frho_\infty (\upsilon_y) - 2\pi\kn\alpha^{-1}| \le \sum_{j=0}^\infty C e^{-\gamma\alpha^2 /4} L^{-\eta j} \le O(e^{-\gamma\alpha^2 /4})$, and this enforces $\frho_\infty (\upsilon_y) = \frho_\infty (\upsilon'_y)$ for sufficiently large $\alpha$.
\end{proof}

We can now restate Proposition~\ref{prop:standardModified} in terms of $\bff$.

\begin{corollary} \label{cor:straightlinecancellationV2}
For \underline{any} standard path $\Upsilon$,
\begin{align}
    \bL[\bP]
    &= -\frac{1}{2} \sum_{t \in \bC} \bA(t) - \frac{1}{2} \Big( \sum_{t \in \bC} \bw (t) \delta_t, \Gamma_{>0} \sum_{t \in \bC} \bw (t) \delta_t \Big)
    \nnb
    &\qquad + \frac{1}{2} \bff^2 ( \one_\Upsilon , R \one_\Upsilon ) + \sum_{j=0}^{N-1} \sum_{b \in \cB_j} \bar\ke_j (b) + O(1/j_\bxy) ,
\end{align}
\end{corollary}

Moreover, by Lemma~\ref{lemma:straightlinecancellation}, $\bff$ satisfies
$2 \ccE^{\rightarrow}_{<\infty} (y_1) +  R^{(1)}  \bff^2 = 0$. 
In particular, $\ccE^{\rightarrow}_{<\infty} (y_1)$ is also independent of $y_1$, so we simply denote
\begin{align}
    \ccE^{\rightarrow}_{<\infty} \equiv \ccE^{\rightarrow}_{<\infty} ( y_1) , 
    \qquad\quad
    2 \ccE^{\rightarrow}_{<\infty} + R^{(1)}  \bff^2 = 0 .
    \label{eq:barkeIndepOfUps}
\end{align}

\subsection{An identity for $a^\lseg$}

Corollary~\ref{cor:straightlinesegIntermediate} also has a reformulation due to the invariance of $\bw (t_1)$.

\begin{corollary} \label{cor:straightlinesegConcl}
For the line segment,
\begin{align}
    & \bL[ \bP^{\lseg}_{\upsilon, \varpi} ]
    = a^\lseg (\upsilon) + a^\lseg (\varpi)
    -\frac{1}{2} \bff^2 \Big( \delta_{t_1} - \delta_{t_2} , \Gamma_{>0} \big( \delta_{t_1} - \delta_{t_2} \big) \Big) + O(1/M)  .    
\end{align}
and there exist $\bar{a}_0, \bar{a}_1 \in \R$ such that
\begin{align}
    a^{\lseg} (\upsilon) = \bar{a}_0 + \bar{a}_1 \upsilon_x .
\end{align}
\end{corollary}
\begin{proof}
The first identity is direct from Corollary~\ref{cor:straightlinesegIntermediate}, Lemma~\ref{lemma:wt1isfinfty} and \ref{lemma:bffinftydefi}. For the second identity, observe that, due to the translation invariance of the measure,
\begin{align}
    & a^{\lseg} (\upsilon) + a^{\lseg} (\varpi) = a^{\lseg} (\upsilon + \bbe_1) + a^{\lseg} (\varpi - \bbe_1) \\
    & a^{\lseg} (\upsilon) + a^{\lseg} (\varpi) = a^{\lseg} (\upsilon + \bbe_2) + a^{\lseg} (\varpi + \bbe_2) .
\end{align}
Then the desired identity holds with $\bar{a}_0 = a^\lseg (0)$ and $\bar{a}_1 = a^\lseg (\bbe_1) - a^\lseg (0)$. 
\end{proof}

\subsection{Charges of a standard-shaped path}
\label{sec:chargesoastndpath}

Before proceeding further, we make a brief note on the charges of a general path of standard shape.

\begin{lemma} \label{lemma:bwcharacterisation}
Let $\bP$ be any path of a standard shape. Then $\bw (t) = \bff$ for $t \in \bT$
that is a source of the path, $\bw (t) = -\bff$ for $t \in \bT$ that is a sink of the path and $\bw (i) = 0$ for $i \in \bI$.
\end{lemma}
\begin{proof}
For $t \in \bT$, the proof is exactly the same as that of Lemma~\ref{lemma:wt1isfinfty}, and $\frho_\infty (\upsilon_y)$ is fixed constant by Lemma~\ref{lemma:bffinftydefi}.

For $i \in \bI$, without loss of generality, suppose $\bP$ enters $i$ from above and leaves $i$ to the right.
Again, by the discrete Stokes' Theorem, whenever $k < j \le M-1$,
\begin{align}
    w_{k,j} (i) = \sum_{x \in \bC_j (i)} d\hat{f}_k (x) = \sum_{e \in \partial\bC_j (i)} \hat{f}_k (x) = \hat{f}_k (\bar{e}_j) - \hat{f}_k ( \hat{e}_j)
\end{align}
where $\partial\bC_j (i)$ is the (counter-clockwise oriented) boundary of $\bC_j (i)$,
$\bar{e}_j$ is the right-most edge of $\bC_j (i)$ and $\hat{e}_j$ is the upper-most edge of $\bC_j (i)$. 
Since both $\bar{e}_j, \hat{e}_j \in \Seg_{j-1}$, we have 
\begin{align}
    \hat{f}_{j-1} (\bar{e}_j) = \hat{f}_{j-1}^{\rightarrow, \upsilon_y} (\bar{e}_j) , \qquad \hat{f}_{j-1} (\hat{e}_j) = \hat{f}_{j-1}^{\rightarrow, \upsilon_x} (\hat{e}_j) , 
\end{align}
so as in \eqref{eq:wt1isfinftyEq1},
\begin{align}
    | \hat{f}_k (\bar{e}_j) - \frho_\infty (\upsilon_y) | , \;\; 
    | \hat{f}_k (\hat{e}_j) - \frho_\infty (\upsilon_x) | \lesssim L^{-\eta k} .
\end{align}
But we have also seen that $\frho_\infty (\upsilon_x) = \frho_\infty (\upsilon_y) = \bff$, 
so this actually says $|\hat{f}_k (\bar{e}_j) - \hat{f}_k (\hat{e}_j) | \lesssim L^{-\eta k}$.
We obtain $\bw (i) = \lim_{k\rightarrow
 \infty} \lim_{j\rightarrow \infty} w_{k,j} (i) = 0$.
\end{proof}

\section{Rectangular path}
\label{sec:rectangularpath}

\begin{figure}
\begin{centering}
\begin{tikzpicture}[scale=0.7, x=1cm, y=1cm, >=stealth]
    \tikzset{
        mid arrow/.style={
            postaction={decorate, decoration={
                markings,
                mark=at position 0.55 with {\arrow[scale=1.5, thick]{>}}
            }}
        },
        mid arrow reverse/.style={
            postaction={decorate, decoration={
                markings,
                mark=at position 0.5 with {\arrow[scale=1.5, thick]{<}}
            }}
        }
    }

    \draw[->, thin] (-0.5, 0) -- (8, 0); 
    \draw[->, thin] (0, -0.5) -- (0, 7);  

    \coordinate (Origin) at (0,0);
    \coordinate (I1) at (2, 2);
    \coordinate (I2) at (6, 2);
    \coordinate (I3) at (6, 6);
    \coordinate (I4) at (2, 6);

    \draw[thick, mid arrow] (I1) -- (I2);         
    \draw[thick, mid arrow] (I2) -- (I3);         
    \draw[thick, mid arrow] (I3) -- (I4); 
    \draw[thick, mid arrow] (I4) -- (I1); 

    \foreach \pt/\lab/\pos in {I1/i_1/below right, I2/i_2/below right, I3/i_3/above, I4/i_4/above} {
        \fill (\pt) circle (2.5pt);
        \node[font=\Large, \pos=8pt] at (\pt) {$\lab$};
    }

    \draw[->, blue, thin] (Origin) -- (1.8, 1.8) node[midway, below=5pt, font=\Large, black] {$\upsilon$};
\end{tikzpicture}
\caption{$\bP^{\square}_\upsilon$ (bold arrowed line). The relative position of $i_1$ with respect to the origin is $\upsilon$.}
\end{centering}
\label{fig:UpsSquare}
\end{figure}

Now, consider a rectangular path defined as the following. 
Given $\upsilon = (\upsilon_x, \upsilon_y ) \in \Z^2$ and $M <N$, let
\begin{align}
    \Upsilon^1_{\upsilon} 
        &= \big\{ \big( (x,\upsilon_y), (x, \upsilon_y+1 ) : x = \upsilon_x +1 , \cdots , 100 L^M - \upsilon_x - 2 \big\} \\
    \Upsilon^2_\upsilon 
        &= \big\{ \big( ( 100 L^M - \upsilon_x -1 , y ) , (100 L^M - \upsilon_x -2 , y) \big) : y = \upsilon_y +1 , \cdots, 100 L^M - \upsilon_y - 2 \big\} \\
    \Upsilon^3_{\upsilon} 
        &= \big\{ \big( (x, 100 L^M - \upsilon_y -1), (x, 100 L^M - \upsilon_y -2 ) \big) : x = \upsilon_x +1 , \cdots , 100 L^M - \upsilon_x - 2 \big\} \\  
    \Upsilon^4_\upsilon 
        &= \big\{ \big( ( \upsilon_x , y ) , (\upsilon_x +1 , y) \big) : y = \upsilon_y +1 , \cdots, 100 L^M - \upsilon_y - 2 \big\}
        .
\end{align}
We also consider the rectangular path
\begin{align}
    \Upsilon^\square_\upsilon = \cup_{i=1}^4 \Upsilon^i_\upsilon , \qquad
    \bP^{\square}_{\upsilon} = ( \Upsilon^{\square}_\upsilon )^\dual .
\end{align}
The path has four corners $i_1 = (\upsilon_x, \upsilon_y)$, $i_2 = (100L^M - \upsilon_x - 2 , \upsilon_y)$, $i_3 = (100L^M - \upsilon_x - 2 , 100L^M - \upsilon_y - 2)$ and $i_4 = (\upsilon_x, 100L^M - \upsilon_y - 2) \in \Lambda^\dual$.
Since $\bP^{\square}_{\upsilon}$ is contractible, Lemma~\ref{lemma:contractible} implies
\begin{align}
    \bL [\bP_\upsilon^\square] =0 
    \label{eq:rectangleContractible}
\end{align}
Since $\bP_\upsilon^\square$ does not have tips, Lemma~\ref{lemma:bwcharacterisation} and symmetries imply
\begin{align}
    \label{eq:rectangleenhanceLineObs}
    \sum_{t \in \bC} \bA (t) = 4 \bA (i_1) , \qquad \bw (t) = 0 \text{ for all } t \in \bC .
\end{align}

\begin{corollary} \label{cor:rectangleEnergyInterm}
When $\Upsilon = \Upsilon^\square_\upsilon$,
\begin{align}
    -2 \bA (i_1) + \frac{1}{2} \bff^2 (\one_\Upsilon , R \one_\Upsilon ) + \sum_{j=0}^{N-1} \bar\ke_j (\Lambda) = O(1/M)
\end{align}
\end{corollary}
\begin{proof}
We get the desired identity after putting together \eqref{eq:rectangleContractible} and \eqref{eq:rectangleenhanceLineObs} with Corollary~\ref{cor:straightlinecancellationV2}.
\end{proof}

\subsection{Self-interaction of the rectangular path}

We first compute the self-interaction energy $(\one_\Upsilon , R \one_\Upsilon)$ for $\Upsilon = \Upsilon^\square_{\upsilon}$.

\begin{lemma} \label{lemma:rectangleselfinteractionE}
If $\Upsilon = \Upsilon^\square_\upsilon$, 
there exists $V^{\square} \in \R$ (independent of $\upsilon$) such that
\begin{align}
    (\one_\Upsilon , R \one_\Upsilon ) = R^{(1)} |\Upsilon| + 8 V^{\square}  + O(L^{-\eta M}) .
\end{align}
\end{lemma}
\begin{proof}
This also follows from Lemma~\ref{eq:Vprimelsegdefi}, the exponential decay of $R$.
To be specific, we take
\begin{align}
    V^{\square} = 2 V^{\lseg} = -  \sum_{k' \le 0} \sum_{k \ge 0} R^{\Z^2} ((k-k')\bbe_1)  ,
\end{align}
Since $R$ and $R^{\Z^2}$ vanish between orthogonal edges and $R$ decays exponentially,
\begin{align}
    R^{(1)} |\Upsilon| - (\one_\Upsilon, R \one_\Upsilon) =  -8 V^{\square} + O(|\Upsilon|^2 e^{-c |\Upsilon|}) ,
\end{align} 
where $O(|\Upsilon|^2 e^{-c |\Upsilon|})$ is due to the interaction between the pair of parallel edges of the rectangle, while $V^\square$ arises because each edge of the rectangle has the shape of a line segment, so has the same self-interaction as the line segment in Lemma~\ref{eq:Vprimelsegdefi}.
\end{proof}

\subsection{Corner energy at the inflection points}

We recall $\cB_j^1 (\upsilon)$ from \eqref{eq:cBjlsegdefi}.
Then we can divide the local line energy into corner components and straight line components, cf. \eqref{eq:hatkelinesegmentdecomp}. 
This gives, by symmetry,
\begin{align}
    & \bar\ke_j (\Lambda)
    = 4 \sum_{b \in \cB_{j+1}^{1} (\upsilon)} \bar\ke_j (b) + 2 \Big[ \big( 100L^{M-j} - 2\hash_j (\upsilon_x) \big) \ccE_j^{\rightarrow} (\upsilon_y)  + \big( 100L^{M-j} - 2\hash_j (\upsilon_y) \big) \ccE_j^{\rightarrow} (\upsilon_x) \Big] . 
    \label{eq:hatkejrectangle} 
\end{align}

\begin{lemma} \label{lemma:kecDefi}
Let
\begin{align}
    \ccE_{<N}^{c}(\upsilon) = \sum_{j<N} \Big[ - \hash_j (\upsilon_x) \ccE_j^{\rightarrow} (\upsilon_y) - \hash_j (\upsilon_y) \ccE_j^{\rightarrow} (\upsilon_x) + \sum_{b \in \cB_{j}^{1} (\upsilon)} \bar\ke_j (b) \Big] .
\end{align}
Then there exists $\ccE_{<\infty}^{c} (\upsilon)$ such that $|\ccE_{<\infty}^{c} (\upsilon) - \ccE_{<N}^{c} (\upsilon)| \lesssim L^{-\eta N}$ and
\begin{align}
    \sum_{j<N} \bar\ke_j (\Lambda) = 400 L^M \ccE_{<\infty}^{\rightarrow} + 4 \ccE_{<\infty}^{c}(\upsilon) + O(L^{-\eta M}) .
    \label{eq:lemmakecDefi2}
\end{align}
\end{lemma}
\begin{proof}
Existence of the limit $\ccE^{c}_{<\infty}(\upsilon)$ follows from Definition~\ref{defi:locallineE}, 
cf. Lemma~\ref{lemma:locallinenEnergylseg}. \eqref{eq:lemmakecDefi2} follows from \eqref{eq:hatkejrectangle}, \eqref{eq:kerarrowDefi} and Lemma~\ref{lemma:barkeinfty}.
Note that, we also used \eqref{eq:barkeIndepOfUps} to remove dependence of $\ccE_{<\infty}^{\rightarrow}(y_1)$ on $y_1$.
\end{proof}

Now, we insert everything inside Corollary~\ref{cor:rectangleEnergyInterm}, and obtain
\begin{align}
    - 2 \bA (i_1) + \frac{1}{2} R^{(1)} \bff^2 |\Upsilon| + 
    4 V^{\square} \bff^2
    + 400 L^M \ccE_{<\infty}^{\rightarrow}  + 4 \ccE_{<\infty}^{c} (\upsilon) 
    =O(1/M) .
    \label{eq:rectangleWOLinecancellation}
\end{align}
We can cancel out the straight line energy, which gives cancellation of the corner energy.

\begin{lemma} \label{lemma:asquarevanishes}
Let $\Upsilon = \Upsilon^\square_{\upsilon}$ and
\begin{align}
    \label{eq:asquareDefi}
    a^{\square} (\upsilon) = -\frac{1}{2} \bA (i_1) - \frac{1}{2} R^{(1)} (\upsilon_x + \upsilon_y + 2) \bff^2 + V^{\square} \bff^2 + \ccE_{<\infty}^{c} (\upsilon) .
\end{align}
Then $a^{\square} (\upsilon) = 0$.
\end{lemma}
\begin{proof}
Using that $|\Upsilon|= 400L^M - 4\upsilon_x - 4\upsilon_y - 8$,
\eqref{eq:rectangleWOLinecancellation} can be restated as
\begin{align}
    4 a^{\square} & +
    400L^M \Big[ \frac{1}{2} R^{(1)} \bff^2 + \ccE^{\rightarrow}_{<\infty} \Big] 
    = O(1/M) ,
\end{align}
but the terms following $400L^M$ vanishes by \eqref{eq:barkeIndepOfUps}, so we have
$4 a^{\square} (\upsilon) = O(1/M)$.
But $a^{\square}$ is $M$-independent, so it just vanishes.
\end{proof}

\section{$\sqcup$-shaped path connecting $\bx$ and $\by$}
\label{sec:sqcupPaths}

Given $\bx = (0, \bx_2)$ and $\by = (\by_1, \by_2) \in \Lambda_N$, consider $\bP_\bxy^\sqcup$, a $\sqcup$-shaped path of a standard shape connecting $\bx$ and $\by$. 
We assume $-100 L^{j_\bxy} \le \bx_2 \le 0$ and $\by_1 \ge \by_2 \vee 0 \ge \bx_2$ without losing generality.
Then standard shape assumption implies $\dist_{\infty} (\bx, \by) \in [L^{j_\bxy} , L^{j_\bxy+1})$ and $\by_1 \in [L^{j_\bxy} , L^{j_\bxy+1})$.
The path has two tips and two inflection points. The contributions from the two tips are controlled by the line segment analysis in Section~\ref{sec:straighlinesegment}, and those from the two corners are controlled by the rectangular path analysis in Section~\ref{sec:rectangularpath} (on rectangles).

\begin{figure}
\centering
\begin{tikzpicture}[scale=0.7, x=1cm, y=1cm, >=stealth]
\tikzset{
        mid arrow/.style={
            postaction={decorate, decoration={
                markings,
                mark=at position 0.55 with {\arrow[scale=1.5, thick]{>}}
            }}
        },
        mid arrow reverse/.style={
            postaction={decorate, decoration={
                markings,
                mark=at position 0.5 with {\arrow[scale=1.5, thick]{<}}
            }}
        }
    }

    \coordinate (Origin) at (0,0);
    \coordinate (I1) at (0,-3);
    \coordinate (I2) at (6,-3);
    \coordinate (X) at (0,-1);
    \coordinate (Y) at (6,1);
    \coordinate (TopRight) at (8.3,3.2);

    \draw[dashed, gray!80] (-1.5, 3.2) -- (9, 3.2); 
    \draw[dashed, gray!80] (8.3, -4) -- (8.3, 5); 

    \draw[->, thin] (-1.5, 0) -- (9, 0); 
    \draw[->, thin] (0, -4) -- (0, 5); 

    \draw[thick, mid arrow] (X) -- (I1);
    \draw[thick, mid arrow] (I1) -- (I2);
    \draw[thick, mid arrow] (I2) -- (Y);    

    \fill (I1) circle (2.5pt) node[below left=2pt] {$i_1$};
    \fill (I2) circle (2.5pt) node[below=5pt] {$i_2$};
    \fill (X) circle (2.5pt) node[right=5pt] {$\bx = t_1$};
    \fill (Y) circle (2.5pt) node[left=5pt] {$\by = t_2$};
    
    \fill[gray] (TopRight) circle (2.5pt) node[above right, black] {$(n_1 L^{j_{xy}} - 1, n_2 L^{j_{xy}} - 1)$};

    \draw[->, blue!60!cyan, thin] (-0.1, 0) -- (-0.1, -0.9) node[midway, left, black] {$\upsilon$};
    
    \draw[->, blue!60!cyan, thin] ($(TopRight)+(-0,-0)$) -- ($(Y)+(0.2,0.2)$) node[midway, above left, black] {$\upsilon'$};
    
    \draw[->, blue!60!cyan, thin] (8.3, -3) -- (6.3, -3) node[midway, above, black] {$\varpi'$};
\end{tikzpicture}
\caption{$\bP^{\sqcup}_{\bxy}$ (bold arrowed line). The relative position of $i_1$ with respect to the origin is $\upsilon$.
\label{fig:UpsSqcup}
}
\end{figure}

To be specific, we consider the path with
\begin{align}
    \text{two inflection points} &\quad i_1 = (0, - 100 L^{j_{\bxy}}) , \quad i_2 = (\by_1 , - 100 L^{j_{\bxy}} ) \quad  \text{and} \\
    \text{two tips} &\quad t_1 = \bx , \quad t_2 = \by .
\end{align}
We define
\begin{align}
    \Upsilon^{\sqcup,1}_{\bxy} 
    & = \big\{
    \big( (0,y), (1,y) \big) : y \in \{ \bx_2 , \cdots, - 100 L^{j_\bxy} + 1 \} 
    \big\} , \\
    \Upsilon^{\sqcup,2}_{\bxy} 
    & = \big\{
    \big( (x, - 100 L^{j_\bxy} ) , (x, -100 L^{j_\bxy} +1) \big) : x \in \{ 1, \cdots, \by_1 \}
    \big\} , \\
    \Upsilon^{\sqcup,3}_{\bxy} 
    & = \big\{
    \big( (\by_1 +1 , y) , (\by_1 , y ) \big) : y \in \{ -100 L^{j_\bxy} + 1 , \cdots ,  \by_2 \} 
    \big\}
\end{align}
and
\begin{align}
    \Upsilon^\sqcup_{\bxy}  = \bigcup_{l=1,2,3} \Upsilon^{\sqcup, l}_{\bx,\by} , \qquad
    \bP^\sqcup_{\bxy} = (\Upsilon^\sqcup_{\bxy})^\dual . 
\end{align}
To study this path in the framework of the previous sections, we take
\begin{align}
    n_1 = \min\{ n \in \Z : n L^{j_\bxy} - 1 \ge \by_1 \}, 
    \qquad
    n_2 = \min\{ n \in \Z : n L^{j_\bxy} - 1 \ge \by_2 \} 
\end{align}
and
\begin{align}
    \upsilon = ( -\bx_2 , 0 ) , \qquad \upsilon' = ( n_2 L^{j_\bxy} - 1- \by_2 , n_1 L^{j_\bxy} - 1 - \by_1 )
    \label{eq:sqcupupsilonDefi}
\end{align}
that characterise the relative position of the tips from $L^{j_\bxy} \Z^2$ and
\begin{align}
    \varpi = (0,0), \qquad \varpi' = (n_1 L^{j_\bxy} -1 - \by_1 , 0) .
    \label{eq:sqcupvarpiDefi}
\end{align}
that characterise the relative position of the inflection points from $L^{j_\bxy} \Z^2$.
Then by Corollary~\ref{cor:straightlinecancellationV2} and Lemma~\ref{lemma:bwcharacterisation}, we obtain for $\Upsilon = \Upsilon_\bxy^\sqcup = \bP^\dual$
\begin{align}
    \bL [ \bP] 
    &= -\frac{1}{2} \sum_{k=1,2} \big( \bA (i_k) + \bA (t_k) \big) +\sum_{j=0}^{N-1} \bar\ke_j (\Lambda) + \frac{1}{2} \bff^2 (\one_{\Upsilon}, R \one_{\Upsilon})  \nnb
    &\qquad -\frac{1}{2} \bff^2 \big( \delta_\bx - \delta_\by , \Gamma_{>0}  (\delta_\bx - \delta_\by )\big) + O(1/j_\bxy) .
    \label{eq:sqcupshapedcurveMGFIntermediate}
\end{align}

\subsection{Self interaction energy of $\Upsilon^\sqcup_\bxy$}

We first consider the term  $(\one_{\Upsilon^\sqcup_{\bxy}}, R \one_{\Upsilon^\sqcup_{\bxy}})$.

\begin{lemma} \label{lemma:sqcupcurveselfinteraction}
If $\Upsilon = \Upsilon^\sqcup_{\bxy}$, then
\begin{align}
    (\one_\Upsilon , R \one_\Upsilon) =  R^{(1)} |\Upsilon| + 4 V^{\lseg} + 4 V^{\square} + O(L^{-\eta j_\bxy})
\end{align}
\end{lemma}
\begin{proof}
The proof is the same as those of Lemma~\ref{eq:Vprimelsegdefi} and \ref{lemma:rectangleselfinteractionE}. $V^{\lseg}$ arises from the two tips of the path and $V^{\square}$ arises from the two inflection points.
The interaction between perpendicular line segments of $\Upsilon$ vanishes because $R$ vanishes between perpendicular edges, and the interaction between parallel line segments ($\Upsilon_\bxy^{\sqcup, 1}$ and $\Upsilon_\bxy^{\sqcup, 3}$) is exponentially small due to the exponential decay of $R$.
\end{proof}

\subsection{Corner energies of $\Upsilon^\sqcup_\bxy$}

To control the local line energy,
we consider the corner blocks at scale $j$. Recall $\cB_j^{1}$ from \eqref{eq:cBjlsegdefi}.
Let $r_{3\pi/2} : (x,y) \mapsto (y,-x)$ be the counter-clockwise rotation by angle $3\pi/2$. Then we define
\begin{align}
    \cB_{j,\bxy}^{\sqcup , 1} &= r_{3\pi/2} \big( \cB_j^{1} (\upsilon) \big) , \\
    \cB_{j,\bxy}^{\sqcup, 2} &= \cB_j^{1} (\varpi) - 100L^{j_\bxy} \bbe_2 .
\end{align}
They are the corner blocks in the neighbourhood of $t_1$ and $i_1$, respectively. $\cB_j^{1}$ needs to be rotated because $\bP^{\sqcup,1}_{\bxy} = (\Upsilon^{\sqcup,1}_{\bxy} )^\dual$ is headed downward instead of rightward.
We define $\cB_{j,\bxy}^{\sqcup, 3}$ and $\cB_{j,\bxy}^{\sqcup , 4}$ similarly as the corner blocks in the neighbourhood of $i_2$ and $t_2$, respectively.

\begin{lemma} \label{lemma:sqcupcurveslocallineenergies}
We have
\begin{align}
\label{eq:lemmasqcupcurveslocallineenergies1}
    \sum_{j<N} \bar\ke_j (\Lambda) =
    L^{j_\bxy} \big( 200 + n_2 + n_1  \big) \ccE_{<\infty}^{\rightarrow}
    + \ccE^{t}_{<\infty} (\upsilon)
    + \ccE^{t}_{<\infty} (\upsilon')
    + \ccE^{c}_{<\infty} (\varpi)
    + \ccE^{c}_{<\infty} (\varpi') + O(L^{-\eta j_\bxy}), 
\end{align}
see \eqref{eq:kerarrowDefi}, Lemma~\ref{lemma:locallinenEnergylseg}  and Lemma~\ref{lemma:kecDefi}  for the definition of all constants on the right side. 
\end{lemma}
\begin{proof}
By Definition~\ref{defi:locallineE}, we have
$\sum_{j\ge j_\bxy} |\bar\ke_j (\Lambda) | \lesssim L^{-\eta j_\bxy}$, so it is sufficient to consider $j < j_\bxy$.
Using the locality of the local line energies again, we can divide the local line energies into straight line components and corner components. This gives
\begin{align}
    \sum_{j<j_\bxy} \bar\ke_j (\Lambda)
    = \sum_{l=1}^{4} \sum_{j<j_\bxy} \sum_{b \in \cB_{j,\bxy}^{\sqcup, l}} \bar\ke_j (b) &+ \sum_{j<j_\bxy} \Big[ \big(100 L^{j_\bxy -j} - \hash_j (\upsilon_x) - \hash_j (\varpi_y) \big) \ccE_j^{\rightarrow} (\upsilon_y) \\
    &\;\; + \big(n_2 L^{j_\bxy -j} - \hash_j (\varpi_x) - \hash_j (\varpi'_x) \big) \ccE_j^{\rightarrow} (\varpi_y) \nnb
    &\;\; + \big( (n_1 + 100) L^{j_\bxy -j} - \hash_j (\varpi'_y) - \hash_j (\upsilon'_x) \big) \ccE_j^{\rightarrow} (\upsilon'_y) \Big] + O(L^{-\eta j_\bxy}) ,
    \nonumber
\end{align}
as we now explain.
The first term on the right-hand side is due to the tips and corners: $l=1$ to $t_1$, $l=2$ to $i_1$, $l=3$ to $i_2$ and $l=4$ to $t_2$.
The straight line components sum up to the rest of \eqref{eq:lemmasqcupcurveslocallineenergies1}: the straight line component of $\Upsilon_{\bxy}^{\sqcup,k}$ ($k=1,2,3$) gives the $k^{\rm th}$ line of \eqref{eq:lemmasqcupcurveslocallineenergies1}.

Now, recalling the definitions 
\eqref{eq:kerarrowDefi} for $\ccE_{<j_{\bxy}}^{\rightarrow} (y_1)$,
Lemma~\ref{lemma:locallinenEnergylseg} for $\ccE^{t}_{<j_{\bxy}} (\upsilon)$ and Lemma~\ref{lemma:kecDefi} for $\ccE^{t}_{<j_{\bxy}} (\varpi)$, we can reorder the terms to see that
\begin{align}
    &= L^{j_\bxy} \big( 100 \ccE_{<j_{\bxy}}^{\rightarrow} (\upsilon_y) + n_2 \ccE_{<j_{\bxy}}^{\rightarrow} (\varpi_y) + (n_1 + 100) \ccE_{<j_{\bxy}}^{\rightarrow} (\upsilon'_y) \big) \nnb
    &\qquad + \ccE^{t}_{<j_{\bxy}} (\upsilon)
    + \ccE^{t}_{<j_{\bxy}} (\upsilon')
    + \ccE^{c}_{<j_{\bxy}} (\varpi)
    + \ccE^{c}_{<j_{\bxy}} (\varpi') + O(L^{-\eta j_\bxy}) \nnb
    &= L^{j_\bxy} \big( 200 + n_2 + n_1  \big) \ccE_{<\infty}^{\rightarrow}
    + \ccE^{t}_{<\infty} (\upsilon)
    + \ccE^{t}_{<\infty} (\upsilon')
    + \ccE^{c}_{<\infty} (\varpi)
    + \ccE^{c}_{<\infty} (\varpi') + O(L^{-\eta j_\bxy}) 
\end{align}
where in the second line, we used Lemma~\ref{lemma:barkeinfty}, \ref{lemma:locallinenEnergylseg} and \ref{lemma:kecDefi} to replace $<N$ by $<\infty$. The second line holds because $\ccE_{<\infty}^{\rightarrow} (\upsilon_y)$ was independent of $\upsilon$ due to \eqref{eq:barkeIndepOfUps}.
\end{proof}

We are now ready to prove the main theorem of Part II.

\begin{proof}[Proof of Theorem~\ref{thm:standard}]

By \eqref{eq:sqcupshapedcurveMGFIntermediate}, Lemma~\ref{lemma:sqcupcurveselfinteraction} and \ref{lemma:sqcupcurveslocallineenergies},
\begin{align}
    \bL [\bP]
    & = - \frac{1}{2} \sum_{k=1,2} \big( \bA (i_k) + \bA (t_k) \big)  + \frac{1}{2} R^{(1)} \bff^2 |\Upsilon| \\
    &\qquad + 2 (V^{\lseg} + V^{\square} ) \bff^2
    -\frac{\bff^2}{2} \big( \delta_\bx - \delta_\by  , \Gamma_{>0} (\delta_\bx - \delta_\by ) \big)
    \nnb
    & \qquad + L^{j_\bxy} \big( 200 + n_2 + n_1  \big) \ccE_{<\infty}^{\rightarrow}
    + \ccE^{t}_{<\infty} (\upsilon)
    + \ccE^{t}_{<\infty} (\upsilon')
    + \ccE^{c}_{<\infty} (\varpi)
    + \ccE^{c}_{<\infty} (\varpi) + O(L^{-\eta j_\bxy})  . \nonumber
\end{align}
By \eqref{eq:sqcupupsilonDefi}--\eqref{eq:sqcupvarpiDefi},
\begin{align}
    |\Upsilon| = 200 L^{j_\bxy} + \bx_2 + \by_1 + \by_2 = (200+n_1+n_2) L^{j_\bxy} - (\upsilon_x + \upsilon'_x + \upsilon'_y +2) 
\end{align} 
(such that $\upsilon'_y = \varpi'_x$). Recalling \eqref{eq:alsegDefi} for $a^{\lseg}$ and \eqref{eq:asquareDefi} for $a^{\square}$,
\begin{align}
    &=
    L^{j_\bxy} (200 + n_2 + n_1) \Big[ \ccE^{\rightarrow}_{<\infty} + \frac{1}{2} R^{(1)}  \bff^2 \Big] 
    \nnb &\quad 
    + a^{\lseg} (\upsilon) + a^{\lseg} (\upsilon') + a^{\square} (\varpi) + a^{\square} (\varpi') -\frac{\bff^2}{2} \big( \delta_\bx - \delta_\by  , \Gamma_{>0} (\delta_\bx - \delta_\by ) \big)
    + O(1/j_\bxy) .
\end{align}
The first line is cancelled due to \eqref{eq:barkeIndepOfUps} and $a^{\square} = 0$ due to Lemma~\ref{lemma:asquarevanishes}, so
\begin{align}
    = a^{\lseg} (\upsilon) + a^{\lseg} (\upsilon') 
    -\frac{\bff^2}{2} \big( \delta_\bx - \delta_\by  , \Gamma_{>0} (\delta_\bx - \delta_\by ) \big)
    + O(1/j_\bxy)
\end{align}
and due to 
Corollary~\ref{cor:straightlinesegConcl},
\begin{align}
    &= 2\bar{a}_0 + \bar{a}_1 (\upsilon_x + \upsilon'_x ) -\frac{\bff^2}{2} \big( \delta_\bx - \delta_\by  , \Gamma_{>0} (\delta_\bx - \delta_\by ) \big)
    + O(1/j_\bxy)  \nnb
    &= 2\bar{a}_0 + \bar{a}_1 (n_2 L^{j_\bxy} - 1 - \bx_2 - \by_2) -\frac{\bff^2}{2} \big( \delta_\bx - \delta_\by  , \Gamma_{>0} (\delta_\bx - \delta_\by ) \big)
    + O(1/j_\bxy) .
\end{align}
However, since the left-hand side $\bL [\bP]$ should be invariant under the translation $(\bx,\by) \mapsto (\bx+\bbe_2, \by+\bbe_2)$, or equivalently $(\bx_2, \by_1, \by_2) \mapsto (\bx_2 +1, \by_1, \by_2+1)$, we actually should have $\bar{a}_1 = 0$. Therefore,
\begin{align}
    \bL = 2\bar{a}_0 - \frac{1}{2} \bff^2 (\delta_{\bx} - \delta_{\by} , \Gamma_{>0} (\delta_{\bx} - \delta_{\by})) + O(1/j_\bxy) .
\end{align}
Due to Lemma~\ref{lemma:Csasymptotic}, we conclude
\begin{align}
    e^{\bL} = e^{C_1 + 2\bar{a}_0 } |\bx-\by|^{- \bff^2 / (2\pi(1+s)) } \big(1 + O(1/\log|\bx - \by| ) \big)
    .
\end{align}
The estimate on $\bff$ was stated in Lemma~\ref{lemma:bffinftydefi}.
\end{proof}

\appendix

\section{Remainder estimate of the RG map} 
\label{app:remainderBnd}

We prove Lemma~\ref{lemma:RPMboundObs} here, with a reference to a theory of RG map estimate from \cite{dgauss1}. The proof is deferred to here because it is out of context, yet involves some serious estimates.

In the following definition, let $(\mathbb{Y}, |\cdot|)$ be any closed subset of a normed vector space.

\begin{definition}
Given $\delta, \xi >0$, we define $\cX (\mathbb{Y})$ as the collection of functions
\begin{align}
    \mathbb{Y} \ni x \mapsto (\delta u_{+} , V, V_{+}^\dagger , K, \bar{K}, \bar{J}, J)
\end{align}
(where each coordinate is a polymer function with domains the same as those defined in Section~\ref{sec:obsRGmap}) that satisfies the following bounds for any $B \in \cB_+$, $Z \in \cP_+$, $\varphi \in \R^\Lambda$ and some $C(\delta, L)$, $C(L)$, $C(A,L) >0$.
\begin{enumerate}
\item When $\bar{V}_{+} = V_{+}^\dagger + \delta u_{+}$, $k=0,1,2$ and $\kV \in \{ V, \bar{V}_{+} \}$,
\begin{align}
    &|\delta u_+ (B)|
        \le C(L) |x| 
    \label{eq:deficX11}
    \\
    & \big| e^{\delta u_+ (B)} - \sum_{m=0}^k \frac{1}{m!} (\delta u_+ (B,\varphi) )^m \big| 
        \le C(L) |x|^{k+1} 
    \label{eq:deficX12} \\
    &\norm{\kV (B,\varphi)}_{h,B,T (\varphi)}
        \le C(\delta, L) (1 + \delta c_w \kappa w^\dagger (B,\varphi) ) |x| 
    \label{eq:deficX13}
    \\
    & \big\| e^{\kV(B,\varphi)} - \sum_{m=0}^k \frac{1}{m!} (\cV (B,\varphi) )^m \big\|_{h,B,T (\varphi)} 
        \le C(\delta , L) \big( G^{\str,\dagger} (B,\varphi) \big)^{\delta} |x|^{k+1} .
    \label{eq:deficX14}    
\end{align}
\item When $\kV' \in \{ V, \bar{V}_+ , \delta u_+  \}$, 
\begin{align}
    & \norm{\rD e^{\kV' (B,\varphi)}}_{h, B, T (\varphi)} \vee \norm{\rD^2 e^{\kV' (B,\varphi)}}_{h, B, T (\varphi)}
        \le C(L) G^{\str} (B,\varphi) 
    \label{eq:deficX21}    
    \\
    & \norm{\rD J (B,Z,\varphi)}_{h, B, T (\varphi)}
        \le C(L) A^{-1} G^{\str} (B,\varphi)
    \label{eq:deficX22}
    \\
    & \norm{\rD \bar{K} (Z,\varphi) }_{h, Z, T(\varphi)}
        \le C(A,L) A^{-(1+\xi)|Z|_{j+1}} G^\dagger (Z,\varphi)
    \label{eq:deficX23}
    \\
    & \norm{\rD \bar{J} (Z,\varphi)}_{h, Z, T(\varphi)}
        \le C(A,L) A^{-(1+\xi)|Z|_{j+1}} G^{\str} (Z,\varphi) ,
    \label{eq:deficX24}
\end{align}
where $\rD$ is the Fr\'echet derivative in $x$.
\end{enumerate}
\end{definition}

This definition is a version of \cite[Definition~7.10]{dgauss1}, with adaptations as appropriate.

\begin{lemma} \label{lemma:RPMbndApp}
Assume Assumption~\ref{asmp:RGasmpBulk} and consider $\RPM^{J} := \PM^J - \LPM^{J}$. 
If $(\delta u_+, V, V_+^\dagger, K, \bar{K},\bar{J}, J) \in \cX (\mathbb{Y})$, there exists $\epsilon_{nl} >0$ such that 
\begin{align}
    \norm{\RPM^{J}}_{h,A,T_+ (G_+^{k})} \le C_2 (A,L) |x|^2
    \label{eq:lemmaRPMbndApp1}
\end{align}
for some $C_2 (A,L) >0$, whenever $|x| < \epsilon_{nl}$.
\end{lemma}
\begin{proof}
This can be deduced from \cite[Lemma~7.12]{dgauss1}.
In this reference, $(\cE K_j , J_j)$ are used in place of $(\bar{J} , J)$ and $\kM^{(k)}_{j+1}$ are defined in \cite[(7.52)--(7.55)]{dgauss1} so that $\sum_{k=1}^4 \kM^{(k)}_{j+1} = \RPM^{J}$.
Also, due to Lemma~\ref{lemma:Gjobsproperties}, the pair $(G, G_+)$ can be replaced by $(G^\dagger , G_{+}^{k})$.
Then the conclusion of \cite[Lemma~7.12]{dgauss1} $\norm{\rD \kM_{j+1}^{(k)}}_{h,A,T(G_+)} \le C_2 (A,L) |x|$, can be restated as \eqref{eq:lemmaRPMbndApp1}.
\end{proof}

We still have to verify that the RG coordinates are in $\cX$. In the following, we consider $\mathbb{X}$ that is the space of $(V,K^\bulk, K^\dagger) \in \cV \times \underline{\cN}^\dagger$ equipped with norm
\begin{align}
    | (V,K^\bulk, K^\dagger) | = \norm{V}_{\cV} + \norm{K^\bulk}_{\bulk} + \norm{K^\dagger}_{\dagger} .
\end{align}

\begin{lemma} \label{lemma:cXverify}
Assume Assumption~\ref{asmp:RGasmpBulk} and take any $\delta >0$. Then there exist $\xi, C(L), C(\delta, L), C(A,L)$, $\epsilon (\delta, L)$ such that
\begin{align}
    (\delta u_+^\dagger , V , V_+^\dagger, K^\dagger , \bar{K}^\dagger , \bar{J}^\obs , \Rae^{\dagger} ) \in \cX (\mathbb{Y})
\end{align}
when $\mathbb{Y} = \{ x \in \mathbb{X} : |x| \le \epsilon (\delta, L) \}$.
\end{lemma}

This is indeed enough for our goal.

\begin{proof}[Proof of Lemma~\ref{lemma:RPMboundObs}]

The conclusion is direct from Lemma~\ref{lemma:RPMbndApp} and \ref{lemma:cXverify}.
\end{proof}

\subsection{Proof of Lemma~\ref{lemma:cXverify}(i)}

By \eqref{eq:VKjbulkdecay} and Lemma~\ref{lemma:VjDjbnd}, for $B\in \cB_{j+1}$,
\begin{align}
    & |\delta u_+^\bulk (B)| \lesssim_L |x|, \\ 
    & \norm{V_+ (B,\varphi)}_{h,B,T_+ (\varphi)} \lesssim_L |x| (h^2 + w (B,\varphi) \lesssim_L \frac{1}{\delta} |x| (h^2 + c_w \kappa w (B,\varphi) )
\end{align}
while by Lemma~\ref{lemma:ptobsRG} and Lemma~\ref{lemma:VjDjbnd},
\begin{align}
    & |\delta u_+^\dagger (B) - \delta u_+^\bulk| \lesssim_L |x| , \\
    & \norm{\delta D_+ (B,\varphi)}_{h,B,T_+ (\varphi)} \lesssim_L |x| (h + w (B,\varphi) ) \lesssim_L \frac{1}{\delta} |x| (h + c_w \kappa w (B,\varphi) ) .
\end{align}
They prove \eqref{eq:deficX13} and \eqref{eq:deficX11}. The other bounds \eqref{eq:deficX12} and \eqref{eq:deficX14} follow from these, because
\begin{align}
    \big\| e^{F} - \sum_{m=0}^{k} \frac{1}{m!} F^m \big\| \le \sum_{m=k+1} \frac{1}{m!} \norm{F}^m \le \norm{F}^{k+1} e^{\norm{F}} .
\end{align}
for any submultiplicative seminorm $\norm{\cdot}$.

\subsection{Proof of Lemma~\ref{lemma:cXverify}(ii)}

First, consider \eqref{eq:deficX21}.
Since $\rD e^{\kV'} = (\rD \kV' ) e^{\kV'}$, $\rD^2 e^{\kV'} = (\rD^2 \kV') e^{\kV'} + (\rD \kV')^2 e^{\kV'}$ and the bound on $e^{\kV'}$ is already proved in \eqref{eq:deficX11}--\eqref{eq:deficX14}, it will be enough to bound $\rD \kV'$ and $\rD^2 \kV'$.
By Definition~\ref{defi:Vplusdefi} and \ref{defi:obsVplusdefi}, each $\kV' \in \{ \delta u_+^\dagger, V , \bar{V}_+ \}$ is a linear function of $x \in \mathbb{X}$, so $\rD^2 \kV' = 0$, while the bound on $\rD \kV'$ can simply be deduced from \eqref{eq:deficX11} and \eqref{eq:deficX13}: by setting $\delta = 1$,
\begin{align}
    \norm{\rD \kV' (B)}_{h,B,T(\varphi)} \lesssim_L (1 + c_w \kappa w^\dagger (B,\varphi)) .
\end{align}

We move on to \eqref{eq:deficX22} and \eqref{eq:deficX24}.
Recalling the definition \eqref{eq:obskEdefi} on $\Rae^{\dagger}$ and using Lemma~\ref{lemma:LocBndd} to bound the localisation, we see
\begin{align}
    \norm{\Rae^{\dagger} (B,Z,\varphi)}_{h,B,T (\varphi)}
    \lesssim (\log L) A^{-|Z|_{j+1}} \big( \norm{K^\bulk}_{\bulk} + \norm{K^\dagger}_{\dagger} \big) G^\str (B,\varphi) ,
\end{align}
and since $\Rae^{\dagger}$ is a linear function of $x \in \mathbb{X}$, \eqref{eq:deficX22} follows. Then, since $\bar{J}^\obs (Z) = \sum_{B \in \cB_+ (Z)} \bar{J}^\obs (B,Z)$,
\begin{align}
    \norm{\rD \bar{\Rae^{\dagger}} (Z,\varphi)}_{h,Z,T (\varphi)}
    \lesssim_L A^{-|Z|_{j+1}} G^\str (Z,\varphi) 
    \lesssim_L A^{4\xi} A^{-(1+\xi) |Z|_{j+1}} G^\str (Z,\varphi) ,
\end{align}
where the final inequality follows because $\bar{J}^\obs (Z)$ does not vanish only for $|Z|_{j+1} \le 2^d = 4$. Thus, \eqref{eq:deficX24} is also proved.

Finally, we study \eqref{eq:deficX23}, the estimate on
\begin{align}
    & \rD \bar{K}^\dagger (Z) = \rD \sum_{Y \in \cP}^{\bar{Y} = Z} e^{V (Z \backslash Y)} K^\dagger (Y) = A (Z) + B (Z) , \qquad \text{where} \\
    & A(Z) = \sum_{Y \in \cP}^{\bar{Y} = Z} e^{V (Z \backslash Y)} (\rD V(Z\backslash Y))  K^\dagger (Y)  ,
    \qquad
    B(Z) = \sum_{Y \in \cP}^{\bar{Y} = Z} e^{V (Z \backslash Y)} \rD K^\dagger (Y) .
\end{align}
By \eqref{eq:deficX13}, we already have a bound on $V$ and $\rD V$, so for $|x| \le \epsilon_{nl}$,
\begin{align}
    \norm{A (Z,\varphi)}_{h,Z,T(\varphi)}
    &\lesssim_L \sum_{Y \in \cP}^{\bar{Y}=Z} G^{\str,\dagger} (Z\backslash Y) (\Adag)^{-|Y|_{j}} G^{\dagger} (Y) ( \epsilon_{nl} )^{|\Comp (Y)|} 
\end{align}
For $B$, we use that
\begin{align}
    \rD K^\dagger (Y) = \rD \prod_{Y' \in \Comp (Y)} K^\dagger (Y') = \sum_{Y' \in \Comp (Y)} (\rD K^\dagger (Y')) K^\dagger (Y\backslash Y') ,
\end{align}
so
\begin{align}
    \norm{B (Z,\varphi)}_{h,Z,T(\varphi)}
    &\lesssim_L \sum_{Y \in \cP}^{\bar{Y}=Z} G^{\str,\dagger} (Z\backslash Y) (\Adag)^{-|Y|_{j}} G^{\dagger} (Y) ( \epsilon_{nl} )^{|\Comp (Y)|-1} 
\end{align}
and putting together the bounds on $A$ and $B$,
\begin{align}
    \norm{(A+B) (Z,\varphi)}_{h,Z,T(\varphi)}
    &\lesssim_L \sum_{Y \in \cP}^{\bar{Y}=Z} G^{\str,\dagger} (Z\backslash Y) (\Adag)^{-|Y|_{j}} G^{\dagger} (Y) ( \epsilon_{nl} )^{|\Comp (Y)|-1} 
    \nnb
    &\lesssim_L G^\dagger (Z) \sum_{Y \in \cP}^{\bar{Y}=Z}  (A/4)^{-|Y|_{j}}  ( \epsilon_{nl} )^{|\Comp (Y)|-1} \nnb
    &\lesssim_{A,L} A^{-(1+\xi) |Z|_{j+1}} G^\dagger (Z)
\end{align}
where the second and the third inequalities are due to Lemma~\ref{lemma:Gjobsproperties} and \ref{lemma:reblockingA}, respectively. 

We will verify the combinatorial bound, Lemma~\ref{lemma:reblockingA}, in the rest of the appendix.

\subsection{Reblocking estimates}

\begin{lemma} \label{lemma:largesets}
\cite[Lemma~6.14--6.15]{MR2523458}
There exists $\xi' >0$ such that the following holds when $L\ge 2^d +1 = 5$. For every $X \in \cP$,
\begin{align}
    (1+\xi') |\bar{X}|_{j+1} \le |X|_j + 8 (1+\xi') |\Comp (X) |
\end{align}
and if $X \in \Conn\backslash \cS$,
\begin{align}
    (1+\xi') |\bar{X}|_{j+1} \le |X|_j .
\end{align}
\end{lemma}

\begin{lemma} \label{lemma:reblockingA}
Let $Z \in \cP_+$ and $A>0$ be sufficiently large (depending on $L$). Then there exists $\xi>0$ such that, for sufficiently small $\epsilon >0$,
\begin{align}
    \sum_{Y \in \cP}^{\bar{Y}=Z} (A/4)^{-|Y|_j} \epsilon^{|\Comp (Y)|-1} \le C(A,L) A^{-(1+\xi) |Z|_{j+1}} 
\end{align}
for some $C(A,L) >0$.
\end{lemma}
\begin{proof}
When $Y \in \cS$, then $|\Comp(Y)|=1$ and $|Z|_{j+1} \le |Y|_j \le 4$, so
\begin{align}
    \sum_{Y \in \cS}^{\bar{Y}=Z} (A/4)^{-|Y|_j} \epsilon^{|\Comp(Y)|-1}  \lesssim L^2 (A/4)^{-|Z|_{j+1}}
    \lesssim L^2 A^{4\xi} A^{-(1+\xi)|Z|_{j+1}} .
\end{align}

When $Y \in \Conn \backslash \cS$, then $|\Comp(Y)|=1$ and by the second part of Lemma~\ref{lemma:largesets}, 
\begin{align}
    \sum_{Y \in \Conn \backslash \cS}^{\bar{Y}=Z} (A/4)^{-|Y|_j} \epsilon^{|\Comp(Y)|-1}
    &\le (A/4)^{-(1+\xi') |Z|_{j+1}} \sum_{Y \in \Conn \backslash \cS}^{\bar{Y}=Z} 1
    \nnb
    &\le (A/4)^{-(1+\xi') |Z|_{j+1}} 2^{L^2 |Z|_{j+1}} ,
\end{align}
and we have the desired bound for $\xi < \xi'$ and sufficiently large $A$.

When $Y \in \cP \backslash \Conn$, then $|\Comp (Y)| \ge 2$, so we can bound $\epsilon^{|\Comp (Y)|-1} \le \epsilon^{|\Comp (Y)|/2}$.
Thus, by Lemma~\ref{lemma:largesets}, if $\epsilon^{1/2} \le A^{-8 (1+\xi')}$,
\begin{align}
    & \sum_{Y \in \cP \backslash \Conn}^{\bar{Y}=Z} (A/4)^{-|Y|_j} \epsilon^{|\Comp (Y)|-1}
    \le \sum_{Y \in \cP}^{\bar{Y}=Z} (A/4)^{-|Y|_j - 8 (1+\xi') |\Comp (Y)| } \nnb
    &\qquad\qquad\qquad  
    \le A^{-(1+\xi') |Z|_{j+1}} \sum_{Y \in \cP}^{\bar{Y}=Z} 1
    \le (A/4)^{-(1+\xi') |Z|_{j+1}} 2^{L^2 |Z|_{j+1}} ,
\end{align}
and we have the desired bound for $\xi < \xi'$ and sufficiently large $A$.
\end{proof}

\section*{Acknowledgements}

The authors thank Roland Bauerschmidt for a useful discussion at the development stage of this paper, when JP was his student at the University of Cambridge.
The authors also thank Christophe Garban and Thomas Spencer for useful discussions concerning the scope of this paper.
JP gratefully acknowledges the support and hospitality of the Department of Mathematics at KAIST during the completion of a part of this work. 
JP is supported by Basic Science Research Program through the National Research Foundation of Korea funded by the Ministry of Science and ICT (RS2025-00518980). WW is supported in part by the Ministry of Science of Technology grant (2021YFA1002700 MOST), NSFC grant (NYTP 20220903), and NYU Shanghai Boost Funds. 


\begin{thebibliography}{10}

\bibitem{aizenman2021depinning}
M.~Aizenman, M.~Harel, R.~Peled, and J.~Shapiro.
\newblock Depinning in the integer-valued {G}aussian {F}ield and the {BKT}
  phase of the {2D} {V}illain model.
\newblock {\em arXiv preprint arXiv:2110.09498}, 2021.

\bibitem{AKM}
S.~Armstrong, T.~Kuusi, and J.-C. Mourrat.
\newblock {\em Quantitative stochastic homogenization and large-scale
  regularity}, volume 352 of {\em Grundlehren der mathematischen Wissenschaften
  [Fundamental Principles of Mathematical Sciences]}.
\newblock Springer, Cham, 2019.

\bibitem{MR4303014}
R.~Bauerschmidt and T.~Bodineau.
\newblock Log-{S}obolev inequality for the continuum sine-{G}ordon model.
\newblock {\em Comm. Pure Appl. Math.}, 74(10):2064--2113, 2021.

\bibitem{dgauss1}
R.~Bauerschmidt, J.~Park, and P.-F. Rodriguez.
\newblock The {D}iscrete {G}aussian model, {I}. {R}enormalisation group flow at
  high temperature.
\newblock {\em The Annals of Probability}, 52(4):1253--1359, 2024.

\bibitem{dgauss2}
R.~Bauerschmidt, J.~Park, and P.-F. Rodriguez.
\newblock The {D}iscrete {G}aussian model, {II}. {I}nfinite-volume scaling
  limit at high temperature.
\newblock {\em The Annals of Probability}, 52(4):1360--1398, 2024.

\bibitem{MR0314399}
V.L. Berezinski\u{\i}.
\newblock Destruction of long-range order in one-dimensional and
  two-dimensional systems having a continuous symmetry group. {I}. {C}lassical
  systems.
\newblock {\em \v{Z}. \`Eksper. Teoret. Fiz.}, 59:907--920, 1970.

\bibitem{boursier2025multipole}
J.~Boursier and S.~Serfaty.
\newblock Multipole and {B}erezinskii-{K}osterlitz-{T}houless transitions in
  the two-component plasma.
\newblock {\em arXiv preprint arXiv:2509.09449}, 2025.

\bibitem{BFL1977}
J.~Bricmont, J.~R. Fontaine, and L.~J. Landau.
\newblock On the uniqueness of the equilibrium state for plane rotators.
\newblock {\em Comm. Math. Phys.}, 52(1):281--296, 1977.

\bibitem{BFLLS}
Jean Bricmont, Jean-Raymond Fontaine, Joel~L. Lebowitz, Elliott~H. Lieb, and
  Thomas Spencer.
\newblock Lattice systems with a continuous symmetry. {III}. {L}ow temperature
  asymptotic expansion for the plane rotator model.
\newblock {\em Comm. Math. Phys.}, 78(4):545--566, 1980/81.

\bibitem{MR1048698}
D.~Brydges and H.-T. Yau.
\newblock Grad {$\phi$} perturbations of massless {G}aussian fields.
\newblock {\em Commun. Math. Phys.}, 129(2):351--392, 1990.

\bibitem{MR2523458}
D.C. Brydges.
\newblock Lectures on the renormalisation group.
\newblock In {\em Statistical mechanics}, volume~16 of {\em IAS/Park City Math.
  Ser.}, pages 7--93. Amer. Math. Soc., 2009.

\bibitem{CW}
S.~Chui and J.~Weeks.
\newblock Phase transition in the two-dimensional {C}oulomb gas, and the
  interfacial roughening transition.
\newblock {\em Physical Review B}, 14(11):4978, 1976.

\bibitem{dario2024massless}
P.~Dario and W.~Wu.
\newblock Massless phases for the {V}illain model in $d\ge3$.
\newblock {\em Ast{\'e}risque}, (447):1--+, 2024.

\bibitem{dario2025parallel}
P.~Dario and W.~Wu.
\newblock Parallel spin wave for the {V}illain model.
\newblock {\em arXiv preprint arXiv:2507.04098}, 2025.

\bibitem{dN}
M.~Den~Nijs.
\newblock Derivation of extended scaling relations between critical exponents
  in two-dimensional models from the one-dimensional {L}uttinger model.
\newblock {\em Physical Review B}, 23(11):6111, 1981.

\bibitem{MR1777310}
J.~Dimock and T.R. Hurd.
\newblock Sine-{G}ordon revisited.
\newblock {\em Ann. Henri Poincar\'e}, 1(3):499--541, 2000.

\bibitem{dubedat2011topics}
J.~Dub{\'e}dat.
\newblock {Topics on abelian spin models and related problems}.
\newblock {\em Probability Surveys}, 8(none):374 -- 402, 2011.

\bibitem{duminil2026gaussian}
H.~Duminil-Copin, K.~K. Kozlowski, P.~Lammers, and I.~Manolescu.
\newblock Gaussian free field convergence of the six-vertex model with
  {$-1\le\Delta\le-1/2$}.
\newblock {\em arXiv preprint arXiv:2603.06268}, 2026.

\bibitem{MR2917175}
P.~Falco.
\newblock Kosterlitz-{T}houless transition line for the two dimensional
  {C}oulomb gas.
\newblock {\em Commun. Math. Phys.}, 312(2):559--609, 2012.

\bibitem{1311.2237}
P.~Falco.
\newblock {C}ritical exponents of the two dimensional {C}oulomb gas at the
  {B}erezinskii-{K}osterlitz-{T}houless transition, 2013.
\newblock Preprint, arXiv:1311.2237.

\bibitem{MR0421531}
J.~Fr{\"o}hlich, B.~Simon, and T.~Spencer.
\newblock Infrared bounds, phase transitions and continuous symmetry breaking.
\newblock {\em Commun. Math. Phys.}, 50(1):79--95, 1976.

\bibitem{MR634447}
J.~Fr{\"o}hlich and T.~Spencer.
\newblock The {K}osterlitz-{T}houless transition in two-dimensional abelian
  spin systems and the {C}oulomb gas.
\newblock {\em Commun. Math. Phys.}, 81(4):527--602, 1981.

\bibitem{garban2023quantitative}
C.~Garban and A.~Sep{\'u}lveda.
\newblock Quantitative bounds on vortex fluctuations in 2d {C}oulomb gas and
  maximum of the integer-valued {G}aussian free field.
\newblock {\em Proceedings of the London Mathematical Society},
  127(3):653--708, 2023.

\bibitem{MR0269252}
J.~Ginibre.
\newblock General formulation of {G}riffiths' inequalities.
\newblock {\em Commun. Math. Phys.}, 16:310--328, 1970.

\bibitem{giuliani2005anomalous}
A.~Giuliani and V.~Mastropietro.
\newblock Anomalous universality in the anisotropic {A}shkin--{T}eller model.
\newblock {\em Communications in Mathematical Physics}, 256(3):681--735, 2005.

\bibitem{giuliani2020non}
A.~Giuliani, V.~Mastropietro, and F.~L. Toninelli.
\newblock Non-integrable dimers: Universal fluctuations of tilted height
  profiles.
\newblock {\em Communications in Mathematical Physics}, 377(3):1883--1959,
  2020.

\bibitem{MR3637384}
A.~Giuliani, V.~Mastropietro, and F.L. Toninelli.
\newblock Haldane relation for interacting dimers.
\newblock {\em J. Stat. Mech. Theory Exp.}, (3):034002, 47, 2017.

\bibitem{JKKN}
J.~V. Jos{\'e}, L.~P Kadanoff, S.~Kirkpatrick, and D.~R. Nelson.
\newblock Renormalization, vortices, and symmetry-breaking perturbations in the
  two-dimensional planar model.
\newblock {\em Physical Review B}, 16(3):1217, 1977.

\bibitem{kos}
J.M. Kosterlitz.
\newblock The critical properties of the two-dimensional {X}{Y} model.
\newblock {\em Journal of Physics C: Solid State Physics}, 7(6):1046--1060,
  1974.

\bibitem{0022-3719-6-7-010}
J.M. Kosterlitz and D.J. Thouless.
\newblock Ordering, metastability and phase transitions in two-dimensional
  systems.
\newblock {\em Journal of Physics C: Solid State Physics}, 6(7):1181, 1973.

\bibitem{arXiv.2211.14365}
P.~Lammers.
\newblock A dichotomy theory for the height functions of the {BKT} transition.
\newblock {\em arXiv preprint arXiv:2211.14365}, 2022.

\bibitem{lammersdeloc}
P.~Lammers.
\newblock Height function delocalisation on cubic planar graphs.
\newblock {\em Probab. Theory Related Fields}, 182(1-2):531--550, 2022.

\bibitem{lammers2023bijecting}
P.~Lammers.
\newblock Bijecting the {BKT} transition.
\newblock {\em arXiv preprint arXiv:2301.06905}, 2023.

\bibitem{lammers2024delocalisation}
P.~Lammers and S.~Ott.
\newblock Delocalisation and absolute-value-{FKG} in the solid-on-solid model.
\newblock {\em Probability Theory and related fields}, 188(1):63--87, 2024.

\bibitem{MR0441179}
O.A. McBryan and T.~Spencer.
\newblock On the decay of correlations in {${\rm SO}(n)$}-symmetric
  ferromagnets.
\newblock {\em Commun. Math. Phys.}, 53(3):299--302, 1977.

\bibitem{PhysRevLett.17.1133}
N.D. Mermin and H.~Wagner.
\newblock Absence of ferromagnetism or antiferromagnetism in one- or
  two-dimensional isotropic {H}eisenberg models.
\newblock {\em Phys. Rev. Lett.}, 17:1133--1136, Nov 1966.

\bibitem{MMSP1978}
A.~Messager, S.~Miracle-Sole, and C.~Pfister.
\newblock Correlation inequalities and uniqueness of the equilibrium state for
  the plane rotator ferromagnetic model.
\newblock {\em Communications in Mathematical Physics}, 58(1):19--29, 1978.

\bibitem{Nie}
B.~Nienhuis.
\newblock Exact critical point and critical exponents of {${\rm O}(n)$}\ models
  in two dimensions.
\newblock {\em Phys. Rev. Lett.}, 49(15):1062--1065, 1982.

\bibitem{BKTHeightSpin2026}
OpenAI.
\newblock Bkt universality for height and planar spin fields.
\newblock September 2026.
\newblock Preprint.

\bibitem{dgauss3}
J.~Park.
\newblock Central limit theorem for multi-point functions of the 2d discrete
  {G}aussian model at high temperature.
\newblock {\em Communications in Mathematical Physics}, 406(10):234, 2025.

\bibitem{MR2251117}
S.~Sheffield.
\newblock Random surfaces.
\newblock {\em Ast\'erisque}, (304):vi+175, 2005.

\bibitem{MR1797142}
T.~Spencer.
\newblock A mathematical approach to universality in two dimensions.
\newblock {\em Phys. A}, 279(1-4):250--259, 2000.
\newblock Statistical mechanics: from rigorous results to applications.

\bibitem{van2023elementary}
D.~van Engelenburg and M.~Lis.
\newblock An elementary proof of phase transition in the planar {XY} model.
\newblock {\em Communications in Mathematical Physics}, 399(1):85--104, 2023.

\bibitem{vDEL}
D.~van Engelenburg and M.~Lis.
\newblock On the duality between height functions and continuous spin models.
\newblock {\em Probab. Math. Phys.}, 6(4):1291--1325, 2025.

\bibitem{villain:jpa-00208289}
J.~Villain.
\newblock {Theory of one- and two-dimensional magnets with an easy
  magnetization plane. II. The planar, classical, two-dimensional magnet}.
\newblock {\em {Journal de Physique}}, 36(6):581--590, 1975.

\end{thebibliography}


\end{document}